\documentclass[11pt,a4paper,reqno]{amsart}
\usepackage[T1]{fontenc}
\usepackage[utf8]{inputenc}
\usepackage{moreenum}

\usepackage{lmodern}
\usepackage{microtype}
\usepackage[export]{adjustbox} 

\usepackage{caption}
\usepackage[english]{babel}

\usepackage{amsthm}
\usepackage{thmtools}
\usepackage{amsmath}
\usepackage{amsfonts} 
\usepackage{amssymb}
\usepackage{mathrsfs}
\usepackage{enumitem}

\usepackage{esint}
\usepackage{xfrac}
\usepackage[hcentering]{geometry}

\usepackage[abbrev]{amsrefs}
\usepackage{thmtools}
\usepackage[dvipsnames]{xcolor}
\usepackage[unicode=true, pdflang={en-GB}, pdfusetitle,hypertexnames=false,allbordercolors=RoyalBlue,
pdfborderstyle={/W .5},
final]{hyperref}
\let\C\undefined
\usepackage[capitalize]{cleveref}

\numberwithin{equation}{section}

\newtheorem{proposition}{Proposition}[section]

\newtheorem{theorem}[proposition]{Theorem}
\newtheorem{lemma}[proposition]{Lemma}

\theoremstyle{definition}
\newtheorem{definition}[proposition]{Definition}

\newtheorem{remark}[proposition]{Remark}

\newcommand{\thprime}{\texorpdfstring{$\smash{'}$}{'}}
\newcommand{\thsecond}{\texorpdfstring{$\smash{''}$}{''}}

\theoremstyle{plain}
\newtheorem{manualtheoreminner}{Theorem}
\newenvironment{manualtheorem}[1]{%
  \renewcommand\themanualtheoreminner{#1}%
  \manualtheoreminner
}{\endmanualtheoreminner}
\crefname{manualtheoreminner}{Theorem}{Theorems}
\Crefname{manualtheoreminner}{Theorem}{Theorems}

\makeatletter
\DeclareRobustCommand\widecheck[1]{{\mathpalette\@widecheck{#1}}}
\def\@widecheck#1#2{%
    \setbox\z@\hbox{\m@th$#1#2$}%
    \setbox\tw@\hbox{\m@th$#1%
       \widehat{%
          \vrule\@width\z@\@height\ht\z@
          \vrule\@height\z@\@width\wd\z@}$}%
    \dp\tw@-\ht\z@
    \@tempdima\ht\z@ \advance\@tempdima2\ht\tw@ \divide\@tempdima\thr@@
    \setbox\tw@\hbox{%
       \raise\@tempdima\hbox{\scalebox{1}[-1]{\lower\@tempdima\box
\tw@}}}%
    {\ooalign{\box\tw@ \cr \box\z@}}}
\makeatother

\makeatletter
\let\save@mathaccent\mathaccent
\newcommand*\if@single[3]{%
  \setbox0\hbox{${\mathaccent"0362{#1}}^H$}%
  \setbox2\hbox{${\mathaccent"0362{\kern0pt#1}}^H$}%
  \ifdim\ht0=\ht2 #3\else #2\fi
  }
\newcommand*\rel@kern[1]{\kern#1\dimexpr\macc@kerna}
\newcommand*\widebar[1]{\@ifnextchar^{{\wide@bar{#1}{0}}}{\wide@bar{#1}{1}}}
\newcommand*\wide@bar[2]{\if@single{#1}{\wide@bar@{#1}{#2}{1}}{\wide@bar@{#1}{#2}{2}}}
\newcommand*\wide@bar@[3]{%
  \begingroup
  \def\mathaccent##1##2{%
    \let\mathaccent\save@mathaccent
    \if#32 \let\macc@nucleus\first@char \fi
    \setbox\z@\hbox{$\macc@style{\macc@nucleus}_{}$}%
    \setbox\tw@\hbox{$\macc@style{\macc@nucleus}{}_{}$}%
    \dimen@\wd\tw@
    \advance\dimen@-\wd\z@
    \divide\dimen@ 3
    \@tempdima\wd\tw@
    \advance\@tempdima-\scriptspace
    \divide\@tempdima 10
    \advance\dimen@-\@tempdima
    \ifdim\dimen@>\z@ \dimen@0pt\fi
    \rel@kern{0.6}\kern-\dimen@
    \if#31
      \overline{\rel@kern{-0.6}\kern\dimen@\macc@nucleus\rel@kern{0.4}\kern\dimen@}%
      \advance\dimen@0.4\dimexpr\macc@kerna
      \let\final@kern#2%
      \ifdim\dimen@<\z@ \let\final@kern1\fi
      \if\final@kern1 \kern-\dimen@\fi
    \else
      \overline{\rel@kern{-0.6}\kern\dimen@#1}%
    \fi
  }%
  \macc@depth\@ne
  \let\math@bgroup\@empty \let\math@egroup\macc@set@skewchar
  \mathsurround\z@ \frozen@everymath{\mathgroup\macc@group\relax}%
  \macc@set@skewchar\relax
  \let\mathaccentV\macc@nested@a
  \if#31
    \macc@nested@a\relax111{#1}%
  \else
    \def\gobble@till@marker##1\endmarker{}%
    \futurelet\first@char\gobble@till@marker#1\endmarker
    \ifcat\noexpand\first@char A\else
      \def\first@char{}%
    \fi
    \macc@nested@a\relax111{\first@char}%
  \fi
  \endgroup
}
\makeatother

\newcommand{\weakto}{\rightharpoonup}

\newcommand{\defeq}{\coloneqq}

\newcommand{\Nset}{\mathbb{N}}
\newcommand{\Zset}{\mathbb{Z}}
\newcommand{\Rset}{\mathbb{R}}

\newcommand{\Sset}{\mathbb{S}}
\newcommand{\Qset}{\mathbb{Q}}
\newcommand{\Bset}{\mathbb{B}}
\newcommand{\Tset}{\mathbb{T}}

\newcommand{\dif}{\,\mathrm{d}}
\usepackage{mathtools}

\newcommand{\compose}{\,\circ\,}
\newcommand{\manifold}[1]{\mathcal{#1}}
\newcommand{\lifting}[1]{\smash{\widetilde{#1}}}

\DeclarePairedDelimiter{\brk}{(}{)}
\DeclarePairedDelimiter{\abs}{\lvert}{\rvert}
\DeclarePairedDelimiter{\norm}{\lVert}{\rVert}
\DeclarePairedDelimiter{\seminorm}{\lvert}{\rvert}
\DeclarePairedDelimiter{\floor}{\lfloor}{\rfloor}
\DeclarePairedDelimiter{\ceil}{\lceil}{\rceil}
\DeclarePairedDelimiterX{\intvc}[2]{[}{]}{#1,#2}
\DeclarePairedDelimiterX{\intvl}[2]{(}{]}{#1,#2}
\DeclarePairedDelimiterX{\intvr}[2]{[}{)}{#1,#2}
\DeclarePairedDelimiterX{\intvo}[2]{(}{)}{#1,#2}
\DeclarePairedDelimiterX{\setcond}[2]{\{}{\}}{#1 \,\delimsize\vert\, #2}

\newcommand{\restr}[1]{\vert_{#1}}
\newcommand{\Deriv}{\mathrm{D}}
\newcommand{\meanosc}{\mathrm{MO}}
\DeclareMathOperator{\id}{id}

\newcommand\stSymbol[1][]{%
\nonscript\;#1\vert
\allowbreak
\nonscript\;
\mathopen{}}
\DeclarePairedDelimiterX\set[1]\{\}{%
\renewcommand\st{\stSymbol[\delimsize]}
#1
}
\providecommand{\st}{\stSymbol}

\DeclareMathOperator{\dist}{dist}
\DeclareMathOperator{\supp}{supp}

\DeclareMathOperator{\diam}{diam}
\DeclareMathOperator{\edge}{edge}
\DeclareMathOperator{\tr}{tr}

\usepackage{constants}

\renewcommand{\PrintDOI}[1]{%
  \href{http://dx.doi.org/#1}{doi:#1}%
}
\BibSpec{book}{%
    +{}  {\PrintPrimary}                {transition}
    +{,} { \textit}                     {title}
    +{.} { }                            {part}
    +{:} { \textit}                     {subtitle}
    +{,} { \PrintEdition}               {edition}
    +{}  { \PrintEditorsB}              {editor}
    +{,} { \PrintTranslatorsC}          {translator}
    +{,} { \PrintContributions}         {contribution}
    +{,} { }                            {series}
    +{,} { \voltext}                    {volume}
    +{,} { }                            {publisher}
    +{,} { }                            {organization}
    +{,} { }                            {address} 
    +{,} { \PrintDateB}                 {date}
    +{,} { }                            {status}
    +{,} { \PrintDOI}                   {doi}
    +{}  { \parenthesize}               {language}
    +{}  { \PrintTranslation}           {translation}
    +{;} { \PrintReprint}               {reprint}
    +{.} { }                            {note}
    +{.} {}                             {transition}
    +{}  {\SentenceSpace \PrintReviews} {review}
}

\title
{%
The extension of traces for Sobolev mappings between manifolds%
}

\author{Jean Van Schaftingen}

\newcommand{\sectref}[1]{\hyperref[#1]{\S \ref*{#1}}}

\begin{document}

\address{
UCLouvain, Institut de Recherche en Math\'ematique et Physique, Chemin du Cyclotron 2 bte L7.01.01, 1348 Louvain-la-Neuve, Belgium}

\email{Jean.VanSchaftingen@UCLouvain.be}

\begin{abstract}
The compact Riemannian manifolds $\mathcal{M}$ and $\mathcal{N}$
for which the trace operator from the first-order Sobolev space of mappings $\smash{\dot{W}}^{1, p} (\mathcal{M}, \mathcal{N})$ to
the fractional Sobolev-Slobodecki\u{\i} space $\smash{\smash{\dot{W}}^{1 - 1/p, p}} (\partial \mathcal{M}, \mathcal{N})$ is surjective when $1 < p < \dim \mathcal{M}$ are characterised.
The traces are extended using a new construction which can be carried out assuming the absence of the known topological and analytical obstructions.
When $p \ge \dim \mathcal{M}$, the same construction provides a Sobolev extension with linear estimates for maps that have a continuous extension, provided that the known analytical obstructions to such a control are not present.
\end{abstract}

\thanks{The author was supported by the Mandat d'Impulsion Scientifique F.4523.17, ``Topological singularities of Sobolev maps'' and by the Projet de Recherche T.0229.21 ``Singular Harmonic Maps and Asymptotics of Ginzburg--Landau Relaxations'' of the Fonds de la Recherche Scientifique--FNRS.
The author acknowledges the hospitality of the Mathematical Institute of the University of Oxford where this work was initiated}

\maketitle

\setcounter{tocdepth}{1}
\tableofcontents
\setcounter{tocdepth}{5}

\section{Introduction}

\subsection{Surjectivity of the trace}
According to Gagliardo's classical trace theory in linear Sobolev spaces \cite{Gagliardo_1957}, given an \(m\)-dimensional Riemannian manifold \(\smash{\manifold{M}}\) with boundary \(\partial \manifold{M}\) under standard assumptions, including the case \(\manifold{M} = \smash{\widebar{\Rset}}^m_+ \defeq \Rset^{m - 1} \times \intvr{0}{\infty}\) and \(\manifold{M}\) with compact and connected boundary \(\partial \manifold{M}\), for each \(p \in \intvo{1}{\infty}\) there is a well-defined surjective trace operator \(\tr_{\partial \manifold{M}} \colon \smash{\dot{W}}^{1, p} \brk{\manifold{M}, \Rset}\to \smash{\dot{W}}^{1-1/p, p} (\partial \manifold{M}, \Rset)\) which coincides with the restriction to the boundary \(\partial \manifold{M}\) on the subset of continuous functions \(\smash{\dot{W}}^{1, p} \brk{\manifold{M}, \Rset} \cap C \brk{\manifold{M}, \Rset}\), with the \emph{linear first-order homogeneous Sobolev space} being defined as
\begin{equation*} 
 \smash{\dot{W}}^{1, p} \brk{\manifold{M}, \Rset}
 \defeq \set[\Big]{U \colon \manifold{M} \to \Rset \st U \text{ is weakly differentiable and }
 \int_{\manifold{M}} \abs{\Deriv U}^p < \infty}
\end{equation*}
and the \emph{linear fractional homogeneous Sobolev space} \(\smash{\dot{W}}^{s, p} \brk{\manifold{M}', \Rset}\) being defined for \(s \in \intvo{0}{1}\) and an \((m-1)\)-dimensional Riemannian manifold \(\manifold{M}'\) as
\begin{multline*}
 \smash{\dot{W}}^{s, p} \brk{\manifold{M}', \Rset}
  \defeq \biggl\{ u \colon \manifold{M}' \to \Rset \st[\bigg] u \text{ is Borel-measurable} \\[-.5em] \text{and}\smashoperator[r]{\iint_{\manifold{M}' \times \manifold{M}'}} \frac{\abs{u (y) - u (z)}^p}{d\brk{y, z}^{m - 1 + sp}}\dif y \dif z < \infty\biggr\},
\end{multline*}
where \(d \colon \manifold{M}' \times \manifold{M}' \to \intvc{0}{\infty}\) is the geodesic distance on \(\manifold{M}'\) induced by its Riemannian metric, with the convention that \(d \brk{x, y}= \infty\) if \(x\) and \(y \in \manifold{M}’\) belong to distinct connected components, so that the integrand then vanishes.
We use homogeneous Sobolev notation without passing explicitly to the quotient by constant maps.
Moreover, there is a constant \(C \in \intvo{0}{\infty}\) such that
for every function \(u \in \smash{\dot{W}}^{1-1/p, p} (\partial \manifold{M}, \Rset)\) there exists a function \(U \in \smash{\dot{W}}^{1, p} (\manifold{M}, \Rset)\) depending linearly on \(u\) such that \(\tr_{\partial \manifold{M}} U = u\)  and 
\begin{equation}
\label{eq_liewi1urahBah7Douvux9roa}
  \int_{\manifold{M}} \abs{\Deriv U}^p
  \le C
  \smashoperator{\iint_{\partial \manifold{M} \times \partial \manifold{M}}} \frac{\abs{u (y) - u (z)}^p}{d\brk{y, z}^{p + m - 2}} + \abs{u \brk{x} - u \brk{y}}^p \dif y \dif z,
\end{equation}
that is, the trace operator \(\tr_{\partial \manifold{M}}\) has a  \emph{bounded linear right inverse}.

\medbreak 

Given a compact Riemannian manifold \(\manifold{N}\), which can be assumed
without loss of generality to be isometrically embedded into the Euclidean space \(\Rset^\nu\) for some \(\nu \in \Nset\), by Nash's isometric embedding theorem \cite{Nash_1956}, let us now consider the homogeneous \emph{Sobolev spaces of mappings}
\begin{equation*}
  \smash{\dot{W}}^{1, p} \brk{\manifold{M}, \manifold{N}}
  \defeq \set[\big]{U \in \smash{\dot{W}}^{1, p} (\manifold{M}, \Rset^\nu) \st U \in \manifold{N} \text{ almost everywhere in \(\manifold{M}\)}}
\end{equation*}
and
\begin{multline*}
  \smash{\dot{W}}^{s, p} \brk{\manifold{M}', \manifold{N}}
  \defeq \biggl\{u \colon \manifold{M}' \to \manifold{N} \st[\bigg] u \text{ is  Borel-measurable}\\[-.5em]
  \text{and}
 \smashoperator[r]{\iint_{\manifold{M}' \times \manifold{M}'}} \frac{d \brk{u (y), u (z)}^p}{d\brk{y, z}^{m - 1 + sp}}\dif y \dif z < \infty\biggr\},
\end{multline*}
where \(d\colon \manifold{N} \times \manifold{N} \to \intvr{0}{\infty}\) in the numerator of the integrand is the geodesic distance on the target manifold \(\manifold{N}\).
Sobolev spaces of mappings appear naturally in calculus of variations and partial differential equations in many contexts, including the theory of \emph{harmonic maps} \cite{Eells_Lemaire_1978}, physical models with \emph{nonlinear order parameters} \cite{Mermin1979}, Cosserat models in elasticity \citelist{\cite{Ericksen_Truesdell_1958}}, and the description of attitudes of a square or a cube in \emph{computer graphics} \cite{Huang_Tong_Wei_Bao_2011}.

As an immediate consequence of the trace theory in linear Sobolev spaces, the trace operator \(\tr_{\partial \manifold{M}}\) is \emph{well-defined} and \emph{continuous} from \(\smash{\dot{W}}^{1, p} \brk{\manifold{M}, \mathcal{N}}\) to \(\smash{\dot{W}}^{1 - 1/p, p} \brk{\partial \manifold{M}, \manifold{N}}\).
On the other hand, the \emph{surjectivity} of the trace is a much more delicate matter. 
Indeed, the classical linear trace theory merely yields an extension \(U \in \smash{\dot{W}}^{1, p} \brk{\manifold{M}, \Rset^\nu}\) and one cannot hope that,  by some miraculous coincidence, the function \(U\) would satisfy the constraint of taking values in the target manifold \(\manifold{N}\) almost everywhere in its domain \(\manifold{M}\).

\medbreak

In the \emph{supercritical} integrability case \(p > m = \dim \manifold{M}\), the \emph{essential continuity} of maps in nonlinear Sobolev spaces \(\smash{\dot{W}}^{1, p} \brk{\manifold{M}, \mathcal{N}}\) and \(\smash{\dot{W}}^{1 - 1/p, p} \brk{\partial \manifold{M}, \manifold{N}}\), which follows from the Morrey-Sobolev embedding and its fractional counterpart, can be used to prove that the extension of traces of Sobolev mappings encounters exactly the \emph{same obstructions} as the extension of continuous mappings: each map in \(\smash{\dot{W}}^{1 - 1/p, p} \brk{\partial \manifold{M}, \manifold{N}}\) is the trace of a map in \(\smash{\dot{W}}^{1, p} \brk{\manifold{M}, \manifold{N}}\) if and only if every continuous map from \(\partial \manifold{M}\) to \(\manifold{N}\) is the restriction of a continuous map from \(\smash{\manifold{M}}\) to \(\manifold{N}\) \cite{Bethuel_Demengel_1995}*{Th.\thinspace{}1}.

In the \emph{critical} integrability case \(p = m = \dim \manifold{M}\), although maps in the spaces \(\smash{\dot{W}}^{1, m} \brk{\manifold{M}, \manifold{N}}\) and \(\smash{\smash{\dot{W}}^{1 - 1/m, m} \brk{\partial \manifold{M}, \manifold{N}}}\) can no longer be identified with continuous ones, the same characterisation remains valid \cite{Bethuel_Demengel_1995}*{Th.\thinspace{}2}; the essential additional ingredient, which goes back to Schoen and Uhlenbeck's approximation of Sobolev mappings \citelist{\cite{Schoen_Uhlenbeck_1982}*{\S 3}\cite{Schoen_Uhlenbeck_1983}*{\S 4}}, is the \emph{vanishing mean oscillation} (VMO) property of the critical Sobolev spaces \(\smash{\smash{\dot{W}}^{1, m} \brk{\manifold{M}, \manifold{N}}}\) and \(\smash{\smash{\dot{W}}^{1-1/m, m} \brk{\partial \manifold{M}, \manifold{N}}}\), which is weaker than continuity, but still strong enough to preserve the essential features of the homotopy and obstruction theories of continuous maps as developed by Brezis and Nirenberg \citelist{\cite{Brezis_Nirenberg_1995}\cite{Brezis_Nirenberg_1996}} (see also \cite{Abbondandolo_1996}).

\medbreak

In the \emph{subcritical} integrability regime \(1 < p < m= \dim \manifold{M}\),  the extension of traces encounters additional \emph{topological and analytical obstructions}.

First a \emph{local topological obstruction} arises when \(p \ge 2\).
If
\begin{equation}
\label{eq_maLaeph9OongieveeSh9wahP}
\smash{\dot{W}}^{1 - 1/p, p} \brk{\partial \manifold{M}, \manifold{N}}
\subseteq
 \tr_{\partial \manifold{M}}
 \brk{\smash{\dot{W}}^{1, p} \brk{\manifold{M}, \manifold{N}}}
\end{equation}
and if \(2 \le p < m\),
then the homotopy group \(\pi_{\floor{p - 1}} \brk{\manifold{N}}\), where \(\floor{t} \in \Zset\) denotes the integer part of \(t \in \Rset\) satisfying \(\floor{t}\le t < \floor{t}+1\), should be \emph{trivial}, as proved by Hardt and Lin \cite{Hardt_Lin_1987}*{\S 6.3} and by Bethuel and Demengel \cite{Bethuel_Demengel_1995}*{Th.\thinspace{}4}.
The archetypal non-extensible map has the hedgehog structure
\begin{equation}
\label{eq_ongohghoom3eiG7baiW3choh}
u \brk{x}= f \brk{x'/\abs{x'}} 
\end{equation}
for \(x = \brk{x', x''}\) in a local chart domain of \(\smash{\Rset^{\floor{p}} \times \Rset^{m - 1 - \floor{p}}}\) containing the origin \(0\) and a smooth map \(f \in C^1 \brk{\smash{\Sset^{\floor{p - 1}}}, \manifold{N}}\) not homotopic to a constant; similar local topological obstructions appear for the strong approximation by smooth mappings of Sobolev maps in \(\smash{\dot{W}}^{1, p} \brk{\manifold{M}, \manifold{N}}\) with \(1 \le p < m\)\citelist{\cite{Schoen_Uhlenbeck_1983}*{\S 4}\cite{Bethuel_Zheng_1988}*{Th.\thinspace{}2}\cite{Bethuel_1991}*{Th.\thinspace{}A0}}, for the corresponding weak approximation when \(p\) is not an integer \cite{Bethuel_1991}*{Th.\thinspace{}3} and for the lifting of Sobolev maps in \(\smash{\smash{\dot{W}}^{s, p} \brk{\manifold{M}, \manifold{N}}}\) over a covering map when  \(0 < s \le 1\) and \(1 \le sp < 2\) \citelist{\cite{Bourgain_Brezis_Mironescu_2000}\cite{Bethuel_Chiron_2007}}.

The extension of traces also encounters a \emph{global topological obstruction},
identified in the works of Bethuel and Demengel \cite{Bethuel_Demengel_1995}*{Proof of Th.\thinspace{}5} and of Isobe \cite{Isobe_2003}.
In order to describe it, given a manifold \(\manifold{M}\) with boundary \(\partial \manifold{M}\) endowed with a triangulation, we denote by \(\manifold{M}^\ell\) for \(\ell \in \set{0, \dotsc, \dim \manifold{M}}\) the \(\ell\)-dimensional component of the triangulation and by
\(\partial \manifold{M}^{\ell}\) for \(\ell \in \set{0, \dotsc, \dim \manifold{M} -1}\) the union of the faces contained in \(\partial \manifold{M}\) so that \(\partial \manifold{M}^{\ell} = \manifold{M}^\ell \cap \partial \manifold{M}\).

The global topological obstruction then reads as follows: if
\eqref{eq_maLaeph9OongieveeSh9wahP} holds and if \(2 \le p < m\), then every continuous map from \(\smash{\partial \manifold{M}^{\floor{p - 1}}}\) to \(\manifold{N}\) should be the restriction of a continuous map from \(\smash{\manifold{M}}^{\floor{p - 1}}\) to \(\manifold{N}\);
a non-extensible Sobolev map can be obtained by starting from a smooth map \(f\colon \smash{\partial \manifold{M}^{\floor{p - 1}}} \to \manifold{N}\) which has no continuous extension to \(\smash{\manifold{M}}^{\floor{p - 1}}\) and extending it homogeneously to higher-dimensional faces of the triangulation of \(\partial \manifold{M}\).
The global obstruction does not occur if, for example, the lower homotopy groups \(\pi_1\brk{\manifold{N}}, \dotsc, \pi_{\floor{p - 1}} \brk{\manifold{N}}\) are trivial.

The local and global obstructions can be bundled into a single equivalent necessary condition that every continuous map from \(\smash{\partial \manifold{M}^{\floor{p - 1}}}\) to \(\manifold{N}\) should be the restriction of a continuous map from \(\smash{\manifold{M}^{\floor{p}}}\) to \(\manifold{N}\).

Similar global obstructions have been identified for the strong and weak approximation by smooth mappings of Sobolev mappings and in the classification of their homotopy classes \citelist{\cite{Hang_Lin_2003_I}\cite{Hang_Lin_2003_II}}.

\medbreak 

The surjectivity of the trace also encounters an \emph{analytical obstruction}: if
\begin{equation}
\label{eq_die4SooJoL3koov8aiwuo4nu}
\overline{
C^1 \brk{\partial \manifold{M}, \manifold{N}}}^{\smash{\dot{W}}^{1 - 1/p, p} \brk{\partial \manifold{M}, \manifold{N}}}
\subseteq
 \tr_{\partial \manifold{M}}
 \brk{\smash{\dot{W}}^{1, p} \brk{\manifold{M}, \manifold{N}}}
\end{equation}
and if \(2 \le p < m\), then the lower homotopy groups \(\pi_1 \brk{\manifold{N}}, \dotsc, \smash{\pi_{\floor{p - 1}}} \brk{\manifold{N}}\) should all be \emph{finite}, as obtained by Bethuel and Demengel for the circle \(\manifold{N} = \Sset^1\) \cite{Bethuel_Demengel_1995}*{Th.\thinspace{}6} and by Bethuel for a general target manifold \cite{Bethuel_2014} (see also \cite{Mironescu_VanSchaftingen_2021_AFST}).
An obstructing map can be obtained by packing together a sequence of smooth maps that have arbitrarily bad lower estimates on the Sobolev energy of their extensions; thanks to scaling properties of the Sobolev energies the resulting map is in \(\smash{\dot{W}}^{1 - 1/p, p} \brk{\partial \manifold{M}, \manifold{N}}\) but has no extension in \(\smash{\dot{W}}^{1, p} \brk{\manifold{M}, \manifold{N}}\).
Such a map is a strong limit in \(\smash{\dot{W}}^{1 - 1/p, p} \brk{\partial \manifold{M}, \manifold{N}}\) of smooth maps into \(\manifold{N}\) and is smooth outside a single point; it has smooth extensions outside small neighbourhoods of the singular point, but the Sobolev energy of the extension blows up when the neighbourhood becomes arbitrarily small.
The analytical obstruction can be seen as a consequence of the failure of linear estimates in Sobolev spaces on the extension of smooth maps combined with a nonlinear uniform boundedness principle \cite{Monteil_VanSchaftingen_2019}, which originated for the  problem of weak approximation of Sobolev mappings \cite{Hang_Lin_2003_III}*{Th.\thinspace{}9.6}.
Similar analytical obstructions were exhibited for the lifting of fractional Sobolev mappings when the covering space is not compact \citelist{\cite{Bourgain_Brezis_Mironescu_2000}*{Th.\thinspace{}2}\cite{Bethuel_Chiron_2007}*{Prop.\thinspace{}2}} and for the weak approximation of Sobolev mappings in \(\smash{\dot{W}}^{1, 4\ell - 1} \brk{\Bset^{m}, \Sset^{2\ell}}\) with \(\ell \in \Nset\setminus \set{0, 1}\) and \(m \ge 4\ell\), and in \(\smash{\dot{W}}^{1, p}\brk{\Bset^m, \Tset^{p + 1} \# \Tset^{p + 1}}\) with \(p \in \Nset\setminus \set{0, 1}\) and \(m \ge p + 1\) \citelist{\cite{Bethuel_2020}\cite{Detaille_VanSchaftingen}}.

A last \emph{critical analytical obstruction} was obtained by Mironescu and the author \cite{Mironescu_VanSchaftingen_2021_AFST}*{Th.\thinspace{}1.5 (b)} when the integrability exponent \(p \in \set{2, \dotsc, m - 1}\), with \(m = \dim \manifold{M}\), is an \emph{integer}, requiring then the \emph{triviality} of the homotopy group  \(\pi_{p - 1} (\manifold{N})\) for \eqref{eq_die4SooJoL3koov8aiwuo4nu} to hold, with a construction somehow similar to the one for the analytical obstruction.
Although a topological obstruction was already known in this case, the corresponding non-extensible map was not a strong limit of smooth maps whereas the one coming from the analytical obstruction is such a limit, showing the strength of the triviality condition on \(\pi_{p - 1} \brk{\manifold{N}}\) when \(p\) is an integer.
A similar critical analytical obstruction appears for the lifting of maps in fractional Sobolev spaces \(\smash{\dot{W}}^{s, p} \brk{\manifold{M}, \manifold{N}}\) when \(0 < s < 1\) and \(sp = 1\) \cite{Mironescu_VanSchaftingen_2021_APDE}.

\medbreak

In the other direction, the trace operator was shown to be \emph{surjective}  from \( \smash{\dot{W}}^{1, p} \brk{\manifold{M}, \manifold{N}}\) to \(\smash{\smash{\dot{W}}^{1 - 1/p, p}} \brk{\partial \manifold{M}, \manifold{N}}\) when the lower homotopy groups \(\pi_1\brk{\manifold{N}}, \dotsc, \smash{\pi_{\floor{p - 1}}} \brk{\manifold{N}}\) are all \emph{trivial} in Hardt and Lin's seminal work \cite{Hardt_Lin_1987}*{Th.\thinspace{}6.2} (see also \cite{Hardt_Kinderlehrer_Lin_1988}), with a method based on the existence of a singular retraction from the ambient Euclidean space \(\Rset^\nu\)  to the target manifold \(\manifold{N}\) which crucially relies on the triviality of these lower homotopy groups.
Combined with the topological and analytical obstructions, this has provided a complete answer to the extension problem when the target manifold \(\manifold{N}\) is a sphere \(\Sset^n\) of any dimension \(n \in \Nset \setminus \set{0}\) (see \cite{Brezis_Mironescu_2021}*{\S 11.3.2}).
Similar methods have been used to prove the density of smooth maps in various Sobolev spaces of mappings \citelist{\cite{Bethuel_Zheng_1988}\cite{Riviere_2000}\cite{Hajlasz_1994}\cite{Bousquet_Ponce_VanSchaftingen_2014}}.
When \(p \ge 3\), the assumption that the fundamental group \(\pi_1 \brk{\manifold{N}}\) is trivial was relaxed into its finiteness thanks to the construction of liftings of fractional Sobolev mappings over a compact covering \cite{Mironescu_VanSchaftingen_2021_APDE} since the universal covering space \(\smash{\widetilde{\manifold{N}}}\) of the target manifold \(\manifold{N}\) is then compact \cite{Mironescu_VanSchaftingen_2021_AFST}.
Liftings can also be used systematically to construct extensions of traces when \(\manifold{N} = \Sset^1\) \cite{Brezis_Mironescu_2021}*{\S 11.1}.

\medbreak

In the present work we give a complete answer to the problem of the surjectivity of traces from \(\smash{\dot{W}}^{1, p} \brk{\manifold{M}, \manifold{N}}\) to \(\smash{\dot{W}}^{1 - 1/p, p} \brk{\partial \manifold{M}, \manifold{N}}\). We prove through a new construction of extensions that there is no other obstruction beyond those already known and presented in the previous paragraphs.

\medbreak

We first state our result in the case of the half-space \(\manifold{M} = \Rset^m_+ = \Rset^{m - 1} \times \intvo{0}{\infty}\).

\begin{theorem}
\label{theorem_extension_halfspace}
If \(1 < p < m\), then 
\[
 \tr_{\Rset^{m - 1}}
 \brk{\smash{\dot{W}}^{1, p} \brk{\smash{\Rset^m_+}, \manifold{N}}} = \smash{\smash{\dot{W}}^{1 - 1/p, p}} (\Rset^{m - 1}, \manifold{N})
\]
if and only if the homotopy groups \(\pi_1 \brk{\manifold{N}}, \dotsc, \smash{\pi_{\floor{p - 2}}} \brk{\manifold{N}}\) are finite and the homotopy group \(\smash{\pi_{\floor{p - 1}} \brk{\manifold{N}}}\) is trivial.\\
Moreover, there is then a constant \(C \in \intvo{0}{\infty}\) such that for every mapping \(u \in \smash{\smash{\dot{W}}^{1 - 1/p, p} \brk{\Rset^{m - 1}, \manifold{N}}}\), one can construct a mapping \(U \in \smash{\smash{\dot{W}}^{1, p}} (\Rset^m_+, \manifold{N})\) such that \(\smash{\tr_{\Rset^{m - 1}} U} = u\) and
\begin{equation}
\label{eq_koQuee7zeex0EithaePheXae}
  \int_{\Rset^{m}_+}
  \abs{\Deriv U}^p
  \le C 
  \smashoperator{
  \iint_{\Rset^{m - 1} \times \Rset^{m - 1}}}
  \frac{d \brk{u \brk{y}, u \brk{z}}^p}{\abs{y -  z}^{p + m - 2}} \dif y \dif z.
\end{equation}
\end{theorem}

When \(1 < p < 2\), the necessary and sufficient condition of \cref{theorem_extension_halfspace} becomes vacuous and the trace operator is surjective for any compact Riemannian target manifold \(\manifold{N}\);
similarly when \(2 \le p < 3\) the condition of \cref{theorem_extension_halfspace} is merely that the target manifold \(\manifold{N}\) should be simply connected.

The estimate \eqref{eq_koQuee7zeex0EithaePheXae} is \emph{linear} in the energy in the sense that it relates the semi-norm of the extension to a multiple of the semi-norm of the boundary datum, although the extension operator constructed here is nonlinear.
In this case, it follows in fact from a nonlinear uniform boundedness principle \cite{Monteil_VanSchaftingen_2019} that if every Sobolev map has a Sobolev extension, then it should be possible to control the extension linearly as in \eqref{eq_koQuee7zeex0EithaePheXae}.

\medbreak

The statement of \cref{theorem_extension_halfspace} extends readily to the case where \(\manifold{M} = \manifold{M}' \times \intvr{0}{1}\) and \(\partial \manifold{M} = \manifold{M}' \times \set{0} \simeq \manifold{M}'\).

\begin{manualtheorem}{\ref*{theorem_extension_halfspace}\texorpdfstring{\('\)}{'}}
\label{theorem_extension_collar}
Let \(\manifold{M}'\) be an \((m - 1)\)-dimensional compact Riemannian manifold.
If \(1 < p < m\),
then 
\[
 \tr_{\manifold{M}' \times \set{0}} \brk{\smash{\dot{W}}^{1, p} \brk{\manifold{M}' \times \intvr{0}{1}, \manifold{N}}}
 = \smash{\dot{W}}^{1 - 1/p, p} \brk{\manifold{M}', \manifold{N}}
\]
if and only if the homotopy groups \(\pi_1 \brk{\manifold{N}}, \dotsc, \smash{\pi_{\floor{p - 2}}} \brk{\manifold{N}}\) are finite and the homotopy group \(\smash{\pi_{\floor{p - 1}} \brk{\manifold{N}}}\) is trivial.\\
Moreover, there is then a constant \(C \in \intvo{0}{\infty}\) such that  for every mapping \(u \in \smash{\smash{\dot{W}}^{1 - 1/p, p} \brk{\manifold{M}', \manifold{N}}}\), one can construct a mapping \(U \in \smash{\dot{W}^{1, p} \brk{\manifold{M}' \times \intvr{0}{1}, \manifold{N}}}\) such that \(\tr_{\manifold{M}' \times \set{0}} U = u\) and 
\begin{equation}
\label{eq_jooFofo1au1aephae7sa0joh}
  \smashoperator{\int_{\manifold{M}' \times \intvr{0}{1}}}
  \abs{\Deriv U}^p
  \le C
  \smashoperator{
  \iint_{\manifold{M}' \times \manifold{M}'}}
  \frac{d \brk{u \brk{y}, u \brk{z}}^p}{d\brk{y, z}^{p + m - 2}} \dif y \dif z.
\end{equation}
\end{manualtheorem}

\Cref{theorem_extension_halfspace} gives in particular an extension to a collar neighbourhood \(\partial \manifold{M} \times \intvr{0}{1}\) of the boundary of any Riemannian manifold \(\manifold{M}\) with compact boundary \(\partial \manifold{M}\).

The estimate \eqref{eq_jooFofo1au1aephae7sa0joh} is still linear,
as imposed by the estimate in the localisation of traces \cite{VanSchaftingen_2026} combined with the nonlinear uniform boundedness principle \cite{Monteil_VanSchaftingen_2019}, which does not apply directly in this case.

\medbreak

Finally, we have the following global counterpart of \cref{theorem_extension_halfspace} and \cref{theorem_extension_collar} for a compact Riemannian manifold \(\manifold{M}\) with compact boundary \(\partial \manifold{M} \ne \emptyset\).

\begin{manualtheorem}{\ref*{theorem_extension_halfspace}\texorpdfstring{\(''\)}{''}}
\label{theorem_extension_global}
Let \(\manifold{M}\) be an \(m\)-dimensional manifold with compact boundary \(\partial \manifold{M} \ne \emptyset\).
If \(1 < p < m\), then 
\[
\tr_{\partial \manifold{M}}\brk{
 \smash{\dot{W}^{1, p} \brk{\manifold{M}, \manifold{N}}}}
 = \smash{\dot{W}}^{1 - 1/p, p} \brk{\partial \manifold{M}, \manifold{N}}
\]
if and only if the homotopy groups \(\pi_1 \brk{\manifold{N}}, \dotsc, \penalty 0 \smash{\pi_{\floor{p - 2}} \brk{\manifold{N}}}\) are finite and every continuous map from \(\smash{\partial \manifold{M}^{\floor{p - 1}}}\) to \(\manifold{N}\) is the restriction of a continuous map from \(\smash{\manifold{M}^{\floor{p}}}\) to \(\manifold{N}\).\\
Moreover, there exists then a convex function \(\Theta \colon  \intvr{0}{\infty} \to \intvr{0}{\infty}\) such that \(\Theta \brk{0} = 0\) and such that for every mapping \(u \in \smash{\dot{W}^{1 - 1/p, p} \brk{\partial \manifold{M}, \manifold{N}}}\), one can construct a mapping \(U \in \smash{\dot{W}^{1, p} \brk{\manifold{M}, \manifold{N}}}\) such that \(\tr_{\partial \manifold{M}} U = u\) and
\begin{equation}
\label{eq_ho0SheiChathesi3fai4nei3}
  \int_{\manifold{M}}
  \abs{\Deriv U}^p
  \le  \Theta \brk[\bigg]{
  \smashoperator[r]{
  \iint_{\partial \manifold{M} \times \partial \manifold{M}}}
  \frac{d \brk{u \brk{y}, u \brk{z}}^p}{d\brk{y, z}^{p + m - 2}}
  +
  d \brk{u \brk{y}, u \brk{z}}^p \dif y \dif z
  }.
\end{equation}
\end{manualtheorem}

The condition of \cref{theorem_extension_global} implies that the homotopy group \(\smash{\pi_{\floor{p - 1}}\brk{\manifold{N}}}\) is trivial; in fact, every continuous map from \(\partial \manifold{M}^{\ell - 1}\) to \(\manifold{N}\) is the restriction of a continuous map from \(\smash{\manifold{M}^{\ell}}\) to \(\manifold{N}\) if and only if the homotopy group \(\pi_{\ell - 1} \brk{\manifold{N}}\) is trivial \emph{and} every continuous map from \(\smash{\partial \manifold{M}^{\ell - 1}}\) to \(\manifold{N}\) is the restriction of a continuous map from \(\manifold{M}^{\ell - 1}\) to \(\manifold{N}\). 

In \cref{theorem_extension_global}, the boundary \(\partial \manifold{M}\) need not be connected; this ensures that important examples of the form \(\manifold{M} = \manifold{M}’\times \intvc{0}{1}\) are covered.
If the homotopy groups \(\pi_1 \brk{\manifold{M}’}, \dotsc, \pi_{\floor{p - 1}} \brk{\manifold{M’}}\) are trivial, the restriction to the \(\smash{\manifold{M}’^{\floor{p - 1}}}\)-dimensional skeleton of \(\manifold{M}\) of any continuous map from \(\smash{\manifold{M}’^{\floor{p}}}\) is homotopic to a constant;
the homotopy extension property and \cref{theorem_extension_global} show then that
\[
\tr_{\manifold{M}’ \times\set{0, 1}}\brk{
 \smash{\dot{W}^{1, p} \brk{\manifold{M}’ \times \intvc{0}{1}, \manifold{N}}}}
 = \smash{\dot{W}}^{1 - 1/p, p} \brk{\manifold{M}’ \times \set{0, 1}, \manifold{N}}
\]
if and only if \(\pi_1 \brk{\manifold{N}}, \dotsc \pi_{\floor{p - 2}} \brk{\manifold{N}}\) are finite and \(\pi_{\floor{p - 1}}\brk{\manifold{N}}\) is trivial; this is the case for example of \(\manifold{M}' = \Sset^{m - 1}\) and \(p \in \intvo{1}{m}\).

\subsection{Characterisation of the range of the trace}
When the trace operator is not surjective for Sobolev mappings, it is natural to seek a characterisation of the image of the trace operator. 
In the simplest case of the half-space \(\manifold{M} = \smash{\widebar{\Rset}}^m_+\), we show that the traces are strong limits of smooth maps when \(p\) is not an integer and the lower homotopy groups are finite.

\begin{theorem}
\label{theorem_characterization_halfspace}
If \(1 < p < m\), then
\[
 \operatorname{tr}_{\Rset^{m - 1}} \brk{\smash{\smash{\dot{W}}^{1, p} \brk{\Rset^m_+, \manifold{N}}}}
 =
 \overline{C^1 \brk{\Rset^{m - 1}, \manifold{N}}}^{\smash{\dot{W}}^{1 - 1/p, p} \brk{\Rset^{m - 1}, \manifold{N}}}
\]
if and only if \(\pi_1 \brk{\manifold{N}}, \dotsc, \smash{\pi_{\floor{p - 1}} \brk{\manifold{N}}}\) are finite and, when \(p \in \Nset\), \(\pi_{p - 1} \brk{\manifold{N}}\) is trivial.\\
Moreover, every map \(u \in\operatorname{tr}_{\Rset^{m - 1}} \brk{\smash{\smash{\dot{W}}^{1, p} \brk{\Rset^m_+, \manifold{N}}}}\)
can then be written as \(u = \tr_{\Rset^{m - 1}} U\) with
\(U \in \smash{\smash{\dot{W}}^{1, p} \brk{\Rset^m_+, \manifold{N}}}\) satisfying the estimate
\eqref{eq_koQuee7zeex0EithaePheXae}.
\end{theorem}

The finiteness of the lower homotopy groups \(\pi_1 \brk{\manifold{N}}, \dotsc, \smash{\pi_{\floor{p - 1}} \brk{\manifold{N}}}\) and the triviality of \(\pi_{p - 1} \brk{\manifold{N}}\) when \(p \in \Nset\) guarantees the \emph{absence of analytical obstructions,} which already occur for the extension of mappings which are the limits of smooth maps.

When \(p \in \Nset\), the triviality of \(\pi_{p - 1} \brk{\manifold{N}}\) is still necessary:
although considering maps that are approximated by smooth maps prevents the topological obstruction, it cannot rule out the analytical obstruction, so that \cref{theorem_characterization_halfspace} does not provide a characterisation of the space of traces beyond \cref{theorem_extension_halfspace}.

The condition for the mapping \(u \in \smash{\smash{\dot{W}}^{1 - 1/p, p} \brk{\Rset^{m - 1}, \manifold{N}}}\) to be the strong limit of a sequence of smooth maps can be characterised through homological conditions when the target is a sphere \(\manifold{N} = \Sset^n\) and \(n + 1 \le p < n + 2\).
This has been achieved thanks to a suitable distributional Jacobian when \(\manifold{N} = \Sset^1\) and \(2 \le p < 3\) \cite{Brezis_Mironescu_2021}*{Th.\thinspace{}10.4 and 10.4\texorpdfstring{\('\)}{'}} and to generalised Cartesian currents \cite{Mucci_2010}.

When \(2 \le p < 3\) so that \(\smash{\floor{p - 1}} = 1\), the condition of \cref{theorem_characterization_halfspace} is also equivalent to the existence of a lifting over the universal covering \(\pi \colon \smash{\widetilde{\manifold{N}}} \to \manifold{N}\) \cite{Mironescu_VanSchaftingen_2021_APDE}*{Th.\thinspace{}3}.

\medbreak

When characterising the Sobolev maps that can be extended to \(\manifold{M}' \times \intvr{0}{1}\) with \(\manifold{M}'\) compact, it is no longer sufficient to be in the closure of smooth maps. This is related to the fact that even when the homotopy group \(\smash{\pi_{\floor{p - 1}}} \brk{\manifold{N}}\) is trivial, smooth maps may not be dense in  \(\smash{\dot{W}^{1 - 1/p, p}} \brk{\smash{\manifold{M}'}, \manifold{N}}\).

There are however still approximations by maps that are not necessarily smooth but whose singular set and whose behaviour next to it are suitably controlled.
If we define for each \(\ell \in \Zset\) such that \(\ell < m - 1\), the set
\begin{equation}
\label{eq_quaeP5sosiC8Nahbaipaek7W}
\begin{split}
 R^1_\ell \brk{\manifold{M}', \manifold{N}}
 \defeq 
 \Bigl\{ u \colon \manifold{M}' \to \manifold{N}
 \st[\Big] & u \in C^1 \brk{\manifold{M}' \setminus \Sigma, \manifold{N}}
 \text{ and } \smash{\smashoperator[r]{\sup_{x \in \manifold{M}'\setminus \Sigma}}} \dist \brk{x, \Sigma} \abs{\Deriv u \brk{x}} < \infty\\
 &\; \text{where \(\Sigma \subseteq \manifold{M}'\) is closed and contained in a finite }\\[-.5em]
 &\;\text{union of \(\ell\)-dimensional embedded smooth submanifolds}
 \Bigr\},
 \end{split}
 \raisetag{3.8em}
\end{equation}
with the understanding that the boundedness condition is automatically satisfied when \(\Sigma = \emptyset\) and
that \(\Sigma = \emptyset\) when \(\ell < 0\) so that \( R^1_\ell \brk{\manifold{M}', \manifold{N}}\) is then the set of functions in \(C^1 \brk{\manifold{M}', \manifold{N}}\), then one has \(R^1_\ell \brk{\manifold{M}' \times \intvr{0}{1}, \manifold{N}}
\subseteq \smash{\dot{W}}^{1, p}\brk{\manifold{M}' \times \intvr{0}{1}, \manifold{N}}\) if and only if \(\ell < m - p\); furthermore the set
\(\smash{R^1_{m - \floor{p} - 1} \brk{\manifold{M}' \times \intvr{0}{1}, \manifold{N}}}\) is strongly dense in \(\smash{\dot{W}}^{1, p}\brk{\manifold{M}' \times \intvr{0}{1}, \manifold{N}}\)
\citelist{\cite{Bethuel_1991}\cite{Hang_Lin_2003_II}}.
Similarly, the inclusion \(R^1_\ell \brk{\manifold{M}', \manifold{N}}
\subseteq \smash{\dot{W}}^{s, p}\brk{\manifold{M}', \manifold{N}}\) holds if and only if \(\ell < m - \floor{sp} - 1\), and the set
\(\smash{R^1_{m - \floor{sp} - 2} \brk{\manifold{M}', \manifold{N}}}\) is strongly dense in the space \(\smash{\smash{\dot{W}}^{1 - 1/p, p}\smash{\brk{\manifold{M}', \manifold{N}}}}\) \citelist{\cite{Bethuel_1995}\cite{Mucci_2009}*{Th.\thinspace{}1}\cite{Mucci_2024}*{Th.\thinspace{}1.1}\cite{Brezis_Mironescu_2015}*{Th.\thinspace{}3}}; in the particular case \(s = 1 -  \frac{1}{p}\) we have \(\smash{R^1_{m - \floor{sp} - 2} \brk{\manifold{M}', \manifold{N}}} = \smash{R^1_{m - \floor{p} - 1} \brk{\manifold{M}', \manifold{N}}}\).

Although the set \(\smash{R^1_{m - \floor{p} - 1} \brk{\manifold{M}', \manifold{N}}}\) is strongly dense in \(\smash{\dot{W}}^{1 - 1/p, p}\brk{\manifold{M}', \manifold{N}}\), it is too large to allow Sobolev extensions of all its maps when the homotopy group \(\smash{\pi_{\floor{p - 1}} \brk{\manifold{N}}}\) is non-trivial, since the counterexamples of the form \eqref{eq_ongohghoom3eiG7baiW3choh} providing the topological obstruction are smooth outside an \(\brk{m - \floor{p} - 1}\)-dimensional subset of \(\manifold{M}'\) and thus belong to  \(\smash{R^1_{m - \floor{p} - 1} \brk{\manifold{M}', \manifold{N}}}\).
On the other hand, the density of the set \(\smash{R^1_{m - \floor{p} - 1} \brk{\manifold{M}' \times \intvr{0}{1}, \manifold{N}}}\) in the space \(\smash{\dot{W}}^{1, p} \brk{\manifold{M}' \times \intvr{0}{1}, \manifold{N}}\) and the
generic \(\brk{m - \smash{\floor{p}} - 2}\)-dimensionality of the intersection of an \(\brk{m - \smash{\floor{p}} - 1}\)-dimensional submanifold with \(\smash{\manifold{M}'} \times \set{0}\) suggest that traces of maps in \(\smash{\dot{W}}^{1, p} \brk{\manifold{M}' \times \intvr{0}{1}, \manifold{N}}\) should be strongly approximable by maps in \(\smash{R^1_{m - \floor{p} - 2} \smash{\brk{\manifold{M}', \manifold{N}}}}\).
This approximation condition does indeed characterise the Sobolev maps which are traces when the target manifold has its lower homotopy groups that are finite.

\begin{manualtheorem}{\ref*{theorem_characterization_halfspace}\texorpdfstring{\('\)}{'}}
\label{theorem_characterization_collar}
Let \(\manifold{M}'\) be an \((m - 1)\)-dimensional compact Riemannian manifold.
If \(1 < p < m\), then
\[
\tr_{\manifold{M}' \times \set{0}}
\brk{ \smash{\smash{\dot{W}}^{1, p} \brk{\manifold{M}' \times \intvr{0}{1}, \manifold{N}}}}
 = \smash{\overline{R^1_{m - \floor{p} - 2} \brk{\smash{\manifold{M}'}, \manifold{N}}}}^{\smash{\smash{\dot{W}}^{1 - 1/p, p} \brk{\manifold{M}', \manifold{N}}}}
\]
if and only if \(\pi_1 \brk{\manifold{N}}, \dotsc, \smash{\pi_{\floor{p - 1}} \brk{\manifold{N}}}\) are finite
and, when \(p \in \Nset\), \(\pi_{p - 1} \brk{\manifold{N}}\) is trivial.\\
Moreover, every map \(u \in\tr_{\manifold{M}' \times \set{0}}
\brk{ \smash{\smash{\dot{W}}^{1, p} \brk{\manifold{M}' \times \intvr{0}{1}, \manifold{N}}}}\) can then be written \(u = \tr_{\manifold{M}' \times \set{0}} U\) with \(U \in \smash{\smash{\dot{W}}^{1, p} \brk{\manifold{M}' \times \intvr{0}{1}, \manifold{N}}}\) satisfying
\eqref{eq_jooFofo1au1aephae7sa0joh}.
\end{manualtheorem}

Since the set  \(\smash{R^1_{m - \floor{p} - 2} \brk{\manifold{M}', \manifold{N}}}\) is dense in \(\smash{\dot{W}}^{1 - 1/p, p} \brk{\manifold{M}', \manifold{N}}\) if and only if either \(1 < p < 2\) or the homotopy group \(\smash{\pi_{\floor{p - 1}} \brk{\manifold{N}}}\) is trivial (see \citelist{\cite{Brezis_Mironescu_2015}\cite{Hang_Lin_2003_II}*{Th.\thinspace{}6.3}}), \cref{theorem_extension_collar} is a consequence of \cref{theorem_characterization_collar}.

If \(1 < p < m - 1\) and if the inclusion \(\manifold{M}'{}^{\floor{p - 1}} \to \manifold{M}'\) is homotopic to a constant, which will be the case for example if the lower homotopy groups \(\pi_{1} \brk{\manifold{M}'}, \dotsc, \smash{\pi_{\floor{p - 1}}} \brk{\manifold{M}'}\) of \(\manifold{M'}\) are trivial, then any mapping from the skeleton \(\smash{\manifold{M}'{}^{\floor{p - 1}}}\) to \(\manifold{N}\) has a continuous extension to the manifold \(\smash{\manifold{M}'}\) if and only if it has a continuous extension to its skeleton \(\manifold{M}'{}^{\floor{p}}\) (see \cite{White_1986}*{\S 6}), and thus any mapping in \(\smash{R^1_{m - \floor{p} - 2} \brk{\manifold{M}', \manifold{N}}}\) is the strong limit in  \(\smash{\smash{\dot{W}}^{1 - 1/p, p}} \brk{\manifold{M}', \manifold{N}}\) of a sequence of smooth maps (cf. \citelist{\cite{Hang_Lin_2003_II}*{Cor.\thinspace{}1.6 \& 6.2}});  it follows then from  \cref{theorem_characterization_collar} that traces are characterised under this assumption as strong limits in \(\smash{\smash{\dot{W}}^{1 - 1/p, p}} \brk{\manifold{M}', \manifold{N}}\) of smooth maps, as in \cref{theorem_characterization_halfspace}.

If \(p \ge m - 1\), then \(\smash{R^1_{m - \floor{p} - 2} \brk{\manifold{M}', \manifold{N}}} = C^1 \brk{\manifold{M}', \manifold{N}}\) by definition, and \cref{theorem_characterization_collar} characterises traces as strong limits of smooth maps in \(\smash{\smash{\dot{W}}^{1 - 1/p, p}} \brk{\manifold{M}', \manifold{N}}\) as in  \cref{theorem_characterization_halfspace}.
More generally, if \(1 < p < m - 1\) and if the homotopy groups \(\smash{\pi_{\floor{p}}} \brk{\manifold{N}}, \dotsc, \pi_{m - 2} \brk{\manifold{N}}\) of the target \(\manifold{N}\) are trivial, then any mapping in \(\smash{R^1_{m - \floor{p} - 2} \brk{\manifold{M}', \manifold{N}}}\) is the strong limit in  \(\smash{\smash{\dot{W}}^{1 - 1/p, p}} \brk{\manifold{M}', \manifold{N}}\) of a sequence of smooth maps (cf.\ \cite{Hang_Lin_2003_II}*{Cor.\thinspace{}1.7 \& 6.2}), and traces are also strong limits of smooth maps in \(\smash{\smash{\dot{W}}^{1 - 1/p, p}} \brk{\manifold{M}', \manifold{N}}\).

If \(2 \le p < 3\), if the fundamental group \(\pi_1 \brk{\manifold{N}}\) is finite and if the fundamental group \(\pi_1 \brk{\manifold{M}'}\) is trivial, then the condition of \cref{theorem_characterization_collar} is equivalent to the existence of a Sobolev lifting over the universal covering \(\pi \colon \smash{\widetilde{\manifold{N}}} \to \manifold{N}\) of \(\manifold{N}\), that is, the existence of some mapping \(\smash{\widetilde{u}} \in W^{1 - 1/p, p} \brk{\manifold{M}', \smash{\widetilde{\manifold{N}}}}\) such that \(\pi \compose \smash{\widetilde{u}} = u\) \cite{Isobe_2003}*{Th.\thinspace{}1.6}.

\medbreak

Finally we consider the characterisation of the traces of \(\smash{\dot{W}}^{1, p} \brk{\manifold{M}, \manifold{N}}\) when \(\manifold{M}\) is a compact Riemannian manifold with boundary \(\partial \manifold{M} \ne \emptyset\).
In order to describe the global obstructions, we define similarly to \eqref{eq_quaeP5sosiC8Nahbaipaek7W} for every \(\ell \in \Zset\) such that \(\ell < m\) the set 
\begin{equation}
\label{eq_TieQuee3iyoh9aequepo9Ahx}
\begin{split}
 R^1_\ell \brk{\manifold{M}, \manifold{N}}
 \defeq 
 \{ U \colon \manifold{M} \to \manifold{N}
 \!\st \!& U \in C^1 \brk{\manifold{M} \setminus \Sigma, \manifold{N}}  \text{ and }\smash{\smashoperator[r]{\sup_{x \in \manifold{M} \setminus \Sigma}}} \dist \brk{x, \Sigma} \abs{\Deriv U \brk{x}} < \infty\\
 &\text{where \(\Sigma \subseteq \manifold{M}\) is closed and contained in a finite union of }\\
 &\text{embedded \(\ell\)-dimensional submanifolds transverse to \(\partial \manifold{M}\)}
 \},
 \end{split}
 \raisetag{4.2em}
\end{equation}
again with the understanding that \(\Sigma = \emptyset\) when \(\ell < 0\).
The transversality assumption in the definition \eqref{eq_TieQuee3iyoh9aequepo9Ahx} of \(R^1_\ell \brk{\manifold{M}, \manifold{N}}\) prevents any tangent space to any of the submanifolds containing \(\Sigma\) from being contained in a tangent space to the boundary \(\partial \manifold{M}\); it ensures that
\[
 \tr_{\partial \manifold{M}} \brk{R^1_{\ell} \brk{\manifold{M}, \manifold{N}}}
 \subseteq R^1_{\ell - 1} \brk{\partial \manifold{M}, \manifold{N}},
\]
that is, if \(U \in R^1_\ell \brk{\manifold{M}, \manifold{N}}\) then \(U \restr{\partial \manifold{M}} \in R^1_{\ell - 1} \brk{\partial \manifold{M}, \manifold{N}}\).
In the particular case \(\ell = 0\), it implies that \(\Sigma \cap \partial \manifold{M} = \emptyset\).

As before, the set \(\smash{R^1_{m - \floor{p} - 1} \brk{\manifold{M}, \manifold{N}}}\) is strongly dense in \(\smash{\dot{W}}^{1, p} \brk{\manifold{M}, \manifold{N}}\) \citelist{\cite{Bethuel_1991}\cite{Hang_Lin_2005_IV}}, and thus, by the continuity of traces, the trace of any map \(U \in \smash{\dot{W}}^{1, p} \brk{\manifold{M}, \manifold{N}}\) is the strong limit in \(\smash{\smash{\dot{W}}^{1 - 1/p, p} \brk{\partial \manifold{M}, \manifold{N}}}\) of a sequence of maps in \(\smash{R^1_{m - \floor{p} - 2} \brk{\partial \manifold{M}, \manifold{N}}}\) that are traces of maps in \(\smash{R^1_{m - \floor{p} - 1} \brk{\manifold{M}, \manifold{N}}}\).
Again, this condition turns out to give a characterisation of the traces of Sobolev mappings when the lower homotopy groups are finite, and when the last of these is trivial for \(p\) integer.

\begin{manualtheorem}{\ref*{theorem_characterization_halfspace}\texorpdfstring{\(''\)}{''}}
\label{theorem_characterization_global}
Let \(\manifold{M}\) be an \(m\)-dimensional manifold with compact boundary \(\partial \manifold{M} \ne \emptyset\).
If \(1 < p < m\),
then
\[
\tr_{\partial \manifold{M}}
 \smash{ \smash{\dot{W}}^{1, p} \brk{\manifold{M}, \manifold{N}}}
 = \smash{\overline{\tr_{\partial \manifold{M}} \brk{R^1_{m - \floor{p} - 1} \brk{\manifold{M}, \manifold{N}}}}}^{\smash{\dot{W}^{1 - 1/p, p} \brk{\partial \manifold{M}, \manifold{N}}}}
\]
if and only
if \(\pi_1 \brk{\manifold{N}}, \dotsc, \smash{\pi_{\floor{p - 1}}} \brk{\manifold{N}}\) are finite, and, when \(p \in \Nset\), \(\pi_{p - 1} \brk{\manifold{N}}\) is trivial.\\
Moreover, every map \(u \in \tr_{\partial \manifold{M}}\brk{
 \smash{ \smash{\dot{W}}^{1, p} \brk{\manifold{M}, \manifold{N}}}}
\) can then be written \(u = \operatorname{tr}_{\partial \manifold{M}} U\) with \(U \in\smash{ \smash{\dot{W}}^{1, p} \brk{\manifold{M}, \manifold{N}}}\) satisfying \eqref{eq_ho0SheiChathesi3fai4nei3}.
\end{manualtheorem}

When \(p \ge m - 1\), we have \(\smash{R^1_{m - \floor{p} - 2} \brk{\partial \manifold{M}, \manifold{N}}} = C^1 \brk{\partial \manifold{M}, \manifold{N}}\) and  \(\smash{R^1_{m - \floor{p} - 1}\brk{\manifold{M}, \manifold{N}}} = \smash{R^1_{0}\brk{\manifold{M}, \manifold{N}}}\); and \cref{theorem_characterization_global} then states that \(u = \operatorname{tr}_{\partial \manifold{M}} U\) for some \(U \in\smash{ \smash{\dot{W}}^{1, p} \brk{\manifold{M}, \manifold{N}}}\) if and only if \(u\) is the strong limit in \(\smash{\dot{W}^{1 - 1/p, p} \brk{\partial \manifold{M}, \manifold{N}}}\) of a sequence of mappings in \(C^1 \brk{\partial \manifold{M}, \manifold{N}}\) that are the traces of mappings with point singularities.

In the case where every mapping in \(\smash{R^1_{m - \floor{p} - 2} \brk{\partial \manifold{M}, \manifold{N}}}\) is the trace of a mapping in \(\smash{R^1_{m - \floor{p} - 1} \brk{\manifold{M}, \manifold{N}}}\), or equivalently, where every continuous map from \(\partial \smash{\manifold{M}^{\floor{p}}}\) to \(\manifold{N}\) is the restriction of a continuous map from \(\smash{\manifold{M}^{\floor{p}}}\) to \(\manifold{N}\), we get from \cref{theorem_characterization_collar} and \cref{theorem_characterization_global} the equivalence between the Sobolev extension to the collar neighbourhood \(\partial \manifold{M} \times \intvo{0}{1}\) and to the entire manifold \(\manifold{M}\), which was already obtained under the slightly stronger condition that any continuous map from \(\smash{\partial \manifold{M}^{\floor{p - 1}}}\) to \(\manifold{N}\) is the restriction of a continuous map from \(\smash{\manifold{M}^{\floor{p}}}\) to \(\manifold{N}\) by Isobe \cite{Isobe_2003}.
The condition for every continuous map from \(\partial \smash{\manifold{M}^{\floor{p}}}\) to \(\manifold{N}\) to be the restriction of a continuous map from \(\smash{\manifold{M}^{\floor{p}}}\) to \(\manifold{N}\) will be satisfied if
there is a retraction from \(\smash{\manifold{M}^{\floor{p}}}\) to \(\partial \smash{\manifold{M}^{\floor{p}}}\), which will be the case for example if \(\manifold{M} = \smash{\widebar{\Rset}}^{m}_+\), if \(\manifold{M} = \Bar{\Bset}^m\) with \(p < m\), or more generally if \(\partial \manifold{M}\) is connected and for every \(j \in \set{1, \dotsc, \floor{p - 1}}\),
the homomorphism \(i_* \colon \pi_{j} \brk{\partial \manifold{M}} \to \pi_{j} \brk{\manifold{M}}\) induced by the canonical inclusion \(i \colon \partial \manifold{M} \to \manifold{M}\) is an isomorphism (see the proofs in  \citelist{\cite{Hatcher_2002}*{\S 4.1}\cite{Whitehead_1949}}).

If \(2 \le p < 3\), if \(\pi \colon \smash{\widetilde{\manifold{N}}} \to \manifold{N}\) is the universal covering of \(\manifold{N}\) and if for every \(F \in C \brk{\manifold{M}^2, \manifold{N}}\),
there exists \(\smash{\widetilde{f}} \in C \brk{\partial \manifold{M}^2,\smash{\widetilde{\manifold{N}}}}\) such that \(F \restr{\partial \manifold{M}^2} = \pi \compose \smash{\widetilde{f}}\)
(this will be the case if the homomorphism \(i_* \colon \pi_1 \brk{\partial \manifold{M}} \to \pi_1 \brk{\manifold{M}}\) induced by the inclusion \(i \colon \partial \manifold{M} \to \manifold{M}\) is trivial, and thus in particular when either \(\pi_1 \brk{\manifold{M}}\) or \(\pi_1 \brk{\partial \manifold{M}}\) is trivial), then for every \(U \in \smash{R^1_{m - \floor{p} - 1}\brk{\manifold{M}, \manifold{N}}}\) there exists \(\smash{\widetilde{u}} \in \smash{R^1_{m - \floor{p} - 2}} \brk{\partial \manifold{M}, \smash{\widetilde{\manifold{N}}}}\) such that \(U \restr{\partial \manifold{M}} = \pi \compose\smash{\widetilde{u}}\), and hence, if the homotopy group \(\smash{\pi_1 \brk{\manifold{N}}}\) is finite, it follows from \cref{theorem_characterization_global} that we have
\(u = \tr_{\partial \manifold{M}}U\) for some \(U \in\smash{ \smash{\dot{W}}^{1, p} \brk{\manifold{M}, \manifold{N}}}\) if and only if \(u = \pi \compose \smash{\widetilde{u}}\) for some \(\smash{\widetilde{u}}\in \smash{\smash{\dot{W}}^{1 - 1/p, p} \brk{\partial \manifold{M}, \lifting{\manifold{N}}}}\) \ \cite{Isobe_2003}*{Th.\thinspace{}1.6}.

\medbreak

When the mapping \(u\colon \partial \manifold{M} \to \manifold{N}\) is assumed to be Lipschitz-continuous, the criteria of Theorems \ref{theorem_characterization_halfspace}, \ref{theorem_characterization_collar} and \ref{theorem_characterization_global} were already described by White \cite{White_1988}*{Th.\thinspace{}4.1}; the striking fact in our work is that the same characterisation holds for boundary data with minimal regularity provided the lower homotopy groups are finite.

The characterisation of the trace spaces of Theorems \ref{theorem_characterization_halfspace}, \ref{theorem_characterization_collar} and \ref{theorem_characterization_global} when the lower homotopy groups are finite seems somehow more satisfactory than the previous more abstract general characterisations of trace spaces \citelist{\cite{Isobe_2003}\cite{Mazowiecka_VanSchaftingen_2023}}.
It would be interesting to investigate whether the ideas of the present work can provide more insightful characterisations when the homotopy groups \(\pi_{1}\brk{\manifold{N}}, \dotsc, \smash{\pi_{\floor{p - 1}} \brk{\manifold{N}}}\) are not necessarily finite or when \(p\) is an integer and \(\pi_{p - 1} \brk{\manifold{N}}\) is not necessarily trivial.
The approximation conditions of Theorems \ref{theorem_characterization_halfspace}, \ref{theorem_characterization_collar} and \ref{theorem_characterization_global} remain then necessary but are not sufficient.

\subsection{Linear estimates with supercritical integrability}
The methods developed in the present work also provide significantly improved estimates for the extension to Sobolev mappings when \(p > m\).
The classical construction of the extension of Sobolev mappings \cite{Bethuel_Demengel_1995} consists in first extending linearly the mapping \(u\colon  \manifold{M}' \to \manifold{N}\) to some function \(V \in \smash{\dot{W}}^{1, p} \brk{\manifold{M}'\times \intvo{0}{\infty}, \Rset^\nu}\) and noting that thanks to the Morrey-Sobolev embedding, one has 
for each \(\brk{x', x_m} \in \manifold{M}'\times \intvo{0}{\infty}\)
\begin{equation}
\label{eq_rahdu5ICh7pei4WieMixaeh7}
 \dist \brk{V \brk{x', x_m}, \manifold{N}}^p
 \le C  x_m^{p - m}
 \smashoperator{\iint_{\manifold{M}' \times \manifold{M}'}}
 \frac{d \brk{u \brk{y}, u \brk{z}}^p}{d \brk{y, z}^{p + m - 2}} \dif y \dif z,
\end{equation}
for some constant \(C \in \intvo{0}{\infty}\).
Therefore, for \(x_m \in \intvo{0}{\infty}\) small enough to ensure that the right-hand side of \eqref{eq_rahdu5ICh7pei4WieMixaeh7} is also small, \(V \brk{x', x_m}\) is close enough to \(\manifold{N}\) to perform a nearest-point projection, 
which combined with a suitable rescaling gives a mapping \(U \in \smash{\dot{W}}^{1, p} \brk{\manifold{M}'\times \intvo{0}{1}, \manifold{N}}\) such that \(\tr_{\manifold{M}' \times \set{0}} U = u\) and 
\begin{equation}  
\label{eq_Aikichoh3elai8iec7faehoh}
\int_{\manifold{M}' \times \intvr{0}{1}}
  \abs{\Deriv U}^p
  \le C'
  \smashoperator{
  \iint_{\manifold{M}' \times \manifold{M}'}}
  \frac{d \brk{u \brk{y}, u \brk{z}}^p}{d\brk{y, z}^{p + m - 2}} \dif y \dif z
  +\brk[\bigg]{ \smashoperator[r]{
  \iint_{\manifold{M}' \times \manifold{M}'}}
  \frac{d \brk{u \brk{y}, u \brk{z}}^p}{d\brk{y, z}^{p + m - 2}} \dif y \dif z}^{1 + \frac{1}{p - m}}.
\end{equation}
A key difference from the classical linear extension of traces of Sobolev functions is that the control \eqref{eq_Aikichoh3elai8iec7faehoh} on the extension of Sobolev mappings is \emph{superlinear}.
When one of the lower homotopy groups \(\pi_1 \brk{\manifold{N}}, \dotsc, \pi_{m - 1}\brk{\manifold{N}}\) is infinite, the same construction as for the analytical obstruction rules out a linear estimate \cite{Bethuel_2014} (see also \cite{Mironescu_VanSchaftingen_2021_AFST}*{Th.\thinspace{}1.10 (a)}).

When the lower homotopy groups satisfy the necessary finiteness condition, we construct nonlinear extensions that satisfy linear Sobolev estimates.

\medbreak

We first have a result on the half-space \(\Rset^m_+\).

\begin{theorem}
\label{theorem_estimate_supercritical_halfspace}
If \(p > m\), then there exists a constant \(C \in \intvo{0}{\infty}\) such that for every mapping  \(u \in \smash{\dot{W}}^{1 - 1/p, p}\brk{\Rset^{m - 1}, \manifold{N}}\) there is a mapping \(U \in \smash{\dot{W}}^{1, p}\brk{\Rset^m_+, \manifold{N}}\) such that \(\tr_{\Rset^{m - 1}} U = u\) and 
\begin{equation*}
  \smashoperator{\int_{\Rset^{m}_+}}
  \abs{\Deriv U}^p
  \le C
  \smashoperator{
  \iint_{\Rset^{m - 1} \times \Rset^{m - 1}}}
  \frac{d \brk{u \brk{y}, u \brk{z}}^p}{d\brk{y, z}^{p + m - 2}} \dif y \dif z
\end{equation*}
if and only if the homotopy groups \(\pi_1 \brk{\manifold{N}}, \dotsc, \pi_{m - 1} \brk{\manifold{N}}\) are finite.
\end{theorem}

The main differences between \cref{theorem_estimate_supercritical_halfspace} and its subcritical counterpart \cref{theorem_extension_halfspace} are that the condition of finiteness of the lower homotopy groups is no longer necessary for the \emph{existence} of an extension but remains so for an extension satisfying a \emph{linear estimate}, and that the number of homotopy groups involved in the condition no longer depends on the exponent \(p\).

\medbreak

We similarly get linear estimates on Sobolev extensions on collar neighbourhoods of the boundaries of manifolds.

\begin{manualtheorem}{\ref*{theorem_estimate_supercritical_halfspace}\thprime}
\label{theorem_estimate_supercritical_collar}
Let \(\manifold{M}'\) be an \((m - 1)\)-dimensional compact Riemannian manifold.
If \(p > m\), then there exists a constant \(C \in \intvo{0}{\infty}\) such that for every mapping \(u \in \smash{\smash{\dot{W}}^{1 - 1/p, p}} \brk{\manifold{M}', \manifold{N}}\) there is a mapping \(U \in \smash{\dot{W}}^{1, p}\brk{\manifold{M}' \times \intvr{0}{1}, \manifold{N}}\) such that \(\tr_{\manifold{M}' \times \set{0}} U = u\) and
\begin{equation}
\label{eq_mei0keegeef7aiVaiSu1aegi}
  \smashoperator[r]{\int_{\manifold{M}' \times \intvr{0}{1}}}
  \abs{\Deriv U}^p
  \le C
  \smashoperator{
  \iint_{\manifold{M}' \times \manifold{M}'}}
  \frac{d \brk{u \brk{y}, u \brk{z}}^p}{d\brk{y, z}^{p + m - 2}} \dif y \dif z
\end{equation}
if and only if the homotopy groups
\(\pi_1 \brk{\manifold{N}}, \dotsc, \pi_{m - 1} \brk{\manifold{N}}\) are finite.
\end{manualtheorem}

\medbreak

In the global case, the same proof as that of Theorems \ref{theorem_estimate_supercritical_halfspace} and \ref{theorem_estimate_supercritical_collar} would provide, when the lower homotopy groups
\(\pi_1 \brk{\manifold{N}}, \dotsc, \pi_{m - 1} \brk{\manifold{N}}\) are all finite, 
for every \(u \in \smash{\dot{W}}^{1 - 1/p, p}\brk{\partial \manifold{M}, \manifold{N}}\) which is the restriction to \(\partial \manifold{M}\) of a continuous map, an extension \(U \in \smash{\dot{W}}^{1, p}\brk{\manifold{M}, \manifold{N}}\) such that \(\tr_{\partial \manifold{M}} U = u\) and \eqref{eq_ho0SheiChathesi3fai4nei3} holds for some convex function \(\Theta \colon  \intvr{0}{\infty} \to \intvr{0}{\infty}\) satisfying \(\Theta \brk{0} = 0\).
However, as one can also simply obtain such a map \(U\) through \eqref{eq_Aikichoh3elai8iec7faehoh} and a compactness argument \citelist{\cite{Petrache_VanSchaftingen_2017}*{Th.\thinspace{}4}\cite{Petrache_Riviere_2014}*{Prop.\thinspace{}2.8}}, without relying on the finiteness of the homotopy groups \(\pi_1 \brk{\manifold{N}}, \dotsc, \pi_{m - 1} \brk{\manifold{N}}\), our methods do not improve on the satisfactory result obtained by the classical approach and we do not pursue the discussion further.

\subsection{Estimates with critical integrability}
In the critical case \(p=m\), there is no estimate such as \eqref{eq_rahdu5ICh7pei4WieMixaeh7}; the extension of Sobolev mappings is performed classically thanks to the vanishing mean oscillation property, and \emph{there is no Sobolev control} on the extension by the semi-norm of the trace, unless the latter is small enough: there exists \(\eta \in \intvo{0}{\infty}\)
 such that if the map \(u \colon \partial \manifold{M} \to \manifold{N}\) satisfies
\begin{equation}
\label{eq_queecaocooh6OhxaiTheel6f}
\smashoperator[r]{\iint_{\partial \manifold{M} \times \partial \manifold{M}}}
  \frac{d \brk{u \brk{y}, u \brk{z}}^m}{d\brk{y, z}^{2 m - 2}} \dif y \dif z
  \le \eta,
\end{equation}
then \(u = \tr_{\partial \manifold{M}} U\) for some mapping \(U \in \smash{\dot{W}}^{1, m} \brk{\manifold{M}, \manifold{N}}\) satisfying 
\begin{equation}
\label{eq_umooceikuoy9raey2ahThahb}
\int_{\manifold{M}} \abs{\Deriv U}^m
  \le 
  C
  \smashoperator{\iint_{\partial \manifold{M} \times \partial \manifold{M}}}
  \frac{d \brk{u \brk{y}, u \brk{z}}^m}{d\brk{y, z}^{2 m - 2}} \dif y \dif z. 
\end{equation}
For reasons similar to the critical analytical obstructions described above when \(p \in \set{2, \dotsc, m - 2}\), a necessary condition for an estimate of the form \eqref{eq_umooceikuoy9raey2ahThahb} without the smallness assumption \eqref{eq_queecaocooh6OhxaiTheel6f} is that the lower homotopy groups \(\pi_1 \brk{\manifold{N}}, \dotsc, \pi_{m - 2} \brk{\manifold{N}}\) are all \emph{finite} and the homotopy group \(\pi_{m - 1} \brk{\manifold{N}}\) is \emph{trivial} \cite{Mironescu_VanSchaftingen_2021_AFST}*{Th.\thinspace{}1.10 (b)}.

\medbreak

We prove that the obstructions described above are the only ones to the construction of Sobolev extensions with linear estimates on the half-space \(\Rset^m_+\).

\begin{theorem}
\label{theorem_estimate_critical_halfspace}
There exists a constant \(C \in \intvo{0}{\infty}\) such that for every mapping \(u \in \smash{\dot{W}}^{1 - 1/m, m}\brk{\Rset^{m - 1}, \manifold{N}}\) there is a mapping \(U \in \smash{\dot{W}}^{1, m}\brk{\Rset^m_+, \manifold{N}}\) such that \(\tr_{\Rset^{m - 1}} U = u\) and 
\begin{equation*}
  \int_{\Rset^m_+}
  \abs{\Deriv U}^m
  \le C
  \smashoperator{
  \iint_{\Rset^{m - 1} \times \Rset^{m - 1}}}
  \frac{d \brk{u \brk{y}, u \brk{z}}^m}{d\brk{y, z}^{2 m - 2}} \dif y \dif z
\end{equation*}
if and only if the homotopy groups \(\pi_1 \brk{\manifold{N}}, \dotsc, \pi_{m - 2} \brk{\manifold{N}}\) are finite and the homotopy group \(\pi_{m - 1}\brk{\manifold{N}}\) is trivial.
\end{theorem}

\medbreak 

Similarly, we get extensions with linear estimates in the collar neighbourhood case.

\begin{manualtheorem}{\ref*{theorem_estimate_critical_halfspace}\thprime}
\label{theorem_estimate_critical_collar}
Let \(\manifold{M}'\) be an \((m - 1)\)-dimensional compact Riemannian manifold.
There exists a constant \(C \in \intvo{0}{\infty}\) such that for every mapping \(u \in \smash{\dot{W}}^{1 - 1/m, m}\brk{\manifold{M}', \manifold{N}}\) there is a mapping \(U \in \smash{\dot{W}}^{1, m} \brk{\manifold{M}' \times \intvr{0}{1}, \manifold{N}}\) such that \(\tr_{\manifold{M}' \times \set{0}} U = u\) and 
\begin{equation}
\label{eq_eing7aeth4uk1yiemee4Wa4N}
  \smashoperator[r]{\int_{\manifold{M}' \times \intvr{0}{1}}}
  \abs{\Deriv U}^m
  \le C
  \smashoperator{
  \iint_{\manifold{M}' \times \manifold{M}'}}
  \frac{d \brk{u \brk{y}, u \brk{z}}^m}{d\brk{y, z}^{2 m - 2}} \dif y \dif z
\end{equation}
if and only if the homotopy groups \(\pi_1 \brk{\manifold{N}}, \dotsc, \pi_{m - 2} \brk{\manifold{N}}\) are finite and the homotopy group \(\pi_{m - 1}\brk{\manifold{N}}\) is trivial.
\end{manualtheorem}

\medbreak 

Finally, we get a sufficient condition for extensions satisfying a nonlinear estimate in the global case.

\begin{manualtheorem}{\ref*{theorem_estimate_critical_halfspace}\thsecond}
\label{theorem_estimate_critical_global}
Let \(\manifold{M}\) be an \(m\)-dimensional manifold with compact boundary \(\partial \manifold{M} \ne \emptyset\).
If the homotopy groups \(\pi_1 \brk{\manifold{N}}, \dotsc, \pi_{m - 2} \brk{\manifold{N}}\) are finite and if the homotopy group \(\pi_{m - 1}\brk{\manifold{N}}\) is trivial, then there exists a convex function \(\Theta \colon  \intvr{0}{\infty} \to \intvr{0}{\infty}\) such that \(\Theta \brk{0} = 0\)  and such that  for every mapping \(u \in \smash{\smash{\dot{W}}^{1 - 1/m, m}} \brk{\partial \manifold{M}, \manifold{N}}\) which is homotopic in \(\mathrm{VMO}\brk{\partial \manifold{M}, \manifold{N}}\) to the restriction to \(\partial \manifold{M}\) of a continuous map from \(\manifold{M}\) to \(\manifold{N}\), there is a mapping \(U \in \smash{\dot{W}}^{1, m}\brk{\manifold{M}, \manifold{N}}\) such that \(\tr_{\partial \manifold{M}} U = u\) and
\begin{equation}
\label{eq_saoN0uZeJor0leichieng7go}
  \int_{\manifold{M}}
  \abs{\Deriv U}^m
  \le  \Theta \brk[\bigg]{
  \smashoperator[r]{
  \iint_{\partial \manifold{M} \times \partial \manifold{M}}}
  \frac{d \brk{u \brk{y}, u \brk{z}}^m}{d\brk{y, z}^{2 m - 2}} + d \brk{u \brk{y}, u \brk{z}}^m\dif y \dif z}.
\end{equation} 
\end{manualtheorem}

In contrast to Theorems \ref{theorem_estimate_critical_halfspace} and \ref{theorem_estimate_critical_collar}, the conditions on the homotopy groups in \cref{theorem_estimate_critical_global} are not proved to be necessary anymore because the known proofs of the necessity of such conditions are known for linear estimates \cite{Mironescu_VanSchaftingen_2021_AFST}*{Th.\thinspace{}1.10 (b)} and  \eqref{eq_saoN0uZeJor0leichieng7go} is not linear.
On the other hand, the conclusions of \cref{theorem_estimate_critical_global}
cannot hold for a general target manifold \(\manifold{N}\): quite dramatically, 
if \(\manifold{M} = \smash{\widebar{\Bset}}^m\) and if the homotopy group \(\pi_{m - 1} \brk{\manifold{N}}\) is non-trivial, then it is not even possible to find an  extension whose Sobolev semi-norm is controlled nonlinearly by the Sobolev semi-norm of the trace; this motivated the construction of singular extensions with Sobolev-Marcinkiewicz estimates
in which some weak-type \(L^m\) quantity of the weak derivative is controlled \citelist{\cite{Petrache_Riviere_2014}\cite{Petrache_VanSchaftingen_2017}\cite{Bulanyi_VanSchaftingen_2025}}.

\subsection{Construction of the extension}
Before outlining how we construct the extension of traces of Sobolev mappings, let us explain why the existing strategies for constructing extensions do not seem to be adaptable to the case where the lower homotopy groups are merely finite rather than trivial.

Given a boundary value \(u \in \smash{\dot{W}}^{1 - 1/p, p} \brk{\partial \manifold{M}, \manifold{N}}\), Hardt and Lin first take in their original construction \cite{Hardt_Lin_1987} a linear extension \(V \in \smash{\smash{\dot{W}}^{1, p} \brk{\manifold{M}, \smash{\Rset^\nu}}}\) such that \(\tr_{\partial \manifold{M}} V = u\) but that does not necessarily satisfy the constraint to take its values in the target manifold \(\manifold{N}\).
The objective is then to suitably modify  \(V\) to satisfy the constraint on the target while keeping the same trace. 
Thanks to the assumption that the lower homotopy groups \(\pi_1 \brk{\manifold{N}}, \dotsc, \smash{\pi_{\floor{p - 1}}\brk{\manifold{N}}}\) are all trivial, there exists a \(\brk{\nu - \floor{p} - 1}\)-dimensional set \(\manifold{S} \subseteq \Rset^\nu\) and a mapping \(\Phi \in C^1 \brk{\Rset^\nu \setminus \manifold{S}, \manifold{N}}\) such that \(\Phi \restr{\manifold{N}} = \id\) and for every \(x \in \Rset^\nu \setminus \manifold{S}\),  \(\abs{\Deriv \Phi (x)}\le C/\dist\brk{x, \manifold{S}}\) \cite{Hardt_Lin_1987}*{Lem.\thinspace{}6.1}.
Because of the singularity, the map \(\Phi \compose V \colon \manifold{M} \to \manifold{N}\) does not necessarily belong to \(\smash{\dot{W}}^{1, p} \brk{\manifold{M}, \manifold{N}}\); however a Fubini-type argument shows that for an appropriate \(h \in \Rset^\nu\) small enough, one has \(\Phi_h \compose V \in \smash{\dot{W}}^{1, p} \brk{\manifold{M}, \manifold{N}}\) with \(\Phi_h \defeq \Phi \brk{\cdot - h}\) and one can then set \(U \defeq \brk{\Phi_h}\restr{\manifold{N}}^{-1} \compose \Phi_h \compose V\) since \(\Phi_h\restr{\manifold{N}}\) is a diffeomorphism and obtain a suitable Sobolev extension.
The nonlinear part of the construction of the extension is performed on the \emph{target side}, through the construction of \(\Phi\) and the composition of a perturbed version of \(\Phi\) with \(V\).
The topological arguments underlying the construction of \(\Phi\)  crucially exploit the triviality of the lower homotopy groups to define \(\Phi\) inductively on skeletons of increasing dimension up to \(\floor{p}\) before performing some singular homogeneous extensions; the triviality of these homotopy groups is actually necessary for the existence of such a \(\Phi\).

The treatment of the case where the fundamental group \(\pi_{1} \brk{\manifold{N}}\) is merely finite \cite{Mironescu_VanSchaftingen_2021_AFST} relies on a classical trick of passing through a lifting over the universal covering, which does not appear to admit a generalisation since there is no working analogue of the universal covering that would kill homotopy groups beyond the fundamental group. (The spaces appearing in the Whitehead tower (see \citelist{\cite{Hu_1959}*{\S V.8}\cite{Bott_Tu_1982}*{\S 18}}) are spaces of paths, with no natural manifold structure.)
%TODO References on Whitehead tower

The construction performed in this work uses a radically different approach, modifying the linear extension \(V\) on the domain side rather than on the target as it has been done since the work of  Hardt and Lin.
The relationship between our method and theirs can be compared to the one between Bethuel's strong approximation of Sobolev mappings \cite{Bethuel_1991} and that of Hajłasz \cite{Hajlasz_1994}, generalising the work of Bethuel and Zheng for spheres \cite{Bethuel_Zheng_1988}*{Th.\thinspace{}1}.

\medskip

We explain now our construction. 
To fix the ideas let us consider the case where we want to extend a mapping \(u \in \smash{\dot{W}}^{1 - 1/p, p} \brk{\Rset^{m - 1}, \manifold{N}}\) to some mapping \(U \in \smash{\dot{W}}^{1, p} \brk{\Rset^{m}_+, \manifold{N}}\).
We start our construction by taking as before a function \(V \in \smash{\dot{W}^{1, p} \brk{\Rset^{m}_+, \Rset^\nu}}\) which satisfies the trace condition \(\tr_{\Rset^{m - 1}} V = u\) \emph{but does not necessarily take its values in the target manifold \(\manifold{N}\).}

Since \(\manifold{N}\) is an isometrically embedded submanifold of the Euclidean space \(\Rset^{\nu}\), the nearest-point projection \(\Pi_{\manifold{N}}\) is well-defined and smooth in a neighbourhood of \(\manifold{N}\) in \(\Rset^\nu\).
It is natural to want to take for the extension \(U\) the mapping \(\Pi_{\manifold{N}} \compose V\) in the \emph{good set} \(\smash{\widehat{\Omega}}^{\mathrm{good}} \subseteq \Rset^{m}_+\) defined as the set of points where \(V\) is sufficiently close to \(\manifold{N}\) so that  \(\Pi_{\manifold{N}} \compose V\) is well-defined there (see \eqref{eq_BuCo1theeQuierowech2fai4} below).
The size of the remaining \emph{bad set} \(\smash{\widehat{\Omega}}^{\mathrm{bad}} = \Rset^m_+ \setminus \smash{\widehat{\Omega}}^{\mathrm{good}}\) can be controlled (see \eqref{eq_ZeeRa8Eiteech3mooxaeh6wu} below) as
\begin{equation}
\label{eq_ietiegeepiukaiGie9xah2Qu}
\int_{\smash{\widehat{\Omega}}^{\mathrm{bad}}}
 \frac{1}{x_m^p}
 \dif x 
 \le 
 \C \smashoperator{\iint_{\Rset^{m - 1} \times \Rset^{m - 1}}}
  \frac{d\brk{u \brk{y}, u \brk{z}}^p}{\abs{y - z}^{p + m - 2}} \dif y \dif z.
\end{equation}
In particular \eqref{eq_ietiegeepiukaiGie9xah2Qu} shows that the bad set is repelled by the boundary, with larger \(p\) producing a stronger repulsion.
The problem is then to extend the map \(\Pi_{\manifold{N}} \compose V\) to the bad set \(\smash{\widehat{\Omega}}^{\mathrm{bad}}\) --- or, more generally, as it will turn out to be necessary, to redefine \(\Pi_{\manifold{N}} \compose V\) on a small part of the  good set \(\smash{\widehat{\Omega}}^{\mathrm{good}}\) in such a way that the new map can be extended to the bad set \(\smash{\widehat{\Omega}}^{\mathrm{bad}}\) with Sobolev regularity and the same trace.

A  first rather technical difficulty with the bad set \(\smash{\widehat{\Omega}}^{\mathrm{bad}}\) is its \emph{lack of structure} as a merely open set.
Instead of the bad set \(\smash{\widehat{\Omega}}^{\mathrm{bad}}\), we consider the collection of \emph{bad cubes} \(\smash{\mathscr{Q}^{\mathrm{bad}}}\) of a Whitney decomposition of \(\Rset^m_+\) which intersect the bad set \(\smash{\widehat{\Omega}}^{\mathrm{bad}}\) (see \eqref{eq_definition_bad_cube} and \cref{figure_badsetcubes} below).
The size of the collection of bad cubes \(\smash{\mathscr{Q}^{\mathrm{bad}}}\) is controlled similarly to the bad set \(\smash{\smash{\widehat{\Omega}}^{\mathrm{bad}}}\) in \eqref{eq_ietiegeepiukaiGie9xah2Qu} as
\begin{equation}
\label{eq_eemi5be3ik4iej5Ang6jiewa}
\int_{\bigcup \mathscr{Q}^{\mathrm{bad}}}
 \frac{1}{x_m^p}
 \dif x 
 \le 
 \C \smashoperator{\iint_{\Rset^{m - 1} \times \Rset^{m - 1}}}
  \frac{d\brk{u \brk{y}, u\brk{z}}^p}{\abs{y - z}^{p + m - 2}} \dif y \dif z,
\end{equation}
provided that the linear extension \(V\) is done by convolution (see \eqref{eq_shee0Ceingooghe2weefei4A} below), which has been the most natural and usual way to perform an extension from the very beginning of the theory of traces in Sobolev spaces \cite{Gagliardo_1957}.
This part of the construction can be thought of as a counterpart for the extension of Sobolev mappings to Bethuel's construction of good and bad cubes for their approximation \cite{Bethuel_1991}.

Thus, given a collection of singular cubes \(\smash{\mathscr{Q}^{\mathrm{sing}}}\), we are led to construct a Sobolev extension from the boundary of its union \( \mathscr{Q}^{\mathrm{sing}}\). 
The key observation in \sectref{section_topological_to_sobolev} is that if the lower homotopy groups \(\pi_{1} \brk{\manifold{N}}, \dotsc, \smash{\pi_{\ell - 1} \brk{\manifold{N}}}\) are all \emph{finite} and if there exists a \emph{continuous extension} to the \(\ell\)-dimensional skeleton of the singular cubes with \(\ell \defeq \min \brk{\floor{p}, m}\), then there exists a \emph{Sobolev extension} \(U\) to \(\bigcup \mathscr{Q}^{\mathrm{sing}}\) such that
\begin{equation}
\label{eq_Ieseixai6eeS4of9oova5eiS}
\int_{\bigcup \mathscr{Q}^{\mathrm{sing}}}
\abs{\Deriv U}^p
\le 
\C 
\int_{\bigcup \mathscr{Q}^{\mathrm{sing}}}
 \frac{1}{x_m^p}
 \dif x 
\end{equation}
(see \cref{proposition_Lipschitz_extension_skeleton}).
\emph{This part of the proof is the only one where the finiteness assumption on the lower homotopy groups \(\pi_{1} \brk{\manifold{N}}, \dotsc, \pi_{\ell - 1} \brk{\manifold{N}}\) plays a role;} essentially it allows to obtain suitable finite constants as maxima among the finitely many homotopy classes to be considered when improving the continuous extension to a Sobolev one on each \(j\)-dimensional face when \(j \in \set{1, \dotsc, \ell - 1}\).

If stronger assumptions à la Hardt and Lin \cite{Hardt_Lin_1987} were satisfied, that is, if the homotopy groups \(\pi_1 \brk{\manifold{N}}, \dotsc, \pi_{\ell - 1} \brk{\manifold{N}}\) were all \emph{trivial}, then the construction of a Sobolev extension would be essentially complete.
Indeed, one could take \( \mathscr{Q}^{\mathrm{sing}} = \mathscr{Q}^{\mathrm{bad}}\) and the existence of a continuous extension to the \(\ell\)-dimensional skeleton of the singular cubes \( \mathscr{Q}^{\mathrm{sing}}\) would follow immediately from the assumption that the lower homotopy groups \(\pi_1 \brk{\manifold{N}}, \dotsc, \smash{\pi_{\ell - 1} \brk{\manifold{N}}}\) are all trivial.
Combining \eqref{eq_eemi5be3ik4iej5Ang6jiewa} and \eqref{eq_Ieseixai6eeS4of9oova5eiS} would then give the conclusion. 
Even in this case, the resulting construction of the extension, based on manipulation on the domain side, is radically different from the one of Hardt and Lin who work on the target side;
Zimmerman implemented this strategy with bad cubes only to construct Sobolev extensions of Lipschitz mappings into \(\brk{m - 1}\)-connected metric spaces \cite{Zimmerman_2019}.

\medbreak

In the general case where all the lower homotopy groups  \(\pi_1 \brk{\manifold{N}}, \dotsc, \pi_{\ell - 1} \brk{\manifold{N}}\) are assumed to be \emph{finite} rather than trivial, we will need to construct a \emph{larger collection of singular cubes} \( \mathscr{Q}^{\mathrm{sing}} \supseteq \mathscr{Q}^{\mathrm{bad}}\) which is \emph{large enough} to bear a continuous extension from the boundary of its union to its \(\floor{p}\)-dimensional skeleton but still \emph{small enough} to keep the right-hand side of  \eqref{eq_Ieseixai6eeS4of9oova5eiS} under control.
Our goal is to choose the singular cubes \(\smash{\mathscr{Q}^{\mathrm{sing}}}\) by adding to the bad cubes \(\smash{\mathscr{Q}^{\mathrm{bad}}}\) as few cubes as possible to simplify enough the topology of \(\bigcup \mathscr{Q}^{\mathrm{sing}}\) so that we can get a continuous extension on its \(\floor{p}\)-dimensional skeleton.
Although we will not be working in those terms, 
one can think of our aim as adding to the bad cubes a controlled small number of cubes so that the lower homotopy groups of the union of singular cubes become trivial.

The easiest case to construct such a collection of singular cubes is the \emph{supercritical} one \(p > m\) developed in \sectref{section_supercritical}. 
Singular cubes can then be constructed according to the following simple rule: they are defined recursively as the bad cubes together with all those that are adjacent to a singular cube at the previous scale (see \cref{figure_singretr2} and the proof of \cref{proposition_modifiedcubes_supercritical} below).
A given bad cube at the scale \(2^{k}\) will generate at most a single cube on any larger scale \(2^i\) with \(i \in \Zset\) and \(i > k\). The total contribution of these singular cubes to the integral on the right-hand side of \eqref{eq_Ieseixai6eeS4of9oova5eiS} is of the order of
\begin{equation}
\label{eq_Reing5aicaitaiM7he8ohsha}
 \sum_{i = k}^{\infty} 2^{i \brk{m - p}} = \frac{2^{k \brk{m-p}}}{1 - 2^{-\brk{p - m}}},
\end{equation}
which is comparable to the contribution \( 2^{k \brk{m - p}}\) of the original bad cube in the integral on the left-hand side of \eqref{eq_ietiegeepiukaiGie9xah2Qu}.
Summing up all the contributions of all the singular cubes that are generated in this way, we then get the estimate
\begin{equation}
\label{eq_ceikat4ooJ5eegheQuaa6eeg}
 \int_{\bigcup \mathscr{Q}^{\mathrm{sing}}}
 \frac{1}{x_m^p}
 \dif x 
 \le \C \int_{\bigcup \mathscr{Q}^{\mathrm{bad}}}
 \frac{1}{x_m^p}
 \dif x,
\end{equation}
which, together with \eqref{eq_eemi5be3ik4iej5Ang6jiewa} and \eqref{eq_Ieseixai6eeS4of9oova5eiS}, gives the supercritical extension result in the half-space \cref{theorem_estimate_supercritical_halfspace}.
Incidentally, the estimates \eqref{eq_Reing5aicaitaiM7he8ohsha} and \eqref{eq_ceikat4ooJ5eegheQuaa6eeg} deteriorate as \(p \searrow m\) and fail in the limit \(p = m\); as explained later, in this case it will be necessary to tweak the argument and assume that the homotopy group \(\pi_{m - 1} \brk{\manifold{N}}\) is trivial.

We now consider the \emph{subcritical} case \(1 < p \le m\).
First, we note that defining singular cubes as we did when \(p > m\) would not work since the sum of the geometric series in \eqref{eq_Reing5aicaitaiM7he8ohsha} would diverge and thus \eqref{eq_ceikat4ooJ5eegheQuaa6eeg} would not hold.

\begin{figure}
\begin{center}
 \includegraphics{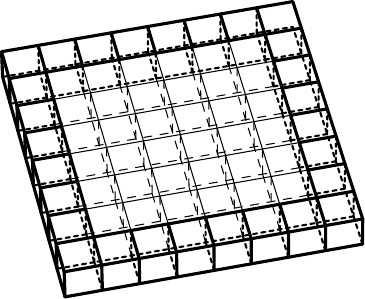}
 \includegraphics{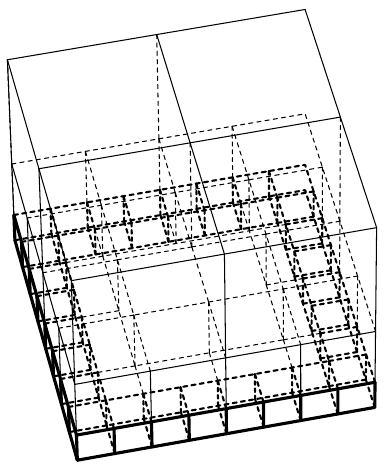}
\end{center}
\caption{Given the collection of bad cubes (with thick edges in both drawings), the \emph{slab construction} on the left fills the cubic toroid horizontally with singular cubes of the same scale; the size of the resulting singular set (see \eqref{eq_theebeipiirequoNa4Iethop}) is asymptotically much larger than that of the original collection of bad cubes (see \eqref{eq_diesheeTeeh3eepus5ohngoo}).
The \emph{dome construction} on the right covers the cubic toroid of bad cubes with successive cubic toroids at larger scales up to the scale where the singular set becomes simply connected; the size of the resulting singular set (see \eqref{eq_lohph8ahL2yuegootheephee}) is asymptotically comparable to the size of the original collection of bad cubes (see \eqref{eq_diesheeTeeh3eepus5ohngoo}).
}
\label{figure-slab-dome}
\end{figure}
Next, in order to fix the ideas and illustrate the situation in spaces of reasonable dimension, let us assume that we want to extend a map \(u \in C^1 \brk{\Rset^2, \manifold{N}}\cap \smash{\dot{W}}^{1 - 1/p, p} \brk{\Rset^2, \manifold{N}}\) to some \(U \in \smash{\dot{W}}^{1, p} \brk{\Rset^3_+, \manifold{N}}\) with \(2 < p < 3\), as in \cref{theorem_characterization_halfspace}.
Then we have to find a collection of singular cubes \(\smash{\mathscr{Q}^{\mathrm{sing}}}\) containing the bad cubes \(\smash{\mathscr{Q}^{\mathrm{bad}}}\) and such that we have a continuous extension to its \(2\)-dimensional skeleton.
For example, the bad cubes could form a cubic toroid in \(\Rset^3_+\); 
in general, a continuous map on the boundary of this cubic toroid cannot be extended to the \(2\)-dimensional faces that it contains.
More concretely, imagine that this cubic toroid consists of \(4 \cdot \brk{2^{j} - 1}\) cubes of edge-length \(2^k\) arranged along a square of side \(2^{j + k}\), corresponding to the cubes with thicker edges in \cref{figure-slab-dome}.
The contribution of this cubic toroid to the size of the collection of bad cubes as in \eqref{eq_eemi5be3ik4iej5Ang6jiewa} is of the order of
\begin{equation}
\label{eq_diesheeTeeh3eepus5ohngoo}
   4\brk{2^{j} - 1} 2^{k \brk{3 - p}} \simeq 
   4 \cdot 2^{j + k \brk{3 - p}}.
\end{equation}
One way of ensuring the existence of a continuous extension to the two-dimensional faces of the singular cubes is to have the singular cubes form a contractible set. 
There are two ways to achieve this:
the \emph{slab construction} in which cubes of edge-length \(2^k\) are used to fill in the square (as illustrated on the left of \cref{figure-slab-dome}), resulting in a contribution to the size of the collection of singular cubes of the order of
\begin{equation}
\label{eq_theebeipiirequoNa4Iethop}
 2^{2 j + k \brk{3 - p}},
\end{equation}
and the \emph{dome construction} which consists in taking as singular cubes all the cubes that are either bad or that are adjacent to singular cubes at the previous scale, and stopping at the scale \(2^{k + j - 1}\) where the singular set becomes contractible (as illustrated on the right of \cref{figure-slab-dome}); the resulting collection of singular cubes then has a size of the order of
\begin{equation}
\label{eq_lohph8ahL2yuegootheephee}
   \smashoperator{\sum_{i = k}^{k + j - 1}} 4\brk{2^{j + k - i} - 1} 2^{i \brk{3 - p}}
   \simeq 
   \smashoperator{\sum_{i = k}^{k + j - 1}} 4\brk{2^{j + k - i\brk{p - 2}}} \le \frac{4 \cdot 2^{j + k \brk{3 - p}}}{1 - 2^{-\brk{p - 2}}}.
\end{equation}
It turns out that while \eqref{eq_diesheeTeeh3eepus5ohngoo} does not control the size of the collection of singular cubes of the slab construction \eqref{eq_theebeipiirequoNa4Iethop}, it does control that of the dome construction \eqref{eq_lohph8ahL2yuegootheephee}, with a control sufficient to achieve \eqref{eq_ceikat4ooJ5eegheQuaa6eeg}.

In general the collection of bad cubes will not have such a simple toroidal structure as the example just outlined, but it is still possible to construct reasonably small collections of singular cubes.
Basically we construct vertical singular faces, by first defining vertical bad faces as vertical faces separating two bad cubes and defining vertical singular faces as either vertical bad faces or vertical faces adjacent to a singular vertical face at the previous scale. 
Although vertical singular faces may propagate vertically indefinitely, it is possible, by an \emph{averaging argument} through the \(2^2 = 4\) different possible configurations of a layer of cubes at each scale, to choose a good dyadic decomposition of \(\Rset^3_+\) such that the size of the collection of singular cubes \(\smash{\mathscr{Q}^{\mathrm{sing}}}\) is controlled as in \eqref{eq_ceikat4ooJ5eegheQuaa6eeg} (see \cref{proposition_general_modification} below), which together with \eqref{eq_ietiegeepiukaiGie9xah2Qu} and \eqref{eq_Ieseixai6eeS4of9oova5eiS} gives then the conclusion for smooth maps (see \cref{theorem_extension_tent_subcritical} below).
Finally, if the homotopy group \(\pi_1 \brk{\manifold{N}}\) is trivial, then \(C^1 \brk{\Rset^2, \manifold{N}}\cap \smash{\smash{\dot{W}}^{1 - 1/p, p}} \brk{\Rset^2, \manifold{N}}\) is strongly dense in the space \(\smash{\smash{\dot{W}}^{1 - 1/p, p}} \brk{\Rset^2, \manifold{N}}\) \citelist{\cite{Bethuel_1995}\cite{Brezis_Mironescu_2015}\cite{Mucci_2009}\cite{Mucci_2024}} so that \cref{theorem_extension_halfspace} follows by a standard approximation and compactness argument.

The construction just described can be adapted to cover the \emph{critical case} \(p = 3\). 
Indeed, the collection of the singular cubes constructed still satisfies the estimate \eqref{eq_ceikat4ooJ5eegheQuaa6eeg} and there is a continuous extension to the faces of the singular cubes, which is not sufficient to get a Sobolev extension when \(p = 3\). 
  However, the additional assumption that \(\pi_{2} \brk{\manifold{N}}\) is trivial can be used to perform a further continuous extension to the interior of the singular cubes, from which one can get a Sobolev extension.

As developed in \sectref{section_critical} (\(p = m\)) and \sectref{section_subcritical} (\(p<m\)), the approach outlined above works for any dimension \(m \in \Nset \setminus \set{0, 1}\) and any \(p \in \intvl{1}{m}\), under the additional assumption that the homotopy group \(\pi_{p - 1} \brk{\manifold{N}}\) is trivial when \(p\) is an integer which allows to upgrade the continuous extension to the \(\brk{p - 1}\)-dimensional skeleton provided by the construction of the singular cubes into a continuous extension to a \(p\)-dimensional skeleton, giving the proofs of Theorems \ref{theorem_extension_halfspace}, \ref{theorem_characterization_halfspace} and \ref{theorem_estimate_critical_halfspace}.

\medbreak 

The extension to a \emph{collar neighbourhood} \(\manifold{M}’\times \intvr{0}{1}\) (Theorems \ref{theorem_extension_collar}, \ref{theorem_characterization_collar}, \ref{theorem_estimate_supercritical_collar} and \ref{theorem_estimate_critical_collar}) is carried out in \sectref{section_collar} by decomposing the domain into suitable subsets on which the construction of the half-space works. Since the first version of the present work, an alternative proof through a localisation principle has been published \cite{VanSchaftingen_2026}, allowing to deduce these from Theorems \ref{theorem_extension_halfspace}, \ref{theorem_characterization_halfspace}, \ref{theorem_estimate_supercritical_halfspace} and  \ref{theorem_estimate_critical_halfspace}.

The construction of the \emph{global extension} (Theorems \ref{theorem_extension_global}, \ref{theorem_characterization_global} and \ref{theorem_estimate_critical_global}) in \sectref{section_global} combines the extension to a collar neighbourhood and suitable compactness arguments for the extensions of nice mappings whose successive traces on skeletons lie in first-order Sobolev spaces.

We close our work in \sectref{section_necessity_approximation} with proofs of the necessity of the approximation conditions in the characterisations of traces of Theorems \ref{theorem_characterization_halfspace}, \ref{theorem_characterization_collar} and \ref{theorem_characterization_global}.

This work also includes  an appendix on the approximation of maps in \(\smash{\dot{W}}^{s, p} \brk{\Rset^m, \manifold{N}}\) by mappings that are constant outside a compact set for \(0 < s \le 1\) and \(sp \le m\) (Appendix \ref{section_approximation_compact}).

\section{Extension outside a controlled bad set}
\label{section_badset}
\resetconstant

\subsection{Good set and bad set}
\begin{figure}
\includegraphics{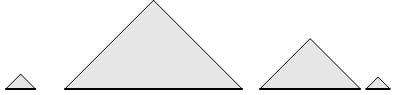}
\caption{The grey region represents the tent set \(\smash{\widehat{\Omega}}\) of \(\Omega\) depicted as a thick black line, as defined in \eqref{eq_epheiY6eilo4phaichoC9ton}.}
\end{figure}
It is natural to perform extensions on the \emph{tent} \(\smash{\widehat{\Omega}}   \subseteq \Rset^m_+\) over a given open set \(\Omega \subseteq \Rset^{m - 1}\) which is defined as 
\begin{equation}
\label{eq_epheiY6eilo4phaichoC9ton}
 \widehat{\Omega}
 \defeq 
 \set{ \brk{x', x_m} \in \Rset^m_+ \st 
 \overline{\Bset^{m - 1}_{x_m} \brk{x'}} \subseteq \Omega},
\end{equation}
where \(\overline{\Bset^\ell_r \brk{a}}\) is the closed ball of radius \(r\) centred at \(a\) in \(\Rset^\ell\).
Following linear extension methods, every Sobolev function \(u \in \smash{\dot{W}^{1 - 1/p, p} \brk{\Omega, \Rset^\nu}}\) is the trace of a function \(V \in W^{1, p} \brk{\smash{\widehat{\Omega}}, \Rset^\nu}\) constructed by averaging,
that is, given a fixed function \(\varphi \in C^\infty (\Rset^{m - 1}, \Rset)\) such that \(\int_{\Rset^{m - 1}} \varphi = 1\) and \(\supp \varphi \subseteq \Bset^{m - 1}\), the function \(V \colon \smash{\widehat{\Omega}} \to\Rset^\nu\) is defined for each \(x = \brk{x', x_m} \in \smash{\widehat{\Omega}}\) by 
\begin{equation}
\label{eq_shee0Ceingooghe2weefei4A}
 V \brk{x} \defeq \frac{1}{x_{m}^{m - 1}} \int_{\Omega} u \brk{z} \,\varphi \brk[\big]{\tfrac{x' - z}{x_{m}}} \dif z 
 = \int_{\Rset^{m - 1}} u\brk{x' - x_{m} y} \, \varphi\brk{y} \dif y.
\end{equation}
Following Gagliardo's classical construction and estimate of extension of traces in Sobolev spaces \cite{Gagliardo_1957} (see also \citelist{\cite{Leoni_2023}*{Th.\thinspace{}9.4}\cite{Mazya_2011}*{\S 10.1.1 Th.\thinspace{}1}\cite{DiBenedetto_2016}}), there exists a constant \(\Cl{cst_uboh3EeShaiHoo8Geehew3Ei} \in \intvo{0}{\infty}\) such that for every function \(u \in \smash{\dot{W}}^{1 - 1/p, p} \brk{\Omega, \Rset^\nu}\),
the function \(V\) defined by \eqref{eq_shee0Ceingooghe2weefei4A} belongs to the Sobolev space \(\smash{\dot{W}}^{1, p} \brk{\smash{\widehat{\Omega}}, \Rset^\nu}\), and satisfies the condition \(\tr_{\Omega \times \set{0}} V = u\) and the inequality
\begin{equation}
\label{eq_oojaeShoo1egh8shaid5sais}
  \int_{\widehat{\Omega}} \abs{\Deriv V}^p
  \le
  \Cr{cst_uboh3EeShaiHoo8Geehew3Ei}
  \smashoperator{\iint_{\Omega \times \Omega}}
  \frac{\abs{u (y) - u (z)}^p}{\abs{y - z}^{p + m - 2}} \dif y \dif z.
\end{equation}

The boundary datum \(u \in \smash{\dot{W}}^{1 - 1/p, p} \brk{\Omega, \manifold{N}}\) that we are interested in satisfies the nonlinear constraint that \(u \in \manifold{N} \subseteq \Rset^\nu\) almost everywhere on \(\Omega\).
Since the manifold \(\manifold{N}\) is  isometrically embedded into the Euclidean space \(\Rset^\nu\), 
the Euclidean distance in \(\Rset^\nu\) is controlled by the geodesic distance in \(\manifold{N}\) and \eqref{eq_oojaeShoo1egh8shaid5sais} readily implies 
\begin{equation}
\label{eq_joos6eteeB6aSh5pa2ahnga9}
   \int_{\widehat{\Omega}} \abs{\Deriv V}^p
  \le
  \Cr{cst_uboh3EeShaiHoo8Geehew3Ei}
  \smashoperator{\iint_{\Omega \times \Omega}}
  \frac{d\brk{u (y), u (z)}^p}{\abs{y - z}^{p + m - 2}} \dif y \dif z.
\end{equation}
Since the manifold \(\manifold{N}\) is also compact, it is a bounded submanifold of  \(\Rset^\nu\) and it follows from the definition of \(V\) in \eqref{eq_shee0Ceingooghe2weefei4A} that for every \(x = \brk{x', x_m} \in \smash{\widehat{\Omega}}\), we also have the pointwise gradient estimate
\begin{equation}
\label{eq_Di5lainai8xoa5ruacuoxeih}
  \abs{\Deriv V \brk{x}} \le \C \frac{\diam \brk{\manifold{N}}}{x_m}.
\end{equation}

\medbreak

Since the submanifold \(\manifold{N}\) is not a convex subset of \(\Rset^\nu\), in general we do not have \(V \in \manifold{N}\) in \(\smash{\widehat{\Omega}}\). 
A classical technique to recover the manifold constraint for Sobolev mappings is the \emph{nearest-point projection} \(\Pi_{\manifold{N}} \colon \manifold{N} + \smash{\overline{\Bset^\nu_{\delta_{\manifold{N}}}}} \to \manifold{N}\) which is well-defined and smooth provided \(\delta_{\manifold{N}} \in \intvo{0}{\infty}\) is small enough \cite{Moser_2005}*{\S 3.1}. 
Defining the \emph{good set}
\begin{equation}
\label{eq_BuCo1theeQuierowech2fai4}
  \smash{\widehat{\Omega}}^{\mathrm{good}} \defeq  \set{x \in \widehat{\Omega} \st  \dist (V\brk{x}, \manifold{N})  \le \delta_{\manifold{N}}}, 
\end{equation}
and the \emph{bad set} 
\begin{equation}
\label{eq_vuaPh9vaengoo9aiHeeshaer}
  \smash{\widehat{\Omega}}^{\mathrm{bad}} \defeq \set{x \in \widehat{\Omega} \st  \dist (V\brk{x}, \manifold{N}) > \delta_{\manifold{N}}},
\end{equation}
so that the tent set \(\smash{\widehat{\Omega}}\) is the disjoint union of \(\smash{\widehat{\Omega}}^{\mathrm{good}}\) and \(\smash{\widehat{\Omega}}^{\mathrm{bad}}\), we get that \(\Pi_{\manifold{N}} \compose \smash{V \restr{\smash{\widehat{\Omega}}^{\mathrm{good}}}} \in \smash{\dot{W}}^{1, p} (\smash{\widehat{\Omega}}^{\mathrm{good}}, \manifold{N})\) and that 
\begin{equation}
\label{eq_xeikeipak0Bi5ie8eethue9c}
 \int_{\smash{\widehat{\Omega}}^{\mathrm{good}}} \abs{\Deriv \brk{\Pi_{\manifold{N}} \compose \smash{V \restr{\smash{\widehat{\Omega}}^{\mathrm{good}}}}}}^p 
 \le \C \smashoperator{\iint_{\Omega \times \Omega}} \frac{\abs{u \brk{y} - u \brk{z}}^p}{\abs{y - z}^{p + m - 2}} \dif y \dif z.
\end{equation}

The challenge of constructing a manifold-constrained Sobolev extension will be to define it on the bad set \(\smash{\smash{\widehat{\Omega}}^{\mathrm{bad}}}\) --- or, as it will turn out, on a larger singular set containing \(\smash{\smash{\widehat{\Omega}}^{\mathrm{bad}}}\)  --- by suitably extending the mapping \(\Pi_{\manifold{N}} \compose \smash{V \restr{\smash{\widehat{\Omega}}^{\mathrm{good}}}}\).
In this approach, it is crucial to guarantee that the Sobolev energy of the extension stays finite and that its trace remains the given mapping \(u\).

\medbreak 

In order to define an extension on the bad set \(\smash{\widehat{\Omega}}^{\mathrm{bad}}\), we first estimate the \emph{size} of this set.
A first estimate can be obtained from the Sobolev estimate on the extension \eqref{eq_joos6eteeB6aSh5pa2ahnga9} via the Chebyshev-Markov inequality, the one-dimensional Hardy inequality and the chain rule:
\begin{equation}
\label{eq_ZeeRa8Eiteech3mooxaeh6wu}
\begin{split}
\int_{\smash{\widehat{\Omega}}^{\mathrm{bad}}}
 \frac{1}{x_m^p}
 \dif x 
 & \le \frac{1}{\delta_{\manifold{N}}^p}
 \int_{\widehat{\Omega}} \frac{\dist \brk{V \brk{x}, \manifold{N}}^p}{x_m^p} \dif x
 \le \frac{p^p}{\brk{p - 1}^p \delta_{\manifold{N}}^p}
 \int_{\widehat{\Omega}} \abs{\Deriv \dist \brk{V , \manifold{N}}}^p 
 \\
 &
 \le \frac{p^p}{\brk{p - 1}^p \delta_{\manifold{N}}^p}
 \int_{\widehat{\Omega}} \abs{\Deriv V}^p 
 \le 
 \frac{\Cr{cst_uboh3EeShaiHoo8Geehew3Ei} p^p}{\brk{p - 1}^p \delta_{\manifold{N}}^p}
  \smashoperator{\iint_{\Omega \times \Omega}}
  \frac{d\brk{u \brk{y}, u \brk{z}}^p}{\abs{y - z}^{p + m - 2}} \dif y \dif z.
\end{split}
\end{equation}
As a consequence of \eqref{eq_ZeeRa8Eiteech3mooxaeh6wu} and of the next criterion for equality of traces, any extension \(U\) of \(\Pi_{\manifold{N}} \compose \smash{V \restr{\smash{\widehat{\Omega}}^{\mathrm{good}}}}\) to a subset of \(\smash{\widehat{\Omega}}\) should have \(u\) as a trace.

\begin{proposition}
\label{proposition_equal_traces}
Let \(U, V \in \smash{\dot{W}}^{1, p} \brk{\smash{\widehat{\Omega}}, \Rset^\nu} \cap L^{\infty} \brk{\smash{\widehat{\Omega}}, \Rset^\nu}\) and assume that \(\tr_{\Omega \times \set{0}} V \in \manifold{N}\) almost everywhere in \(\Omega\).
If
\begin{equation}
\label{eq_chi8zaivao7Shahmob3eimah}
 \lim_{\varepsilon \to 0} 
 \frac{\mathcal{L}^{m} \brk{A_\varepsilon}}{\varepsilon}
 = 0,
\end{equation}
where
\begin{equation}
\label{eq_eiquee2Iax3mei2OobaeXooZ}
 A_\varepsilon \defeq \set{\brk{x', x_m} \in \widehat{\Omega} \st x_m < \varepsilon \text{ and }
 \abs{U (x)- V(x)} > \dist \brk{V (x), \manifold{N}}}
\end{equation}
then
\begin{align*}
  \tr_{\Omega \times \set{0}} U &= \tr_{\Omega \times \set{0}} V
  &&\text{almost everywhere on \(\Omega \times \set{0}\).}
\end{align*}
\end{proposition}

Here and in the sequel, \(\mathcal{L}^m\) denotes Lebesgue's measure on the Euclidean space \(\Rset^m\).

If \(U\) coincides with \(\Pi_{\manifold{N}} \compose V\) on \(\smash{\widehat{\Omega}}^{\mathrm{good}}\), then \(A_\varepsilon \subseteq \smash{\widehat{\Omega}}^{\mathrm{bad}}\), and the condition \eqref{eq_chi8zaivao7Shahmob3eimah} follows from the estimate \eqref{eq_ZeeRa8Eiteech3mooxaeh6wu}; the statement of \cref{proposition_equal_traces} leaves the possibility of taking an extension \(U\) of the restriction of \(\Pi_{\manifold{N}} \compose V\) to a subset of the good set \(\smash{\widehat{\Omega}}^{\mathrm{good}}\), provided that the condition \eqref{eq_chi8zaivao7Shahmob3eimah} still holds.

\begin{proof}[Proof of \cref{proposition_equal_traces}]
Setting \(W \defeq U - V\) and \(w \defeq \tr_{\Omega \times \set{0}} W\), we have by the \(L^p\) theory of traces (see for example \citelist{\cite{Brezis_2011}*{Lem.\thinspace{}9.9}\cite{DiBenedetto_2016}*{Prop.\thinspace{}16.2}\cite{Willem_2013}*{Lem.\thinspace{}6.2.2}}) and a cut-off argument, if the set \(K \subseteq \Omega\) is compact and if \(\varepsilon \in \intvo{0}{\infty}\) is small enough,
\begin{equation}
\label{eq_us2jief1iafuleeJahtabaem}
 \int_{K} \abs{w}^p
 \le 
 \C 
 \brk[\bigg]{\varepsilon^{p - 1} \int_{K \times \intvo{0}{\varepsilon}} \abs{\Deriv W}^p 
 + \frac{1}{\varepsilon} \int_{K \times \intvo{0}{\varepsilon}} \abs{W}^p
 }.
\end{equation}
Moreover, we have by convexity
\begin{equation}
\label{eq_coeGhaogh7kaeF3Phieh8ugh}
\int_{K \times \intvo{0}{\varepsilon}} \abs{\Deriv W}^p 
\le 2^{p - 1}\int_{K \times \intvo{0}{\varepsilon}} \abs{\Deriv U}^p + \abs{\Deriv V}^p, 
\end{equation}
and by definition of the set \(A_\varepsilon\) in \eqref{eq_eiquee2Iax3mei2OobaeXooZ} and by Hardy's inequality, since \(\tr_{K \times \set{0}} \dist \brk{V, \manifold{N}} = 0\) on \(K \times \set{0}\),
\begin{equation}
\label{eq_jaiy1nahf9zae8thohVahqui}
\begin{split}
\int_{K \times \intvo{0}{\varepsilon}} \abs{W}^p 
&\le \int_{K \times \intvo{0}{\varepsilon}} \dist \brk{V, \manifold{N}}^p
+ \norm{U - V}_{L^\infty\brk{\smash{\widehat{\Omega}}}}^p \mathcal{L}^m \brk{A_\varepsilon} \\
&\le \frac{p^p \varepsilon^{p}}{\brk{p - 1}^p} \int_{K \times \intvo{0}{\varepsilon}} \abs{\Deriv V}^p
+ \norm{U - V}_{L^\infty\brk{\smash{\widehat{\Omega}}}}^p \mathcal{L}^m \brk{A_\varepsilon}.
\end{split}
\end{equation}
Inserting \eqref{eq_coeGhaogh7kaeF3Phieh8ugh} and \eqref{eq_jaiy1nahf9zae8thohVahqui} in \eqref{eq_us2jief1iafuleeJahtabaem} we get
\begin{equation}
\label{eq_Phiach7oadeiteequee8eyo9}
 \int_{K} \abs{w}^p
 \le 
 \C \brk[\bigg]{\varepsilon^{p - 1} \int_{K \times \intvo{0}{\varepsilon}} \brk[\big]{\abs{\Deriv U}^p + \abs{\Deriv V}^p} + \norm{U - V}_{L^\infty\brk{\smash{\widehat{\Omega}}}}^p \frac{\mathcal{L}^m \brk{A_\varepsilon}}{\varepsilon}}.
\end{equation}
Letting \(\varepsilon \to 0\) in  \eqref{eq_Phiach7oadeiteequee8eyo9}, we reach the conclusion.
\end{proof}

\medbreak

Although the estimate \eqref{eq_ZeeRa8Eiteech3mooxaeh6wu} provides some quantitative information about the size of the bad set \(\smash{\widehat{\Omega}}^{\mathrm{bad}}\), its formulation and its proof are not versatile enough to be applied to enlargements of the bad set \(\smash{\widehat{\Omega}}^{\mathrm{bad}}\).
We will rely on a more subtle analysis of the bad set, based on the fundamental observation that the distance to the target manifold is controlled by the \emph{mean oscillation}, which arose in the approximation of Sobolev mappings with critical integrability by Schoen and Uhlenbeck \citelist{\cite{Schoen_Uhlenbeck_1982}*{\S 3}\cite{Schoen_Uhlenbeck_1983}*{\S 4}} and was systematised by Brezis and Nirenberg \citelist{\cite{Brezis_Nirenberg_1995}\cite{Brezis_Nirenberg_1996}}.

This method starts from the estimate that for each \(x = \brk{x', x_m} \in \smash{\widehat{\Omega}}\)
\begin{equation}
\label{eq_phaejateiGh6ceir1kohngis}
\begin{split}
\dist (V \brk{x}, \manifold{N})
&\le \Cl{cst_ahbie4ooreng8Maidoo0WaiG} \meanosc_0 \brk{u} \brk{x', x_m}\\
&\le \Cr{cst_ahbie4ooreng8Maidoo0WaiG} \brk{\delta + \meanosc_\delta \brk{u} \brk{x', x_m}},
\end{split}
\end{equation}
where, for each \(\delta \in \intvr{0}{\infty}\) we have defined the \emph{truncated mean oscillation}  \(\meanosc_\delta \brk{u}\colon \smash{\widehat{\Omega}} \to \intvr{0}{\infty}\) of \(u\) at the point \(x = \brk{x', x_m} \in \smash{\widehat{\Omega}}\) as 
\begin{equation}
\label{eq_moo7thiequ2zug8ahSoh3oth}
 \meanosc_\delta \brk{u} \brk{x} \defeq \frac{1}{\mathcal{L}^{m-1} (\Bset^{m - 1}_{x_m} \brk{x'})^2}
 \smashoperator{\iint_{\brk{\Bset^{m - 1}_{x_m} \brk{x'}}^2}} \brk{d \brk{u\brk{y}, u\brk{z}} - \delta}_+ \dif y \dif z,
\end{equation}
with \(t_+ \defeq \max\set{t, 0}\) denoting the positive part of \(t \in \Rset\).
Setting 
\begin{equation*}
 \delta_* \defeq \frac{\delta_{\manifold{N}}}{2 \Cr{cst_ahbie4ooreng8Maidoo0WaiG}},
\end{equation*}
we have then by \eqref{eq_phaejateiGh6ceir1kohngis} and by definition of the good set \(\smash{\widehat{\Omega}}^{\mathrm{good}}\) in \eqref{eq_BuCo1theeQuierowech2fai4} and of the bad set \(\smash{\widehat{\Omega}}^{\mathrm{bad}}\) in \eqref{eq_vuaPh9vaengoo9aiHeeshaer} the inclusions
\begin{equation}
  \label{eq_ahhae0phoTh6iekais4eghae}
 \smash{\widehat{\Omega}}^{\mathrm{good}}
 \supseteq 
 \set{x \in \widehat{\Omega} \st \meanosc_{\delta_*} \brk{u} \brk{x} \le \delta_*}
\end{equation}
and 
\begin{equation}
\label{eq_haw0fai3ohlie0Eitha9Aes2}
 \smash{\widehat{\Omega}}^{\mathrm{bad}}
 \subseteq 
 \set{x \in \widehat{\Omega} \st \meanosc_{\delta_*} \brk{u} \brk{x} > \delta_*}.
\end{equation}
In particular, it follows from \eqref{eq_haw0fai3ohlie0Eitha9Aes2}, \eqref{eq_moo7thiequ2zug8ahSoh3oth} and Fubini's theorem that
\begin{equation}
\label{eq_Ahka5kieng3ooGhouFaew5au}
\begin{split}
\int_{\smash{\widehat{\Omega}}^{\mathrm{bad}}}
 \frac{1}{x_m^p}
 \dif x 
 &\le \C \int_{\widehat{\Omega}} \frac{1}{x_{m}^{p + 2m - 2}} \smashoperator{\iint_{\brk{\Bset^{m - 1}_{x_m} \brk{x'}}^2}} \brk{d\brk{u \brk{y}, u\brk{z}} -  \delta_*}_+ \dif y \dif z  \dif x\\
 &\le \Cl{cst_cho2aiXo3ahHoohi9eiyoh4e} \smashoperator{\iint_{\Omega \times \Omega}} \frac{\brk{d \brk{u \brk{y}, u \brk{z}} - \delta_*}_+}{\abs{y - z}^{p + m - 2}} \dif y \dif z,
\end{split}
\end{equation}
and therefore,  since for every \(t \in \intvr{0}{\infty}\) one has \(\brk{t - \delta_*}_+ \le t^p/\delta_*^{p - 1}\),
\begin{equation}
\label{eq_eeM3miequieno5oom4Layav5}
 \int_{\smash{\widehat{\Omega}}^{\mathrm{bad}}}
 \frac{1}{x_m^p}\dif x
 \le \frac{\Cr{cst_cho2aiXo3ahHoohi9eiyoh4e}}{\delta_*^{p - 1}} \smashoperator{\iint_{\Omega \times \Omega}} \frac{d \brk{u \brk{y}, u \brk{z}}^p}{\abs{y - z}^{p + m - 2}} \dif y \dif z,
\end{equation}
similarly to what was given by  \eqref{eq_ZeeRa8Eiteech3mooxaeh6wu}; the inclusion \eqref{eq_haw0fai3ohlie0Eitha9Aes2} thus encodes all the quantitative information we have about the bad set \(\smash{\widehat{\Omega}}^{\mathrm{bad}}\) at this point.

We finally note that by \eqref{eq_moo7thiequ2zug8ahSoh3oth} and by Hölder's inequality
\begin{equation}
\label{eq_sixaika7eitie6deLaem4ae0}
  \meanosc_\delta \brk{u} \brk{x} \le \C\, x_m^{1 - \frac{m}{p}} \brk[\bigg]{
   \smashoperator[r]{\iint_{\brk{\Bset^{m - 1}_{x_m} \brk{x'}}^2}} \frac{d \brk{u\brk{y}, u\brk{z}}^p}{\abs{y - z}^{m + p}} \dif y \dif z}^\frac{1}{p},
\end{equation}
so that, if \( p \ge m\), by the absolute continuity of the integral,
\begin{equation}
\label{eq_eeGhi9Thosh8ephai7avai8l}
 \inf \set{\abs{x - y} \st x \in \Rset^{m - 1}\times \set{0} \text{ and } y \in  \smash{\widehat{\Omega}}^{\mathrm{bad}}} > 0.
\end{equation}
(The estimate \eqref{eq_sixaika7eitie6deLaem4ae0} and the inclusion \eqref{eq_eeGhi9Thosh8ephai7avai8l} are a formulation of the vanishing mean oscillation (VMO) property of Sobolev spaces with critical and supercritical integrability \citelist{\cite{Brezis_Nirenberg_1995}\cite{Brezis_Nirenberg_1996}}.)

\subsection{Good cubes and bad cubes}
\label{section_dyadic_decomposition}
In order to define a Sobolev extension over the bad set \(\smash{\smash{\widehat{\Omega}}^{\mathrm{bad}}}\),
we will instead consider a \emph{cubical complex} covering it. 

The construction of cubical complexes will be based on a Whitney-type dyadic decomposition of the half-space into cubes.

\begin{definition}
\label{definition_dyadic_decomposition}
A family \(\mathscr{Q} = \bigcup_{k \in \Zset} \mathscr{Q}_k\) is a \emph{dyadic decomposition} of \(\Rset^m_+\) whenever there exists a family \(\brk{\xi_{k}}_{k \in \Zset}\) in \(\Rset^{m - 1}\) such that for every \(k \in \Zset\),
\begin{equation}
\label{eq_toovahg2ashahc4ohtohRouc}
\xi_{k} \in 2^{k - 1} \Zset^{m - 1} + \xi_{k - 1}
\end{equation}
and
\begin{equation}
\label{eq_gooqu5ohrethuxiephoh2ETh}
  \mathscr{Q}_k
  =
  \set{ \brk{\intvc{0}{2^k}^{m - 1} + \zeta}\times \intvc{2^k}{2^{k + 1}} \st \zeta \in 2^k \Zset^{m - 1} + \xi_k }.
\end{equation}
\end{definition}

Geometrically, for every \(k \in \Zset\), the set \(\mathscr{Q}_k\) is a subdivision of the slab \(\Rset^{m - 1} \times \intvc{2^k} {2^{k + 1}}\) into cubes of edge-length \(2^{k}\) with sides parallel to the coordinate axes, satisfying the condition that the intersection with the hyperplane \(\Rset^{m - 1} \times \set{2^{k}}\) of the previous layer \(\mathscr{Q}_{k - 1}\)
is a subdivision of the intersections with the same hyperplane of the cubes of \(\mathscr{Q}_{k}\).

Any cube \(Q \in \mathscr{Q}\) satisfies \(\edge \brk{Q} = 2^k\) and \(\dist \brk{Q, \partial \Rset^m_+} = 2^k\),  for some \(k \in \Zset\), and therefore the identity 
\begin{equation}
\label{eq_edge_lambda_dist}
 \edge \brk{Q} = \dist \brk{Q, \partial \Rset^m_+}.
\end{equation}

\begin{remark}
\label{remark_solenoid}
From a more algebraic perspective, it appears from \eqref{eq_gooqu5ohrethuxiephoh2ETh} in  \cref{definition_dyadic_decomposition} that the vector \(\xi_k\) can be taken in the torus \(\Rset^{m - 1}/2^k \Zset^{m - 1} \simeq \brk{\Rset/2^k \Zset}^{m - 1}\) 
and from the condition \eqref{eq_toovahg2ashahc4ohtohRouc} that the increment \(\xi_{k} - \xi_{k - 1}\) can be taken in the finite abelian group  \(2^{k-1} \Zset^{m - 1}/2^{k} \Zset^{m - 1}\simeq \brk{2^{k - 1} \Zset/2^{k} \Zset}^{m - 1} \simeq \brk{\Zset/2\Zset}^{m - 1}\).
In other words, the dyadic decomposition \(\mathscr{Q}\) described by \(\brk{\xi_k}_{k \in \Zset}\) is an element of \(\brk{\boldsymbol{\Sigma}_2}^{m - 1}\), 
where \(\boldsymbol{\Sigma}_2\) is the \emph{dyadic solenoid} \citelist{\cite{Vietoris_1927}\cite{VanDantzig}} (see also \cite{Hewitt_Ross_1979}*{Ch.\thinspace{}2 Def.\thinspace{}(10.12)}); although our proofs will not rely on these deeper properties of solenoids, it is conceptually interesting to note that the dyadic decompositions of \(\Rset^m_+\) form a \emph{compact abelian topological group}.
\end{remark}

The \(k\)-th generation of \(\mathscr{C}\) of any subset \(\mathscr{C} \subseteq \mathscr{Q}\) of a dyadic decomposition from \cref{definition_dyadic_decomposition} is defined as 
\begin{equation*}
  \mathscr{C}_k \defeq \mathscr{C} \cap \mathscr{Q}_k. 
\end{equation*}
Given an open set \(\Omega \subseteq \Rset^{m - 1}\) and its tent \(\smash{\widehat{\Omega}}\) defined in \eqref{eq_epheiY6eilo4phaichoC9ton},
if we define the collection of dyadic cubes contained in \(\smash{\widehat{\Omega}}\) as 
\begin{equation}
\mathscr{Q}^{\vert \Omega}
\defeq \set{Q \in \mathscr{Q} \st Q \subseteq \widehat{\Omega}},
\end{equation}
then \(\bigcup \mathscr{Q}^{\vert \Omega}\) covers a smaller and flatter tent included in \(\smash{\widehat{\Omega}}\):
we have 
\begin{equation}
\label{eq_vu4jate7raeHae1ohN0wohdo}
  \bigcup \mathscr{Q}^{\vert \Omega}
 \supseteq \smash{\widehat{\Omega}}^\kappa,
\end{equation}
where 
\begin{equation*}
  \smash{\widehat{\Omega}}^\kappa
  \defeq 
  \set{x \in \Rset^{m}_+
  \st \overline{{\Bset^{m - 1}_{\kappa x_m}} (x')} \subseteq \Omega}
\end{equation*}
and 
\begin{equation}
\label{eq_eosh4xaen1ohvohkea0aSeeV}
 \kappa \defeq 2 \brk{1 + \sqrt{m - 1}}.
\end{equation}
Indeed, if \(x = \brk{x', x_m} \in \smash{\widehat{\Omega}}^\kappa\) and if the cube \(Q \in \mathscr{Q}\) contains \(x\), then for every \(y = \brk{y', y_m} \in Q\) and \(z' \in \smash{\overline{\smash{\Bset^{m-1}_{y_m}} \brk{y'}}}\), we have in view of \eqref{eq_edge_lambda_dist}
\begin{equation*}
\begin{split}
 \abs{z' - x'} 
 &\le \abs{z' - y'} + \abs{y' - x'} \\
 &\le \brk[\Big]{1 + \sqrt{m - 1} \, \frac{\edge \brk{Q}}{\dist \brk{Q, \partial \Rset^m_+}}}y_m
 =  \brk[\big]{1 + \sqrt{m - 1}}y_m
\end{split}
\end{equation*}
and 
\begin{equation*}
 y_m \le \brk[\Big]{1 + \frac{\edge \brk{Q}}{\dist \brk{Q, \partial \Rset^m_+}}}x_m 
 = 2 x_m,
\end{equation*}
so that \(\overline{\smash{\Bset^{m-1}_{y_m}} \brk{y'}} \subseteq \overline{\smash{\Bset^{m-1}_{\kappa x_m}} \brk{x'\vphantom{y'}}}\), with \(\kappa\) defined by \eqref{eq_eosh4xaen1ohvohkea0aSeeV}, and thus \(y \in \smash{\widehat{\Omega}}\).
We have thus proved that \(Q \subseteq \smash{\widehat{\Omega}}\) so that \(Q \in \smash{\mathscr{Q}^{\vert \Omega}}\), \(x \in \bigcup \smash{\mathscr{Q}^{\vert \Omega}}\) and the inclusion \eqref{eq_vu4jate7raeHae1ohN0wohdo} is proved.

\medbreak

We define the collection of \emph{good cubes} of a dyadic decomposition \(\mathscr{Q}\) as the set of cubes that are contained in the good set \(\smash{\widehat{\Omega}}^{\mathrm{good}}\) defined in \eqref{eq_BuCo1theeQuierowech2fai4}
\begin{equation}
\label{eq_definition_good_cube}
\begin{split}
 \mathscr{Q}^{\mathrm{good}}
 &\defeq \set{Q \in \mathscr{Q}^{\vert \Omega} \st \sup_{x \in Q} \dist \brk{V\brk{x}, \manifold{N}} \le \delta_{\manifold{N}}}\\
 & = \set{Q \in \mathscr{Q}^{\vert \Omega} \st Q \subseteq \smash{\widehat{\Omega}}^{\mathrm{good} }}
\end{split}
\end{equation}
and its complement the collection of \emph{bad cubes} of the dyadic decomposition \(\mathscr{Q}\) as the set of cubes that have a non-trivial intersection with the bad set \(\smash{\widehat{\Omega}}^{\mathrm{bad}}\) defined in \eqref{eq_vuaPh9vaengoo9aiHeeshaer}
\begin{equation}
\label{eq_definition_bad_cube}
\begin{split}
 \mathscr{Q}^{\mathrm{bad}}
 &\defeq \set{Q \in \mathscr{Q}^{\vert \Omega} \st \sup_{x \in Q} \dist \brk{V\brk{x}, \manifold{N}} > \delta_{\manifold{N}}}\\
 &= \set{Q \in \mathscr{Q}^{\vert \Omega} \st Q \cap \smash{\widehat{\Omega}}^{\mathrm{bad}}\ne \emptyset},
\end{split}
\end{equation}
so that the collection \(\mathscr{Q}^{\vert \Omega}\) is the disjoint union of \(\smash{\mathscr{Q}^{\mathrm{bad}}}\) and \(\smash{\mathscr{Q}^{\mathrm{good}}}\). 
The bad set and the bad cubes are illustrated on \cref{figure_badsetcubes}.

\begin{figure}
\begin{center}
 \includegraphics[alt={The bad set and the bad cubes whose complements are the good set and the good cubes respectively}]{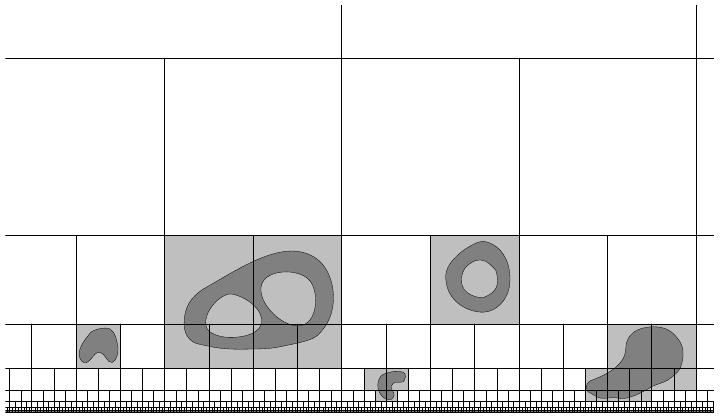}
\end{center}
\caption {On the bad set \(\smash{\widehat{\Omega}}^{\mathrm{bad}}\) defined in \eqref{eq_vuaPh9vaengoo9aiHeeshaer} (darker gray) and on the bad cubes \(\smash{\mathscr{Q}^{\mathrm{bad}}}\) defined in \eqref{eq_definition_bad_cube} (lighter gray) the linear extension \(V\) takes some values far away from the target manifold \(\manifold{N}\), whereas on the good set \(\smash{\smash{\widehat{\Omega}}^{\mathrm{good}}}\) and on the good cubes \(\smash{\mathscr{Q}^{\mathrm{good}}}\), defined as their complements respectively, the function \(V\) takes its value sufficiently close to \(\manifold{N}\) for its nearest-point projection \(\Pi_{\manifold{N}}\) to be well-defined and smooth.}
\label{figure_badsetcubes}
\end{figure}

In order to estimate the size of the collection of bad cubes \(\smash{\mathscr{Q}^{\mathrm{bad}}}\), we start from the observation that if \(Q \in \smash{\mathscr{Q}^{\vert \Omega}}\), then by the definition of the tent \(\smash{\widehat{\Omega}}\) in \eqref{eq_epheiY6eilo4phaichoC9ton}, for every \(x = \brk{x', x_m} \in Q\), we have \(x_m \le 2 \edge \brk{Q}\) and thus 
\begin{equation}
\label{eq_OaQuee5peeweoShuk8zah4aW}
 \overline{\Bset^{m - 1}_{x_m} \brk{x'}} \subseteq \Omega \cap Q'_5,
\end{equation}
where \(Q'\) is the projection of \(Q\) on \(\Rset^{m - 1}\times \set{0}\) and \(Q'_5\) denotes the cube of \(\Rset^{m - 1}\) with the same center \(z'\) as \(Q'\) and five times its edge-length, that is,
\(
 Q'_5 \defeq \set{5 x' - 4z' \st x' \in Q'}
\).

If \(Q \in \smash{\mathscr{Q}^{\vert \Omega}}\), then by definition of the truncated mean oscillation \eqref{eq_moo7thiequ2zug8ahSoh3oth} and by the inclusion \eqref{eq_OaQuee5peeweoShuk8zah4aW}, we have for every \(x \in Q\), 
\begin{equation}
\label{eq_ahx6xeexizaShieshohpui8e}
 \meanosc_{\delta_*} \brk{u} (x)
 \le \frac{\Cl{cst_gishiS2oequ2chi8quan0eet}}{\mathcal{L}^{m - 1} \brk{Q'}^2}
 \smashoperator{\iint_{\brk{\Omega \cap Q'_5}\times \brk{\Omega \cap Q'_5}}} \brk{d \brk{u \brk{y}, u \brk{z}} - \delta_*}_+ \dif y \dif z .
\end{equation}
In particular it follows from \eqref{eq_haw0fai3ohlie0Eitha9Aes2}, \eqref{eq_definition_bad_cube} and \eqref{eq_ahx6xeexizaShieshohpui8e} that if \(Q \in \mathscr{Q}^{\mathrm{bad}}\),
\begin{equation}
\label{eq_nequiecaeth6aeShiedohF7L}
 \delta_*
 \le \frac{\Cr{cst_gishiS2oequ2chi8quan0eet}}{\mathcal{L}^{m - 1} \brk{Q'}^2}
 \smashoperator{\iint_{\brk{\Omega \cap Q'_5}\times \brk{\Omega \cap Q'_5}}} \brk{d \brk{u \brk{y}, u \brk{z}} - \delta_*}_+ \dif y \dif z .
\end{equation}
Summing the estimate \eqref{eq_nequiecaeth6aeShiedohF7L} over the cubes \(Q\) in \(\mathscr{Q}^{\mathrm{bad}}_k\), we get 
\begin{equation}
\label{eq_TeiNieh7uz5sahs7aewo1shu}
\begin{split}
 \# \mathscr{Q}^{\mathrm{bad}}_k
 & \le
 \frac{5^{2\brk{m - 1}}\Cr{cst_gishiS2oequ2chi8quan0eet}}{2^{2 k \brk{m - 1}}\delta_*} \smashoperator{\iint_{\substack{\brk{y, z} \in \Omega \times \Omega\\
 \abs{y - z} \le 5 \sqrt{m - 1} \, 2^k}} }
 \brk{d \brk{u \brk{y}, u \brk{z}} - \delta_*}_+ \dif y \dif z.
\end{split}
\end{equation}
The summation of the inequality \eqref{eq_TeiNieh7uz5sahs7aewo1shu} with respect to \(k\) over \(\Zset\) gives then 
\begin{equation}
\label{eq_yie2te0ahrie6eeL2mah5Boh}
 \sum_{k \in \Zset} 2^{k\brk{m - p}} \# \mathscr{Q}^{\mathrm{bad}}_k
 \le \C \smashoperator{\iint_{\Omega \times \Omega}}
 \frac{\brk{d \brk{u \brk{y}, u \brk{z}} - \delta_*}_+}{\abs{y - z}^{p + m - 2}} \dif y \dif z.
\end{equation}
In particular 
the size of the collection of bad cubes \(\bigcup \mathscr{Q}^{\mathrm{bad}}\) satisfies a similar and practically identical estimate \eqref{eq_yie2te0ahrie6eeL2mah5Boh} as the inequality \eqref{eq_eeM3miequieno5oom4Layav5} controlling the original bad set \(\smash{\widehat{\Omega}}^{\mathrm{bad}}\):
indeed it follows immediately from the estimate \eqref{eq_yie2te0ahrie6eeL2mah5Boh} that 
\begin{equation}
 \smashoperator[r]{\int_{\bigcup \mathscr{Q}^{\mathrm{bad}}}}
 \frac{1}{x_m^p} \dif x
 \le 
 \C \smashoperator{\iint_{\Omega \times \Omega}}
 \frac{\brk{d \brk{u \brk{y}, u \brk{z}} - \delta_*}_+}{\abs{y - z}^{p + m - 2}} \dif y \dif z.
\end{equation}

If we assumed for a moment that we were able to perform an extension \(U \colon \bigcup \smash{\mathscr{Q}^{ \vert \Omega}} \to \manifold{N}\) such that
\(\smash{U \restr{\bigcup \mathscr{Q}^{\mathrm{good}}}} = \smash{\Pi_{\manifold{N}} \compose V \restr{\bigcup \mathscr{Q}^{\mathrm{good}}}}\)
and such that for every \(x = \brk{x',x_m} \in \smash{\bigcup \mathscr{Q}^{\mathrm{bad}}}\) we had 
\begin{equation}
\label{eq_deiZeef7ahlieP0eNgaoqua2}
   \abs{\Deriv U \brk{x}} \le \frac{\C}{x_m},
\end{equation}
that is, \(U\) satisfied an estimate \eqref{eq_deiZeef7ahlieP0eNgaoqua2} similar to the one satisfied thanks to \eqref{eq_Di5lainai8xoa5ruacuoxeih} by the value imposed by \(\Pi_{\manifold{N}} \compose \smash{V  \restr{\bigcup \mathscr{Q}^{\mathrm{good}}}}\) on any of the good cubes \(\smash{\mathscr{Q}^{\mathrm{good}}}\), then we would have constructed a suitable Sobolev extension of \(u\).
Although it will not be possible to enforce a condition as strong as  \eqref{eq_deiZeef7ahlieP0eNgaoqua2}, we will ensure a weaker condition that will still be strong enough for our purpose.

\section{From a partial topological extension to a Sobolev extension }
\label{section_topological_to_sobolev}

\resetconstant

As we have observed in \sectref{section_badset}, the boundary datum \(u\) has an extension \(\smash{\Pi_{\manifold{N}} \compose V\restr{\bigcup \mathscr{Q}^{\mathrm{good}}}}\) with appropriate regularity and corresponding estimates outside a collection of dyadic cubes \(\smash{\mathscr{Q}^{\mathrm{bad}}}\) satisfying the estimate \eqref{eq_yie2te0ahrie6eeL2mah5Boh}. 
It therefore remains to perform an appropriate construction to take care of the bad cubes. 

We show in the present section that if the lower homotopy groups \(\pi_1 \brk{\manifold{N}}, \dotsc, \pi_{\ell - 1} \brk{\manifold{N}}\), with \(\ell \defeq \min \brk{\floor{p}, m}\), are all finite, then the construction of a \emph{Sobolev} extension can be reduced to the production of a \emph{continuous} extension on the union of the \(\ell\)-dimensional faces of the bad cubes of \(\smash{\mathscr{Q}^{\mathrm{bad}}}\), or, as it will turn out to be useful later, on a larger collection \(\smash{\mathscr{Q}^{\mathrm{sing}}}\) of singular cubes.
\emph{It is in this part only of the proof that the finiteness assumption on the lower homotopy groups intervenes.}

\medbreak

In order to give precise statements, we now describe the notion of lower-dimensional skeletons.
Given a collection of cubes \(\mathscr{C} \subseteq \mathscr{Q}\), where \(\mathscr{Q}\) is a dyadic decomposition of \(\Rset^m_+\) as defined in \cref{definition_dyadic_decomposition}, we define the set of \emph{\(\ell\)-dimensional horizontal faces} of \(\mathscr{C}\) with \(\ell \in \set{0, \dotsc, m - 1}\), by first setting for each \(k \in \Zset\),
\begin{equation}
  \brk{\mathscr{C}}^{\ell, \top}_k
  \defeq \set{\sigma \st \sigma \text{ is an \(\ell\)-dimensional face of \(Q \cap \brk{\Rset^{m - 1} \times \set{2^{k + 1}}}\) for some \(Q \in \mathscr{C}_k\)}}
\end{equation}
and then defining
\begin{equation}
 \brk{\mathscr{C}}^{\ell, \top} \defeq \bigcup_{k \in \Zset} \brk{\mathscr{C}}^{\ell, \top}_k,
\end{equation}
and the set of \emph{\(\ell\)-dimensional vertical faces} of \(\mathscr{C}\), with \(\ell \in \set{1, \dotsc, m}\), by letting for each \(k \in \Zset\),
\begin{equation}
\label{eq_deeyooy8gah0vai7pheiSo0E}
  \brk{\mathscr{C}}^{\ell, \perp}_k
  = \set{\sigma' \times \intvc{2^k}{2^{k + 1}} \st \sigma' \times \set{2^{k + 1}} \in \brk{\mathscr{C}}^{\ell - 1, \top}_k} 
\end{equation}
and then setting
\begin{equation}
\label{eq_Xai3tekei9ao1Ahsh9rae9ae}
 \brk{\mathscr{C}}^{\ell, \perp} \defeq \bigcup_{k \in \Zset} \brk{\mathscr{C}}^{\ell, \perp}_k.
\end{equation}
We finally define the set of \emph{\(\ell\)-dimensional faces} of \(\mathscr{C}\) with \(\ell \in \set{0, \dotsc, m}\) as
\begin{equation*}
 \brk{\mathscr{C}}^{\ell} = \brk{\mathscr{C}}^{\ell, \top} \cup \brk{\mathscr{C}}^{\ell, \perp},
\end{equation*}
with the convention that \( \brk{\mathscr{C}}^{m, \top} = \emptyset\) and \( \brk{\mathscr{C}}^{0, \perp} = \emptyset\) so that \(\brk{\mathscr{C}}^{m} =\brk{\mathscr{C}}^{m, \perp} = \mathscr{C}\) and \(\brk{\mathscr{C}}^{0} =\brk{\mathscr{C}}^{0, \top}\).

The Sobolev extension that we will construct will be  continuous outside the \emph{dual skeleton} \(L^j \subseteq \Rset^m_+\)  of dimension \(j \in \set{-1, \dotsc, m - 1}\) relative to the dyadic decomposition \(\mathscr{Q}\) defined as follows.
If \(j = - 1\), one sets \(L^{-1} \defeq \emptyset\).
If \(j = 0\), then \(L^0\) is defined as the set of centres of the cubes in \(\mathscr{Q}\).
In general for \(j \in \set{1, \dotsc, m - 1}\), one defines \(L^{j, m - j}\) to be the set of the centres of faces in \(\brk{\mathscr{Q}}^{m - j}\), then recursively for \(i \in \set{m - j + 1, \dotsc, m}\), the set \(L^{j, i}\) as the union of segments joining the centre of any face \(\sigma \in \brk{\mathscr{Q}}^{i}\) with a point in \(\sigma \cap L^{j, i - 1}\) and finally letting \(L^{j} \defeq L^{j, m}\).
The \(0\)-dimensional dual skeleton \(L^0\) and the \(1\)-dimensional dual skeleton \(L^1  \) in the half-plane \(\Rset^2_+\) are illustrated on \cref{figure_dual}.

\begin{figure}
\begin{center}
\includegraphics[alt={The 0- and 1- dimensional skeletons of a dyadic decomposition of the halfspace}]{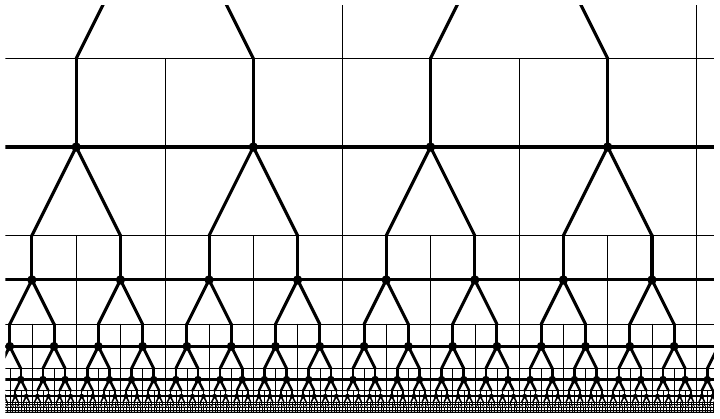}
\end{center}
\caption{The \(0\)-dimensional dual skeleton \(L^0\) of a dyadic decomposition \(\mathscr{Q}\) of \(\smash{\Rset^2_+}\) consists of the centres of its squares while its \(1\)-dimensional dual skeleton consists of the segment joining those points to the centres of its vertices. }
\label{figure_dual}
\end{figure}

\begin{proposition}
\label{proposition_Lipschitz_extension_skeleton}
Let \(\ell \in \set{1, \dotsc, m}\) and assume that the lower homotopy groups \(\pi_1 \brk{\manifold{N}}, \dotsc, \pi_{\ell - 1} \brk{\manifold{N}}\) are finite.
For every \(\alpha \in \intvo{0}{\infty}\), there exists \(\beta \in \intvo{0}{\infty}\) such that 
given a dyadic decomposition \(\mathscr{Q}\) of \(\Rset^m_+\), 
collections of cubes \(\mathscr{Q}^{\mathrm{reg}}, \mathscr{Q}^{\mathrm{sing}} \subseteq \mathscr{Q}\) such that \(\mathscr{Q}^{\mathrm{reg}} \cap \mathscr{Q}^{\mathrm{sing}} = \emptyset\),
and a mapping \(W \in C \brk{\smash{\bigcup \mathscr{Q}^{\mathrm{reg}}\cup \bigcup \brk{\mathscr{Q}^{\mathrm{sing}}}^{\ell}}, \manifold{N}}\) such that for every \(x = \brk{x', x_m} \in \bigcup \smash{\mathscr{Q}^{\mathrm{reg}}}\), 
\begin{equation}
\label{eq_aed3oush2ahz6IexahYeich4}
  \operatorname{Lip} W \brk{x} \le \frac{\alpha}{x_{m}},
\end{equation}
there exists a mapping \(U \in C \brk{\bigcup \mathscr{Q}^{\mathrm{reg}} \cup \brk{\bigcup \mathscr{Q}^{\mathrm{sing}} \setminus L^{m - \ell - 1}}, \manifold{N}}\) such that 
\(\smash{U \restr{\bigcup \mathscr{Q}^{\mathrm{reg}}}} = \smash{W\restr{\bigcup \mathscr{Q}^{\mathrm{reg}}}}\) and such that for every \(x = \brk{x', x_m} \in\bigcup \mathscr{Q}^{\mathrm{sing}} \setminus L^{m - \ell - 1}\) 
\begin{align}
\label{eq_Umech5Ueyeic6oa7Ohngeevi}
\operatorname{Lip} U \brk{x} &\le \frac{\beta}{x_m} &&\text{if \(\ell = m\)},
\intertext{or}
\label{eq_wohch3beevohDuegeraiB6zu}
 \operatorname{Lip} U \brk{x} &\le \frac{\beta}{\dist \brk{x, L^{m - \ell - 1}}}
 &&\text{if \(\ell < m\)}.
\end{align}
Moreover, for every \(p \in \intvo{1}{\infty}\) such that either \(\ell = m\) or \(\ell > p - 1\), there exists a constant \(C \in \intvo{0}{\infty}\) such that if \(\smash{W\restr{\bigcup \mathscr{Q}^{\mathrm{reg}}}} \in \smash{\dot{W}}^{1, p} \brk{\bigcup \mathscr{Q}^{\mathrm{reg}}, \manifold{N}}\) and if 
\begin{equation}
\label{eq_ee9wei0noyah1zee0xuQuie2}
\sum_{k \in \Zset} 2^{k\brk{m - p}} \# \mathscr{Q}^{\mathrm{sing}}_k < \infty, 
\end{equation}
then \(U \in \smash{\dot{W}}^{1, p} \brk{\bigcup \mathscr{Q}^{\mathrm{reg}} \cup \bigcup \mathscr{Q}^{\mathrm{sing}}, \manifold{N}}\) and 
\begin{equation}
\label{eq_xuoY7IuceiwooH5Aicadai9a}
 \int_{\bigcup \brk{\mathscr{Q}^{\mathrm{reg}}\cup \mathscr{Q}^{\mathrm{sing}}}} \abs{\Deriv U}^p
 \le 
 \int_{\bigcup \mathscr{Q}^{\mathrm{reg}}} \abs{\Deriv W}^p
 +  C \beta^p \sum_{k \in \Zset} 2^{k\brk{m - p}} \# \mathscr{Q}^{\mathrm{sing}}_k.
\end{equation}

\end{proposition}

The assumption in \cref{proposition_Lipschitz_extension_skeleton} that the \emph{lower homotopy groups} \(\pi_1 \brk{\manifold{N}}, \dotsc, \pi_{\ell - 1} \brk{\manifold{N}}\) are finite plays a \emph{crucial role} in the proof; the discussion about its necessity for \cref{lemma_cube_extension_lipschitz_homotopic} also applies here for \cref{proposition_Lipschitz_extension_skeleton}.

For a set \(A \subseteq \Rset^m\), a mapping \(f \colon A \to \manifold{N}\) and \(x \in A\), the quantity \(\operatorname{Lip} f(x)\) in the estimates \eqref{eq_Umech5Ueyeic6oa7Ohngeevi} and \eqref{eq_wohch3beevohDuegeraiB6zu} of \cref{proposition_Lipschitz_extension_skeleton} is the \emph{Lipschitz number} of the mapping \(f\) at \(x\) defined as 
\begin{equation*}
 \operatorname{Lip} f(x)
 \defeq \limsup_{y \to x} \frac{d\brk{f \brk{y}, f\brk{x}}}{\abs{y - x}} \in \intvc{0}{\infty},
\end{equation*}
with the additional convention that \(\operatorname{Lip} f(x) = 0\) when \(x\) is an isolated point of \(A\); 
it is related to the \emph{Lipschitz constant} of \(f\) on \(A\) as follows:
\begin{equation*}
 \operatorname{Lip} f(x)
 \le \abs{f}_{\mathrm{Lip}} \defeq \sup_{y, z \in A} \frac{d\brk{f (y), f(z)}}{\abs{y - z}}.
\end{equation*}

\medbreak

The proof of \cref{proposition_Lipschitz_extension_skeleton} is done through suitable constructions to propagate smoothness properties from the boundary of a face \(\sigma \in \brk{\mathscr{Q}}^j\) to its interior for \(j \in \set{2, \dotsc, m}\).

The simplest construction is performed for every face \(\sigma \in \brk{\mathscr{Q}}^j\) when \(j > \ell\). 
We extend then the mapping from the boundary \(\partial \sigma\) to the interior of the face \(\sigma\) homogeneously. To simplify the notation, we perform the construction on the unit cube \(\Qset^j \defeq \intvc{-1}{1}^j\).

\begin{lemma}
\label{lemma_homogeneous_extension}
For every \(j \in \Nset \setminus \set{0, 1}\) and every mapping \(f \colon \partial \Qset^j \to \manifold{N}\), the mapping \(F\colon \Qset^j \setminus \set{0} \to \manifold{N}\) defined for every \(x \in \Qset^j \setminus \set{0}\) by 
\begin{equation*}
 F (x) \defeq f \brk[\big]{\tfrac{x}{\abs{x}_\infty}}
\end{equation*}
has the following properties:
\begin{enumerate}[label=(\roman*)]
 \item 
 \label{it_iuphei7Lahng3co3aeju0woh}
 one has 
 \begin{equation*}
  \sup_{x \in \Qset^j \setminus \set{0}} \abs{x} \operatorname{Lip} F \brk{x} 
  \le j
  \sup_{x \in \partial \Qset^j} 
   \operatorname{Lip} f \brk{x},
 \end{equation*}
\item 
\label{it_IWigh3phei9Ohjoolaineequ}
if \(A \subseteq \partial \Qset^j\) and \(A \ne \emptyset\), then 
 \begin{equation*}
  \sup_{x \in \Qset^j \setminus \set{0}}\dist\brk{x, A^*} \operatorname{Lip} F \brk{x}
  \le \sqrt{j}\sup_{x \in \partial \Qset^j} 
  \dist\brk{x, A} \operatorname{Lip} f \brk{x},
 \end{equation*}
   where \(A^* \defeq \set{ tx \st x \in A \text{ and }t \in \intvc{0}{1}}\).
   \end{enumerate}
\end{lemma}
\begin{proof}
Defining the map \(\Psi \colon \Qset^j\setminus \set{0} \to \partial \Qset^j\) for every \(x \in \Qset^j \setminus \set{0}\) by \(\Psi \brk{x} \defeq x/\abs{x}_\infty\), one has 
\(
  \operatorname{Lip} \Psi \brk{x} = \abs{x}/\abs{x}_\infty^2
\) and \(\abs{x}\le \sqrt{j} \abs{x}_\infty\),
so that
\begin{equation*}
 \operatorname{Lip} F \brk{x}
 \le \frac{\operatorname{Lip} f \brk{x/\abs{x}_\infty} \abs{x}}{\abs{x}_\infty^2}
 \le \frac{j \operatorname{Lip} f \brk{x/\abs{x}_\infty}} {\abs{x}},
\end{equation*}
and the assertion \ref{it_iuphei7Lahng3co3aeju0woh} follows.
Observing that \(\dist \brk{x, A^*} \le \abs{x}_\infty \dist \brk{x/\abs{x}_\infty, A}\), we have 
\begin{equation}
\begin{split}
  \dist\brk{x, A^*} \operatorname{Lip} F \brk{x}
  &\le \frac{\dist \brk{x/\abs{x}_\infty, A} \operatorname{Lip} f \brk{x/\abs{x}_\infty} \abs{x}}{\abs{x}_\infty}\\
  &\le \sqrt{j} \dist \brk{x/\abs{x}_\infty, A} \operatorname{Lip} f \brk{x/\abs{x}_\infty},
\end{split}
\end{equation}
and thus the assertion \ref{it_IWigh3phei9Ohjoolaineequ} holds.
\end{proof}

The second construction to propagate smoothness properties from the boundary of a face follows from the well-known fact that every Lipschitz-continuous mapping on the boundary of the cube with a continuous extension to the cube also has a Lipschitz-continuous extension to the cube, and combines this with a compactness argument to obtain a {uniform bound on the extension}. 

\begin{lemma}
\label{lemma_cube_extension_lipschitz}
Let \(j \in \Nset \setminus \set{0}\). 
For every \(\alpha \in \intvo{0}{\infty}\) there exists \(\beta \in \intvo{0}{\infty}\) such that for every mapping \(F \in C \brk{\Qset^j, \manifold{N}}\) satisfying
\(
 \seminorm{F \restr{\partial \Qset^j}}_{\mathrm{Lip}} \le \alpha
\),
there exists a mapping \(G \in C \brk{\Qset^j, \manifold{N}}\) such that 
\(G \restr{\partial \Qset^j} = F \restr{\partial \Qset^j}\) and \(\seminorm{G}_{\mathrm{Lip}} \le \beta\).
\end{lemma}
\begin{proof}
We consider the set of mappings 
\begin{equation*}
 \mathfrak{F}_\alpha
 \defeq 
 \set{F \restr{\partial \Qset^j} \st F \in C \brk{\Qset^j, \manifold{N}} \text{ and } \seminorm{F \restr{\partial \Qset^j}}_{\mathrm{Lip}} \le \alpha}
 \subseteq C \brk{\partial \Qset^j, \manifold{N}}.
\end{equation*}
By the classical Arzelà-Ascoli compactness criterion for continuous functions for the uniform distance, there are finitely many mappings \(f_1, \dotsc, f_M \in \mathfrak{F}_\alpha\) such that for every mapping \(F \in C \brk{\Qset^j, \manifold{N}}\) satisfying \(
 \seminorm{F \restr{\partial \Qset^j}}_{\mathrm{Lip}} \le \alpha
\), there exist some \(i \in \set{1, \dotsc, M}\) for which \(d\brk{F \restr{\partial \Qset^j}, f_i} \le \delta_{\manifold{N}}\) everywhere on \(\partial \Qset^j\), where \(\delta_{\manifold{N}}\) is a distance from \(\manifold{N}\) up to which the nearest-point projection  \(\Pi_{\manifold{N}}\) is well-defined and smooth.
We now choose for each \(i \in \set{1, \dotsc, M}\) a mapping \(G_i\in C\brk{\Qset^j, \manifold{N}}\) such that \(\seminorm{G_i}_{\mathrm{Lip}} < \infty\) and \(G_i \restr{\partial \Qset^j} = f_i\).

Given a mapping \(F \in C \brk{\Qset^j, \manifold{N}}\) satisfying \(
 \seminorm{F \restr{\partial \Qset^j}}_{\mathrm{Lip}} \le \alpha
\), there is some \(i \in \set{1, \dotsc, M}\) such that we have \(d\brk{F \restr{\partial \Qset^j}, G_i \restr{\partial \Qset^j}} \le \delta_{\manifold{N}}\) everywhere on \(\partial \Qset^j\).
We define then the mapping \(G \colon \Qset^j \to \manifold{N}\) for every \(x \in \Qset^j\) by
\begin{equation}
 G \brk{x}
 \defeq 
 \begin{cases}
   \Pi_{\manifold{N}} \brk[\big]{\brk{2\abs{x}_\infty - 1} F \brk{\tfrac{x}{\abs{x}_\infty}} + 2 \brk{1 -  \abs{x}_\infty} G_i \brk{\tfrac{x}{\abs{x}_\infty}}} & \text{if \(\abs{x}_{\infty} \ge 1/2\)},\\
   G_i \brk{2 x} & \text{otherwise}.
 \end{cases}
\end{equation}
The conclusion then follows.
\end{proof}

\Cref{lemma_cube_extension_lipschitz} is not strong enough for our purposes, 
because \emph{it does not guarantee any relationship between the original continuous extension \(F\) and the resulting Lipschitz-continuous extension \(G\).} 
In particular, if the mapping \(F\) had a continuous extension to a higher-dimensional face of a dyadic cubical complex, there is no guarantee whatsoever that the mapping \(G\) would also have one.
However, if we could ensure that the Lipschitz-continuous extension \(G\) is \emph{homotopic} to the extension \(F\) \emph{relatively to the boundary of the face,} then, thanks to the homotopy extension property, we could replace the mapping \(F\) by the more regular mapping \(G\) without losing the crucial property of having a continuous extension to higher-dimensional faces.
A key observation in the present work is that, as  known in quantitative homotopy theory  \citelist{\cite{Siegel_Williams_1989}\cite{Ferry_Weinberger_2013}}, this is indeed possible on a \(j\)-dimensional face, provided that the homotopy group \(\pi_{j} \brk{\manifold{N}}\) is \emph{finite}.

\begin{lemma}
\label{lemma_cube_extension_lipschitz_homotopic}
Let \(j \in \Nset \setminus \set{0}\). If the homotopy group \(\pi_{j} \brk{\manifold{N}}\) is finite, then for every \(\alpha \in \intvo{0}{\infty}\) there exists \(\beta \in \intvo{0}{\infty}\) such that for every mapping \(F \in C \brk{\Qset^j, \manifold{N}}\) satisfying
\(
 \seminorm{F \restr{\partial \Qset^j}}_{\mathrm{Lip}} \le \alpha
\),
there exists a mapping \(G \in C \brk{\Qset^j, \manifold{N}}\) such that 
\(G \restr{\partial \Qset^j} = F \restr{\partial \Qset^j}\), the maps \(F\) and \(G\) are homotopic relatively to \(\partial \Qset^j\) and \(\seminorm{G}_{\mathrm{Lip}} \le \beta\).
\end{lemma}

\Cref{lemma_cube_extension_lipschitz_homotopic} is the only place where we are using directly the assumption of finiteness on homotopy groups of \(\manifold{N}\); in all other statements where it appears, its use in the proof is mediated through \cref{lemma_cube_extension_lipschitz_homotopic}.

The assumption that the homotopy group \(\pi_{j}\brk{\manifold{N}}\) is finite is \emph{necessary} in \cref{lemma_cube_extension_lipschitz_homotopic}. 
Indeed, if the conclusion holds for some \(\alpha\), taking \(F\) to be constant on \(\partial \Qset^j\), by the Arzelà-Ascoli compactness criterion there would be a compact set that intersects any homotopy class of such maps, implying that there are only finitely many such homotopy classes and thus that the homotopy group \(\pi_{j} \brk{\manifold{N}}\) is finite.

The proof of \cref{lemma_cube_extension_lipschitz_homotopic} will elaborate the proof of its counterpart \cref{lemma_cube_extension_lipschitz} without the homotopy constraint and rely on the following homotopy result.

\begin{lemma}
\label{lemma_homotopy_bijection}
Given a homotopy \(H \in C\brk{\intvc{0}{1} \times \partial \Qset^j, \manifold{N}}\) and maps \(F_0, F_1 \in C \brk{\Qset^j, \manifold{N}}\) such that for each \(i \in \set{0, 1}\), \(F_i \restr{\partial \Qset^j} = H\brk{0, \cdot}\), if we define the mappings \(G_0, G_1 \in C \brk{\Qset^j, \manifold{N}}\) for every \(x \in \Qset^j\) by
\begin{equation*}
 G_i (x)
 \defeq
 \begin{cases}
   H \brk[\big]{2\abs{x}_{\infty}  - 1, \tfrac{x}{\abs{x}_\infty}} &\text{if \(\abs{x}_\infty \ge 1/2\)},\\
   F_i \brk{2 x} & \text{otherwise},
  \end{cases}
\end{equation*}
then \(F_0\) and \(F_1\) are homotopic relatively to \(\partial \Qset^j\) if and only if \(G_0\) and \(G_1\) are homotopic relatively to \(\partial \Qset^j\).
\end{lemma}
\begin{proof}
If \(F_0\) and \(F_1\) are homotopic relatively to \(\partial \Qset^j\), then one has immediately that \(G_0\) and \(G_1\) are homotopic relatively to \(\partial \Qset^j\).
Conversely, if \(G_0\) and \(G_1\) are homotopic relatively to \(\partial \Qset^j\), then an application of the homotopy extension property yields the required homotopy.
\end{proof}

\begin{proof}[Proof of \cref{lemma_cube_extension_lipschitz_homotopic}]
We proceed as in the proof of \cref{lemma_cube_extension_lipschitz} in order to define the maps \(f_1, \dotsc, f_M\).
Since all these maps have a continuous extension to \(\Qset^j\), we observe then that thanks to \cref{lemma_homotopy_bijection} for every \(i \in \set{1, \dotsc, M}\), there exist
mappings \(G_i^{1}, \dotsc, G_i^{k} \in  C\brk{\Qset^{j}, \manifold{N}}\), with \(k \defeq \# \pi_{j} \brk{\manifold{N}}\) such that for every \(h \in \set{1, \dotsc, k}\), \(\smash{\seminorm{G_i^h}_{\mathrm{Lip}}} < \infty\) and
\(\smash{G_i^{h}} \restr{\partial \Qset^j} = f_i\) and for every \(F \in C \brk{\Qset^j, \manifold{N}}\) such that \(\smash{F \restr{\partial \Qset^j}} = f_i\),
there is some \(h \in \set{1, \dotsc, k}\) for which the map \(F\) is homotopic to \(\smash{G_i^h}\) relatively to \(\smash{\partial \Qset^j}\).

We continue as in the  proof of \cref{lemma_cube_extension_lipschitz}, choosing instead of \(G_i\) a map \(\smash{G_i^h}\) which is homotopic to \(F\) relatively to \(\partial \Qset^j\).
\end{proof}

We have now all the ingredients to prove \cref{proposition_Lipschitz_extension_skeleton} as illustrated on \cref{figure_singext}.

\begin{figure}

 \begin{center}
  \includegraphics[alt={The construction of a controlled singular Sobolev extension from a continuous extension to a lower-dimensional skeleton}]{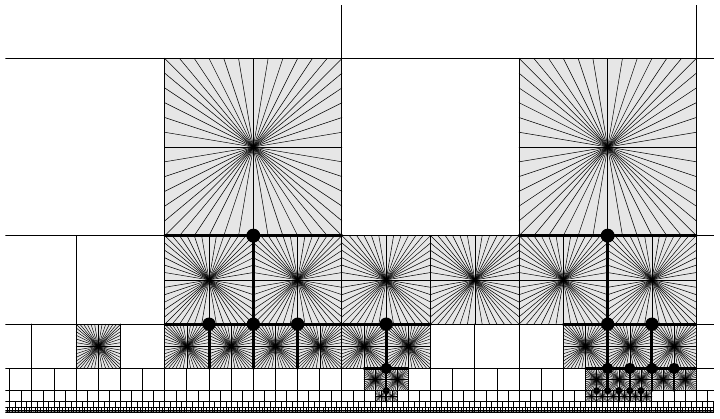}
 \end{center}
\caption{The construction of a controlled singular Sobolev extension from a continuous extension to a lower-dimensional skeleton of  \cref{proposition_Lipschitz_extension_skeleton} when \(m = 2\) and \(\ell = 1\):
on the lower-dimensional faces of \(\brk{\mathscr{Q}^{\mathrm{sing}}}^0\), represented as fat vertices,  \cref{lemma_cube_extension_lipschitz_homotopic} is used to replace the mapping by a homotopic map with controlled smoothness, on the critical-dimensional faces of \(\brk{\mathscr{Q}^{\mathrm{sing}}}^1\), represented as thick edges, \cref{lemma_cube_extension_lipschitz} is applied to replace the mapping by a map with controlled smoothness with no homotopy constraint, and finally on the higher-dimensional faces of \(\brk{\mathscr{Q}^{\mathrm{sing}}}^2\), represented as gray squares, a homogeneous extension is performed thanks to \cref{lemma_homogeneous_extension}.}
\label{figure_singext}
\end{figure}

\begin{proof}[Proof of \cref{proposition_Lipschitz_extension_skeleton}]
The map \(U\) will be constructed recursively over faces of increasing dimension \(j \in \set{0, \dotsc, m}\), treating separately the successive cases \(j = 0\), \(1 \le j \le \ell - 1\), \(j = \ell\), \(j = \ell + 1\) and \(\ell + 2 \le j \le m\).
(When \(\ell \in \set{m - 1, m}\), the last two or the very last step will be irrelevant and will not be performed; the resulting map will nevertheless satisfy the required properties.)

When \(j = 0\), we define the map \(U^0 \in C \brk{\bigcup \brk{\mathscr{Q}^{\mathrm{reg}} \cup \brk{\mathscr{Q}^{\mathrm{sing}}}^0}, \manifold{N}}\) in such a way that \(U^0 = W\) on \(\bigcup \mathscr{Q}^{\mathrm{reg}}\) and \(U^0 = b\) on \(\bigcup \brk{\mathscr{Q}^{\mathrm{sing}}}^0 \setminus \bigcup \mathscr{Q}^{\mathrm{reg}}\), where \(b \in \manifold{N}\) is a given fixed point. 
Since \(\manifold{N}\) is compact and since \(\bigcup \brk{\mathscr{Q}^{\mathrm{sing}}}^0\) consists only of isolated points, we immediately deduce from \eqref{eq_aed3oush2ahz6IexahYeich4}, that for every \(x = \brk{x', x_m} \in \bigcup \brk{\mathscr{Q}^{\mathrm{reg}} \cup \brk{\mathscr{Q}^{\mathrm{sing}}}^{0}}\), we have
\begin{equation}
\label{eq_aFaigohThoidu6uch7Yaed3f}
 \operatorname{Lip} U^0 \brk{x} \le \frac{\beta_0}{x_m}
\end{equation}
with \(\beta_0 \defeq \alpha\).
By the homotopy extension property there exists a mapping \(W^0 \in C \brk{\bigcup \brk{\mathscr{Q}^{\mathrm{reg}}\cup \brk{\mathscr{Q}^{\mathrm{sing}}}^{\ell}}, \manifold{N}}\) such that 
\begin{equation*}
U^0 = W^0 \restr{\bigcup \brk{\mathscr{Q}^{\mathrm{reg}} \cup \brk{\mathscr{Q}^{\mathrm{sing}}}^0}}.
\end{equation*}

Next, for \(j \in \set{1, \dotsc, \ell - 1}\), we assume that we have maps \(U^{j - 1} \in C \brk{\bigcup \brk{\mathscr{Q}^{\mathrm{reg}} \cup  \brk{\mathscr{Q}^{\mathrm{sing}}}^{j - 1}}, \manifold{N}}\) and \(W^{j - 1} \in C \brk{\bigcup \brk{\mathscr{Q}^{\mathrm{reg}} \cup \brk{\mathscr{Q}^{\mathrm{sing}}}^{\ell}}, \manifold{N}}\)
such that 
\begin{align*}
U^{j - 1} &= W^{j - 1} \restr{\bigcup \mathscr{Q}^{\mathrm{reg}}
\cup \brk{\mathscr{Q}^{\mathrm{sing}}}^{j - 1}}
&
\text{ and }
&
&
U^{j - 1} \restr{\bigcup \mathscr{Q}^{\mathrm{reg}}} &= W \restr{\bigcup \mathscr{Q}^{\mathrm{reg}}}
\end{align*}
and such that for every \(x= \brk{x', x_m} \in \bigcup \mathscr{Q}^{\mathrm{reg}}
\cup \brk{\mathscr{Q}^{\mathrm{sing}}}^{j - 1}\) we have 
\begin{equation}
\label{eq_ahgeibahtoh1ohx6Eis1wah8}
 \operatorname{Lip} U^{j - 1} \brk{x} \le \frac{\beta_{j - 1}}{x_m},
\end{equation}
with \(\beta_{j - 1}\) a constant that only depends on \(j - 1\) and \(\manifold{N}\).
(When \(j = 1\), this assumption is indeed satisfied according to the construction in the previous paragraph and \eqref{eq_aFaigohThoidu6uch7Yaed3f}; hence, the induction process is properly initialised.)
We define the map \(U^j \in C \brk{\bigcup \brk{\mathscr{Q}^{\mathrm{reg}} \cup \brk{\mathscr{Q}^{\mathrm{sing}}}^j}, \manifold{N}}\) to be \(U^{j - 1}\) on \(\bigcup \mathscr{Q}^{\mathrm{reg}}
\cup \brk{\mathscr{Q}^{\mathrm{sing}}}^{j - 1}\) so that it will remain to define it on every face \(\sigma \in \brk{\mathscr{Q}^\mathrm{sing}}^j \setminus \brk{\mathscr{Q}^\mathrm{reg}}^j\).
For such a face the mapping \(U^{j - 1}\) is defined on the relative boundary \(\partial \sigma \subseteq \bigcup \brk{\mathscr{Q}^{\mathrm{reg}} \cup \mathscr{Q}^{\mathrm{sing}}}^{j - 1}\) of the face \(\sigma\). 
Since the homotopy group \(\pi_j \brk{\manifold{N}}\) is finite by assumption, we can define \(U^j\) on \(\sigma\) with 
\cref{lemma_cube_extension_lipschitz_homotopic}.
In view of \eqref{eq_ahgeibahtoh1ohx6Eis1wah8}, \cref{lemma_cube_extension_lipschitz_homotopic} and a scaling argument, there exists a constant \(\beta_{j} \in \intvo{0}{\infty}\) depending just on \(j\), \(\beta_{j - 1}\) and \(\manifold{N}\) such that 
for every \(x= \brk{x', x_m} \in \bigcup \brk{\mathscr{Q}^{\mathrm{reg}} \cup \brk{\mathscr{Q}^{\mathrm{sing}}}^j}\),
\begin{equation}
\label{eq_ukee9aqu2shakaefohhahCha}
 \operatorname{Lip} U^j \brk{x} \le \frac{\beta_j}{x_m}.
\end{equation}
Moreover, we have 
\[
  U^j \restr{\bigcup \mathscr{Q}^{\mathrm{reg}}} = U^{j - 1} \restr{\bigcup \mathscr{Q}^{\mathrm{reg}}} = W \restr{\bigcup \mathscr{Q}^{\mathrm{reg}}}.
\]
By the homotopy extension property, there exists also a map \(W^j \in C \brk{\bigcup \brk{\mathscr{Q}^{\mathrm{reg}}\cup \brk{\mathscr{Q}^{\mathrm{sing}}}^{\ell}}, \manifold{N}}\) such that 
\begin{equation*}
U^j = W^j \restr{\bigcup \brk{\mathscr{Q}^{\mathrm{reg}} \cup \brk{\mathscr{Q}^{\mathrm{sing}}}^j}}.
\end{equation*}

When \(j = \ell\), we proceed as in the previous case, relying on \cref{lemma_cube_extension_lipschitz} instead of \cref{lemma_cube_extension_lipschitz_homotopic} so that no finiteness assumption has to be imposed on the homotopy group \(\pi_\ell \brk{\manifold{N}}\), and omitting the construction of the mapping \(W^\ell\).
We get thus a map \(U^\ell \in C \brk{\bigcup \brk{\mathscr{Q}^{\mathrm{reg}} \cup \brk{\mathscr{Q}^{\mathrm{sing}}}^{\ell}}, \manifold{N}}\) such that 
\[U^\ell \restr{\bigcup \mathscr{Q}^{\mathrm{reg}}} = U^{\ell - 1} \restr{\bigcup \mathscr{Q}^{\mathrm{reg}}} = W \restr{\bigcup \mathscr{Q}^{\mathrm{reg}}}
\]
and, in view of \eqref{eq_ukee9aqu2shakaefohhahCha}, for every \(x = \brk{x', x_m} \in \bigcup \brk{\mathscr{Q}^{\mathrm{reg}} \cup \brk{\mathscr{Q}^{\mathrm{sing}}}^{\ell}}\), 
\begin{equation}
 \operatorname{Lip} U^\ell \brk{x} \le \frac{\beta_\ell}{x_m}.
\end{equation}

For \(j = \ell + 1\), we define \(U^{\ell + 1}\) to coincide with \(U^\ell\) on \(\bigcup \brk{\mathscr{Q}^{\mathrm{reg}}\cup \brk{\mathscr{Q}^{\mathrm{sing}}}^{\ell}}\) and we apply \cref{lemma_homogeneous_extension} to define the mapping \(U^{\ell + 1}\) on every face \(\sigma \in \brk{\mathscr{Q}^{\mathrm{sing}}}^{\ell + 1}
\setminus \brk{\mathscr{Q}^{\mathrm{reg}}}^{\ell + 1}\).
The resulting map satisfies the condition 
\[
U^{\ell + 1} \restr{\bigcup \mathscr{Q}^{\mathrm{reg}}} = U^{\ell} \restr{\bigcup \mathscr{Q}^{\mathrm{reg}}} = W \restr{\bigcup \mathscr{Q}^{\mathrm{reg}}}
\]
and can be estimated for every \(x = \brk{x', x_m} \in \bigcup \mathscr{Q}^{\mathrm{reg}}\cup \brk{\bigcup \brk{\mathscr{Q}^{\mathrm{sing}}}^{\ell + 1} \setminus L^{m - \ell - 1}}\) in view of \cref{lemma_homogeneous_extension} \ref{it_iuphei7Lahng3co3aeju0woh} as 
\begin{equation}
 \operatorname{Lip} U^{\ell + 1} \brk{x} \le \frac{\beta_{\ell + 1}}{\dist \brk{x,L^{m - \ell - 1}}}.
\end{equation}

When \(\ell + 2\le j \le m\), assuming that there is a map \(U^{j - 1} \in C \brk{\bigcup \mathscr{Q}^{\mathrm{reg}} \cup \brk{\bigcup \brk{\mathscr{Q}^{\mathrm{sing}}}^{j - 1} \setminus L^{m - \ell - 1}}, \manifold{N}}\)
such that 
\[U^{j - 1} \restr{\bigcup \mathscr{Q}^{\mathrm{reg}}} = W \restr{\bigcup \mathscr{Q}^{\mathrm{reg}}},
\]
such that for every \(x \in \bigcup \mathscr{Q}^{\mathrm{reg}} 
\cup \brk{\bigcup \brk{\mathscr{Q}^{\mathrm{sing}}}^{j - 1} \setminus L^{m - \ell - 1} }\)
\begin{equation}
 \operatorname{Lip} U^{j - 1} \brk{x} \le \frac{\beta_{j - 1}}{\dist \brk{x,L^{m - \ell - 1}}}
\end{equation}
and such that 
\[
  U^{j - 1}\restr{\bigcup \mathscr{Q}^{\mathrm{reg}}} = W \restr{\bigcup \mathscr{Q}^{\mathrm{reg}}}
\]
(we have just proved this was indeed the case for \(j = \ell + 2\)),
we define the map \(U^j \in C \brk{\bigcup \mathscr{Q}^{\mathrm{reg}} \cup \brk{\bigcup \brk{\mathscr{Q}^{\mathrm{sing}}}^{j} \setminus L^{m - \ell - 1}}, \manifold{N}}\) by taking \(U^{j - 1}\) on the set \(\bigcup \mathscr{Q}^{\mathrm{reg}} \cup \brk{\bigcup \brk{\mathscr{Q}^{\mathrm{sing}}}^{j - 1} \setminus L^{m - \ell - 1}}\) and applying \cref{lemma_homogeneous_extension} on every face \(\sigma \in \brk{\mathscr{Q}^{\mathrm{sing}}}^{j} \setminus \brk{\mathscr{Q}^{\mathrm{reg}}}^{j}\); the resulting map satisfies in view of \cref{lemma_homogeneous_extension} \ref{it_IWigh3phei9Ohjoolaineequ} for every \(x \in \bigcup \mathscr{Q}^{\mathrm{reg}} \cup \brk{\bigcup \brk{\mathscr{Q}^{\mathrm{sing}}}^j \setminus L^{m - \ell - 1}}\) the estimate
\begin{equation}
 \operatorname{Lip} U^{j} \brk{x} \le \frac{\beta_{j}}{\dist \brk{x,L^{m - \ell - 1}}}
\end{equation}
and  the condition 
\[
U^j \restr{\bigcup \mathscr{Q}^{\mathrm{reg}}} = U^{j - 1} \restr{\bigcup \mathscr{Q}^{\mathrm{reg}}} = W \restr{\bigcup \mathscr{Q}^{\mathrm{reg}}}.
\]

In order to conclude we take \(U \defeq U^m\).
We check immediately that 
\[U \restr{\bigcup \mathscr{Q}^{\mathrm{reg}}} = U^{m} \restr{\bigcup \mathscr{Q}^{\mathrm{reg}}} = W \restr{\bigcup \mathscr{Q}^{\mathrm{reg}}}, 
\]
and that \(U\) satisfies the estimates
 \eqref{eq_Umech5Ueyeic6oa7Ohngeevi} and  \eqref{eq_wohch3beevohDuegeraiB6zu} when \(\ell = m\) and \(\ell < m\) respectively.
Finally, it follows from \eqref{eq_wohch3beevohDuegeraiB6zu} and \eqref{eq_Umech5Ueyeic6oa7Ohngeevi} that \(U\) is weakly differentiable and that if \(\ell = m\) one has 
\begin{equation}
\label{eq_do4aephiusheegh7tiChaeha}
 \int_{\bigcup \mathscr{Q}^{\mathrm{sing}}} \abs{\Deriv U}^p
 \le \C \beta^p \sum_{k \in \Zset} 2^{k\brk{m - p}} \# \mathscr{Q}^{\mathrm{sing}}_k,
\end{equation}
whereas if \(\ell > p - 1\),
\begin{equation}
\label{eq_aevioyoo6shohsh3coh7Naiz}
 \int_{\bigcup \mathscr{Q}^{\mathrm{sing}}} \abs{\Deriv U}^p
 \le \frac{\C \beta^p}{\ell + 1 - p}  \sum_{k \in \Zset} 2^{k\brk{m - p}} \# \mathscr{Q}^{\mathrm{sing}}_k.
\end{equation}
In view of \eqref{eq_ee9wei0noyah1zee0xuQuie2} and either \eqref{eq_do4aephiusheegh7tiChaeha} or \eqref{eq_aevioyoo6shohsh3coh7Naiz}, 
it follows that \(U \in \smash{\dot{W}}^{1, p} \brk{\bigcup \mathscr{Q}^{\mathrm{reg}} \cup \bigcup \mathscr{Q}^{\mathrm{sing}}, \manifold{N}}\), with the estimate \eqref{eq_xuoY7IuceiwooH5Aicadai9a}. 
\end{proof}

At this point of our construction of the extension of Sobolev mappings, we have gathered all the tools to recover Hardt and Lin's extension result, in which they assumed the lower homotopy groups \(\pi_1 \brk{\manifold{N}}, \dotsc, \smash{\pi_{\floor{p - 1}}} \brk{\manifold{N}}\) to be \emph{all trivial} \cite{Hardt_Lin_1987} (see also \cite{Hardt_Kinderlehrer_Lin_1988}).
Indeed, taking the collections \(\mathscr{Q}^{\mathrm{sing}} \defeq \mathscr{Q}^{\mathrm{bad}}\) and
\(\mathscr{Q}^{\mathrm{reg}} \defeq \mathscr{Q}^{\mathrm{good}}\)
and the map \(W\) to be \(\Pi_{\manifold{N}} \compose V\) on \(\bigcup \mathscr{Q}^{\mathrm{reg}}\), which can be extended continuously to \(\smash{\bigcup \smash{\brk{\mathscr{Q}^{\mathrm{sing}}}^{\ell}}}\) with \(\ell =\min\set{\smash{\floor{p}}, m}\) since the homotopy groups \(\pi_1 \brk{\manifold{N}}, \dotsc, \pi_{\ell - 1} \brk{\manifold{N}}\) are trivial, and
applying then \cref{proposition_Lipschitz_extension_skeleton}, we get the map \(U \colon \bigcup \smash{\mathscr{Q}^{\vert \Omega}} \to \manifold{N}\); thanks to \eqref{eq_xuoY7IuceiwooH5Aicadai9a} and \eqref{eq_yie2te0ahrie6eeL2mah5Boh}, \(U \in \smash{\smash{\dot{W}}^{1, p} \brk{\bigcup \mathscr{Q}^{\vert \Omega}, \manifold{N}}}\) satisfies the required estimates;
by \cref{proposition_equal_traces}, we also have \(\tr_{\Omega \times \set{0}} U = u\) so that \(U\) has the prescribed trace.

The resulting proof of the extension of traces when the lower homotopy groups are trivial,
based on manipulations in the domain of the linear extension, is \emph{radically different} from the original construction of Hardt and Lin by manipulations in the target manifold thanks to a singular retraction.
In this particular case, where the lower homotopy groups are all trivial, we also note that the proof of \cref{proposition_Lipschitz_extension_skeleton} can be simplified, since one can apply at each step \(j \in \set{1, \dots, \ell - 1}\) the simpler procedure used in the proof when \(j = \ell\); in particular \cref{lemma_cube_extension_lipschitz_homotopic} can be replaced by the less delicate \cref{lemma_cube_extension_lipschitz} and the successive auxiliary maps \(W^0, \dotsc, W^{\ell - 1}\) do not need to be constructed in the \(\ell\) first inductive steps of the proof of \cref{proposition_Lipschitz_extension_skeleton}.

Zimmerman's construction of Sobolev extensions of Lipschitz mappings \cite{Zimmerman_2019} is similar to the one just described; thanks to the Lipschitz continuity of the mapping to be extended, he can afford to treat all cubes as bad cubes; good cubes would moreover be difficult to deal with a target space which is merely an \(\brk{m - 1}\)-connected metric space.

\section{Supercritical integrability extension}
\label{section_supercritical}
\resetconstant

If \(u \in \smash{\dot{W}}^{1 - 1/p, p} \brk{\Omega, \manifold{N}}\), if \(p > m\) and if \(\mathscr{Q}\) is a dyadic decomposition of \(\Rset^m_+\), 
one can take \(k_0 \in \Zset\) small enough so that the collection of bad cubes \(\smash{\mathscr{Q}^{\mathrm{bad}}}\) defined in \eqref{eq_definition_bad_cube}
satisfies in view of the estimate \eqref{eq_yie2te0ahrie6eeL2mah5Boh} the inclusion
\begin{equation}
\label{eq_chiemooReix5bahChee8Vohn}
\mathscr{Q}^{\mathrm{bad}} \subseteq 
\bigcup_{k = k_0}^{\infty} \mathscr{Q}_k,
\end{equation}
with even a quantitative estimate on \(k_0\); 
hence the construction of \sectref{section_badset} readily yields a map \(\Pi_{\manifold{N}} \compose V \in \smash{\dot{W}}^{1, p} \brk{\smash{\widehat{\Omega}} \cap \brk{\Rset^{m - 1} \times \intvo{0}{2^{k_0}}}, \manifold{N}}\); and it remains to perform a suitable construction on \(\smash{\widehat{\Omega}} \cap \smash{\brk{\Rset^{m - 1} \times \intvo{2^{k_0}}{\infty}}}\).
In view of \cref{proposition_Lipschitz_extension_skeleton} it will be sufficient to perform a continuous extension on a collection of singular cubes \(\smash{\mathscr{Q}^{\mathrm{sing}}}\) that contains the collection of bad cubes \(\smash{\mathscr{Q}^{\mathrm{bad}}}\).

In order to postpone to the main proof the more tedious handling of the domain \(\Omega\), of its tent \(\smash{\widehat{\Omega}}\) and of the associated dyadic cubes \(\mathscr{Q}^{\vert \Omega}\), we introduce the notion of descending map.

\begin{definition}
\label{definition_descending}
Given \(A \subseteq \Rset^m_+\), a 
mapping \(\Psi \colon A \to \Rset^{m}_+\) is \emph{descending} with respect to a dyadic decomposition \(\mathscr{Q}\) of \(\Rset^m_+\) whenever for every \(k \in \Zset\) and for every cube \(Q = Q' \times \intvc{2^{k}}{2^{k + 1}}  \in \mathscr{Q}_k\), one has 
\(\Psi \brk{A \cap Q} \subseteq Q' \times \intvl{0}{2^{k+1}}\).
\end{definition}

Geometrically, a map is descending with respect to \(\mathscr{Q}\) whenever the image of the intersection of its domain with any dyadic cube of \(\mathscr{Q}\) is contained in the convex hull of the cube and its orthogonal projection on the hyperplane \(\Rset^{m - 1} \times \set{0}\).
% The crucial property of a descending mapping  \(\Psi \colon A \to \Rset^{m}_+\)  for our purpose is that for every open set \(\Omega \subseteq \Rset^{m}_+\), one has 
% \[
% \Psi \brk{A \cap \smash{\bigcup \mathscr{Q}^{\vert \Omega}}} \subseteq \smash{\bigcup \mathscr{Q}^{\vert \Omega}}.
% \]

The next proposition provides a collection of singular cubes \(\smash{\mathscr{Q}^{\mathrm{sing}}}\) that contains the collection of bad cubes \(\smash{\mathscr{Q}^{\mathrm{bad}}}\) and on which one can define a descending continuous mapping to non-singular cubes that can be used to perform extensions.

\begin{proposition}
\label{proposition_modifiedcubes_supercritical}
For every dyadic decomposition \(\mathscr{Q}\) of \(\Rset^{m}_+\), every \(k_0 \in \Zset\) and every collection of cubes 
\(
\mathscr{Q}^{\mathrm{bad}}
 \subseteq  \bigcup_{k = k_0}^{\infty} \mathscr{Q}_k
\),
there exists a collection of cubes \(\mathscr{Q}^{\mathrm{sing}} \subseteq \bigcup_{k = k_0}^{\infty} \mathscr{Q}_k\) and a mapping \(\Psi^{\mathrm{sing}}\colon \bigcup  \mathscr{Q} \to 
\bigcup \brk{\mathscr{Q} \setminus \mathscr{Q}^\mathrm{sing}}\)
such that 
\begin{enumerate}[label=(\roman*)]
  \item 
  \label{it_Naejishaes0taighei7bo4Ui}
  \(\mathscr{Q}^{\mathrm{sing}} \supseteq \mathscr{Q}^{\mathrm{bad}}\),
  \item 
  \label{it_jee7giebu5Moo3rei5aewohp}
  for every \(k \ge k_0\),
  \begin{equation}
    \label{eq_eSu0lae0quatheekusochie6}
    \# \mathscr{Q}^{\mathrm{sing}}_k
    \le 
    \smashoperator[r]{\sum_{i = k_0}^k} \# \mathscr{Q}^{\mathrm{bad}}_i,
  \end{equation}
  \item 
  \label{it_dee7moib9iihoNgieWahloog}
  \(\Psi^{\mathrm{sing}} = \id\) on 
  \(\bigcup \brk{\mathscr{Q} \setminus \mathscr{Q}^{\mathrm{sing}}}\),
  \item 
  \label{it_aic8oaw3VekeWiecawah4pho}
  \(\Psi^{\mathrm{sing}}\) is descending.
\end{enumerate}
\end{proposition}

Although the statement of \cref{proposition_modifiedcubes_supercritical} uses  the same notation \(\smash{\mathscr{Q}^{\mathrm{bad}}}\) as in \sectref{section_badset}, the result and its proof are purely combinatorial and do not rely in any way on the definition of the collection of bad cubes in \eqref{eq_definition_bad_cube}.

\medbreak

Given \(j \in \set{1, \dotsc, m}\), each vertical face \(\sigma \in \brk{\mathscr{Q}}^{j, \perp}\) (as defined in  \eqref{eq_deeyooy8gah0vai7pheiSo0E} and \eqref{eq_Xai3tekei9ao1Ahsh9rae9ae}) can be written as \(\sigma = \sigma' \times \intvc{2^{k}}{2^{k + 1}}\)
for some \(k \in \Zset\) and \(\sigma' \subseteq \Rset^{m - 1}\) such that \(\sigma' \times \set{2^{k + 1}} \in \smash{\brk{\mathscr{Q}}^{j - 1, \top}_k}\),
its \emph{upper boundary} is defined as 
\begin{equation}
\label{eq_hoofi5Emae8xaepa8queef7r}
\partial_{\top} \sigma
\defeq \sigma' \times \set{2^{k + 1}}
\end{equation}
and its \emph{enclosing boundary} as
\begin{equation}
\label{eq_ap8pee3oquaiH5iig9phahsh}
 \partial_{\sqcup} \sigma \defeq 
 \brk{\partial \sigma' \times \intvc{2^k}{2^{k + 1}}}
 \cup 
 \brk{\sigma' \times \set{2^k}}
\end{equation}
(\( \partial_{\sqcup} \sigma \) is also known as the parabolic boundary, although the relevance of this terminology is weak in the present work where there is no parabolic partial differential equation).
By the following lemma and a change of variables, there always exists a descending retraction from a vertical face \(\sigma\) to its enclosing boundary \(\partial_{\sqcup} \sigma\).

\begin{lemma}
\label{lemma_descending}
Let \(W \subseteq \Rset^{m - 1}\) be convex, open and bounded. If \(0 \in W\), then the map \(\Psi \colon W \times \intvc{0}{1} \to \partial_{\sqcup} \brk{W \times \intvc{0}{1}}\), where \(\partial_{\sqcup}\brk{W \times \intvc{0}{1}} \defeq \brk{W \times \set{0}} \cup \brk{\partial W \times \intvc{0}{1}}\) defined for each \(x = \brk{x', x_m} \in W \times \intvc{0}{1}\) by 
\begin{equation*}
\Psi \brk{x}
\defeq 
\begin{cases}
 \brk[\big]{\frac{2x'}{2 - x_m}, 0} &\text{if \(\gamma_W (x') \le 1 - \frac{x_m}{2},\)}\\
 \brk[\big]{\frac{x'}{\gamma\brk{x'}}, 2 - \frac{2 - x_m}{\gamma \brk{x'}}}
 &\text{otherwise,}
\end{cases}
\end{equation*}
where \(\gamma_W\colon \Rset^{m - 1} \to \intvr{0}{\infty}\) is the Minkowski functional  (gauge) of \(W\) defined for \(x' \in W\) by 
\begin{equation*}
 \gamma_{W}\brk{x'} \defeq \inf \set{t \in \intvo{0}{\infty} \st x'/t \in W},
\end{equation*}
is continuous and \(\Psi \restr{\partial_{\sqcup} \brk{W \times \intvc{0}{1}}} = \id_{\partial_{\sqcup} \brk{W \times \intvc{0}{1}}}\).
\end{lemma}

\begin{proof}
Since \(W\) is open and bounded, and since \(0 \in W\), the function \(0 \le \gamma_W <1\) on \(W\). The continuity of \(\gamma_W\) follows from the inequality
\[
 \gamma_W \brk{x' + y'} \le \gamma_W \brk{x'} + \gamma_W \brk{y'}
\]
(for the details, see for example \cite{Brezis_2011}*{Lemma 1.2}).
One can then compute directly the boundary properties.
\end{proof}

We will also need the notion of adjacent faces:
two faces \(\sigma, \tau \in \brk{\mathscr{Q}}^j\) are \emph{adjacent} whenever \(\sigma \cap \tau \in \brk{\mathscr{Q}}^{j - 1}\).

The next lemma shows the relation between the upper boundary of faces and vertical faces of dyadic cubes of different generations.

\begin{lemma}
\label{lemma_cubes_top_boundary}
Let \(\mathscr{Q}\) be a dyadic decomposition of \(\Rset^m_+\), let \(\ell \in \set{1, \dotsc, m}\) and let \(j \in \set{1, \dotsc, \ell}\).
If \(\sigma \in \smash{\brk{\mathscr{Q}}^{j, \perp}_{k - 1}}\), if \(\tau \in \smash{\brk{\mathscr{Q}}^{\ell, \perp}_{k}}\) and if \(\partial_{\top} \sigma \subseteq \tau\),
then there exists \(\rho \in \smash{\brk{\mathscr{Q}}^{\ell, \perp}_{k - 1}}\) such that \(\rho\) is adjacent to \(\tau\) and \(\sigma \subseteq \rho  \). 
\end{lemma}

\begin{proof}
This follows from the structure of the \(\ell\)-dimensional complex \(\brk{\mathscr{Q}}^{\ell, \perp}\) defined in  \eqref{eq_deeyooy8gah0vai7pheiSo0E} and \eqref{eq_Xai3tekei9ao1Ahsh9rae9ae}. 
\end{proof}

\begin{figure}
\begin{center}
 \includegraphics[alt={The construction of the singuler cubes and the associated retraction in the supercritical case}]{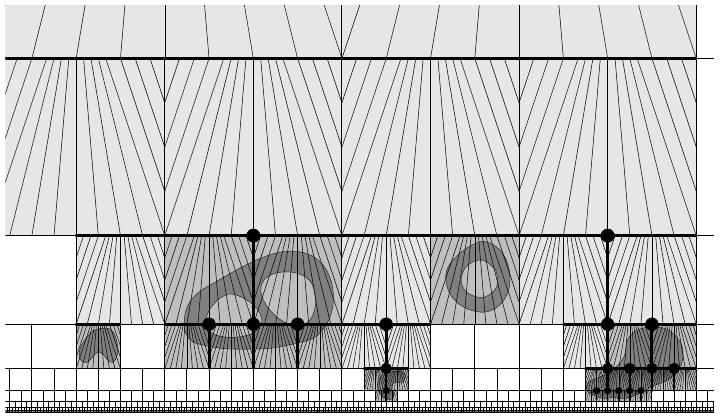}
\end{center}
\caption{
In the supercritical construction represented here when \(m = 2\) in \cref{proposition_modifiedcubes_supercritical}, the bad cubes \(\smash{\mathscr{Q}^{\mathrm{bad}}}\), in medium gray, are defined as those intersecting the bad set \(\smash{\widehat{\Omega}}^{\mathrm{bad}}\), in dark gray; singular cubes \(\smash{\mathscr{Q}^{\mathrm{sing}}}\), in light gray, are defined recursively as being any bad cube or any cube adjacent to a singular cube at the previous scale.
The number of singular cubes at any scale is controlled by the number of bad cubes at previous scales (see \eqref{eq_eSu0lae0quatheekusochie6}).
The union of the singular cubes can be retracted on its boundary along the thick vertical lines and then along the bundles of lines visible on each square.
}
\label{figure_singretr2}
\end{figure}

\begin{proof}[Proof of \cref{proposition_modifiedcubes_supercritical}]
We define recursively the collection \(\smash{\mathscr{Q}^{\mathrm{sing}}_k} \subseteq \mathscr{Q}_k\) for every \(k \in \Zset\). 
We first set for \(k < k_0\), \(\smash{\mathscr{Q}^{\mathrm{sing}}_{k}} \defeq \emptyset\). 
Next assuming that \(k \ge k_0\) and that \(\smash{\mathscr{Q}^{\mathrm{sing}}_{k - 1}}\) has already been defined, we let  
\begin{equation}
\label{eq_daiXeivoht6ohXiew6oiShei}
\mathscr{Q}^\mathrm{prop}_k
\defeq 
 \set{Q \in \mathscr{Q}_k 
 \st Q \text{ is adjacent to some } R \in \mathscr{Q}^{\mathrm{sing}}_{k - 1}}
\end{equation}
and then 
\begin{equation}
\label{eq_nei9quaongie3gaivahguHie}
 \mathscr{Q}^{\mathrm{sing}}_{k}
 \defeq \mathscr{Q}^{\mathrm{bad}}_k
 \cup \mathscr{Q}^\mathrm{prop}_k.
\end{equation}
Since a cube \(R \in \mathscr{Q}_{k - 1}\) is adjacent to a single larger-scale cube \(Q \in \mathscr{Q}_{k}\),
any cube in \(\smash{\mathscr{Q}^{\mathrm{sing}}_{k - 1}}\) is used at most once in the definition \eqref{eq_daiXeivoht6ohXiew6oiShei} of \(\mathscr{Q}^\mathrm{prop}_k\) so that 
\begin{equation}
\label{eq_ziez3taFeerae5ai0chied2n}
 \# \mathscr{Q}^\mathrm{prop}_k
 \le \# \mathscr{Q}^\mathrm{sing}_{k - 1},
\end{equation}
and thus by \eqref{eq_nei9quaongie3gaivahguHie} and \eqref{eq_ziez3taFeerae5ai0chied2n},
\begin{equation}
 \# \mathscr{Q}^\mathrm{sing}_{k}
 \le \# \mathscr{Q}^\mathrm{sing}_{k - 1} + \# \mathscr{Q}^{\mathrm{bad}}_k,
\end{equation}
from which the inequality \eqref{eq_eSu0lae0quatheekusochie6} follows.
We have thus proved the assertions \ref{it_Naejishaes0taighei7bo4Ui} and \ref{it_jee7giebu5Moo3rei5aewohp}.

In order to prove the assertion \ref{it_dee7moib9iihoNgieWahloog}, we need to define the mapping \(\Psi^{\mathrm{sing}}\).
First, we prescribe \(\Psi^{\mathrm{sing}}\) to be the identity on \(\smash{\bigcup \mathscr{Q} \setminus \mathscr{Q}^{\mathrm{sing}}}\).
Since 
\(
  \smash{\bigcup \mathscr{Q}^{\mathrm{sing}}} \subseteq
  \smash{
   \bigcup_{k = k_0}^{\infty} \mathscr{Q}_k}
\),
we proceed now by induction over \(k \in \Zset\) with \(k \ge k_0\) to define \(\Psi^{\mathrm{sing}}\) on \(\smash{\bigcup \mathscr{Q}_k}\).
We observe that \(\Psi^{\mathrm{sing}}\) is then the identity on \(\smash{\bigcup \mathscr{Q}_{k - 1}}\) by definition when \(k \le k_0\) and by induction when \(k \ge k_0 + 1\).
We define inductively for every \(j \in \set{1, \dotsc, m}\) the map \(\Psi^{\mathrm{sing}}\) on each face 
\(\sigma \in \smash{\brk{\mathscr{Q}}^{j, \perp}_k \setminus  \brk{\mathscr{Q} \setminus \mathscr{Q}^{\mathrm{sing}}}^{j, \perp}_k}\) by noting that by \cref{lemma_cubes_top_boundary} and by \eqref{eq_daiXeivoht6ohXiew6oiShei},
\(\partial_{\top} \sigma \not \in \smash{\brk{\mathscr{Q} \setminus \mathscr{Q}^{\mathrm{sing}}}^{j - 1, \top}_k}\) so that \(\smash{\Psi^{\mathrm{sing}}}\) can be defined on \(\sigma\) thanks to a descending retraction from \(\sigma\) onto its enclosing boundary \(\partial_{\sqcup} \sigma\) given by \cref{lemma_descending} so that the assertions  \ref{it_dee7moib9iihoNgieWahloog} and \ref{it_aic8oaw3VekeWiecawah4pho} are satisfied.
\end{proof}

% \begin{remark}
% It can be observed in the proof of \cref{proposition_modifiedcubes_supercritical}, 
% that the mapping \(\Psi^{\mathrm{sing}}\) is a \emph{deformation retraction}, since the map provided by \cref{lemma_descending} also is.
% As a consequence, \(\pi_j \brk{\Gamma}\) is trivial for any connected component \(\Gamma\)
% of \(\bigcup \mathscr{Q}^{\mathrm{sing}}\).
% \end{remark}

We are now in position to extend traces and estimate the resulting extension under the supercritical integrability assumption.

\begin{theorem}
\label{theorem_extension_tent_supercritical}
Let \(p \in \intvo{m}{\infty}\).
Assume that the homotopy groups \(\pi_{1} \brk{\manifold{N}}, \dotsc, \pi_{m - 1} \brk{\manifold{N}}\) are finite.
There exists a constant \(C \in \intvo{0}{\infty}\) such that 
for every open set \(\Omega \subseteq \Rset^{m - 1}\), every dyadic decomposition \(\mathscr{Q}\) of \(\Rset^m_+\) and every mapping \(u \in \smash{\smash{\dot{W}}^{1 - 1/p, p} \brk{\Omega, \manifold{N}}}\),
there exists a map \(U \colon \smash{\bigcup \mathscr{Q}^{\vert \Omega}} \to \manifold{N}\) 
such that 
\begin{enumerate}[label=(\roman*)]
 \item 
 \label{it_uqueiphealohyahc5tooZooh}
 \(U \in \smash{\dot{W}}^{1, p} \brk{\bigcup \mathscr{Q}^{\vert \Omega}, \manifold{N}}\) and
 \begin{equation*}
   \int_{\bigcup \mathscr{Q}^{\vert \Omega}} \abs{\Deriv U}^p
   \le C
   \smashoperator{\iint_{\Omega \times \Omega}} \frac{d \brk{u \brk{y}, u \brk{z}}^p}{\abs{y - z}^{p + m - 2}} \dif y \dif z,
 \end{equation*}
 \item 
 \label{it_vai7ne5eiwaNa3kaNgequ1ui}
 \(\tr_{\Omega \times \set{0}} U = u\),
 \item
 \label{it_AiDiecaeh7xachoogh0Auchi}
 for almost every \(x \in \bigcup \mathscr{Q}^{\vert \Omega}\),
 \begin{equation*}
    \abs{\Deriv U \brk{x}} \le
    \frac{C}{\dist \brk{x, \partial \Rset^{m}_+}}.
 \end{equation*}
\end{enumerate}

\end{theorem}

\begin{proof}
We define the linear extension \(V \colon \widehat{\Omega} \to \Rset^\nu\) and the collection of bad cubes \(\mathscr{Q}^{\mathrm{bad}} \subseteq \mathscr{Q}^{\vert \Omega}\) according to \eqref{eq_shee0Ceingooghe2weefei4A} and \eqref{eq_definition_bad_cube} respectively.
We let the collection \(\smash{\mathscr{Q}^{\mathrm{sing}}} \subseteq \mathscr{Q}\) and the mapping \(\smash{\Psi^{\mathrm{sing}}} \colon \Rset^m_+ \to \bigcup \mathscr{Q} \setminus \smash{\mathscr{Q}^{\mathrm{sing}}}\) be given by \cref{proposition_modifiedcubes_supercritical},
and we define the map \(W \defeq \Pi_{\manifold{N}} \compose V \compose \smash{\Psi^{\mathrm{sing}}} \colon 
\smash{\bigcup \mathscr{Q}^{\vert \Omega}} \to \manifold{N}\).
Since \(\Psi^{\mathrm{sing}}\) is descending, the mapping \(W\) is well-defined.
We let \(U \colon \bigcup \mathscr{Q}^{\vert \Omega} \to \manifold{N}\) be the map given by \cref{proposition_Lipschitz_extension_skeleton} with \(\ell = m\) and the collections \(\smash{\mathscr{Q}^{\vert \Omega}} \cap \smash{\mathscr{Q}^{\mathrm{sing}}}\) and \(\smash{\mathscr{Q}^{\vert \Omega}} \setminus \smash{\mathscr{Q}^{\mathrm{sing}}}\)
instead of \(\smash{\mathscr{Q}^{\mathrm{sing}}}\) and \(\smash{\mathscr{Q}^{\mathrm{reg}}}\) in the statement respectively.

The pointwise estimate \ref{it_AiDiecaeh7xachoogh0Auchi} follows from \eqref{eq_Umech5Ueyeic6oa7Ohngeevi} in \cref{proposition_Lipschitz_extension_skeleton}.

In order to get the estimate \ref{it_uqueiphealohyahc5tooZooh} as a consequence of  \eqref{eq_xuoY7IuceiwooH5Aicadai9a} in \cref{proposition_Lipschitz_extension_skeleton},
we estimate the size of the singular cubes \(\smash{\mathscr{Q}^{\mathrm{sing}}}\) by \eqref{eq_chiemooReix5bahChee8Vohn}, \eqref{eq_eSu0lae0quatheekusochie6} and by \eqref{eq_yie2te0ahrie6eeL2mah5Boh}
\begin{equation}
\label{eq_oujoh1zaicae8OonaiHiefoa}
\begin{split}
\smashoperator[r]{\sum_{k = k_0}^\infty} 2^{k\brk{m - p}}\# \mathscr{Q}^{\mathrm{sing}}_k
 &\le \sum_{k =k_0}^\infty \smashoperator[r]{\sum_{i = k_0}^k}  2^{k\brk{m - p}} \# \mathscr{Q}^{\mathrm{bad}}_i\\
 &= \sum_{i = k_0}^\infty \sum_{k = i}^\infty 2^{-k\brk{p- m}} \# \mathscr{Q}^{\mathrm{bad}}_i\\
 &\le  \frac{1}{1 - 2^{-\brk{p - m}}} \smashoperator{\sum_{i = k_0}^\infty} 2^{i\brk{m- p}} \# \mathscr{Q}^{\mathrm{bad}}_i\\
 &\le \C
 \smashoperator{\iint_{\Omega \times \Omega}} \frac{\brk{d \brk{u(y), u (z)} - \delta_*}_+}{\abs{y - z}^{p + m - 2}} \dif y \dif z ,
\end{split}
\end{equation}
relying on the condition \(p > m\) in the summation of the geometric series.

The trace condition of assertion \ref{it_vai7ne5eiwaNa3kaNgequ1ui} follows from the fact that \(\tr_{\Omega \times \set{0}} V = u\) and from the fact that \(U = \Pi_{\manifold{N}} \compose V\) in a neighbourhood of \(\Omega \times \set{0}\), and the proof is complete.
\end{proof}

The estimate \eqref{eq_oujoh1zaicae8OonaiHiefoa} of the size of the singular cubes is reminiscent of integral estimates on the quantity \(d \brk{x}\) describing in several works the size of the neighbourhood in which a linear extension stays close to the target manifold
for \emph{critical Sobolev estimates} on the \emph{Brouwer degree} \cite{Bourgain_Brezis_Mironescu_2005}*{Th.\thinspace{}1.1 and Th.\thinspace{}0.6}, on the \emph{lifting} \cite{Bourgain_Brezis_Mironescu_2005}*{Th.\thinspace{}0.1} and on the \emph{distributional Jacobian} \cite{Bourgain_Brezis_Mironescu_2005}*{Th.\thinspace{}0.8}.  

\medbreak

As a byproduct of \cref{theorem_extension_tent_supercritical}, we immediately deduce the linear estimate for the extension to the half-space of traces of mappings in Sobolev spaces with supercritical integrability of \cref{theorem_estimate_supercritical_halfspace}.

\begin{proof}[Proof of \cref{theorem_estimate_supercritical_halfspace}]
This follows from \cref{theorem_extension_tent_supercritical} with \(\Omega = \Rset^{m - 1}\), since one then has
\(\smash{\bigcup \mathscr{Q}^{\vert \Omega}} = \smash{\bigcup \mathscr{Q}} = \smash{\widehat{\Omega}} = \Rset^m_+\).
\end{proof}

\section{Critical integrability extension}
\label{section_critical}

We now turn our attention to the critical integrability case \(p =  m\).
As already discussed for the supercritical case in \sectref{section_supercritical}, we know from \sectref{section_badset} that the nearest-point projection \(\Pi_{\manifold{N}} \compose V\) of the linear extension \(V\) is a suitable extension in each of the good cubes of the collection \(\smash{\mathscr{Q}^{\mathrm{good}}}\) and that the size of the collection of bad cubes \(\smash{\mathscr{Q}^{\mathrm{bad}}}\) is controlled by the Gagliardo energy of the boundary datum \(u\) by \eqref{eq_yie2te0ahrie6eeL2mah5Boh}. 
Thus, our goal is to perform a continuous extension on the \(m\)-dimensional skeleton \(\smash{\brk{\mathscr{Q}^{\mathrm{sing}}}^{m}}\) of a collection of singular cubes \(\smash{\mathscr{Q}^{\mathrm{sing}}} \supseteq \smash{\mathscr{Q}^{\mathrm{bad}}}\) so that the construction of the Sobolev extension can be realised thanks to \cref{proposition_Lipschitz_extension_skeleton}.

When \(p > m\), we could find in \cref{proposition_modifiedcubes_supercritical} a collection of singular cubes \(\mathscr{Q}^{\mathrm{sing}} \supseteq \mathscr{Q}^{\mathrm{bad}}\) to which the map \(\Pi_{\manifold{N}} \compose V\) could be extended to continuously while the size of the collection of singular cubes \(\smash{\mathscr{Q}^{\mathrm{sing}}}\) was controlled by the size of the collection of bad cubes \(\smash{\mathscr{Q}^{\mathrm{bad}}}\) by \eqref{eq_eSu0lae0quatheekusochie6} and then by the fractional energy of the boundary datum in  \eqref{eq_oujoh1zaicae8OonaiHiefoa}.
This proof will not work when \(p \le m\), because the geometric series that appears in \eqref{eq_oujoh1zaicae8OonaiHiefoa} then diverges.

In this section we introduce an averaging argument that allows to improve on the estimate \eqref{eq_oujoh1zaicae8OonaiHiefoa} and get a continuous extension to the \(\brk{m - 1}\)-dimensional skeleton \(\smash{\brk{\mathscr{Q}^{\mathrm{sing}}}^{m - 1}}\) of a collection of singular cubes \(\smash{\mathscr{Q}^{\mathrm{sing}}} \supseteq \smash{\mathscr{Q}^{\mathrm{bad}}}\). We then use the triviality of \(\pi_{m - 1}\brk{\manifold{N}}\) to get a continuous extension to \(m\)-dimensional skeleton \(\smash{\brk{\mathscr{Q}^{\mathrm{sing}}}^{m}}\).

 \subsection{Propagation and decay of singular cubes}
As explained above, the estimate \eqref{eq_eSu0lae0quatheekusochie6} is too weak to guarantee a suitable bound on the size of the collection of singular cubes when \(p \le m\) due to the divergence of the geometric series in \eqref{eq_oujoh1zaicae8OonaiHiefoa}.

To follow the same strategy we need to improve the bound \eqref{eq_eSu0lae0quatheekusochie6}. For \(p \le m\), a better estimate on the collection of singular cubes is possible if we accept an extension to a skeleton of the collection of singular cubes of lower dimension \(\ell\).
A key point in the proof is that we can perform it thanks to the crucial observation that the fraction of vertical faces of \(\smash{\brk{\mathscr{Q}}^{\ell, \perp}_{k - 1}}\) that are adjacent to vertical faces at the next scale of \(\smash{\brk{\mathscr{Q}}^{\ell, \perp}_{k}}\) is \(\smash{1/2^{m - \ell}}\); a suitable averaging argument shows that such a small fraction can be effectively achieved in the propagation of vertical singular faces so that we obtain a suitable exponential decay of the number of singular cubes generated by bad cubes.
The resulting bound will be sufficient to close the argument when \(\ell < p < \ell + 1\);
when \(p\) is an integer, the dimension \(p - 1\) of the skeleton to which we can continuously extend the map will not match the dimension \(p\) required in \cref{proposition_Lipschitz_extension_skeleton} to perform a singular extension with the desired Sobolev integrability, and we will crucially need the triviality of the homotopy group \(\pi_{p - 1} \brk{\manifold{N}}\), which may seem like an ad hoc trick in the proof but turns out to be a necessary condition because of the critical analytical obstruction \cite{Mironescu_VanSchaftingen_2021_AFST}*{Th.\thinspace{}1.5 (b)}.

%TODO Initialise the procedure: take singular cubes

\begin{proposition}
\label{proposition_general_modification}
Let \(\ell \in \set{1, \dotsc, m}\).
Given a dyadic
decomposition \(\smash{\mathscr{Q}^{\infty}}\) of
\(\Rset^{m}_+\), a collection of cubes \(\smash{\mathscr{Q}^{\mathrm{sing}, \infty}} \subseteq \smash{\mathscr{Q}^{\infty}}\), an open set \(\smash{A^{\mathrm{bad}}} \subseteq \bigcup \smash{\mathscr{Q}^{\mathrm{sing}, \infty}} \),
and a descending map \(\smash{\Psi^{\mathrm{sing}, \infty}} \colon \bigcup \smash{\brk{\mathscr{Q}^\infty}^\ell}
\to \Rset^m_+ \setminus  \smash{A^{\mathrm{bad}}}\) such that \(\Psi^{\mathrm{sing}, \infty} = \id\) on
\(\smash{\bigcup \brk{\smash{\mathscr{Q}^{\infty}} \setminus \smash{\smash{\mathscr{Q}^{\mathrm{sing}, \infty}}}}^\ell}\),
then for every \(k_0 \in \Zset\), there exists a dyadic decomposition \(\smash{\mathscr{Q}^{k_0}}\) of \(\Rset^m_+\),
a collection of cubes \(\smash{\mathscr{Q}^{\mathrm{sing}, k_0} \subseteq \mathscr{Q}^{k_0}}\),
and a descending mapping \(\smash{\Psi^{\mathrm{sing}, k_0} \colon \bigcup \brk{\mathscr{Q}^{k_0}}^{\ell} \to \Rset^m_+ \setminus \smash{A^{\mathrm{bad}}}}
\)
such that if
\begin{equation}
\label{eq_che7ahweeghieCi4aqu7mi5b}
  \mathscr{Q}^{\mathrm{bad}, k_0}
  \defeq
  \set{Q \in \mathscr{Q}^{k_0} \st Q \cap \smash{A^{\mathrm{bad}}} \ne \emptyset },
\end{equation}
then
\begin{enumerate}[label=(\roman*)]
 \item
 \label{it_Dih1Poh0nohhahTh2chi8sah}
 for every \(k \in \Zset\) such that \(k < k_0\), one has
 \(
  \mathscr{Q}^{k_0}_k
  = \smash{\mathscr{Q}^{\infty}_k}\) and
 \(
  \mathscr{Q}^{\mathrm{sing}, k_0}_{k}= \smash{\mathscr{Q}^{\mathrm{sing}, \infty}_{k}}
 \),

 \item
 \label{eq_EiKah0oreekaBohwohbighie}
 \(
 \mathscr{Q}^{\mathrm{bad}, k_0}
 \subseteq
   \mathscr{Q}^{\mathrm{sing}, k_0}\),
 \item
 \label{it_omei1eoC1Iebibohhoh8so0z}
 \(\Psi^{\mathrm{sing}, k_0} = \id\)
 on \(\bigcup  \brk{\mathscr{Q}^{k_0} \setminus \mathscr{Q}^{\mathrm{sing}, k_0}}^{\ell}\),
 \item
 \label{it_taa6web7ien1quae8Xuey1xa}
 for every \(k  \in \Zset\) such that \(k \ge k_0\),
\begin{equation}
\label{eq_mai6Xieph9uachahChengoon}
  \# \mathscr{Q}^{\mathrm{sing}, k_0}_{k}
  \le
 \tbinom{m - 1}{\ell - 1}
  \brk[\Big]{2^{-\brk{k - k_0}\brk{m - \ell}} \#  \smash{\mathscr{Q}^{\mathrm{sing},\infty}_{k_0 - 1}} + \smashoperator{\sum_{i = k_0}^{k}}
  2^{-\brk{k - i - 1}\brk{m - \ell}} \# \mathscr{Q}^{\mathrm{bad}, k_0}_i}.
\end{equation}
\end{enumerate}
\end{proposition}

\begin{figure}
\begin{center}
 \includegraphics[alt={The construction of the singuler cubes and the associated retraction in the subcritical case}]{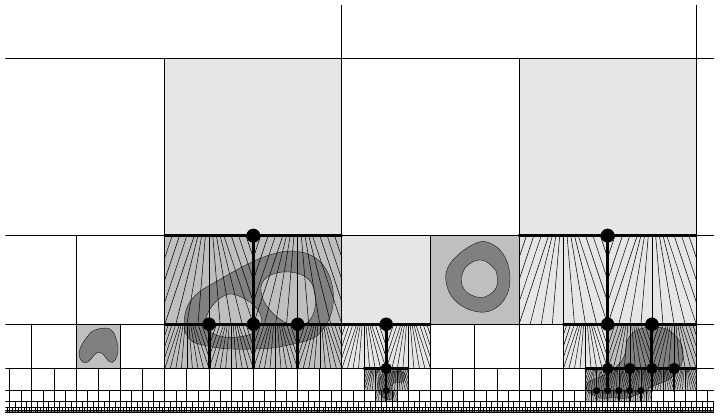}
\end{center}
\caption{
In the subcritical construction represented here when \(m = 2\) and \(\ell = 1\) in \cref{proposition_general_modification}, the singular vertical faces \(\smash{\mathscr{T}^{\mathrm{sing}, k_0}_k}\) are recursively defined as those that either touch only bad cubes of \(\smash{\mathscr{Q}^{\mathrm{bad}}}\) or those that touch a singular vertical face at a lower scale; the singular cubes \(\smash{\mathscr{Q}^{\mathrm{sing}}}\)  are then defined as those that either are bad or contain the upper boundary \(\partial_{\top} \sigma\) of a singular face \(\sigma\).
The number of singular cubes at each scale is controlled by the number of bad cubes at the previous scales (see \eqref{eq_mai6Xieph9uachahChengoon}), provided that the dyadic decomposition is well chosen at each step thanks to an averaging argument.
The continuous extension is performed on the singular vertical lines and then extended to the thick horizontal singular lines along the bundles of lines visible on each square.}
\end{figure}

As discussed for \cref{proposition_modifiedcubes_supercritical},
whereas the same notation \(\smash{\mathscr{Q}^{\mathrm{bad}}}\) as in \sectref{section_badset} appears in the statement of  \cref{proposition_general_modification}, the result and its proof are purely combinatorial and do not rely in any way on the definition of bad cubes in \eqref{eq_definition_bad_cube}, but merely on the presence of an arbitrary open set \(\smash{A^{\mathrm{bad}}} \subseteq \Rset^m_+\).

\begin{proof}[Proof of \cref{proposition_general_modification}]
We are first going to define inductively vectors \(\xi_k \in \Rset^{m - 1} \simeq \Rset^{m - 1} \times \set{0}\) and collections of vertical faces
\begin{equation}
\label{eq_av8chae6av3Niashahquoo5i}
  \mathscr{T}^{\mathrm{sing}, k_0}_{k} \subseteq \smash{\mathscr{Q}^{\ell, \perp}_{k}}
 \defeq \set{\sigma + \xi_{k}  \st \sigma \in \brk{\smash{\mathscr{Q}^{\infty}}}^{\ell, \perp}_{k}}
\end{equation}
for each \(k \in \Zset\) such that \(k \ge k_0 - 1\).
We let \(\xi_{k_0 - 1} \defeq 0 \in \Rset^{m - 1}\)
and
\begin{equation}
\label{eq_og1quaePhaelahbahchail5y}
 \mathscr{T}^{\mathrm{sing}, k_0}_{k_0 - 1}
 \defeq \set{\sigma \in \brk{\mathscr{Q}^\infty}^{\ell, \perp}_{k_0 - 1} \st \sigma \not \subseteq \textstyle \bigcup \brk{\mathscr{Q}^\infty \setminus \smash{\mathscr{Q}^{\mathrm{sing}, \infty}}}}
 \subseteq \smash{\mathscr{Q}^{\ell, \perp}_{k_0 - 1}} .
\end{equation}
Assuming that \(k \ge k_0\) and that the vector \(\xi_{k - 1} \in \Rset^{m - 1}\) and the collection
\(
 \smash{\mathscr{T}^{\mathrm{sing},k_0}_{k - 1}}
 \subseteq \smash{\mathscr{Q}^{\ell, \perp}_{k - 1}}\)
have already been defined, for every \(\zeta \in \set{0, 2^{k - 1}}^{m - 1} \subseteq \Rset^{m - 1}\), we set
\begin{equation}
\label{eq_uzahcohYe1haiche0gaiX1bo}
  \mathscr{Q}^{\ell, \perp}_{k, \zeta}
  \defeq \set{\sigma + \xi_{k - 1} + \zeta \st \sigma \in \brk{\smash{\mathscr{Q}^{\infty}}}^{\ell, \perp}_{k}}
\end{equation}
and
\begin{equation}
\label{eq_kuz0tooRuph1eevahs4ieb7r}
 \mathscr{T}^{\mathrm{prop}, k_0}_{k, \zeta}
 \defeq
 \set{\sigma \in
  \mathscr{Q}^{\ell, \perp}_{k, \zeta} \st
  \sigma \text{ is adjacent to some \(\tau \in \mathscr{T}^{\mathrm{sing}, k_0}_{k - 1}\) }
  }.
\end{equation}
Noting that for every face \(\tau \in \smash{\mathscr{Q}^{\ell, \perp}_{k - 1}}\) there are exactly \(2^{\ell - 1}\) translations \(\zeta \in \set{0, 2^{k - 1}}^{m - 1}\) such that \(\tau\) is adjacent to some vertical face of \(\sigma \in \smash{\mathscr{Q}^{\ell, \perp}_{k, \zeta}}\),
we have
\begin{equation}
 \smashoperator[r]{\sum_{\zeta \in \set{0, 2^{k - 1}}^{m - 1}}}
 \# \mathscr{T}^{\mathrm{prop}, k_0}_{k, \zeta}
 \le 2^{\ell - 1} \# \mathscr{T}^{\mathrm{sing}, k_0}_{k - 1}.
\end{equation}
We choose now \(\zeta_k \in \set{0, 2^{k - 1}}^{m - 1}\) such that
\begin{equation}
\label{eq_eiDeuGi8dee6aeB9phua9uel}
 \# \mathscr{T}^{\mathrm{prop}, k_0}_{k, \zeta_k}
 \le
 \frac{1}{2^{m - 1}}\hspace{-1em}
  \smashoperator[r]{\sum_{\zeta \in \set{0, 2^{k - 1}}^{m - 1}}}
 \# \mathscr{T}^{\mathrm{prop}, k_0}_{k, \zeta}
 \le 2^{-\brk{m - \ell}}
 \# \mathscr{T}^{\mathrm{sing},k_0}_{k - 1}.
\end{equation}
and we set
\begin{equation}
\label{eq_eo6oneiQuee1aic4oon6ek9w}
\xi_{k} \defeq \xi_{k - 1} + \zeta_k,
\end{equation}
 so that in view of \eqref{eq_av8chae6av3Niashahquoo5i}, \eqref{eq_uzahcohYe1haiche0gaiX1bo} and \eqref{eq_eo6oneiQuee1aic4oon6ek9w}, we have \(\mathscr{Q}^{\ell, \perp}_{k} = \mathscr{Q}^{\ell, \perp}_{k, \zeta_k}\).
We then
define successively
\begin{equation}
\label{eq_oov6aer7rouFueghiefoovua}
 \mathscr{T}^{\mathrm{prop}, k_0}_{k} \defeq \mathscr{T}^{\mathrm{prop}, k_0}_{k, \zeta_k} \subseteq \mathscr{Q}^{\ell, \perp}_{k},
\end{equation}
\begin{equation}
\label{eq_Da8pei5Va7aihaig1aethoch}
 \mathscr{T}^{\mathrm{bad}, k_0}_k
 \defeq \set{\sigma \in \mathscr{Q}^{\ell, \perp}_{k} \st
 \sigma \cap \smash{A^{\mathrm{bad}}} \ne \emptyset},
\end{equation}
and
\begin{equation}
\label{eq_Zo4ux0ooj6uXeev9ahphejov}
  \mathscr{T}^{\mathrm{sing}, k_0}_k
 \defeq
 \mathscr{T}^{\mathrm{prop}, k_0}_k
 \cup
 \mathscr{T}^{\mathrm{bad}, k_0}_k \subseteq \mathscr{Q}^{\ell, \perp}_{k}.
\end{equation}
It follows from \eqref{eq_eiDeuGi8dee6aeB9phua9uel}, \eqref{eq_oov6aer7rouFueghiefoovua} and  \eqref{eq_Zo4ux0ooj6uXeev9ahphejov} that
\begin{equation}
 \# \mathscr{T}^{\mathrm{sing}, k_0}_k
 \le 2^{-\brk{m - \ell}} \# \mathscr{T}^{\mathrm{sing}, k_0}_{k - 1} + \# \mathscr{T}^{\mathrm{bad}, k_0}_{k},
\end{equation}
which implies by induction in view of \eqref{eq_eiDeuGi8dee6aeB9phua9uel} and  \eqref{eq_oov6aer7rouFueghiefoovua} again that
\begin{equation}
\label{eq_quaiFee2sheixau3oocaekoo}
 \# \mathscr{T}^{\mathrm{sing}, k_0}_k
 \le 2^{-\brk{k + 1 - k_0}\brk{m - \ell}} \# \mathscr{T}^{\mathrm{sing}, k_0}_{k_0 - 1} + \smash{\sum_{i = k_0}^{k}} 2^{-\brk{k - i}\brk{m - \ell}}\# \mathscr{T}^{\mathrm{bad}, k_0}_i.
\end{equation}

We define the collection
\begin{equation}
  \mathscr{Q}^{k_0} \defeq \set{Q + \xi_k \st k \in \Zset \text{ and }Q \in \mathscr{Q}^\infty_k},
\end{equation}
with the convention that \(\xi_k = 0\) when \(k < k_0\).
Since, in view of \eqref{eq_eo6oneiQuee1aic4oon6ek9w}, \(\xi_k - \xi_{k - 1} = \zeta_k \in 2^{k - 1} \Zset^{m - 1}\), the collection \(\mathscr{Q}^{k_0}\) is a dyadic decomposition of \(\Rset^m_+\) in the sense of \cref{definition_dyadic_decomposition}.
The collection \(\smash{\mathscr{Q}^{\mathrm{bad}, k_0}}\) can then be defined by \eqref{eq_che7ahweeghieCi4aqu7mi5b}.
We note that if \(\sigma \in \smash{\mathscr{T}^{\mathrm{bad}, k_0}_k}\),
then any of the \(2^{m - \ell}\) cubes in \(\smash{\mathscr{Q}^{k_0}_k}\) containing \(\sigma\) should be in \(\smash{\mathscr{Q}^{\mathrm{bad}, k_0}}\); on the other hand any of the latter cubes has \(2^{m - \ell} \smash{\binom{m - 1}{\ell - 1}}\) vertical faces of dimension \(\ell\), and hence we have
\begin{equation}
\label{eq_eighiequooc3xeir3AeNgie1}
  \# \mathscr{T}^{\mathrm{bad}, k_0}_k
  \le \tbinom{m - 1}{\ell - 1} \# \mathscr{Q}^{\mathrm{bad}, k_0}_k;
\end{equation}
similarly, we have in view of \eqref{eq_og1quaePhaelahbahchail5y}
\begin{equation}
\label{eq_Mio5oovengahmoh4xiPusuzo}
  \# \mathscr{T}^{\mathrm{sing}, k_0}_{k_0 - 1}
  \le \tbinom{m - 1}{\ell - 1} \# \mathscr{Q}^{\mathrm{sing},\infty}_{k_0 - 1}.
\end{equation}

We also define for every \(k \in \Zset\) such that \(k \ge k_0\), the collection
\begin{equation}
 \mathscr{Q}^{\mathrm{prop}, k_0}_k
 \defeq \set{Q \in \mathscr{Q}^{k_0}_k \st Q \supseteq \partial_{\top} \sigma \text{ for some \(\sigma \in \mathscr{T}^{\mathrm{sing}, k_0}_{k - 1}\)
 }},
\end{equation}
with the upper boundary \(\partial_{\top}\sigma\) defined in \eqref{eq_hoofi5Emae8xaepa8queef7r}.
Every \(\sigma \in \smash{\brk{\mathscr{Q}^{k_0}}_{k - 1}^{\ell, \perp}}\) has its upper boundary \(\partial_{\top} \sigma\) being  contained in \(2^{m - \ell}\) cubes of \(\smash{\mathscr{Q}^{k_0}_k}\), so that
\begin{equation}
\label{eq_eSooPhae6ohcoseiceeFoiba}
 \# \mathscr{Q}^{\mathrm{prop}, k_0}_k
 \le 2^{m - \ell} \# \mathscr{T}^{\mathrm{sing}, k_0}_{k - 1}.
\end{equation}
Inserting \eqref{eq_eighiequooc3xeir3AeNgie1}, \eqref{eq_Mio5oovengahmoh4xiPusuzo} and \eqref{eq_eSooPhae6ohcoseiceeFoiba} in \eqref{eq_quaiFee2sheixau3oocaekoo}, we have for every \(k \in \Zset\) such that \(k \ge k_0\)
\begin{equation}
\label{eq_aiw7thee9zuelie3Ong3ohSe}
\begin{split}
 \# \mathscr{Q}^{\mathrm{prop}, k_0}_k
 & \le 2^{-\brk{k - k_0}\brk{m - \ell}} \# \mathscr{T}^{\mathrm{sing}, k_0}_{k_0 - 1} + \smashoperator{\sum_{i = k_0}^{k - 1}} 2^{-\brk{k - i - 1}\brk{m - \ell}}\# \mathscr{T}^{\mathrm{bad}, k_0}_i\\[-.8em]
 &\le
 \tbinom{m - 1}{\ell - 1}
 \brk[\Big]{2^{-\brk{k - k_0}\brk{m - \ell}} \# \smash{\mathscr{Q}^{\mathrm{sing}, \infty}_{k_0 - 1}} + \smashoperator{\sum_{i = k_0}^{k - 1}} 2^{-\brk{k - i - 1}\brk{m - \ell}}\# \mathscr{Q}^{\mathrm{bad}, k_0}_i}.
\end{split}
\end{equation}
Defining the collection of cubes
\begin{equation}
\label{eq_phuupahg4Aj0uev6uhoojaW9}
 \mathscr{Q}^{\mathrm{sing},k_0}
 \defeq
 \smashoperator{\bigcup_{k = -\infty}^{k_0 - 1}} \smash{\mathscr{Q}^{\mathrm{sing}, \infty}_k}
 \cup
 \smashoperator{
 \bigcup_{k = k_0}^{\infty}}
 \mathscr{Q}^{\mathrm{bad}, k_0}_k
 \cup
 \smashoperator{\bigcup_{k = k_0}^{\infty}}
 \mathscr{Q}^{\mathrm{prop}, k_0}_k,
\end{equation}
we have by \eqref{eq_aiw7thee9zuelie3Ong3ohSe} and \eqref{eq_phuupahg4Aj0uev6uhoojaW9} for every \(k \in \Zset\) such that \(k \ge k_0\)
\begin{equation}
 \# \mathscr{Q}^{\mathrm{sing}, k_0}_k
 \le
 \tbinom{m - 1}{\ell - 1}
 \brk[\Big]{2^{-\brk{k - k_0}\brk{m - \ell}} \# \smash{\mathscr{Q}^{\mathrm{sing}, \infty}_{k_0 - 1}} + \smashoperator{\sum_{i = k_0}^{k}}  2^{-\brk{k - i - 1}\brk{m - \ell}}\# \mathscr{Q}^{\mathrm{bad}, k_0}_i},
\end{equation}
and we have thus proved the assertions \ref{it_Dih1Poh0nohhahTh2chi8sah}, \ref{eq_EiKah0oreekaBohwohbighie} and \ref{it_taa6web7ien1quae8Xuey1xa}.

In order to define the mapping \(\Psi^{\mathrm{sing}, k_0}\) satisfying assertion \ref{it_omei1eoC1Iebibohhoh8so0z},
we first define
\(\Psi^{\mathrm{sing}, k_0} \defeq \Psi^{\mathrm{sing}, \infty}\)
on the set \(\smash{\bigcup_{k = -\infty}^{k_0 - 1} \smash{\brk{\mathscr{Q}^{k_0}}^{\ell, \perp}_k}
\cup \bigcup_{k = -\infty}^{k_0 - 2} \brk{\mathscr{Q}^{k_0}}^{\ell, \top}_k}\).
Next we extend \(\Psi^{\mathrm{sing}, k_0}\) inductively to the set \(\bigcup \smash{\brk{\mathscr{Q}^{k_0}}^{\ell, \perp}_k}\) for every \(k \in \Zset\) such that \(k \ge k_0\).

If \(\Psi^{\mathrm{sing}, k_0}\) has already been defined for some \(k \in \Zset\) on the set
\(\smash{\bigcup \brk{\mathscr{Q}^{k_0}}^{\ell, \perp}_{k - 1}}\) in such a way that \(\smash{\Psi^{\mathrm{sing}, k_0}} = \id\)
 on \(\smash{\bigcup \brk{\brk{\mathscr{Q}^{k_0}}^{\ell, \perp}_{k - 1} \setminus \mathscr{T}^{\mathrm{sing}, k_0}_{k - 1}}}\),
we are going to extend it to the set \(\smash{\bigcup \brk{\mathscr{Q}^{k_0}}^{j, \perp}_{k}}\)
for every \(j \in \set{0, \dotsc, \ell}\).
The case \(j = 0\) is trivial since there is no vertical face of dimension \(0\).

We suppose now that \(j \in \set{1, \dotsc, \ell}\) and that the mapping \(\Psi^{\mathrm{sing}, k_0}\) has been defined on the set \(\smash{\bigcup \brk{\mathscr{Q}^{k_0}}^{\ell, \perp}_{k - 1}
\cup \bigcup \brk{\mathscr{Q}^{k_0}}^{j - 1, \perp}_k}\)
in such a way that \(\Psi^{\mathrm{sing}, k_0}= \id\) on the set \(\brk{\bigcup \smash{\brk{\mathscr{Q}^{k_0}}^{\ell, \perp}_{k - 1}} \setminus \smash{\mathscr{T}^{\mathrm{sing}, k_0}_{k - 1}}} \cup
\brk{\bigcup \smash{\brk{\mathscr{Q}^{k_0}}^{j - 1, \perp}_k} \cap \brk{\bigcup \smash{\brk{\mathscr{Q}^{k_0}}^{\ell, \perp}_{k}} \setminus \smash{\mathscr{T}^{\mathrm{sing}, k_0}_{k}}}}\).
Given any face \(\sigma \in \smash{\brk{\mathscr{Q}^{k_0}}^{j, \perp}_{k}}\), we observe that \(\partial_{\sqcup} \sigma \subseteq \bigcup \smash{\brk{\mathscr{Q}^{k_0}}^{\ell, \perp}_{k - 1}} \cup \bigcup \smash{\brk{\mathscr{Q}^{k_0}}^{j - 1, \perp}_{k}}\), with the enclosing boundary \(\partial_{\sqcup} \sigma\) defined in \eqref{eq_ap8pee3oquaiH5iig9phahsh}, and thus \(\Psi^{\mathrm{sing}, k_0}\) is already defined on \(\partial_{\sqcup} \sigma\).
We first consider  the case where \(\sigma \subseteq \bigcup \smash{\brk{\brk{\mathscr{Q}^{k_0}}^{\ell, \perp}_k \setminus \mathscr{T}^{\mathrm{sing}, k_0}_k}}\).
By definition of \(\smash{\mathscr{T}^{\mathrm{sing},k_0}_k}\) in \eqref{eq_Zo4ux0ooj6uXeev9ahphejov} we have
\(\sigma \cap \smash{A^{\mathrm{bad}}} = \emptyset\) by \eqref{eq_Da8pei5Va7aihaig1aethoch} and
\(\partial_{\sqcup} \sigma \subseteq \brk{\bigcup \smash{\brk{\mathscr{Q}^{k_0}}^{\ell, \perp}_{k - 1}} \setminus \smash{\mathscr{T}^{\mathrm{sing}, k_0}_{k - 1}}} \cup
\brk{\bigcup \smash{\brk{\mathscr{Q}^{k_0}}^{j - 1, \perp}_k} \cap \bigcup \brk{\smash{\brk{\mathscr{Q}^{k_0}}^{\ell, \perp}_{k}} \setminus \smash{\mathscr{T}^{\mathrm{sing}, k_0}_{k}}}}\) by \eqref{eq_kuz0tooRuph1eevahs4ieb7r} and by \cref{lemma_cubes_top_boundary}, so that \(\smash{\Psi^{\mathrm{sing}, k_0}}\) is the identity on \(\partial_{\sqcup} \sigma\) and we extend then \(\smash{\Psi^{\mathrm{sing}, k_0}}\) to \(\sigma\) by taking the identity on \(\sigma\).
Otherwise, we can take \(\smash{\Psi^{\mathrm{sing}, k_0}}\) on \(\sigma\) to be a descending extension to \(\sigma\) of \(\smash{\Psi^{\mathrm{sing}, k_0}}\) already defined on \(\partial_{\sqcup} \sigma\) thanks to \cref{lemma_descending}.

It remains now to extend \(\Psi^{\mathrm{sing}, k_0}\) to
\(\bigcup \smash{\brk{\mathscr{Q}^{k_0}}^{\ell, \top}_{k}}\) inductively for \(k \ge k_0 - 1\).
We assume that it has already been done for \(\bigcup \smash{\brk{\mathscr{Q}^{k_0}}^{\ell, \top}_{k - 1}}\);
as this is indeed the case for \(k = k_0 - 1\), we will thus have made a construction by induction.
Given any face \(\sigma \in \smash{\brk{\mathscr{Q}^{k_0}}^{\ell, \top}_{k}}\),
we either have \(\partial \sigma \subseteq \bigcup \brk{\smash{\brk{\mathscr{Q}^{k_0}}^{\ell, \perp}_k} \setminus \smash{\mathscr{T}^{\mathrm{sing}, k_0}_k}}\) and \(\sigma \cap \smash{A^{\mathrm{bad}}} = \emptyset\), in which case \(\smash{\Psi^{\mathrm{sing}, k_0}}\) is the identity on \(\partial \sigma\) and we can extend \(\smash{\Psi^{\mathrm{sing}, k_0}}\) by the identity on \(\sigma\);
otherwise we have \(\sigma = \partial_{\top} \tau\) for some \(\tau \in \smash{\brk{\mathscr{Q}^{k_0}}^{\ell + 1, \perp}_{k}}\), the map \(\Psi^{\mathrm{sing}, k_0}\) is then defined on \(\sigma = \partial_{\top} \tau\) by a descending retraction of \(\tau\) on \(\partial_{\sqcup} \tau\) thanks to \cref{lemma_descending}; in this case we have \(\sigma \not \subseteq \bigcup \brk{ \mathscr{Q}^{k_0} \setminus \mathscr{Q}^{\mathrm{sing}, k_0}}\). This concludes the construction of \(\Psi^{\mathrm{sing},k_0}\) satisfying the assertion \ref{it_omei1eoC1Iebibohhoh8so0z}.
\end{proof}

\begin{remark}
\label{remark_probabilistic_propagation_decay}
\Cref{proposition_general_modification} also has a probabilistic formulation. Thanks to \cref{remark_solenoid}, we can represent dyadic decompositions of \(\Rset^m_+\) by  \(\smash{\brk{\boldsymbol{\Sigma}_2}^{m - 1}}\),
where \(\boldsymbol{\Sigma}_2\) is the dyadic solenoid;
the latter can be endowed with its Haar probability measure, which is the inverse limit of the usual Haar probability measure on the unit circle.
If the dyadic decomposition \(\mathscr{Q}^{k_0}\) is taken randomly under the condition \(\mathscr{Q}^\infty_k = \mathscr{Q}^{k_0}_k\) for \(k < k_0\), then
the expectations of \(\smash{\# \mathscr{Q}^{\mathrm{sing},k_0}_k}\) satisfy \ref{it_taa6web7ien1quae8Xuey1xa};
the key ingredient is the averaging estimate \eqref{eq_eSooPhae6ohcoseiceeFoiba}, which has a manifest probabilistic flavour.
\end{remark}

\subsection{Extension of traces under critical integrability}
We now have all the tools to prove an extension result on the set \(\bigcup \mathscr{Q}^{\vert \Omega}\) from which the extension and its estimate in \cref{theorem_estimate_critical_halfspace} will follow.

\begin{theorem}
\label{theorem_extension_tent_critical}
Assume that the homotopy groups \(\pi_{1} \brk{\manifold{N}}, \dotsc, \smash{\pi_{m - 2}} \brk{\manifold{N}}\) are finite and that the group \(\pi_{m - 1} \brk{\manifold{N}}\) is trivial.
There exists a constant \(C \in \intvo{0}{\infty}\), such that
for every open set \(\Omega \subseteq \smash{\Rset^{m - 1}}\) and every map \(u \in \smash{\smash{\dot{W}}^{1 - 1/m, m}} \brk{\Omega, \manifold{N}}\), there exists
a dyadic decomposition \(\mathscr{Q}\) of \(\Rset^m_+\) and a map \(U \colon \smash{\bigcup \mathscr{Q}^{\vert \Omega}} \to \manifold{N}\)  such that
\begin{enumerate}[label=(\roman*)]
 \item
 \label{it_iutiegu2pah8Eivegheedei9}
 \(U \in \smash{\dot{W}}^{1, m} \brk{\bigcup \smash{\mathscr{Q}^{\vert \Omega}}, \manifold{N}}\)
 and
 \begin{equation}
 \label{eq_Ru6saphai7shee9aJahngai0}
   \int_{ \bigcup \mathscr{Q}^{\vert \Omega}} \abs{\Deriv U}^m
   \le C
   \smashoperator{\iint_{\Omega \times \Omega}} \frac{d \brk{u \brk{y}, u \brk{z}}^m}{\abs{y - z}^{2m - 2}} \dif y \dif z,
 \end{equation}
 \item
 \label{it_zieZ5zoh7iey8door3jei8Oh}
 \(\tr_{\Omega \times \set{0}} U = u\),
 \item
 \label{it_quaePaothierohx1quoehuew}
 for almost every \(x \in \bigcup \mathscr{Q}^{\vert \Omega}\),
 \begin{equation*}
    \abs{\Deriv U \brk{x}} \le
    \frac{C}{\dist \brk{x, \partial \Rset^{m}_+}},
 \end{equation*}
\end{enumerate}
\end{theorem}

\begin{proof}[Proof of \cref{theorem_extension_tent_critical}]
We begin as in the proof of \cref{theorem_extension_tent_supercritical},
defining \(V \colon \widehat{\Omega} \to \Rset^\nu\) as the linear extension of \(u\) by \eqref{eq_shee0Ceingooghe2weefei4A}, taking an arbitrary dyadic decomposition \(\mathscr{Q}^\infty\) of \(\Rset^m_+\)
and defining the associated collection of bad cubes given by \eqref{eq_definition_bad_cube}.
Thanks to \eqref{eq_sixaika7eitie6deLaem4ae0} and \eqref{eq_eeGhi9Thosh8ephai7avai8l}, there exists \(k_0\) such that for every \(k \le k_0\), we have \(\smash{\mathscr{Q}^{\mathrm{bad}, \infty}_k} = \emptyset\).

We let \(\mathscr{Q}^{k_0}\) be the dyadic decomposition of \(\Rset^m_+\), \(\mathscr{Q}^{\mathrm{bad}, k_0}\), \(\mathscr{Q}^{\mathrm{sing}, k_0}\) be the collections of cubes and \(\smash{\Psi^{\mathrm{sing}, k_0}} \colon \smash{ \bigcup \brk{\mathscr{Q}^{k_0}}^{m - 1} \to \Rset^{m}_+ \setminus \smash{\smash{\widehat{\Omega}}^{\mathrm{bad}}}}\) be the mapping given by \cref{proposition_general_modification} with \(A^{\mathrm{bad}} = \smash{\widehat{\Omega}}^{\mathrm{bad}}\).
We define the mapping 
\[
W \colon \smash{\bigcup \brk{\mathscr{Q}^{k_0 \vert \Omega} \setminus \mathscr{Q}^{\mathrm{sing}, k_0}}}\cup \smash{\bigcup \brk{\mathscr{Q}^{k_0 \vert \Omega} \cap \mathscr{Q}^{\mathrm{sing}, k_0}}^{m - 1}} \to \manifold{N}
\]
as \(\Pi_{\manifold{N}} \compose V \) on  \(\smash{\bigcup \brk{\mathscr{Q}^{k_0 \vert \Omega} \setminus \mathscr{Q}^{\mathrm{sing}, k_0}}}\)
and \(\smash{\Pi_{\manifold{N}}} \compose V \compose\smash{\Psi^{\mathrm{sing}, k_0}}\) on \(\smash{\bigcup \brk{\mathscr{Q}^{k_0 \vert \Omega} \cap \mathscr{Q}^{\mathrm{sing}, k_0}}^{m - 1}}\).
Since the map \(\Psi^{\mathrm{sing}, k_0}\) is descending, \(W\) is well defined on the latter set.
In order to estimate the size of the singular cubes \(\smash{\mathscr{Q}^{\mathrm{sing}, k_0}}\), we have thanks to \ref{it_Dih1Poh0nohhahTh2chi8sah} and \ref{it_taa6web7ien1quae8Xuey1xa} in \cref{proposition_general_modification}
\begin{equation}
\label{eq_ese7Phoihosh4vi5Wah2uiw8}
\begin{split}
  \sum_{k \in \Zset}
    \# \mathscr{Q}^{\mathrm{sing}, k_0}_{k}
  &=
  \smashoperator{\sum_{k = -\infty}^{k_0 - 1}}
  \# \smash{\mathscr{Q}^{\mathrm{sing}, \infty}_k}
  + \smashoperator{\sum_{k = k_0}^{\infty}}
  \# \mathscr{Q}^{\mathrm{sing}, k_0}_{k}\\
   &\le \Cl{cst_ohvie0biciecahdepaiget8X} \brk[\Big]{\smashoperator[r]{\sum_{k = k_0}^{\infty}}  \smashoperator[r]{\sum_{i = k_0}^{k}} 2^{-\brk{k - i}}\# \mathscr{Q}^{\mathrm{bad}, k_0}_i}\\
   &= \Cr{cst_ohvie0biciecahdepaiget8X}
   \sum_{i = k_0}^{\infty} 2^{i}
  \sum_{k = i}^\infty 2^{-k}  \# \mathscr{Q}^{\mathrm{bad}, k_0}_i
  = 2 \Cr{cst_ohvie0biciecahdepaiget8X} \sum_{i = k_0}^\infty \# \mathscr{Q}^{\mathrm{bad}, k_0}_i.
\end{split}
\end{equation}

Since \(\pi_{m - 1}\brk{\manifold{N}}\) is trivial, the map \(W\) can be extended to \(\smash{\bigcup \brk{\mathscr{Q}^{k_0 \vert \Omega}}}\).
We let \(\mathscr{Q} \defeq  \mathscr{Q}^{k_0}\) and \(U \colon \bigcup \mathscr{Q}^{\vert \Omega} \to \manifold{N}\) be the mapping defined by \cref{proposition_Lipschitz_extension_skeleton}
with \(\smash{\mathscr{Q}^{\mathrm{sing}}} \defeq \smash{\mathscr{Q}^{k_0 \vert \Omega}} \cap \smash{\mathscr{Q}^{\mathrm{sing},k_0}}\) and \(\smash{\mathscr{Q}^{\mathrm{reg}}} \defeq \smash{\mathscr{Q}^{k_0 \vert \Omega}} \setminus \smash{\mathscr{Q}^{\mathrm{sing},k_0}}\). By \eqref{eq_yie2te0ahrie6eeL2mah5Boh}, \eqref{eq_xuoY7IuceiwooH5Aicadai9a} and \eqref{eq_ese7Phoihosh4vi5Wah2uiw8}, the estimate \eqref{eq_Ru6saphai7shee9aJahngai0} holds, proving \ref{it_iutiegu2pah8Eivegheedei9}.

Since \(U = \Pi_{\manifold{N}}\compose V\) close to \(\Omega \times \set{0}\), we have
\(\tr_{\Omega \times \set{0}} U = \tr_{\Omega \times \set{0}} V = u\), proving \ref{it_zieZ5zoh7iey8door3jei8Oh}.
The estimate of \ref{it_quaePaothierohx1quoehuew} also follows from  \eqref{eq_Umech5Ueyeic6oa7Ohngeevi} in \cref{proposition_Lipschitz_extension_skeleton}.
\end{proof}

\begin{remark}
\label{remark_probablistic_critical}
\Cref{theorem_extension_tent_critical} has a probabilistic proof and formulation.
As explained in \cref{remark_solenoid}, dyadic decompositions of \(\Rset^m_+\) can be represented by  \(\smash{\brk{\boldsymbol{\Sigma}_2}^{m - 1}}\),
where \(\boldsymbol{\Sigma}_2\) is the dyadic solenoid that can be endowed with its Haar probability measure, which is the inverse limit of the usual Haar probability measure on the unit circle.
Even though the cardinals of the subcollections \(\smash{\# \mathscr{Q}^{\mathrm{sing},k_0}_k}\) are not independent random variables, the linearity property of the expectation gives an estimate on the expectation of the size of bad cubes similar to \eqref{eq_ese7Phoihosh4vi5Wah2uiw8} so that
the conclusion \ref{it_aighie9iush4hei9opoo3Quo} holds for the expectation of the left-hand side of \eqref{eq_Ru6saphai7shee9aJahngai0} over random dyadic decompositions of \(\Rset^m_+\).
\end{remark}

\begin{remark}
The estimate \eqref{eq_Ru6saphai7shee9aJahngai0} can also be localised.
An inspection of the proof, shows that there exists a constant \(C\) such that for every \(\Bar{k}\in \Zset\), one has
 \begin{equation}
   \smashoperator[r]{\int_{ \brk{\Rset^{m - 1} \times \intvo{0}{2^{\Bar{k}}} } \,\cap\, \bigcup \mathscr{Q}^{\vert \Omega} }} \quad \abs{\Deriv U}^m
   \le
   C
   \smashoperator{\iint_{\substack{\brk{y, z} \in \Omega \times \Omega\\
   d\brk{y, z} \le 2^{\Bar{k}}}}} \frac{d \brk{u \brk{y}, u \brk{z}}^m}{\abs{y - z}^{2m - 2}} \dif y \dif z.
 \end{equation}
\end{remark}

Thanks to \cref{theorem_extension_tent_critical}, we are in position to construct an extension satisfying a linear estimate in the critical case.

\begin{proof}[Proof of \cref{theorem_estimate_critical_halfspace}]
\label{proof_theorem_estimate_critical_halfspace}
This follows from \cref{theorem_extension_tent_critical} with \(\Omega = \Rset^{m - 1}\).
\end{proof}

\section{Subcritical extension}
\label{section_subcritical}
When \(1 < p < m\), the averaging argument of \cref{proposition_modifiedcubes_supercritical} has provided a continuous extension on the \(\ceil{p - 1}\)-dimensional skeleton \(\smash{\brk{\mathscr{Q}^{\mathrm{sing}}}^{\ceil{p - 1}}}\) of a collection of singular cubes \(\smash{\mathscr{Q}^{\mathrm{sing}}} \supseteq \smash{\mathscr{Q}^{\mathrm{bad}}}\) so that the construction of the Sobolev extension can be realised thanks to \cref{proposition_Lipschitz_extension_skeleton}, providing if either \(p \not \in \Nset\) or \(\pi_{p - 1}\brk{\manifold{N}}\simeq \set{0}\), a continuous extension on the \(\floor{p}\)-dimensional skeleton \(\smash{\brk{\mathscr{Q}^{\mathrm{sing}}}^{\floor{p}}}\).

When \(p < m\), in order to use the estimate \eqref{eq_mai6Xieph9uachahChengoon} in   \cref{proposition_general_modification}, we will not anymore be able to use the continuity or the vanishing mean oscillation property of \eqref{eq_sixaika7eitie6deLaem4ae0} to guarantee that \( \smash{\mathscr{Q}^{\mathrm{sing}, \infty}_{k_0 - 1}} = \emptyset \).
The construction of an extension will rely on a more delicate spawning construction.

\subsection{Spawning cubes over the bad cubes near the boundary}
When \(p \le m\), the boundary datum \(u \in  \smash{\smash{\dot{W}}^{1 - 1/p, p} \brk{\Omega, \manifold{N}}}\) need not be continuous and so there is no reason for the inclusion \eqref{eq_chiemooReix5bahChee8Vohn} to hold; indeed it may fail if \(p < m\).
A natural way to bypass this issue would be to first construct an extension for a continuous boundary datum \(u\) and then approximate any boundary datum by continuous ones, and pass to the limit in Sobolev spaces in the extension thanks to uniform Sobolev estimates on the extension.

However uniformly continuous mappings of the Sobolev space \(\smash{\dot{W}}^{1 - 1/p, p} \brk{\Omega, \manifold{N}}\) are \emph{not} in general weakly sequentially dense in the space \(\smash{\smash{\dot{W}}^{1 - 1/p, p} \brk{\Omega, \manifold{N}}}\) because of global topological obstructions.
Nonetheless, if one defines for the open set \(\Omega \subseteq \Rset^{m - 1}\) and for every \(\ell \in \set{0, \dotsc, m - 2}\), similarly to \eqref{eq_quaeP5sosiC8Nahbaipaek7W} and \eqref{eq_TieQuee3iyoh9aequepo9Ahx}, the set
\begin{equation}
\label{eq_Quaibaibai5iajiel3yah9ie}
\begin{split}
 R^1_\ell \brk{\Omega, \manifold{N}}
 \defeq
 \Bigl\{ u \colon \Omega \to \manifold{N}
 \st[\Big] & u \in C^1 \brk{\Omega \setminus \Sigma, \manifold{N}}
 \text{ and } \smash{\sup_{x \in \Omega \setminus \Sigma}} \dist \brk{x, \Sigma} \abs{\Deriv u \brk{x}} < \infty\\
 &\; \text{where \(\emptyset \subsetneq \Sigma \subseteq \Rset^{m - 1}\) is compact and contained in a finite }\\[-.5em]
 &\;\text{union of \(\ell\)-dimensional embedded smooth submanifolds}
 \Bigr\},
 \end{split}
 \raisetag{4.5em}
\end{equation}
then the set \(\smash{R^1_{m - \floor{p} - 1}\brk{\Omega, \manifold{N}}}\) is still strongly dense in the  space \(\smash{\smash{\dot{W}}^{1 - 1/p, p} \brk{\Omega, \manifold{N}}}\) when \(\Omega\) is bounded with a smooth boundary \citelist{\cite{Mucci_2009}*{Th.\thinspace{}1}\cite{Mucci_2024}*{Th.\thinspace{}1.1}\cite{Brezis_Mironescu_2015}*{Th.\thinspace{}3}} or is the Euclidean space \(\Rset^{m - 1}\) (see \cref{theorem_R_dense_Wsp} below).

If moreover the homotopy group \(\smash{\pi_{\floor{p - 1}} \brk{\manifold{N}}}\) is trivial, as we assume in Theorems \ref{theorem_extension_halfspace}, \ref{theorem_extension_collar} and \ref{theorem_extension_global}, then it follows from the same methods (see \cite{Hang_Lin_2003_III}*{\S 4}), that the smaller set \(\smash{R^1_{m - \floor{p} - 2}\brk{\Omega, \manifold{N}}}\) of somehow less discontinuous mappings is still strongly dense in the fractional Sobolev space \(\smash{\dot{W}}^{1 - 1/p, p} \brk{\Omega, \manifold{N}}\).
We will thus analyse the structure of the collection of bad cubes originating from some mapping \(u \in \smash{R^1_{m - \floor{p} - 2}\brk{\Omega, \manifold{N}}}\).

We begin by observing that the singular set of a mapping \(u \colon \Omega \to \manifold{N}\) controls the bad set \(\smash{\widehat{\Omega}}^{\mathrm{bad}}\) defined in \eqref{eq_vuaPh9vaengoo9aiHeeshaer} of its linear extension \(V\colon \smash{\widehat{\Omega}} \to \Rset^\nu\) given by \eqref{eq_shee0Ceingooghe2weefei4A}.

\begin{proposition}
\label{proposition_neat_singular_set}
For every Borel-measurable mapping \(u \colon \Omega \to \manifold{N}\) and every non-empty closed set \(\Sigma \subseteq \Rset^{m - 1}\), if \(u \in C^1 \brk{\Omega \setminus \Sigma, \manifold{N}}\) and if
\begin{equation}
\label{eq_ahshohvoeBeemu1Chuphooxa}
  \sup_{y \in \Omega \setminus \Sigma}  \dist \brk{y, \Sigma} \abs{\Deriv u \brk{y}} < \infty,
\end{equation}
then there exists \(\kappa \in \intvo{0}{\infty}\) such that
\begin{equation}
\label{eq_hediexie5aiz6ree2ieg2aeL}
 \smash{\widecheck{\Sigma}}^{\kappa} \cap
 \widehat{\Omega}^{\mathrm{bad}}  = \emptyset,
\end{equation}
where
\begin{equation}
\label{eq_eeh3aeguoDao3fahj6dah2Ae}
 \smash{\widecheck{\Sigma}}^{\kappa}
 \defeq 
 \set[\big]{x = \brk{x', x_m} \in \Rset^m_+ \st \Bset^{m - 1}_{\kappa x_m}\brk{x'} \cap \Sigma = \emptyset}.
\end{equation}
\end{proposition}
\begin{proof}
By \eqref{eq_haw0fai3ohlie0Eitha9Aes2}, the definition of the mean oscillation in \eqref{eq_moo7thiequ2zug8ahSoh3oth}, the Poincaré inequality and by the assumption \eqref{eq_ahshohvoeBeemu1Chuphooxa}, we have for every \(x = \brk{x', x_m} \in \smash{\widehat{\Omega}}^{\mathrm{bad}}\)
\begin{equation*}
\begin{split}
 \delta_* \le \meanosc_{\delta_*} \brk{u} \brk{x}
 &\le \frac{\C}{x_m^{2 \brk{m - 1}}} \smashoperator{\iint_{\brk{\Bset^{m-1}_{x_m}\brk{x'}}^2}} \abs{u \brk{y} - u \brk{z}} \dif y \dif z\\
 &\le \frac{\C}{x_m^{m - 2}} \int_{\Bset^{m-1}_{x_m}\brk{x'}} \abs{\Deriv u}
 \le \frac{\Cl{cst_Noh7aquaeghooh8niseeph3m} x_m}{\brk{\dist \brk{x', \Sigma}- x_m}_+},
 \end{split}
\end{equation*}
and therefore
\[
  \dist \brk{x', \Sigma} \ge \brk{1 + \Cr{cst_Noh7aquaeghooh8niseeph3m}/\delta_*} x_m.
  \]
Setting \(\kappa \defeq \brk{1 + \Cr{cst_Noh7aquaeghooh8niseeph3m}/\delta_*}\), we have the conclusion \eqref{eq_hediexie5aiz6ree2ieg2aeL}.
\end{proof}

The set \(\Rset^{m}_+ \setminus \brk{\Sigma \times \intvo{0}{\infty}}\) can be retracted to the set \(\smash{\widecheck{\Sigma}}^{\kappa}\) given by \eqref{eq_eeh3aeguoDao3fahj6dah2Ae}.

\begin{lemma}
\label{lemma_retraction_invtent}
Let \(\Sigma \subseteq \Rset^{m - 1}\) be closed and \(\kappa \in \intvo{0}{\infty}\).
The map \(\Psi_{\Sigma, \kappa} \colon \brk{\Rset^{m - 1} \setminus \Sigma} \times \intvo{0}{\infty} \to \smash{\widecheck{\Sigma}}^{\kappa}\) defined by
\begin{equation}
\label{eq_eetoongekuo5Ooth2Yapiatu}
 \Psi_{\Sigma, \kappa}\brk{x}
 \defeq
 \begin{cases}
   x &\text{if \(x \in \smash{\widecheck{\Sigma}}^{\kappa}\)},\\
   \brk{x', \dist \brk{x', \Sigma}/\kappa} & \text{otherwise},
 \end{cases}
\end{equation}
is a well-defined and continuous retraction and is a descending map with respect to every dyadic decomposition.
\end{lemma}
\begin{proof}
This follows from the continuity of the distance function, from the definition of the set \(\smash{\widecheck{\Sigma}}^{\kappa}\) in  \eqref{eq_eeh3aeguoDao3fahj6dah2Ae} and from the definition of descending map (\cref{definition_descending}).
\end{proof}

Now we transfer the construction of \cref{proposition_neat_singular_set} and \cref{lemma_retraction_invtent} to a dyadic decomposition of \(\Rset^m_+\).

\begin{proposition}[Spawning cubes]
\label{proposition_cubes_spawning}
Let \(\ell \in \set{0, \dotsc, m - 2}\).
If \(\Sigma \subseteq \Rset^{m - 1}\) is closed and contained in a finite union of compact \(\brk{m - \ell - 2}\)-dimensional submanifolds of \(\Rset^{m - 1}\) and if \(\kappa \in \intvo{0}{\infty}\),
then there exist a dyadic decomposition \(\smash{\mathscr{Q}^{\infty}}\) of \(\Rset^m_+\), a collection of cubes \(\smash{\mathscr{Q}^{\mathrm{sing}, \infty}} \subseteq \smash{
\mathscr{Q}^{\infty}}\) and a mapping \(\Psi^{\mathrm{sing}, \infty} \colon \smash{\bigcup} \smash{\brk{\smash{\mathscr{Q}^{\infty}}}^{\ell}} \to \smash{\smash{\widecheck{\Sigma}}^{\kappa}}\) such that
\begin{enumerate}[label=(\roman*)]
 \item 
 \label{it_feta4AQu4ahja7ci7Esie5ah}
 \(\smash{\widecheck{\Sigma}}^{\kappa} \cup \smash{\bigcup \smash{\mathscr{Q}^{\mathrm{sing}, \infty}}}= \Rset^m_+\),
 \item 
 \label{it_OaKei0aizi8Cheedo0ohqu7v}
 \(\smash{\Psi^{\mathrm{sing}, \infty}} = \id\) on \(\bigcup \brk{\mathscr{Q}^{\infty} \setminus \smash{\mathscr{Q}^{\mathrm{sing}, \infty}}}^{\ell}\),
 \item 
 \label{it_ohd6RaeyuiWooGoyo6easiey}
 there exists a constant \(M \in \intvo{0}{\infty}\), depending on \(\Sigma\) and \(\kappa\),
 such that for every \(k \in \Zset\), 
 \begin{equation*}
   \# \smash{\mathscr{Q}^{\mathrm{sing}, \infty}_k}
   \le M \, \brk{2^{-k\brk{m - \ell - 2}} + 1}.
 \end{equation*}
\end{enumerate}
\end{proposition}

\begin{proof}
Let \(\mathscr{Q}\) be an arbitrary dyadic decomposition of \(\Rset^m_+\).
Since the set \(\Sigma\) is contained in a finite union of compact \(\brk{m - \ell - 2}\)-dimensional submanifolds of \(\Rset^{m - 1}\), for almost every \(\xi \in \Rset^{m - 1} \times \set{0} \simeq \Rset^{m - 1}\), we have \[
  \brk{\Sigma \times \intvo{0}{\infty}} \cap \brk{\xi + \textstyle \bigcup \brk{\mathscr{Q}}^\ell} = \emptyset.                                                                        
\]
Fixing such a \(\xi \in \Rset^{m - 1} \times \set{0}\), we define successively the collections
\begin{align*}
\mathscr{Q}^\infty &\defeq \set{Q + \xi \st Q \in \mathscr{Q}},&
 \smash{\mathscr{Q}^{\mathrm{sing}, \infty}}
 &\defeq \set{Q \in \mathscr{Q}^{\infty} \st Q \not \subseteq \smash{\widecheck{\Sigma}}^{\kappa}}
\end{align*}
and the mapping
\begin{equation*}
\Psi^{\mathrm{sing}, \infty} \defeq \Psi_{\Sigma, \kappa} \restr{\bigcup \brk{\mathscr{Q }^{\infty}}^\ell },
\end{equation*}
with the map \(\Psi_{\Sigma, \kappa}\) defined in \eqref{eq_eetoongekuo5Ooth2Yapiatu}.
The assertion \ref{it_feta4AQu4ahja7ci7Esie5ah} is immediate and, thanks to \cref{lemma_retraction_invtent}, the assertion \ref{it_OaKei0aizi8Cheedo0ohqu7v} holds.

We also note that, by the structure of the set \( \Sigma \),
\[
\begin{split}
2^{k \brk{m- 1}} \smash{\mathscr{Q}^{\mathrm{sing}, \infty}_k}
 &= \mathcal{H}^{m - 1}\brk[\big]{ \bigcup \smash{\mathscr{Q}^{\mathrm{sing}, \infty}_k} \cap \Rset^{m - 1}\times 2^{k}}\\
 &\le \mathcal{H}^{m - 1}\brk[\big]{ \smash{\widecheck{\Sigma}}^{\kappa + \sqrt{m - 1}} \cap \Rset^{m - 1}\times 2^{k}}\\
 &\le \C \brk{2^{k \brk{\ell + 1}} + 2^{k \brk{m - 1}}},
\end{split}
\]
so that \ref{it_ohd6RaeyuiWooGoyo6easiey} holds.
\end{proof}

\begin{remark}
\label{remark_probabilistic_spawing}
The statement and the proof of \cref{proposition_cubes_spawning} can be interpreted probabilistically as in \cref{remark_probablistic_critical} above.
Indeed, as explained in \cref{remark_solenoid}, dyadic decompositions of \(\Rset^m_+\) can be represented by  \(\smash{\brk{\boldsymbol{\Sigma}_2}^{m - 1}}\), 
where \(\boldsymbol{\Sigma}_2\) is the dyadic solenoid that can be endowed with its Haar probability measure, which is the inverse limit of the usual Haar probability measure on the unit circle.
The proof of \cref{proposition_cubes_spawning} shows that if the dyadic decomposition \(\smash{\mathscr{Q}^\infty}\) is taken randomly, then the conclusion of \cref{proposition_cubes_spawning} holds almost surely, with a constant \(M\) in \ref{it_ohd6RaeyuiWooGoyo6easiey} which is deterministic, that is, the quantity \(M\) depends on the mapping \(u\) but not on the dyadic decomposition \(\mathscr{Q}^\infty\).
\end{remark}

\begin{remark}
The fact that the singular set \(\Sigma\) has dimension at most \(m - \ell - 2\) plays a crucial role in the transversality argument in the first part of the proof of \cref{proposition_cubes_spawning}, which eventually ensures that the mapping \(\Psi^{\mathrm{sing}, \infty}\) is well-defined and continuous.
Indeed, if such a mapping \(\smash{\Psi^{\mathrm{sing}, \infty}}\) existed and if \(\Sigma \subseteq \Omega\) with \(\Omega \subseteq \Rset^{m - 1}\) open, then for every \(f \in C \brk{\Sset^{\ell - 1}, \Omega \setminus \Sigma}\) which is homotopic to a constant in \(C \brk{\smash{\Sset^{\ell - 1}}, \Omega}\),  the map
 \(\smash{\Psi^{\mathrm{sing}, \infty}}\) could be used to construct a homotopy to a constant in \(C \brk{\smash{\Sset^{\ell - 1}}, \Omega \setminus \Sigma}\), which is not possible in general if \(\dim \Sigma = m - \ell - 1\).
In particular if \(\ell = \floor{p}\), then \cref{proposition_cubes_spawning} does not apply in the context of mappings in \(\smash{R^1_{m - \ell - 1} \brk{\Omega, \manifold{N}}}\) which, as a strongly dense subset of \(\smash{\smash{\dot{W}}^{1 - 1/p, p} \brk{\Omega, \manifold{N}}}\) when \(\Omega\) is bounded with a smooth boundary, would be the most natural set to work with.
When \(p \not \in \Nset\), \cref{proposition_cubes_spawning} turns out to be the only place where we use the necessary assumptions in Theorems \ref{theorem_extension_halfspace}, \ref{theorem_extension_collar} and \ref{theorem_extension_global} that the homotopy group \(\smash{\pi_{\floor{p - 1}}}\brk{\manifold{N}}\) is trivial
and that \(u\) has a suitable approximation in Theorems \ref{theorem_characterization_halfspace}, \ref{theorem_characterization_collar} and \ref{theorem_characterization_global};
in the case \(p \in \Nset\) it will also appear at a second place for analytical reasons.
\end{remark}

% \begin{remark}
% The mapping \(\Psi^{\mathrm{sing},\infty}\) provided by   \cref{proposition_cubes_spawning} is a \emph{deformation retraction} from the \(\ell\)-dimensional component of 
% 
% 
% , and therefore \(\pi_j \brk{\Gamma}\) is trivial for any connected component \(\Gamma\)
% of \(\bigcup \mathscr{Q}^{\mathrm{sing}}\).
% \end{remark}

Given \(u \in \smash{R^1_{m - \floor{p} - 2} \brk{\Omega, \manifold{N}}}\),
let \(\Sigma \subseteq \Rset^{m - 1}\) be the closed set given by the definition \eqref{eq_Quaibaibai5iajiel3yah9ie} of \(\smash{R^1_{m - \floor{p} - 2}} \brk{\Omega, \manifold{N}}\) and let \(\kappa \in \intvo{0}{\infty}\) be given by \cref{proposition_neat_singular_set}.
With the dyadic decomposition \(\smash{\mathscr{Q}^\infty}\) of \(\smash{\Rset^m_+}\), the
collection \(\smash{\mathscr{Q}^{\mathrm{sing}, \infty}} \subseteq \smash{\mathscr{Q}^\infty}\) and the mapping  \(\smash{\Psi^{\mathrm{sing}, \infty}} \colon \smash{\bigcup} \smash{\brk{\smash{\mathscr{Q}^{\infty}}}^{\floor{p}}} \to \smash{\smash{\widecheck{\Sigma}}^{\kappa}}\) given by \cref{proposition_cubes_spawning} with \(\ell = \floor{p}\), applying \cref{proposition_Lipschitz_extension_skeleton}
to the collections \(\smash{\mathscr{Q}^{\mathrm{reg}}} \defeq \smash{\mathscr{Q}^{\infty \mid \Omega} \setminus \mathscr{Q}^{\mathrm{sing}, \infty}}\) and \(\smash{\mathscr{Q}^{\mathrm{sing}} \defeq \smash{\mathscr{Q}^{\infty \mid \Omega}} \cap \smash{\mathscr{Q}^{\mathrm{sing}, \infty}}}\)
and to the mapping
\[W \defeq \Pi_{\manifold{N}} \compose V \compose \Psi^{\mathrm{sing},\infty}\restr{\bigcup \mathscr{Q}^{\mathrm{reg}}
\,\cup\, \bigcup \brk{\mathscr{Q}^{\mathrm{sing}}}^{\floor{p}}},
\]
which is well defined in view of  \eqref{eq_hediexie5aiz6ree2ieg2aeL},
we get a map \(U \colon \bigcup \mathscr{Q}^{\infty \vert \Omega} \to \manifold{N}\) such that, in view of \eqref{eq_xuoY7IuceiwooH5Aicadai9a}, for every \(k_0 \in \Zset\)
\begin{equation}
\label{eq_Neuchathie3teef1Oogh5ree}
\begin{split}
  \smashoperator[r]{\int_{\brk{\Rset^{m - 1} \times \intvo{0}{2^{k_0}}} \,\cap\, \bigcup \mathscr{Q}^{\vert \Omega}}}
  \abs{\Deriv U}^p &\le \Cl{cst_sheepeilie3oozupooghaiQu} \smashoperator{\iint_{\Omega \times \Omega}}
   \frac{d \brk{u \brk{y}, u \brk{z}}^p}{\abs{y - z}^{p + m - 2}} \dif y \dif z + \C \smashoperator{\sum_{k = -\infty}^{k_0 - 1}} 2^{k \brk{m - p}} \# \smash{\mathscr{Q}^{\mathrm{sing}, \infty}_k}\\
  &\le 
  \Cr{cst_sheepeilie3oozupooghaiQu} \brk[\bigg]{\smashoperator{\iint_{\Omega \times \Omega}}
   \frac{d \brk{u \brk{y}, u \brk{z}}^p}{\abs{y - z}^{p + m - 2}} \dif y \dif z +
   \C\,M \,
  2^{k_0 \brk{2 + \floor{p} - p}}},
\end{split}
\end{equation}
so that in particular the second term in right-hand side of \eqref{eq_Neuchathie3teef1Oogh5ree} goes to \(0\) as \(k_0 \to -\infty\), since \(\floor{p} < p + 1\).
It should be noted that the constant \(M\) in \eqref{eq_Neuchathie3teef1Oogh5ree} depends on the bounds on the derivative of the function \(u\) and on the structure of its singular set \(\Sigma\) only, so that the constant \(M\) will only play a perturbative role in the  estimate \eqref{eq_Neuchathie3teef1Oogh5ree}.

\subsection{Extension of traces under subcritical integrability}
We now have all the tools to prove an extension result on the set \(\bigcup \mathscr{Q}^{\vert \Omega}\) from which the extension and its estimate in \cref{theorem_extension_halfspace,theorem_characterization_halfspace,theorem_estimate_critical_halfspace} will follow.

\begin{theorem}
\label{theorem_extension_tent_subcritical}
Let \(p \in \intvo{1}{m}\).
Assume that the homotopy groups \(\pi_{1} \brk{\manifold{N}}, \dotsc, \smash{\pi_{\floor{p - 1}}} \brk{\manifold{N}}\) are finite and, when \(p \in \Nset\), that the group \(\pi_{p - 1} \brk{\manifold{N}}\) is trivial.
There exists a constant \(C \in \intvo{0}{\infty}\), such that
for every open set \(\Omega \subseteq \smash{\Rset^{m - 1}}\) and every map \(u \in \smash{\smash{\dot{W}}^{1 - 1/p, p}} \brk{\Omega, \manifold{N}}\) which belongs to the strong closure of \(\smash{R^1_{m - \floor{p} - 2}}\brk{\Omega, \manifold{N}} \cap \smash{\smash{\dot{W}}^{1 - 1/p, p}} \brk{\Omega, \manifold{N}} \), there exists
a dyadic decomposition \(\mathscr{Q}\) of \(\Rset^m_+\), a collection of cubes \(\mathscr{Q}^{\mathrm{sing}} \subseteq \smash{\mathscr{Q}^{\vert \Omega}}\) and a map \(U \colon \smash{\bigcup \mathscr{Q}^{\vert \Omega}} \to \manifold{N}\)  such that
\begin{enumerate}[label=(\roman*)]
 \item
 \label{it_aighie9iush4hei9opoo3Quo}
 \(U \in \smash{\dot{W}}^{1, p} \brk{\bigcup \smash{\mathscr{Q}^{\vert \Omega}}, \manifold{N}}\)
 and
 \begin{equation}
 \label{eq_IeV9Fo4agazu5ceiTheewad0}
   \int_{ \bigcup \mathscr{Q}^{\vert \Omega}} \abs{\Deriv U}^p
   \le C
   \smashoperator{\iint_{\Omega \times \Omega}} \frac{d \brk{u \brk{y}, u \brk{z}}^p}{\abs{y - z}^{p + m - 2}} \dif y \dif z,
 \end{equation}
 \item
 \label{it_Aihah8tixo4nu6quuuphee4s}
 \(\tr_{\Omega \times \set{0}} U = u\),
 \item
 \label{it_aShee7jee4iu6OhYie0noofo}
 for almost every \(x \in \bigcup \mathscr{Q}^{\vert \Omega} \setminus \brk{L^{m - \floor{p} - 1} \cap \bigcup \mathscr{Q}^{\mathrm{sing}}} \),
 \begin{equation}
 \label{eq_sooji5coo9Ainoviez4aite9}
    \abs{\Deriv U \brk{x}} \le
    \frac{C}{\dist \brk{x, \partial \Rset^{m}_+ \cup \brk{L^{m - \floor{p} - 1} \cap \bigcup \mathscr{Q}^{\mathrm{sing}}}}},
 \end{equation}
 \item
 \label{it_bug3Hee0chaophaita7pheew}
 one has
  \begin{equation}
  \label{eq_oo4quep0ahShuqu8aib8aib5}
\sum_{k \in \Zset} 2^{k \brk{m - p}} \# \mathscr{Q}^{\mathrm{sing}}_k
   \le C
   \smashoperator{\iint_{\substack{\Omega \times \Omega}}} \frac{\brk{d \brk{u \brk{y}, u \brk{z}} - \delta_*}_+}{\abs{y - z}^{p + m - 2}} \dif y \dif z.
 \end{equation}
\end{enumerate}
\end{theorem}

The set \(L^{m - \floor{p} - 1}\) is the dual skeleton defined at the beginning of \sectref{section_topological_to_sobolev}.

%TODO Expand the compactness argument
\begin{proof}[Proof of \cref{theorem_extension_tent_subcritical}]
We first assume that \(u\in \smash{R^1_{m - \floor{p} - 2} \brk{\Omega, \manifold{N}}}\cap \smash{\smash{\dot{W}}^{1 - 1/p, p}} \brk{\Omega, \manifold{N}}\) and that \(p \in \intvo{\ell}{\ell + 1}\)
for some \(\ell \in \set{0, \dotsc, m - 1}\), so that \(\ell = \floor{p} < p \not \in \Nset\).
By definition of the set \(\smash{R^1_{m - \ell - 2}} \brk{\Omega, \manifold{N}}\) in \eqref{eq_Quaibaibai5iajiel3yah9ie},
there is a compact set \(\Sigma \subseteq \Rset^{m - 1}\) contained in finitely many smooth submanifolds of dimension \(m - \ell - 2\) of \(\Rset^{m - 1}\) such that
\(u \in C^1 \brk{\Omega \setminus \Sigma, \manifold{N}}\) and \eqref{eq_ahshohvoeBeemu1Chuphooxa} holds.

If \(\Sigma \ne \emptyset\), we let \(\kappa \in \intvo{0}{\infty}\) be given by \cref{proposition_neat_singular_set} so that \eqref{eq_hediexie5aiz6ree2ieg2aeL} holds and we let \(\smash{\mathscr{Q}^{\infty}}\) be the dyadic decomposition of \(\smash{\Rset^m_+}\), \(\smash{\mathscr{Q}^{\mathrm{sing}, \infty}} \subseteq \smash{\mathscr{Q}^{\infty}}\)
be the collection of cubes and \(\smash{\Psi^{\mathrm{sing}, \infty}} \colon \smash{\bigcup \brk{\mathscr{Q}^{\infty \vert \Omega}}^\ell} \to \smash{\widecheck{\Sigma}}^{\kappa} \subseteq \Rset^m_+ \setminus \smash{\widehat{\Omega}}^{\mathrm{bad}}\) be the mapping
given for \(\Sigma\) and \(\kappa\) by \cref{proposition_cubes_spawning}.

Given \(k_0 \in \Zset\), we let \(\mathscr{Q}^{k_0}\)
be the dyadic decomposition of \(\Rset^m_+\), \(\mathscr{Q}^{\mathrm{bad}, k_0}\) and \(\mathscr{Q}^{\mathrm{sing}, k_0}\) be the collections of cubes and \(\smash{\Psi^{\mathrm{sing}, k_0}} \colon \smash{ \bigcup \brk{\mathscr{Q}^{k_0}}^{\ell} \to \Rset^{m}_+ \setminus \smash{\smash{\widehat{\Omega}}^{\mathrm{bad}}}}\) be the mapping given by \cref{proposition_general_modification} with \(A^{\mathrm{bad}} = \smash{\widehat{\Omega}}^{\mathrm{bad}}\).
We define the mapping \(W \colon \smash{\bigcup \brk{\mathscr{Q}^{k_0 \vert \Omega} \setminus \mathscr{Q}^{\mathrm{sing}, k_0}}}\cup \smash{\bigcup \brk{\mathscr{Q}^{k_0 \vert \Omega} \cap \mathscr{Q}^{\mathrm{sing}, k_0}}^\ell} \to \manifold{N}\)
to be \(\Pi_{\manifold{N}} \compose V \) on the set  \(\smash{\bigcup \brk{\mathscr{Q}^{k_0 \vert \Omega} \setminus \mathscr{Q}^{\mathrm{sing}, k_0}}}\)
and \(\smash{\Pi_{\manifold{N}}} \compose V \compose\smash{\Psi^{\mathrm{sing}, k_0}}\) on the set \(\smash{\bigcup \brk{\mathscr{Q}^{k_0 \vert \Omega} \cap \mathscr{Q}^{\mathrm{sing}, k_0}}^\ell}\).
Since the map \(\Psi^{\mathrm{sing}, k_0}\) is descending, \(W\) is well defined on the latter set.
In order to estimate the size of the singular cubes \(\smash{\mathscr{Q}^{\mathrm{sing}, k_0}}\), we decompose in view of \ref{it_Dih1Poh0nohhahTh2chi8sah} in \cref{proposition_general_modification}
\begin{equation}
\label{eq_it_taa6web7ien1quae8Xuey1xa}
\begin{split}
  \sum_{k \in \Zset}
  2^{k \brk{m - p}}
  \# \mathscr{Q}^{\mathrm{sing}, k_0}_{k}
  =
  \smashoperator{\sum_{k = -\infty}^{k_0 - 1}} 2^{k \brk{m - p}}
  \# \smash{\mathscr{Q}^{\mathrm{sing}, \infty}_k}
  + \smashoperator{\sum_{k = k_0}^{\infty}} 2^{k \brk{m - p}}
  \# \mathscr{Q}^{\mathrm{sing}, k_0}_{k}.
\end{split}
\end{equation}
By \cref{proposition_cubes_spawning} \ref{it_ohd6RaeyuiWooGoyo6easiey}, we have for every \(k \in \Zset\)
\begin{equation}
\label{eq_Aa2waiveeNgeethe1ahdagho}
 \# \smash{\mathscr{Q}^{\mathrm{sing}, \infty}_k}
 \le M \, \brk{2^{-k \brk{m - \ell - 2}} + 1},
\end{equation}
so that we can estimate the first term on the right-hand side of \eqref{eq_it_taa6web7ien1quae8Xuey1xa} as
\begin{equation}
\label{eq_au5oom5sic4YaiGhoogeichu}
\begin{split}
 \smashoperator[r]{\sum_{k = -\infty}^{k_0 - 1}} 2^{k \brk{m - p}}
  \# \smash{\mathscr{Q}^{\mathrm{sing}, \infty}_k}
  &\le M \smashoperator{\sum_{k = -\infty}^{k_0 - 1}} \brk{ 2^{k \brk{\ell + 2 - p}}+ 2^{k \brk{m - p}}}\\
  &\le M \brk[\Big]{\frac{2^{k_0 \brk{\ell + 2 - p}}}{2^{\ell + 2 - p}-1} + \frac{2^{k_0 \brk{m - p}}}{2^{m - p}-1}}.
\end{split}
\end{equation}

On the other hand, in order to estimate the second term on the right-hand side of \eqref{eq_it_taa6web7ien1quae8Xuey1xa}, we have by \cref{proposition_general_modification} \ref{it_taa6web7ien1quae8Xuey1xa}
\begin{equation}
\label{eq_zaex0MemieFaejahy5quaiZ1}
\begin{split}
 \smashoperator[r]{\sum_{k = k_0}^{\infty}} 2^{k \brk{m - p}}\# \mathscr{Q}^{\mathrm{sing},k_0}_k
 &\le \C \brk[\Big]{\smashoperator[r]{\sum_{k = k_0}^{\infty}} 2^{k\brk{\ell - p}+ k_0 \brk{m - \ell}} \# \smash{\mathscr{Q}^{\mathrm{sing}, \infty}_{k_0 - 1}}\\[-1em]
 &\qquad \qquad + \smashoperator[r]{\sum_{k = k_0}^{\infty}} 2^{k \brk{m - p}} \smashoperator[r]{\sum_{i = k_0}^{k}} 2^{-\brk{k - i}\brk{m - \ell}}\# \mathscr{Q}^{\mathrm{bad}, k_0}_i}.
\end{split}
\end{equation}
For the first term in the right-hand side of \eqref{eq_zaex0MemieFaejahy5quaiZ1}, by \eqref{eq_Aa2waiveeNgeethe1ahdagho} again, we have, since \(\ell < p\)
\begin{equation}
\label{eq_qua5phoh3Naeg6queefeshah}
\begin{split}
 \smashoperator[r]{\sum_{k = k_0}^{\infty}} 2^{k\brk{\ell - p}+ k_0 \brk{m - \ell}} \# \smash{\mathscr{Q}^{\mathrm{sing}, \infty}_{k_0 - 1}}
 &\le M \smashoperator{\sum_{k = k_0}^{\infty}} \brk{2^{k\brk{\ell - p}+ 2 k_0} + 2^{k\brk{\ell - p}+ k_0 \brk{m - \ell}}}\\
 &\le M \frac{
  2^{k_0 \brk{\ell + 2 - p}} + 
  2^{k_0 \brk{m - p}}}
  {1 - 2^{-\brk{p - \ell}}}.
\end{split}
\end{equation}
We finally estimate the first term in the right-hand side of \eqref{eq_zaex0MemieFaejahy5quaiZ1}, since \(p > \ell\), as
\begin{equation}
\label{eq_gie6Eimie5chezoos6kout4t}
\begin{split}
  \smashoperator[r]{\sum_{k = k_0}^{\infty}} 2^{k \brk{m - p}} \smashoperator[r]{\sum_{i = k_0}^{k}} 2^{-\brk{k - i}\brk{m - \ell}}\# \mathscr{Q}^{\mathrm{bad}, k_0}_i
  &= \sum_{i = k_0}^{\infty}
  \sum_{k = i}^\infty 2^{-k \brk{p - \ell}} 2^{i\brk{m - \ell}} \# \mathscr{Q}^{\mathrm{bad}, k_0}_i\\[-.3em]
  &= \frac{1}{1 - 2^{-\brk{p - \ell}}} \sum_{i = k_0}^\infty 2^{i \brk{m - p}}\# \mathscr{Q}^{\mathrm{bad}, k_0}_i.
\end{split}
\end{equation}
Combining \eqref{eq_zaex0MemieFaejahy5quaiZ1}, \eqref{eq_qua5phoh3Naeg6queefeshah} and \eqref{eq_gie6Eimie5chezoos6kout4t}, we control the second term in the right-hand side of \eqref{eq_it_taa6web7ien1quae8Xuey1xa} by 
\begin{equation}
\label{eq_Bee3thoM2Iene9ov2xohZ7fo}
 \smashoperator[r]{\sum_{k = k_0}^{\infty}} 2^{k \brk{m - p}}\# \mathscr{Q}^{\mathrm{sing}, k_0}_k
 \le \C
  \brk[\Big]{M\,  \brk{2^{k_0 \brk{\ell + 2 - p}} + 2^{k_0 \brk{m - p}}} + \smashoperator[r]{\sum_{k \in \Zset }} 2^{k \brk{m - p}}\# \mathscr{Q}^{\mathrm{bad}, k_0}_k}.
\end{equation}
Since \(m > p\) and \(\ell + 2 - p > 0\), taking \(k_0 \in \Zset\) small enough, depending on \(u\) through \(\Sigma\) and \(\kappa\), we get by \eqref{eq_it_taa6web7ien1quae8Xuey1xa}, \eqref{eq_au5oom5sic4YaiGhoogeichu} and \eqref{eq_Bee3thoM2Iene9ov2xohZ7fo},
\begin{equation}
\label{eq_Kiebieb4ieL5kaiphaet5ohy}
\sum_{k \in \Zset} 2^{k \brk{m - p}}
  \# \mathscr{Q}^{\mathrm{sing}, k_0}_{k}
  \le
  \C \sum_{k \in \Zset } 2^{k \brk{m - p}}
  \# \mathscr{Q}^{\mathrm{bad}, k_0}_k.
\end{equation}

We let \(\mathscr{Q} \defeq  \mathscr{Q}^{k_0}\) and \(U \colon \bigcup \mathscr{Q}^{\vert \Omega} \to \manifold{N}\) be the mapping defined by \cref{proposition_Lipschitz_extension_skeleton}
with \(\smash{\mathscr{Q}^{\mathrm{sing}}} \defeq \smash{\mathscr{Q}^{k_0 \vert \Omega}} \cap \smash{\mathscr{Q}^{\mathrm{sing},k_0}}\) and \(\smash{\mathscr{Q}^{\mathrm{reg}}} \defeq \smash{\mathscr{Q}^{k_0 \vert \Omega}} \setminus \smash{\mathscr{Q}^{\mathrm{sing},k_0}}\). By \eqref{eq_yie2te0ahrie6eeL2mah5Boh}, \eqref{eq_xuoY7IuceiwooH5Aicadai9a} and \eqref{eq_Kiebieb4ieL5kaiphaet5ohy}, the estimates \eqref{eq_IeV9Fo4agazu5ceiTheewad0}, \eqref{eq_sooji5coo9Ainoviez4aite9} and \eqref{eq_oo4quep0ahShuqu8aib8aib5} hold so that \ref{it_aighie9iush4hei9opoo3Quo}, \ref{it_aShee7jee4iu6OhYie0noofo} and \ref{it_bug3Hee0chaophaita7pheew} are proved.

In order to prove \ref{it_Aihah8tixo4nu6quuuphee4s}, we note that
\[
A \defeq \set{\brk{x', x_m} \in \bigcup \mathscr{Q}^{\vert \Omega}  \st
 \abs{U (x)- V(x)} > \dist \brk{V (x), \manifold{N}}}
 \subseteq
 \bigcup \mathscr{Q}^{\mathrm{sing}}
\]
and that if \(k_* \le k_0\)
\[
  \bigcup \mathscr{Q}^{\mathrm{sing}} \cap \Rset^{m - 1}\times \intvl{0}{2^{k_*}}  = \cup_{k \le k_* - 1} \bigcup \mathscr{Q}^{\mathrm{sing}, k_0}_k.
\]
Therefore, in view of \eqref{eq_Aa2waiveeNgeethe1ahdagho},
\[
\begin{split}
  \frac{\mathcal{L}^m \brk{A \cap \Rset^{m - 1}\intvl{0}{2^{k_*}}}}{2^{k_*}}
  &\le \sum_{k \le k_*} \smashoperator[r]{\sum_{k = -\infty}^{k_*}} 2^{k m -k_*}
  \# \smash{\mathscr{Q}^{\mathrm{sing}, \infty}_k}\\
  &\le M \smashoperator{\sum_{k = -\infty}^{k_* - 1}} \brk{2^{k m-k_*} + 2^{k \brk{\ell + 2}-k_*}} =  M \brk[\Big]{\frac{2^{k_* \brk{m - 1}}}{2^{m}-1} + \frac{2^{k_* \brk{\ell + 1}}}{2^{\ell + 2}-1}}.
\end{split}
\]
Since \(m > 1\) and \(\ell + 1 > 0\), \ref{it_Aihah8tixo4nu6quuuphee4s} follows from \cref{proposition_equal_traces}.

\medbreak

Still assuming that \(p \in \intvo{\ell}{\ell + 1}\), we assume now that
\(u\) is the strong limit in \(\smash{\smash{\dot{W}}^{1 - 1/p, p}} \brk{\Omega, \manifold{N}}\) of a sequence \(\brk{u_n}_{n \in \Nset}\) in \(\smash{R^1_{m - \floor{p} - 2} \brk{\Omega, \manifold{N}}}\cap \smash{\smash{\dot{W}}^{1 - 1/p, p}} \brk{\Omega, \manifold{N}}\).
For each \(n \in \Nset\), the first part of the proof gives us a dyadic decomposition \(\mathscr{Q}_n\) be the dyadic decomposition, a collection of cubes \(\mathscr{Q}_n^{\mathrm{sing}}\subseteq \smash{\mathscr{Q}_n^{\vert \Omega}}\) and a map \(U_n \colon \smash{\bigcup \mathscr{Q}_n^{\vert \Omega}} \to \manifold{N}\).
Since the \(\xi_k\) appearing in the definition of dyadic decomposition (\cref{definition_dyadic_decomposition}) are only defined modulo \(2^{k} \Zset^{m - 1}\), standard compactness and diagonal arguments show that \(\mathscr{Q}_n\) converges to some dyadic decomposition \(\mathscr{Q}\) in the sense of the convergence of the translation factors. (This also follows from the solenoid structure of the dyadic decompositions described in \cref{remark_solenoid}.)

In particular, there exists a function \(U \in \dot{W}^{1, p}\brk{\bigcup \mathscr{Q}^{\vert \Omega}, \manifold{N}}\), if \(\omega \subseteq \Rset^m\) is open and if
\begin{equation}
\label{eq_Roth3yaewoobeinefiephoni}
  \Bar{\omega} \subseteq \operatorname{int}\brk[\bigg]{\bigcup \mathscr{Q}^{\vert \Omega}},
\end{equation}
then, up to a subsequence, \(U_n\restr{\omega} \weakto U\restr{\omega}\) weakly in \( \dot{W}^{1, p}\brk{\omega, \manifold{N}}\),
and
\[
 \int_{\omega}\abs{\Deriv U}^p
 \le \liminf_{n \to \infty} \int_{\omega}\abs{\Deriv U_n}^p\le \liminf_{n \to \infty} \int_{ \bigcup \mathscr{Q}_n^{\vert \Omega}} \abs{\Deriv U_n}^p.
\]
By \ref{it_aighie9iush4hei9opoo3Quo} applied to \(u_n\) and \(U_n\), we have
\begin{equation}
\begin{split}
 \int_{ \bigcup \mathscr{Q}^{\vert \Omega}} \abs{\Deriv U}^p
 &\le \liminf_{n \to \infty} \int_{ \bigcup \mathscr{Q}_n^{\vert \Omega}} \abs{\Deriv U_n}^p\\
  & \le \Cl{cst_PhengaeShoocuu0wah6aet9c}
   \smashoperator{\iint_{\Omega \times \Omega}} \frac{d \brk{u_n \brk{y}, u_n \brk{z}}^p}{\abs{y - z}^{p + m - 2}} \dif y \dif z\\
  &\le  \Cr{cst_PhengaeShoocuu0wah6aet9c}
   \smashoperator{\iint_{\Omega \times \Omega}} \frac{d \brk{u \brk{y}, u \brk{z}}^p}{\abs{y - z}^{p + m - 2}} \dif y \dif z,
\end{split}
 \end{equation}
 so that \ref{it_aighie9iush4hei9opoo3Quo} holds for \(u\) and \(U_n\).

By the weak continuity of the trace operator,  \ref{it_Aihah8tixo4nu6quuuphee4s} holds.

By \ref{it_bug3Hee0chaophaita7pheew} applied to \(u_n\) and \(\mathscr{Q}_n\), we can assume, up to a subsequence, that there exists \(\mathscr{Q}^{\mathrm{sing}}\) such that
\begin{equation}
\label{eq_iej8aiTeeyauT3eid8iejie2}
 \mathcal{L}^{m}\brk[\Big]{\bigcup \mathscr{Q}^{\mathrm{sing}} \triangle \bigcup \mathscr{Q}^{\mathrm{sing}}_n} = 0
\end{equation}
and
\[
 \sum_{k \in \Zset} 2^{k \brk{m - p}} \# \mathscr{Q}^{\mathrm{sing}}_k
 \le \lim_{n \to \infty}\sum_{k \in \Zset} 2^{k \brk{m - p}} \# \mathscr{Q}^{\mathrm{sing}}_{n, k}.
\]
Finally, defining
\[
 F \brk{x} \defeq \frac{1}{\dist \brk{x, \partial \Rset^{m}_+ \cup \brk{L^{m - \floor{p} - 1} \cap \bigcup \mathscr{Q}^{\mathrm{sing}}}}}
\]
and for \(n \in \Nset\)
\[
  F_n \brk{x} \defeq \frac{1}{\dist \brk{x, \partial \Rset^{m}_+ \cup \brk{L^{m - \floor{p} - 1} \cap \bigcup \mathscr{Q}^{\mathrm{sing}}_n }}},
\]
we have thanks to \eqref{eq_iej8aiTeeyauT3eid8iejie2} and the structure of dyadic cubes, \(F_n\restr{\omega} \to F\restr{\omega}\) for every open set \(\omega \subseteq \Rset^m_+\) satisfying \eqref{eq_Roth3yaewoobeinefiephoni}, we have
\(F_n \to F\) in \(L^p \brk{\omega, \Rset}\). The assertion \ref{eq_sooji5coo9Ainoviez4aite9} follows from the stability of pointwise inequalities under weak convergence in \(L^p \brk{\omega, \Rset}\).

\medbreak

We now consider the case where \(p \in \Nset\).
Defining \(\ell \defeq p - 1\), we proceed as in the case \(p \not \in \Nset\) up to the point where we have a mapping
\(W \colon \smash{\bigcup \brk{\smash{\mathscr{Q}^{k_0 \vert \Omega}} \setminus \smash{\mathscr{Q}^{\mathrm{sing}, k_0}}}} \cup \smash{\bigcup \brk{\smash{\mathscr{Q}^{k_0 \vert \Omega}} \cap \smash{\mathscr{Q}^{\mathrm{sing}, k_0}}}^{p - 1}} \to \manifold{N}\).
Since the homotopy group \(\pi_{p - 1} \brk{\manifold{N}}\) is trivial, the mapping \(W\) can be further extended into a continuous map from \(\smash{\bigcup \brk{\mathscr{Q}^{k_0 \vert \Omega} \setminus \smash{\mathscr{Q}^{\mathrm{sing}, k_0}}}} \cup \smash{\bigcup \brk{\smash{\mathscr{Q}^{k_0 \vert \Omega}} \cap \smash{\mathscr{Q}^{\mathrm{sing}, k_0}}}^{p}}\) to \(\manifold{N}\) and we then apply \cref{proposition_Lipschitz_extension_skeleton} with \(\ell = p\).
\end{proof}

\begin{remark}
\Cref{theorem_extension_tent_subcritical} has a probabilistic proof and formulation, in the prolongation of \cref{remark_probablistic_critical}, \cref{remark_probabilistic_propagation_decay} and \cref{remark_probabilistic_spawing}.
Indeed although the cardinals of the subcollections \(\smash{\# \mathscr{Q}^{\mathrm{sing},k_0}_k}\) are not independent random variables, one can still use the linearity property of the expectation to get an estimate on the expectation of the size of bad cubes similar to \eqref{eq_Bee3thoM2Iene9ov2xohZ7fo};
since moreover the quantity \(M\) in \cref{proposition_cubes_spawning} is deterministic as explained in \cref{remark_probabilistic_spawing},
the conclusions \ref{it_aighie9iush4hei9opoo3Quo} and \ref{it_bug3Hee0chaophaita7pheew} hold for the expectation of the left-hand side over random dyadic decompositions of \(\Rset^m_+\).
\end{remark}

\begin{remark}
\label{remark_p_dependence}
The constants of the estimates \eqref{eq_IeV9Fo4agazu5ceiTheewad0} and \eqref{eq_oo4quep0ahShuqu8aib8aib5} can be tracked in the proof of \cref{theorem_extension_tent_subcritical}: if \(\ell \in \set{1, \dotsc, m - 1}\), then there exists a constant \(C \in \intvo{0}{\infty}\) such that if \(\ell < p < \ell + 1\) and if the lower homotopy groups \(\pi_1\brk{\manifold{N}}, \dotsc, \pi_{\ell - 1} \brk{\manifold{N}}\) are all finite, then
\begin{multline}
\label{eq_thuiShohr1omah4ebeeThech}
   \int_{\bigcup \mathscr{Q}^{\vert \Omega}} \abs{\Deriv U}^p
   \le
   C
   \brk[\bigg]{
   \smashoperator{\iint_{\Omega \times \Omega}} \frac{d \brk{u \brk{y}, u \brk{z}}^p}{\abs{y - z}^{p + m - 2}} \dif y \dif z\\[-1em]
   +
   \frac{1}{\brk{\ell + 1 - p}\brk{p - \ell}}
   \smashoperator{\iint_{\Omega \times \Omega}}  \frac{\brk{d \brk{u \brk{y}, u \brk{z}} - \delta_*}_+}{\abs{y - z}^{p + m - 2}} \dif y \dif z }
 \end{multline}
and
\begin{equation}
   \smashoperator[r]{\sum_{k \in \Zset}} 2^{k \brk{m - p}} \# \mathscr{Q}^{\mathrm{sing}}_k
   \le
   \frac{C}{p - \ell}
   \smashoperator{\iint_{\Omega \times \Omega}} \frac{\brk{d \brk{u \brk{y}, u \brk{z}} - \delta_*}_+}{\abs{y - z}^{p + m - 2}} \dif y \dif z.
\end{equation}
The crucial estimates for the dependence are \eqref{eq_gie6Eimie5chezoos6kout4t} in the summation of the geometric series and \eqref{eq_xuoY7IuceiwooH5Aicadai9a} for the integrability of the resulting function.
Because of the critical analytical obstruction  \cite{Mironescu_VanSchaftingen_2021_AFST}*{Th.\thinspace{}1.5 (b)}, the estimate \eqref{eq_thuiShohr1omah4ebeeThech} is optimal as \(p \to \ell\) and as \(p \to \ell + 1\) when the homotopy groups \(\pi_{\ell - 1} \brk{\manifold{N}}\) and \(\pi_{\ell} \brk{\manifold{N}}\) are non-trivial respectively.
\end{remark}

\begin{remark}
The estimates \eqref{eq_IeV9Fo4agazu5ceiTheewad0} and \eqref{eq_oo4quep0ahShuqu8aib8aib5} can also be localized.
Through a careful inspection of the proof and a suitable limiting argument when \(k_0 \to -\infty\), it can be proved that there exists a constant \(C\) such that for every \(\Bar{k}\in \Zset\), one has
 \begin{equation}
   \smashoperator[r]{\int_{ \brk{\Rset^{m - 1} \times \intvo{0}{2^{\Bar{k}}} } \,\cap\, \bigcup \mathscr{Q}^{\vert \Omega} }} \quad \abs{\Deriv U}^p
   \le
   C
   \smashoperator{\iint_{\substack{\brk{y, z} \in \Omega \times \Omega\\
   d\brk{y, z} \le 2^{\Bar{k}}}}} \frac{d \brk{u \brk{y}, u \brk{z}}^p}{\abs{y - z}^{p + m - 2}} \dif y \dif z,
 \end{equation}
and
  \begin{equation}
   \smashoperator[r]{\sum_{k = -\infty}^{\Bar{k}}} 2^{k \brk{m - p}} \# \mathscr{Q}^{\mathrm{sing}}_k
   \le C
   \smashoperator{\iint_{\substack{\brk{y, z} \in \Omega \times \Omega\\
   d\brk{y, z} \le 2^{\Bar{k}}}}} \frac{\brk{d \brk{u \brk{y}, u \brk{z}} - \delta_*}_+}{\abs{y - z}^{p + m - 2}} \dif y \dif z.
 \end{equation}
The localisation of the present remark can be combined with the explicit dependence in \(p\) of \cref{remark_p_dependence}.
\end{remark}

We are now in a position to prove the subcritical extension results on the half-space \(\Rset^m_+\).

\begin{proof}[Proof of \cref{theorem_extension_halfspace}]
This follows from \cref{theorem_extension_tent_subcritical} with \(\Omega = \Rset^{m - 1}\), so that \(\smash{\widehat{\Omega}} = \Rset^m_+\)
and from the fact that, since the homotopy group \(\smash{\pi_{\floor{p - 1}}} \brk{\manifold{N}}\) is trivial, the set \(\smash{R^{1}_{m - \floor{p} - 2}}\smash{\brk{\Rset^{m - 1}, \manifold{N}}}\) is strongly dense in the space \(\smash{\smash{\dot{W}}^{1 - 1/p, p}\brk{\Rset^{m - 1}, \manifold{N}}}\) \citelist{\cite{Bethuel_1995}\cite{Brezis_Mironescu_2015}\cite{Mucci_2009}\cite{Mucci_2024}} (see also \cref{theorem_R_dense_Wsp} below).
\end{proof}

The use of the density result for fractional Sobolev mappings \citelist{\cite{Mucci_2024}\cite{Brezis_Mironescu_2015}\cite{Mucci_2009}} is essential in the proof of \cref{theorem_extension_halfspace} and accounts by itself for a substantial part of the length and complexity of a self-contained proof of \cref{theorem_extension_halfspace}.

\medbreak

We can also prove the extension of limits of continuous maps; the other part of \cref{theorem_characterization_halfspace} will be proved in \cref{proposition_approximation_Rm}.

\begin{proof}[Proof of the extension result \cref{theorem_characterization_halfspace}]
This follows immediately from \cref{theorem_extension_tent_subcritical} with \(\Omega = \Rset^{m - 1}\).
\end{proof}

\section{Collar neighbourhood extension}
\resetconstant 

\label{section_collar}
In order to get the extensions to collar neighbourhoods \(\manifold{M}' \times \intvr{0}{1}\) when \(\manifold{M}'\) is a compact \(\brk{m - 1}\)-dimensional Riemannian manifold of 
\cref{theorem_extension_collar,theorem_characterization_collar,theorem_estimate_supercritical_collar,theorem_estimate_critical_collar}, our strategy will be to show that, up to a deformation, the boundary datum \(u\) can be extended to \(\smash{\manifold{M}'^{m - 2}}\times \intvo{0}{1}\), where \(\smash{\manifold{M}'^{m - 2}}\) is the \(\brk{m - 2}\)-dimensional component of a given triangulation of \(\manifold{M}’\), and then to perform extensions to \(\sigma \times \intvo{0}{1}\) for every \(\brk{m - 1}\)-dimensional face \(\sigma\) of this triangulation.

In a first step, we show the existence of a deformation for which the mapping \(u\) behaves well on the \(\brk{m - 2}\)-dimensional component of the triangulation.
This is essentially a version for triangulations of the Fubini-type classical slicing property for fractional Sobolev spaces on the Euclidean space \cite{Strichartz_1968} (see \cite{Leoni_2023}*{Th.\thinspace{}6.35}) and to the notion of good restriction \cite{Brezis_Mironescu_2001}*{App.\thinspace{}B}.

\begin{proposition}
\label{proposition_fractional_Fubini_perturbation}
Let \(\manifold{M}'\) be an \(\brk{m - 1}\)-dimensional compact triangulated Riemannian manifold embedded into \(\Rset^\mu\).
If \(p \in \intvr{1}{\infty}\) and if \(\delta \in \intvo{0}{\infty}\) is small enough, then there exists a constant \(C \in \intvo{0}{\infty}\) such that 
for every \(u \in \smash{\dot{W}}^{1 - 1/p, p} \brk{\manifold{M}', \manifold{N}}\) and 
for every \(h \in \Bset^{\mu}_{\delta}\), the map \(u_h\colon \manifold{M}' \to \manifold{N}\) defined for each \(x \in \manifold{M}'\) by 
\begin{equation}
\label{eq_foh5eewaivaigol1aezi8VeB}
 u_h \brk{x} 
 \defeq u
 \brk{\Pi_{\manifold{M}'}\brk{x + h}}
\end{equation}
satisfies
\begin{equation}
\label{eq_vaino8raCh2yeechieVohqua}
\int_{\Bset^{\mu}_{\delta}}
\brk[\bigg]{\smashoperator[r]{\iint_{\manifold{M}' \times \manifold{M}'^{m  - 2}}}
\frac{d \brk{u_h \brk{y}, u_h \brk{z}}^p}{d \brk{y, z}^{p + m - 2}}\dif y \dif z} \dif h 
\le 
C \smashoperator{\iint_{\manifold{M}' \times \manifold{M}'}}
\frac{d \brk{u \brk{y}, u \brk{z}}^p}{d \brk{y, z}^{p + m - 2}}\dif y \dif z.
\end{equation}
\end{proposition}
In particular, the inner integral in the left-hand side of \eqref{eq_vaino8raCh2yeechieVohqua} is finite for almost every \(h \in \Bset^{\mu}_{\delta}\).

\begin{proof}[Proof of \cref{proposition_fractional_Fubini_perturbation}]
This follows from Fubini's theorem and transversality properties of the nearest-point projection \(\Pi_{\manifold{M}'}\).
\end{proof}

The boundedness of the inner integral in the left-hand side \eqref{eq_vaino8raCh2yeechieVohqua} implies that the boundary datum map \(u\) can be extended to a suitable \(\brk{m - 1}\)-dimensional skeleton.
Given \(\manifold{M}'\) an \(\brk{m - 1}\)-dimensional triangulated compact Riemannian manifold and \(\ell \in \set{0, \dotsc, m - 1}\), we define 
\begin{equation}
\label{eq_Oth9sieg6aiyahVahChe1gok}
 \manifold{M}'{}^{\sqcup, \ell}
 \defeq 
 \brk{\manifold{M}'{}^{\ell} \times \set{0}} \cup 
 \brk{\manifold{M}'{}^{\ell - 1} \times \intvc{0}{1}};
\end{equation}
in other words \(\manifold{M}'{}^{\sqcup, \ell}\) is the union of the bottom \(\ell\)-skeleton and the vertical cylinder over the \(\brk{\ell - 1}\)-skeleton.

\begin{proposition}
\label{proposition_fractional_Sobolev_triangulation_collar}
Let \(\manifold{M}'\) be an \(\brk{m - 1}\)-dimensional triangulated compact Riemannian manifold.
If \(p \in \intvr{1}{\infty}\), then there exists a constant \(C \in \intvo{0}{\infty}\) such that for every Borel-measurable mapping \(u \colon \manifold{M}' \to \manifold{N}\) one has 
\begin{multline}
\label{eq_ieB2Uvohv6Eifahgeejoal8j}
\smashoperator[r]{\iint_{\manifold{M}'{}^{\sqcup, m - 1} \times \manifold{M}'{}^{\sqcup, m - 1}}}
\frac{d \brk{u \brk{y'}, u \brk{z'}}^p}{d \brk{y, z}^{p + m - 2}}\dif y \dif z\\[-1em]
\le \smashoperator{\iint_{\manifold{M}' \times \manifold{M}'}}
\frac{d \brk{u \brk{y}, u \brk{z}}^p}{d \brk{y, z}^{p + m - 2}}\dif y \dif z
+
C \smashoperator{\iint_{\manifold{M}' \times \smash{\manifold{M}'^{m - 2}}}}
\frac{d \brk{u \brk{y}, u \brk{z}}^p}{d \brk{y, z}^{p + m - 2}}\dif y \dif z,
\end{multline}
with \(y = \brk{y', y_m} \in \manifold{M}' \times \intvc{0}{1}\) and \(z = \brk{z', z_m}\in \manifold{M}' \times \intvc{0}{1}\) inside the left-hand side integral.
\end{proposition}

The proof of \cref{proposition_fractional_Sobolev_triangulation_collar} will rely on the following estimate.

\begin{proposition}
\label{proposition_fractional_Fubini_mixed}
Under the assumptions of \cref{proposition_fractional_Sobolev_triangulation_collar}, there exists a constant \(C \in \intvo{0}{\infty}\) such that for every Borel-measurable mapping \(u \colon \manifold{M}' \to \manifold{N}\) one has 
\begin{equation}
\label{eq_eeyoob6aegie4GaXai6too6k}
\smashoperator{\iint_{\manifold{M}'{}^{m - 2}\times \manifold{M}'{}^{m - 2}}}
 \frac{d \brk{u \brk{y}, u \brk{z}}^p}{d \brk{y, z}^{p + m - 3}}\dif y \dif z
 \le C 
 \smashoperator{\iint_{ \manifold{M}' \times \manifold{M}'{}^{m - 2}}}
 \frac{d \brk{u \brk{y}, u \brk{z}}^p}{d \brk{y, z}^{p + m - 2}}\dif y \dif z.
 \end{equation}
\end{proposition}
\begin{proof}
By convexity and by the triangle inequality, we have for every \(y, z \in \manifold{M}'{}^{m - 2}\)
\begin{equation}
\label{eq_iuGhoShieGahphieth6wah7e}
  d \brk{u (y), u (z)}^p
  \le \frac{\Cl{cst_shahCa1aim7ahFaBo7ael7ee}}{d \brk{y, z}^{m - 1}}
  \brk[\bigg]{
  \smashoperator[r]{\int_{\substack{x \in \manifold{M}'\\ 
  d (x, y) \le d \brk{y, z}}}}
  d \brk{u (x), u (y)}^p \dif x 
  +
 \smashoperator{\int_{\substack{x \in \manifold{M}'\\ 
  d \brk{x, z} \le d \brk{y, z}}}}
  d \brk{u (x), u (z)}^p \dif x 
 }.
\end{equation}
By integration of \eqref{eq_iuGhoShieGahphieth6wah7e} and by symmetry, and then by interchanging the integrals, we get
\begin{equation}
\label{eq_IeSh9Uuth7zuiSail5ieyieP}
\begin{split}
\smashoperator[r]{\iint_{\manifold{M}'{}^{m - 2}\times \manifold{M}'{}^{m - 2}}}
 \frac{d \brk{u \brk{y}, u \brk{z}}^p}{d \brk{y, z}^{p + m - 3}}\dif y \dif z
&\le 2\,\Cr{cst_shahCa1aim7ahFaBo7ael7ee}
  \smashoperator[l]{\iint_{\manifold{M}'{}^{m - 2} \times \manifold{M}'{}^{m - 2}}}
  \smashoperator[r]{\int_{\substack{x \in \manifold{M}'\\ 
  d \brk{x, z} \le d \brk{y, z}}}}
  \frac{d \brk{u (x), u (z)}^p}{d \brk{y, z}^{p + 2m - 4}} \dif x \dif y \dif z\\
&=  2\,\Cr{cst_shahCa1aim7ahFaBo7ael7ee}
  \int_{\manifold{M}'{}^{m - 2}} \int_{\manifold{M}'}
  \brk[\bigg]{
  \smashoperator[r]{\int_{\substack{y \in \manifold{M}'{}^{m - 2}\\ 
  d \brk{x, z} \le d \brk{y, z}}}}
  \frac{d \brk{u (x), u (z)}^p}{d \brk{y, z}^{p + 2m - 4}}
  \dif y}
  \dif x \dif z.
  \end{split}
  \raisetag{2em}
\end{equation}
The conclusion \eqref{eq_eeyoob6aegie4GaXai6too6k} follows from a direct estimate on the innermost integral in the right-hand side of \eqref{eq_IeSh9Uuth7zuiSail5ieyieP}.
\end{proof}

\begin{proof}[Proof of \cref{proposition_fractional_Sobolev_triangulation_collar}]
We first have  by integration of \(z_m\) over \(\intvo{0}{1}\)
\begin{equation}
\label{eq_Tuolooturavaenoopeichue2}
\begin{split}
 \smashoperator[r]{\iint_{\manifold{M}'\times \brk{\manifold{M}'{}^{m - 2} \times \intvo{0}{1}}}}
 \frac{d \brk{u \brk{y}, u \brk{z'}}^p}{d \brk{y, z}^{p + m - 2}}\dif y \dif z
 &\le \C \smashoperator{\iint_{\manifold{M}' \times \smash{\manifold{M}'^{m - 2}}}}
\frac{d \brk{u \brk{y}, u \brk{z}}^p}{d \brk{y, z}^{p + m - 3}}\dif y \dif z\\
&\le \C \smashoperator{\iint_{\manifold{M}' \times \smash{\manifold{M}'^{m - 2}}}}
\frac{d \brk{u \brk{y}, u \brk{z}}^p}{d \brk{y, z}^{p + m - 2}}\dif y \dif z,
\end{split}
\end{equation}
since \(\manifold{M}'\) is bounded.
Next we have by integration with respect to \(y_m\) and \(z_m\) over \(\intvo{0}{1}\)
and then by \cref{proposition_fractional_Fubini_mixed}
\begin{equation}
\label{eq_Nohp5foo2chahshaawaijahp}
\begin{split}
\smashoperator[r]{\iint_{\brk{\manifold{M}'{}^{m - 2} \times \intvo{0}{1}}\times \brk{\manifold{M}'{}^{m - 2} \times \intvo{0}{1}}}}
 \frac{d \brk{u \brk{y'}, u \brk{z'}}^p}{d \brk{y, z}^{p + m - 2}}\dif y \dif z
 &\le \C\smashoperator{\iint_{\manifold{M}'{}^{m - 2}\times \manifold{M}'{}^{m - 2}}}
 \frac{d \brk{u \brk{y}, u \brk{z}}^p}{d \brk{y, z}^{p + m - 3}}\dif y \dif z\\
&\le \C \smashoperator{\iint_{\manifold{M}' \times \smash{\manifold{M}'^{m - 2}}}}
\frac{d \brk{u \brk{y}, u \brk{z}}^p}{d \brk{y, z}^{p + m - 2}}\dif y \dif z.
\end{split}
\end{equation}
The conclusion \eqref{eq_ieB2Uvohv6Eifahgeejoal8j} follows then from \eqref{eq_Tuolooturavaenoopeichue2} and \eqref{eq_Nohp5foo2chahshaawaijahp}.
\end{proof}

\medbreak 
We can now perform the extensions on collar neighbourhoods of Theorems \ref{theorem_extension_collar}, \ref{theorem_characterization_collar}, \ref{theorem_estimate_supercritical_collar} and \ref{theorem_estimate_critical_collar}.

\begin{proof}[Proof of Theorems \ref{theorem_extension_collar}, \ref{theorem_characterization_collar} (sufficiency for the extension), \ref{theorem_estimate_supercritical_collar} and \ref{theorem_estimate_critical_collar}]
By standard approximation and compactness arguments, we can assume that \(u \in \smash{R^1_{m - \floor{p} - 2}\brk{\manifold{M}', \manifold{N}}}\) if \(p < m\) and \(u \in \smash{C^1} \brk{\manifold{M}', \manifold{N}}\) if \(p \ge m\).
Applying \cref{proposition_fractional_Fubini_perturbation}, we get some \(h_* \in \Bset^{\mu}_{\delta}\) such that if we set \(u_* \defeq u_{h_*}\) we have 
\begin{equation}
\label{eq_quoyaisohchoT1aeJ0coo0wi}
\smashoperator[r]{\iint_{\manifold{M}' \times \smash{\manifold{M}'^{m - 2}}}}
\frac{d \brk{u_* \brk{y}, u_* \brk{z}}^p}{d \brk{y, z}^{p + m - 2}}\dif y \dif z
\le 
\C \smashoperator{\iint_{\manifold{M}' \times \manifold{M}'}}
\frac{d \brk{u \brk{y}, u \brk{z}}^p}{d \brk{y, z}^{p + m - 2}}\dif y \dif z.
\end{equation}
We can assume moreover that \(h_*\in \Bset^{\mu}_{\delta}\) has been chosen so that we also have
\(\smash{u_*\restr{\manifold{M}'{}^{m - 2}}}\in \smash{R^1_{m - \floor{p} - 3} \smash{\brk{\manifold{M}'{}^{m - 2}, \manifold{N}}}}\).

We define the mapping \(u_*^{\sqcup} \colon \manifold{M}'{}^{\sqcup, m - 1} \to \manifold{N}\) for every \(x = \brk{x', x_m} \in \manifold{M}'{}^{\sqcup, m - 1} \subseteq \manifold{M}' \times \intvc{0}{1}\) by \(u_*^{\sqcup} \brk{x', x_m} \defeq u_* \brk{x'}\).
By \cref{proposition_fractional_Sobolev_triangulation_collar}, we have 
\begin{equation}
\label{eq_Eem4esue7Ieho1voa2shai2C}
\begin{split}
\smashoperator[r]{\iint_{\manifold{M}'{}^{\sqcup, m - 1} \times \manifold{M}'{}^{\sqcup, m - 1}}}&
\frac{d \brk{u_*^{\sqcup} \brk{y'}, u_*^{\sqcup} \brk{z'}}^p}{d \brk{y, z}^{p + m - 2}}\dif y \dif z\\
&\le \C \bigg(\smashoperator[r]{\iint_{\manifold{M}' \times \smash{\manifold{M}'^{m - 2}}}}
\frac{d \brk{u_* \brk{y}, u_* \brk{z}}^p}{d \brk{y, z}^{p + m - 2}}\dif y \dif z+ \smashoperator{\iint_{\manifold{M}' \times \manifold{M}'}}
\frac{d \brk{u_* \brk{y}, u_* \brk{z}}^p}{d \brk{y, z}^{p + m - 2}}\dif y \dif z\bigg).
\end{split}
\end{equation}
We also have 
\(u_*^{\sqcup} \in \smash{
R^1_{m - \floor{p} - 2} \brk{\manifold{M}'{}^{\sqcup, m - 1}, \manifold{N}}}\).

For every \(\brk{m - 1}\)-dimensional simplex \(\sigma\) in the triangulation of \(\manifold{M}'\), we apply \cref{theorem_extension_tent_supercritical}, \cref{theorem_extension_tent_critical} or \cref{theorem_extension_tent_subcritical}, depending on whether \(p > m\), \(p = m\) or \(p < m\), and a suitable change of variables, to define the map \(U_*\) on \(\sigma \times \intvo{0}{1}\) as a map in \(\smash{\smash{\dot{W}}^{1, p}} \brk{\sigma \times \intvo{0}{1}, \manifold{N}}\) having \(u_*^{\sqcup}\) as trace on the set \(\partial_{\sqcup} \sigma = \brk{\sigma \times \set{0}} \cup \brk{\partial \sigma \times \intvc{0}{1}}
\subseteq \smash{\manifold{M}'{}^{\sqcup, m - 1}}\).
The map \(U_* \colon \manifold{M}' \times \intvr{0}{1} \to \manifold{N}\) is then well defined and satisfies the estimate
\begin{equation}
\label{eq_lah0aegu7BaeD9ise4axetee}
\smashoperator[r]{\int_{\manifold{M}' \times \intvr{0}{1}}} \abs{\Deriv U_*}^p 
\le 
\C \; \smashoperator{\iint_{\manifold{M}'{}^{\sqcup, m - 1} \times \manifold{M}'{}^{\sqcup, m - 1}}}
\frac{d \brk{u_*^{\sqcup} \brk{y'}, u_*^{\sqcup} \brk{z'}}^p}{d \brk{y, z}^{p + m - 2}}\dif y \dif z.
\end{equation}

Finally, we define the mapping \(U \colon \manifold{M}' \times \intvr{0}{1}\to \manifold{N}\) for each \(x = \brk{x', x_m} \in  \manifold{M}' \times \intvr{0}{1}\) by 
\begin{equation*}
 U \brk{x', x_m} 
 \defeq 
 U_* \brk{\Pi_{\manifold{M}'} \restr{\manifold{M}' + h_*}^{-1} \brk{x'} - h_*, x_m},
\end{equation*}
so that by the estimates \eqref{eq_quoyaisohchoT1aeJ0coo0wi}, \eqref{eq_Eem4esue7Ieho1voa2shai2C} and \eqref{eq_lah0aegu7BaeD9ise4axetee} and by a change of variables, we have \(U \in \smash{\smash{\dot{W}}^{1, p}} \brk{\manifold{M}' \times \intvr{0}{1}, \manifold{N}}\) and 
\begin{equation*}
\smashoperator[r]{\int_{\manifold{M}' \times \intvr{0}{1}}} \abs{\Deriv U}^p 
\le \C \smashoperator{\iint_{\manifold{M}' \times \manifold{M}'}}
\frac{d \brk{u \brk{y}, u \brk{z}}^p}{d \brk{y, z}^{p + m - 2}}\dif y \dif z,
\end{equation*}
and thus \(U\) satisfies the announced estimate \eqref{eq_jooFofo1au1aephae7sa0joh}, \eqref{eq_mei0keegeef7aiVaiSu1aegi} or \eqref{eq_eing7aeth4uk1yiemee4Wa4N},
while in view of the definition of \(u_* = u_{h_*}\) in  \eqref{eq_foh5eewaivaigol1aezi8VeB} we have for almost every \(x' \in \manifold{M}'\)
\begin{equation*}
\begin{split}
 \tr_{\manifold{M}' \times \set{0}} U \brk{x'}
 &= \tr_{\manifold{M}' \times \set{0}} U_* \brk{\Pi_{\manifold{M}'} \restr{\manifold{M}' + h_*}{}^{-1} \brk{x'} - h_*}\\
 &= u_* \brk{\Pi_{\manifold{M}'} \restr{\manifold{M}' + h_*}{}^{-1} \brk{x'} - h_*}
 = u \brk{\Pi_{\manifold{M}'} \brk{\Pi_{\manifold{M}'} \restr{\manifold{M}' + h_*}{}^{-1}\brk{x'}}}
 = u \brk{x'},
 \end{split}
\end{equation*}
so that the map \(U\) has all the required properties.
\end{proof}

\section{Global extension}
\label{section_global}
\resetconstant

We finally consider the global Sobolev extension of traces to a manifold \(\manifold{M}\) from its boundary \(\partial \manifold{M}\) of  \cref{theorem_extension_global,theorem_characterization_global}. 

\subsection{Global extension of a nice mapping}
Our first step towards the global Sobolev extension of traces is to extend mappings that already possess the required Sobolev regularity \(\smash{W^{1, p}}\) on the boundary and on high-dimensional triangulations of the boundary, and also have a continuous extension on a suitable skeleton.

\begin{proposition}
\label{proposition_Sobolev_homotopy_extension_triangulation}
Let \(\manifold{M}\) be an \(m\)-dimensional triangulated Riemannian manifold with boundary \(\partial \manifold{M}\), let \(p \in \intvo{1}{m}\) and let \(\ell \in \set{1, \dotsc, m}\).
If \(p - 1 < \ell < p + 1\), then there exists a constant \(C \in \intvo{0}{\infty}\) and a convex function \(\Theta \colon \intvr{0}{\infty} \to \intvr{0}{\infty}\) such that \(\Theta \brk{0} = 0\) and such that if  \(u \in \smash{\smash{\dot{W}}^{1, p}} \brk{\partial \manifold{M}, \manifold{N}}\), if for every \(j \in \set{\ell - 1, \dotsc, m - 1}\),
\begin{equation*}u^j \defeq \tr_{\partial \manifold{M}^{j}} u^{j + 1} = u \restr{\partial \manifold{M}^{j}} \in \smash{\dot{W}}^{1, p} \brk{\partial \manifold{M}^{j}, \manifold{N}}, 
\end{equation*}
with \(u^m \defeq u\)
and if there exists some mapping \(W^\ell \in C \brk{\manifold{M}^{\ell},\manifold{N}}\) such that \(W^\ell \restr{\partial \manifold{M}^{\ell - 1}} = u^{\ell - 1}\), then there exists a mapping \(U \in \smash{\smash{\dot{W}}^{1, p}} \brk{\manifold{M}, \manifold{N}}\) such that 
\(\tr_{\partial \manifold{M}} U = u\) and 
\begin{equation}
\label{eq_ubei7leruth5eechaep6eaSh}
 \int_{\manifold{M}} \abs{\Deriv U}^p
 \le C \sum_{i = \ell}^{m - 1} \int_{\partial \manifold{M}^i} \abs{\Deriv u^i}^p
 + 
 \Theta 
 \brk[\bigg]{\int_{\partial \manifold{M}^{\ell - 1}} \abs{\Deriv u^{\ell - 1}}^p + \smashoperator{\iint_{\partial \manifold{M}^{\ell - 1}\times \partial \manifold{M}^{\ell - 1}}} \abs{u \brk{x} - u \brk{y}}^p\dif x \dif y}.
\end{equation}
\end{proposition}

Our first step towards the proof of \cref{proposition_Sobolev_homotopy_extension_triangulation}
is exhibiting a controlled approximate extension to \(\manifold{M}^\ell\).
\begin{lemma}
\label{proposition_skeleton_close_extension}
Let \(\manifold{M}\) be an \(m\)-dimensional triangulated Riemannian manifold with boundary \(\partial \manifold{M}\), let \(\ell \in \set{1, \dotsc, m}\).
If \(p > \ell - 1\), then there exists a constant \(C \in \intvo{0}{\infty}\) and, for every \(\delta \in \intvo{0}{\infty}\), there is a convex function \(\Theta \colon \intvr{0}{\infty} \to \intvr{0}{\infty}\) such that \(\Theta = 0\) on a neighbourhood of \(0\) and such that
if \(W^{\ell} \in C \brk{\manifold{M}^{\ell}, \manifold{N}}\) and 
\begin{equation*}
u^{\ell - 1} \defeq W^\ell\restr{\partial \manifold{M}^{\ell -1}}
\in \smash{\dot{W}}^{1, p}\brk{\partial \manifold{M}^{\ell -1}, \manifold{N}},
\end{equation*}
then there exists a mapping \(V^{\ell} \in \smash{\dot{W}}^{1, p} \brk{\manifold{M}{}^{\ell}, \manifold{N}}\)
such that 
\begin{equation*}v^{\ell - 1} \defeq \tr_{\partial \manifold{M}^{\ell - 1}} V^{\ell} \in \smash{\dot{W}}^{1, p} \brk{\partial \manifold{M}^{\ell - 1}, \manifold{N}},
\end{equation*}
\begin{equation}
\label{eq_aif2bieheiX8Iaph8cae6nee}
d\brk{v^{\ell - 1}, u^{\ell - 1}}^p \le \min \set[\bigg]{\delta^p,
C\brk[\bigg]{\int_{\partial \manifold{M}^{\ell - 1}} \abs{\Deriv u^{\ell - 1}}^p + \smashoperator{\iint_{\partial \manifold{M}^{\ell - 1}\times \partial \manifold{M}^{\ell - 1}}} \abs{u \brk{x} - u \brk{y}}^p\dif x \dif y}}
\end{equation}
almost everywhere on \(\partial \manifold{M}^{\ell - 1}\),
\begin{equation}
\label{eq_joofeengie7cahCielong9ka}
  \int_{\partial \manifold{M}^{\ell - 1}} \abs{\Deriv v^{\ell - 1}}^p
  \le C \brk[\bigg]{\int_{\partial \manifold{M}^{\ell - 1}} \abs{\Deriv u^{\ell - 1}}^p+\smashoperator{\iint_{\partial \manifold{M}^{\ell - 1}\times \partial \manifold{M}^{\ell - 1}}} \abs{u \brk{x} - u \brk{y}}^p\dif x \dif y},
\end{equation}
and 
\begin{equation}
\label{eq_dah9Yeeteepaequiteijeequ}
 \int_{\manifold{M}^\ell} \abs{\Deriv V^\ell}^p 
 \le \Theta \brk[\bigg]{\int_{\partial \manifold{M}^{\ell - 1}} \abs{\Deriv u^{\ell - 1}}^p + \smashoperator{\iint_{\partial \manifold{M}^{\ell - 1}\times \partial \manifold{M}^{\ell - 1}}} \abs{u \brk{x} - u \brk{y}}^p\dif x \dif y}.
\end{equation}
\end{lemma}

The convex function \(\Theta \colon \intvr{0}{\infty} \to \intvr{0}{\infty}\) appearing in \cref{proposition_Sobolev_homotopy_extension_triangulation} and in Theorems~\ref{theorem_extension_global}, \ref{theorem_characterization_global} and \ref{theorem_estimate_critical_global} comes from \cref{proposition_skeleton_close_extension} only.

The essential ingredient of the proof of \cref{proposition_skeleton_close_extension} is a compactness argument (see for example \citelist{\cite{Petrache_VanSchaftingen_2017}*{Th.\thinspace{}4}\cite{Petrache_Riviere_2014}*{Prop.\thinspace{}2.8}}).

\begin{proof}[Proof of \cref{proposition_skeleton_close_extension}]
Since \(p >  \ell - 1\), in view of the Morrey-Sobolev embedding, we have for almost every \(x, y \in \partial \manifold{M}^{\ell - 1}\)
\begin{equation}
\label{eq_Keu9aigethuD7aigh9phieze}
 d \brk{u^{\ell- 1} (x), u^{\ell - 1} (y)}^p \le
 \Cl{cst_Muphaidaisathei3Jai9zah4}
 \brk[\bigg]{\int_{\partial \manifold{M}^{\ell - 1}} \abs{\Deriv u^{\ell - 1}}^p
 + \smashoperator{\iint_{\partial \manifold{M}^{\ell - 1}\times \partial \manifold{M}^{\ell - 1}}} \abs{u \brk{x} - u \brk{y}}^p\dif x \dif y}.
\end{equation}
We set
\begin{equation}
\label{eq_chae7Lee5Kah0zeeWah6me8i}
\eta  \defeq \delta^p/\Cr{cst_Muphaidaisathei3Jai9zah4}
\end{equation}

We first treat the case where
\begin{equation*}
\int_{\partial \manifold{M}^{\ell - 1}} \abs{\Deriv u^{\ell - 1}}^p  + \smashoperator{\iint_{\partial \manifold{M}^{\ell - 1}\times \partial \manifold{M}^{\ell - 1}}} \abs{u \brk{x} - u \brk{y}}^p\dif x \dif y \le \eta.
\end{equation*}
It follows from \eqref{eq_Keu9aigethuD7aigh9phieze} and \eqref{eq_chae7Lee5Kah0zeeWah6me8i} that
we have for almost every  \(x, y \in \partial \manifold{M}^{\ell - 1}\)
\begin{equation*}
 d \brk{u^{\ell- 1} (x), u^{\ell - 1} (y)}^p \le \delta^p.
\end{equation*}
We take then \(V^\ell \defeq u^{\ell - 1} \brk{a}\) for a suitable point \(a \in \partial \manifold{M}^{\ell - 1}\).

Next, we assume that  \(k \in \Nset\) and that
\begin{equation}
\label{eq_ahd9xeth1ie9queThii0uphe}
2^{k} \eta < \int_{\partial \manifold{M}^{\ell - 1}} \abs{\Deriv u^{\ell - 1}}^p  + \smashoperator{\iint_{\partial \manifold{M}^{\ell - 1}\times \partial \manifold{M}^{\ell - 1}}} \abs{u \brk{x} - u \brk{y}}^p\dif x \dif y \le 2^{k + 1} \eta;
\end{equation}
since \(p > \ell - 1\), 
thanks to the Morrey-Sobolev embedding again and the Arzelà-Ascoli compactness criterion, there exist \(I_k \in \Nset\) and mappings \(V_{k, 1}, \dotsc, V_{k, I_k} \in C^1 \brk{\manifold{M}^\ell, \manifold{N}}\) such that if the mapping \(u^{\ell - 1}\) satisfies the condition \eqref{eq_ahd9xeth1ie9queThii0uphe} and is the restriction to \(\partial \manifold{M}^{\ell - 1}\) of a continuous function \(W^{\ell} \colon \manifold{M}^\ell \to \manifold{N}\), then
\begin{equation}
\label{eq_yieveePutaShohz2aej3aiKo}
 d \brk{u^{\ell - 1}, V_{k, i} \restr{\partial \manifold{M}^{\ell- 1}}} \le \delta
\end{equation}
almost everywhere in \(\partial \manifold{M}^{\ell - 1}\) for some \(i \in \set{1, \dotsc, I_k}\) (see a similar argument in the proof of \cref{lemma_cube_extension_lipschitz}).
Moreover we can assume that for every \(i \in \set{1, \dotsc, I_k}\), 
\begin{equation}
\label{eq_beiz8aingeeng4ieGezoiPoo}
  \int_{\partial \manifold{M}^{\ell - 1}} \abs{\Deriv V_{k, i}\restr{\partial \manifold{M}^{\ell - 1}}}^p
  \le 2^{k+ 2} \eta.
\end{equation}

We take the function \(\Theta \colon \intvr{0}{\infty} \to \intvr{0}{\infty}\) to be convex and to satisfy the conditions that \(\Theta = 0\) on \(\intvc{0}{\eta}\) and that for every \(k \in \Nset\) and every \(i \in  \set{1, \dotsc, I_k}\)
\begin{equation}
\label{eq_iwaeYo1IyohVohgh2ohJievi}
 \Theta \brk{2^{k} \eta} \ge \int_{\manifold{M}^{\ell}} \abs{\Deriv V_{k,i}}^p.
\end{equation}
In order to reach the conclusion, we take \(V^{\ell} \defeq V_{k, i}\) for suitable \(k \in \Nset\) and \(i \in \set{1, \dotsc, I_k}\).
The estimate \eqref{eq_aif2bieheiX8Iaph8cae6nee} follows from \eqref{eq_ahd9xeth1ie9queThii0uphe} and  \eqref{eq_yieveePutaShohz2aej3aiKo},
the inequality \eqref{eq_joofeengie7cahCielong9ka} is a consequence of \eqref{eq_beiz8aingeeng4ieGezoiPoo} and \eqref{eq_ahd9xeth1ie9queThii0uphe},
while \eqref{eq_dah9Yeeteepaequiteijeequ} follows from \eqref{eq_ahd9xeth1ie9queThii0uphe} and \eqref{eq_iwaeYo1IyohVohgh2ohJievi}.
\end{proof}

\medbreak

We now prove \cref{proposition_Sobolev_homotopy_extension_triangulation}.

\begin{proof}[Proof of \cref{proposition_Sobolev_homotopy_extension_triangulation}]
We take the mapping \(V^\ell \in \smash{\dot{W}}^{1, p}\brk{\manifold{M}^\ell, \manifold{N}}\) and the mapping \(v^{\ell - 1} = \tr_{\partial \manifold{M}^{\ell - 1}} V^\ell \in \smash{\dot{W}}^{1, p}\brk{\partial \manifold{M}^{\ell - 1}, \manifold{N}}\)
given by \cref{proposition_skeleton_close_extension} with \(\delta \in \intvo{0}{\infty}\) chosen small enough so that the map \(Y^{\ell} \colon \partial \manifold{M}^{\sqcup, \ell} \to \manifold{N}\), with \(\partial \manifold{M}^{\sqcup, \ell}\) defined in \eqref{eq_Oth9sieg6aiyahVahChe1gok}, given for
\(x = \brk{x', x_m} \in \partial \manifold{M}^{\sqcup, \ell}\) by 
\begin{equation}
\label{eq_Oow2aepah2Pieghu9lauv0ik}
 Y^{\ell} \brk{x', x_m} 
 \defeq \begin{cases}
 \Pi_{\manifold{N}} \brk{\brk{1 - x_m} u^{\ell - 1} \brk{x'} + x_m v^{\ell - 1} \brk{x'}} & \text{if \(x' \in \partial \manifold{M}^{\ell - 1}\)},\\
 u^{\ell} (x) & \text{otherwise}
 \end{cases}
\end{equation}
is well defined
 and satisfies, in view of \eqref{eq_aif2bieheiX8Iaph8cae6nee} and \eqref{eq_joofeengie7cahCielong9ka} the estimate
\begin{equation}
\label{eq_iNgee0xool1cheeyaewohje7}
\begin{split}
 \int_{\partial \manifold{M}^{\sqcup, \ell}}
 \abs{\Deriv Y^\ell}^p 
 &\le \int_{\partial \manifold{M}^{\ell}}\abs{\Deriv u^{\ell}}^p
 + \C  \int_{\partial \manifold{M}^{\ell - 1}}\brk[\big]{\abs{\Deriv u^{\ell - 1}}^p + d\brk{u^{\ell - 1}, v^{\ell - 1}}^p + \abs{\Deriv v^{\ell - 1}}^p}\\
 &\qquad + \C \smashoperator{\iint_{\partial \manifold{M}^{\ell - 1}\times \partial \manifold{M}^{\ell - 1}}} \abs{u \brk{x} - u \brk{y}}^p\dif x \dif y\\
 &\le \C \brk[\bigg]{\sum_{i = \ell -1}^{\ell}
 \int_{\partial \manifold{M}^{i}}\abs{\Deriv u^{i}}^p + \smashoperator{\iint_{\partial \manifold{M}^{\ell - 1}\times \partial \manifold{M}^{\ell - 1}}} \abs{u \brk{x} - u \brk{y}}^p\dif x \dif y }.
\end{split}
\end{equation}

We define next inductively the mappings 
\(Y^{j}\colon \partial \manifold{M}^{\sqcup, j} \to \manifold{N}\) for \(j \in \set{\ell + 1, \dotsc, m}\)
by imposing \(\smash{Y^j \restr{\partial \manifold{M}^{j - 1} \times \set{0}} = u^{j} \restr{\partial \manifold{M}^{j - 1}}}\) 
and  defining \(Y^{j}\) on \(\sigma \times \intvc{0}{1}\) for each \(\brk{j - 1}\)-dimensional face \(\sigma\) of \(\partial \manifold{M}^{j - 1}\) by a retraction of \(\sigma \times \intvc{0}{1}\) to 
\(\partial_{\sqcup} \brk{\sigma \times \intvc{0}{1}} = \partial \sigma \times \intvc{0}{1} \cup \sigma \times \set{0} \subseteq \partial \manifold{M}^{\sqcup, j - 1}\) as in \cref{lemma_descending}. 
One gets in such a way a mapping \(Y^{j} \in \smash{\dot{W}}^{1, p} \brk{\partial \manifold{M}^{\sqcup, j}, \manifold{N}}\) such that \(\tr_{\partial \manifold{M}^{\sqcup, j - 1}} Y^j = Y^{j - 1}\), \(\tr_{\partial \manifold{M}^{j - 1} \times \set{0}} Y^j = u^{j - 1}\) and, by an induction argument based on \eqref{eq_iNgee0xool1cheeyaewohje7}
\begin{equation}
\label{eq_zied3lungich2ohvoM1quai3}
\begin{split}
 \int_{\partial \manifold{M}^{\sqcup, j}}
 \abs{\Deriv Y^j}^p 
 &\le \int_{\partial \manifold{M}^j} \abs{\Deriv u^j}^p  + \C  \int_{\partial \manifold{M}^{\sqcup, j - 1}}
 \abs{\Deriv Y^{j - 1}}^p \\
 &\le \C \brk[\bigg]{\sum_{i = \ell - 1}^{j} \int_{\partial \manifold{M}^{i}}\abs{\Deriv u^{i}}^p + \smashoperator{\iint_{\partial \manifold{M}^{\ell - 1}\times \partial \manifold{M}^{\ell - 1}}} \abs{u \brk{x} - u \brk{y}}^p\dif x \dif y}.
\end{split}
\end{equation}

Moreover, it also follows from the construction and from \eqref{eq_zied3lungich2ohvoM1quai3} that we also have for every \(j \in \set{\ell, \dotsc, m - 1}\), 
\begin{equation*}
v^{j} 
\defeq 
Y^{j + 1} \restr{\partial \manifold{M}^{j}\times \set{1}} 
= 
\tr_{\partial \manifold{M}^{j}\times \set{1}} Y^{j + 1} 
\in \smash{\dot{W}^{1, p}} \brk{\partial \manifold{M}^{j}, \manifold{N}}, 
\end{equation*}
together with 
\begin{equation}
\label{eq_quieg4Dohn0jahh0aph1ofae}
\begin{split}
 \int_{\partial \manifold{M}^{j}}
 \abs{\Deriv v^{j}}^p 
 &\le \C\int_{\partial \manifold{M}^{\sqcup, j}}
 \abs{\Deriv Y^{j}}^p\\
 &\le \C \brk[\bigg]{\sum_{i = \ell - 1}^{j} \int_{\partial \manifold{M}^{i}}\abs{\Deriv u^{i}}^p +  \smashoperator{\iint_{\partial \manifold{M}^{\ell - 1}\times \partial \manifold{M}^{\ell - 1}}} \abs{u \brk{x} - u \brk{y}}^p\dif x \dif y},
\end{split}
\end{equation}
and \(\tr_{\partial \manifold{M}^{j - 2}} v^{j - 1} = v^{j - 2}\).

\medbreak

We proceed now to define a mapping \(\widebar{V}{}^m \in \smash{\dot{W}}^{1, p} \brk{\manifold{M}, \manifold{N}}\) such that \(\tr_{\partial \manifold{M}} \widebar{V}{}^m = v^{m - 1}\).
We first define \(\widebar{V}{}^\ell\) by the condition that \(\widebar{V}{}^\ell\restr{\partial \manifold{M}^\ell} = v^\ell\) and 
\(\widebar{V}{}^\ell\restr{\manifold{M}^\ell \setminus \partial \manifold{M}^{\ell}} = V^\ell\). It satisfies the estimate by \eqref{eq_dah9Yeeteepaequiteijeequ}, and by \eqref{eq_quieg4Dohn0jahh0aph1ofae}
\begin{equation}
\label{eq_gie9zailiedoo8ooG8Aethix}
\begin{split}
  \int_{\manifold{M}^{\ell}}
 \abs{\Deriv \widebar{V}{}^\ell}^p
 &= \int_{\manifold{M}^{\ell}}
 \abs{\Deriv V^\ell}^p
 + \int_{\partial \manifold{M}^\ell}
 \abs{\Deriv v^\ell}^p\\
 &\le  \Theta \brk[\bigg]{\int_{\partial \manifold{M}^{\ell - 1}} \abs{\Deriv u^{\ell - 1}}^p +  \smashoperator{\iint_{\partial \manifold{M}^{\ell - 1}\times \partial \manifold{M}^{\ell - 1}}} \abs{u \brk{x} - u \brk{y}}^p\dif x \dif y }\\
 &\qquad \qquad
 + \sum_{i = \ell - 1}^{\ell} \int_{\partial \manifold{M}^{i}}\abs{\Deriv u^{i}}^p +  \smashoperator{\iint_{\partial \manifold{M}^{\ell - 1}\times \partial \manifold{M}^{\ell - 1}}} \abs{u \brk{x} - u \brk{y}}^p\dif x \dif y.
\end{split}
\end{equation}

Next we construct inductively the mapping \(\widebar{V}{}^j \colon \manifold{M}^j \to \manifold{N}\) for every \(j \in \set{\ell + 1, \dotsc, m}\) as follows:
we take \(\widebar{V}{}^j\restr{\partial \manifold{M}^{j}} = v^{j}\) (with the understanding that when \(j=m\), we have \(\partial \manifold{M}^{j} = \emptyset\) and we do thus nothing); for every \(j\)-dimensional face \(\sigma\) of \(\manifold{M}^j \setminus \partial \manifold{M}^j\) we  use a homogeneous extension similar to \cref{lemma_homogeneous_extension} to define \(\widebar{V}{}^j\) on \(\sigma\) since \(p < \ell + 1 \le j\).
By induction, we have \(\widebar{V}{}^j \in \smash{\smash{\dot{W}}^{1, p}} \brk{\manifold{M}^{j}, \manifold{N}}\), \(\tr_{\partial \manifold{M}^{j - 1}} \widebar{V}{}^j = v^{j - 1}\) and
\begin{equation}
\label{eq_Ko8johQueeworees6Dijaisi}
\begin{split}
 \int_{\manifold{M}^{j}}
 \abs{\Deriv \widebar{V}{}^j}^p
 &\le 
 \C
 \brk[\bigg]{\int_{\manifold{M}^{j - 1}} \abs{\Deriv \widebar{V}{}^{j - 1}}^p
  + \int_{\partial \manifold{M}^{j}} \abs{\Deriv v^{j}}^p}\\
  &\le \C
 \brk[\bigg]{\int_{\manifold{M}^{\ell}} \abs{\Deriv \widebar{V}{}^{\ell}}^p
  + \sum_{i = \ell + 1}^{j} \int_{\partial \manifold{M}^i} \abs{\Deriv v^i}^p}\\
  &
   \le \C \brk[\bigg]{ \sum_{i = \ell - 1}^{j} \int_{\partial \manifold{M}^i} \abs{\Deriv u^i}^p + \smashoperator{\iint_{\partial \manifold{M}^{\ell - 1}\times \partial \manifold{M}^{\ell - 1}}} \abs{u \brk{x} - u \brk{y}}^p\dif x \dif y\\
   &\qquad \qquad
 + 
 \Theta 
 \brk[\bigg]{\int_{\partial \manifold{M}^{\ell - 1}} \abs{\Deriv u^{\ell - 1}}^p + \smashoperator{\iint_{\partial \manifold{M}^{\ell - 1}\times \partial \manifold{M}^{\ell - 1}}} \abs{u \brk{x} - u \brk{y}}^p\dif x \dif y}},
  \end{split}
\end{equation}
in view of \eqref{eq_gie9zailiedoo8ooG8Aethix}.

In order to conclude, we note that by the collar neighbourhood theorem the manifolds \(\manifold{M}\) and \(\partial \manifold{M} \times \intvc{0}{1}\) can be glued in such a way that \(\partial \manifold{M}\) and \(\partial \manifold{M} \times \set{1}\) are identified and that the resulting manifold with boundary is diffeomorphic to \(\manifold{M}\).
Since by construction \(\smash{\tr_{\partial \manifold{M} \times \set{1}}Y^m} = v^{m - 1} =  \tr_{\partial \manifold{M}} \widebar{V}{}^m\), the maps \(Y^m\) and \(\widebar{V}{}^m\) can be glued to define the map \(U\) that has the required properties and satisfies as a consequence of \eqref{eq_zied3lungich2ohvoM1quai3} and \eqref{eq_Ko8johQueeworees6Dijaisi} the announced estimate \eqref{eq_ubei7leruth5eechaep6eaSh} with an appropriate constant \(C \in \intvo{0}{\infty}\) and a suitable convex function \(\Theta \colon \intvr{0}{\infty} \to \intvr{0}{\infty}\).
\end{proof}

\subsection{Good restrictions to skeletons}

In order to apply \cref{proposition_Sobolev_homotopy_extension_triangulation},
we will show that the collar neighbourhood extensions given by \cref{theorem_extension_collar,theorem_characterization_collar} satisfy the required condition up to a diffeomorphism.

\begin{proposition}
\label{proposition_collar_generic_restriction_triangulation_trace}
Given \(U \in \smash{\dot{W}}^{1, p} \brk{\manifold{M}'\times \intvo{0}{1}, \manifold{N}}\),
define for every \(h = \brk{h', h_{\mu + 1}} \in \Bset^{\mu}_\delta \times \intvo{\frac{1}{2}}{1}\) the mapping
\(U_h \colon \manifold{M}' \times \intvr{0}{1} \to \manifold{N}\) for every \(\brk{x', x_m} \in \manifold{M}' \times \intvr{0}{1}\) by 
\begin{equation*}
 U_h \brk{x} 
 \defeq U
 \brk{\Pi_{\manifold{M}'} \brk{x'+ h'}, x_m h_{\mu + 1}}
 ,
\end{equation*}
for every \(j \in \set{1, \dotsc, m}\),
\[
 U_h^j \defeq  U_h \restr{\manifold{M}'{}^{j - 1} \times \intvo{0}{1}}
\]
 and for every \(j \in \set{0, \dotsc, m - 1}\)
\[
  v_h^{j} \defeq  U_h \restr{\manifold{M}'{}^{j} \times \set{1}}.
\]
Then, the following assertions hold
\begin{enumerate}[label=(\roman*)]
\item for every \(j \in \set{1, \dotsc, m}\) and for almost every \(h \in \Bset^{\mu}_{\delta} \times \intvo{\frac{1}{2}}{1}\), we have 
\(U_{h}^{j}\in \smash{\dot{W}}^{1, p} \brk{\manifold{M}'{}^{j - 1} \times \intvo{0}{1}, \manifold{N}}\) and 
\begin{equation*}
 \int_{\Bset^{\mu}_\delta \times \intvo{\frac{1}{2}}{1}} \brk[\bigg]{\int_{\manifold{M}'{}^{j - 1} \times \intvo{0}{1}}
 \abs{\Deriv U_h^j}^p} \dif h
 \le C
 \int_{\manifold{M}'{} \times \intvo{0}{1}} \abs{\Deriv U}^p,
\end{equation*}
\item for every \(j \in \set{1, \dotsc, m - 1}\) and for almost every \(h \in \Bset^{\mu}_{\delta} \times \intvo{\frac{1}{2}}{1}\), we have 
\(\smash{U_{h}^{j}} = \smash{\tr_{\manifold{M}'^{j - 1} \times \intvo{0}{1}} U_h^{j + 1}}\),
\item for every \(j \in \set{0, \dotsc, m - 1}\) and for almost every \(h \in \Bset^{\mu}_{\delta} \times \intvo{\frac{1}{2}}{1}\), we have  
\(
v_{h}^{j}  \in \smash{\dot{W}^{1, p}} \brk{\manifold{M}'{}^j, \manifold{N}}\) and 
\begin{equation*}
 \int_{\Bset^{\mu}_\delta \times \intvo{\frac{1}{2}}{1}} \brk[\bigg]{\int_{\manifold{M}'{}^j}
 \abs{\Deriv v_h^j}^p} \dif h
 \le C
 \int_{\manifold{M}'{} \times \intvo{0}{1}} \abs{\Deriv U}^p,
\end{equation*}
\item for every \(j \in \set{0, \dotsc, m - 2}\), we have
 \begin{equation*}
   \int_{\Bset^{\mu}_\delta \times \intvo{\frac{1}{2}}{1}} \brk[\bigg]{\smashoperator{\iint_{\manifold{M}'{}^j\times \manifold{M}'{}^j}}
 \abs{v_h^j \brk{x} - v_h^j \brk{y}}^p}\dif x \dif y \dif h
 \le C
 \brk[\bigg]{\smashoperator{\iint_{\brk{\manifold{M}'\times \intvo{0}{1}}\times \brk{\manifold{M}'\times \intvo{0}{1}}}}
 \abs{U \brk{x} - U \brk{y}}^p}\dif x \dif y,
 \end{equation*}
\item for every \(j \in \set{0, \dotsc, m - 2}\) and for almost every \(h \in \Bset^{\mu}_{\delta} \times \intvo{\frac{1}{2}}{1}\), we have  
\(v_{h}^{j}  = \smash{ \tr_{\manifold{M}'^{j}} v_h^{j + 1}}\),
\item for every \(j \in \set{0, \dotsc, m - 1}\) and for almost every \(h \in \Bset^{\mu}_{\delta} \times \smash{\intvo{\frac{1}{2}}{1}}\), we have 
\(v_h^j = \smash{\tr_{\manifold{M}'{}^{j} \times \set{1}} U_h^{j + 1}}\).
\end{enumerate}
If moreover \(u \defeq U \restr{\manifold{M}' \times \set{0}} = \tr_{\manifold{M}' \times \set{0}} U \in \smash{\dot{W}}^{1, p} \brk{\manifold{M}', \manifold{N}}\),
defining for every \(
h = \brk{h', h_{\mu + 1}} \in \Bset^{\mu}_{\delta} \times \intvo{\frac{1}{2}}{1}\) and \(x' \in \manifold{M}'\),
\begin{equation*}
 u_h \brk{x'} \defeq  u
 \brk{\Pi_{\manifold{M}'} \brk{x'+ h'}}
\end{equation*}
and for every \(j \in \set{0, \dotsc, m - 1}\), \(u_h^j \defeq u_h \restr{\manifold{M}'{}^j}\), the following assertions hold
\begin{enumerate}[label=(\roman*),resume]
 \item for every \(j \in \set{0, \dotsc, m - 2}\) and for almost every \(h \in \Bset^{\mu}_{\delta} \times \intvo{\frac{1}{2}}{1}\), we have 
 \(u^j_h  \in \smash{\dot{W}}^{1, p} \brk{\manifold{M}'{}^{j}, \manifold{N}}\)
 and 
 \begin{equation*}
   \int_{\Bset^{\mu}_\delta \times \intvo{\frac{1}{2}}{1}} \brk[\bigg]{\int_{\manifold{M}'{}^j}
 \abs{\Deriv u_h^j}^p} \dif h
 \le C
 \int_{\manifold{M}'} \abs{\Deriv u}^p,
 \end{equation*}
\item 
for every \(j \in \set{0, \dotsc, m - 2}\) and for almost every \(h \in \Bset^{\mu}_{\delta} \times \intvo{\frac{1}{2}}{1}\), we have 
 \(u^j_h = \tr_{\manifold{M}'{}^{j}} u^{j + 1}_h\),
\item for every \(j \in \set{0, \dotsc, m - 1}\) and for almost every \(h \in \Bset^{\mu}_{\delta} \times \intvo{\frac{1}{2}}{1}\), we have \(u^j_h = \tr_{\manifold{M}'{}^{j} \times \set{0}} \smash{U^{j + 1}_h}\).
\end{enumerate}
\end{proposition}
\begin{proof}
We assume without loss of generality \(\manifold{M}'\) is triangulated and isometrically embedded into \(\Rset^\mu\).
Taking a sequence \(\brk{U_n}_{n \in \Nset}\) in \(C^{\infty}\brk{\manifold{M}' \times \intvr{0}{1}, \Rset^{\nu}}\) such that 
\begin{equation}
\label{eq_shong1weDeepheiNee5iedae}
 \sum_{n \in \Nset} \int_{\manifold{M}' \times \intvr{0}{1}} 
 \abs{\Deriv U_{n} - \Deriv U}^p
 \le \int_{\manifold{M}' \times \intvr{0}{1}} \abs{\Deriv U}^p
\end{equation}
and
\begin{equation}
\label{eq_ooroh5iej6oozeij5riQuoh8}
 \sum_{n \in \Nset} \int_{\manifold{M}' \times \intvr{0}{1}} 
 \abs{U_{n} - U}^p
 < \infty ,
\end{equation}
and defining the functions \(G \colon \manifold{M}' \times \intvr{0}{1} \to \intvr{0}{\infty}\) and \(H \colon \manifold{M}' \times \intvr{0}{1} \to \intvr{0}{\infty}\) by
\begin{align*}
  G &\defeq \abs{\Deriv U}^p + \sum_{n \in \Nset} \abs{\Deriv U_{n} - \Deriv U}^p&
  &\text{ and }&
  H&\defeq \sum_{n \in \Nset} \abs{U_{n} - U}^p,
\end{align*}
we have by \eqref{eq_shong1weDeepheiNee5iedae} and \eqref{eq_ooroh5iej6oozeij5riQuoh8}
\begin{align}
\label{eq_Eb8ahRaenahngai9Ohngoor1}
 \int_{\manifold{M}' \times \intvr{0}{1}} G
 &\le 2 \int_{\manifold{M}' \times \intvr{0}{1}} \abs{\Deriv U}^p&
 &\text{ and }&
 \int_{\manifold{M}' \times \intvr{0}{1}} H
 &< \infty.
\end{align}
Defining also for  every \(h = \brk{h', h_{\mu + 1}} \in \Bset^{\mu}_\delta \times \intvo{\frac{1}{2}}{1}\) the functions \(G_h \colon \manifold{M}' \times \intvr{0}{1} \to \intvr{0}{\infty}\) and 
\(H_h \colon \manifold{M}' \times \intvr{0}{1} \to \intvr{0}{\infty} \)
for each \(\brk{x', x_m} \in \manifold{M}' \times \intvr{0}{1}\) by 
\begin{align*}
 G_h \brk{x} 
 &\defeq 
 G
 \brk{\Pi_{\manifold{M}'} \brk{x'+ h'}, x_m h_{\mu + 1}}&
 &\text{ and }&
 H_h \brk{x} 
 &\defeq 
 H
 \brk{\Pi_{\manifold{M}'} \brk{x'+ h'}, x_m h_{\mu + 1}}, 
\end{align*}
we have by Fubini's theorem and by a transversality argument 
\begin{equation}
\label{eq_ethee4dei7zeeS5cooSh0yah}
 \int_{\Bset^{\mu}_\delta \times \intvo{\frac{1}{2}}{1}}
 \brk[\bigg]{
 \sum_{j = 0}^{m - 1}  \int_{\manifold{M}'{}^j \times \set{1}} G_h
 +  \sum_{j = 1}^{m - 1} \int_{\manifold{M}'^{j - 1} \times \intvo{0}{1}} G_h}
 \le \C \int_{\manifold{M}' \times \intvr{0}{1}} G;
\end{equation}
and 
\begin{equation}
\label{eq_uteusahzo3eid7quiej2aiNg}
\int_{\Bset^{\mu}_\delta \times \intvo{\frac{1}{2}}{1}}
 \brk[\bigg]{
 \sum_{j = 0}^{m - 1}  \int_{\manifold{M}'{}^j \times \set{1}} H_h
 +  \sum_{j = 1}^{m - 1} \int_{\manifold{M}'^{j - 1} \times \intvo{0}{1}} H_h}
 \le \C \int_{\manifold{M}' \times \intvr{0}{1}} H; 
\end{equation}
hence by \eqref{eq_Eb8ahRaenahngai9Ohngoor1},
the integrand of the outermost integral in the left-hand side of \eqref{eq_ethee4dei7zeeS5cooSh0yah} and \eqref{eq_uteusahzo3eid7quiej2aiNg} is finite for almost every \(h \in \Bset^{\mu}_\delta \times \intvo{\frac{1}{2}}{1}\); for any such \(h\), defining 
\begin{equation*}
  U_{n, h} \brk{x} 
 \defeq 
 U_{n}
 \brk{\Pi_{\manifold{M}'} \brk{x'+ h'}, x_m h_{\mu + 1}},
\end{equation*}
the sequences \(\smash{\brk{U_{n, h} \restr{\manifold{M}'{}^{j} \times \set{1}}}_{n \in \Nset}}\) and \(\smash{\brk{U_{n, h} \restr{\manifold{M}'{}^{j - 1} \times \intvo{0}{1}}}_{n \in \Nset}}\) converge in the Sobolev spaces
\(\smash{\dot{W}}^{1, p} \brk{\manifold{M}'{}^{j} \times \set{1}, \Rset^\nu}\) and \(\smash{\dot{W}}^{1, p} \brk{\manifold{M}'{}^{j} \times \intvo{0}{1}, \Rset^\nu}\) to \(v^j_h\) and \(U^j_h\)
when \(j \in \set{0, \dotsc, m - 1}\) and \(j \in \set{1, \dotsc, m1}\) respectively; this allows to show  that the mappings \(\smash{v^j_h} \colon \smash{\manifold{M}'{}^{j}} \times \set{1} \to \manifold{N}\) and \(\smash{U^j_h}\colon \smash{\manifold{M}'{}^{j - 1}} \times \intvo{0}{1} \to \manifold{N}\) have all the required properties and satisfy the announced estimates.

If we assume moreover that \(u \defeq \tr_{\manifold{M}' \times \set{0}} U \in \smash{\dot{W}}^{1, p} \brk{\manifold{M}', \manifold{N}}\), then we proceed as above with an additional approximation of \(u\).
\end{proof}

We can now prove the global extension result.

\begin{proof}[Proof of  \cref{theorem_extension_global}]
This follows from \cref{theorem_extension_collar}, \cref{proposition_Sobolev_homotopy_extension_triangulation} and \cref{proposition_collar_generic_restriction_triangulation_trace}, together with a gluing argument as in the proof of \cref{proposition_skeleton_close_extension}, since the extension \(U_{\mathrm{col}}\) given by \cref{theorem_extension_collar} satisfies the estimate
\[
\begin{split}
 \smashoperator[r]{\iint_{\brk{\partial \manifold{M}\times \intvo{0}{1}} \times \brk{\partial \manifold{M}\times \intvo{0}{1}} }}&
 \abs{U_{\mathrm{col}} \brk{x} - U_{\mathrm{col}} \brk{y}}^p\dif x \dif y\\
 & \le \C\brk[\bigg]{
 \int_{\partial \manifold{M}\times \intvo{0}{1}} \abs{\Deriv U_{\mathrm{col}}}^p
  + \smashoperator[r]{\iint_{\partial \manifold{M} \times \partial \manifold{M}}}
 \abs{u\brk{x} - u \brk{y}}^p\dif x \dif y}.\qedhere
\end{split}
\]
\end{proof}

\subsection{Proof of the global characterisations}

In order to prove the characterisation for the global extension, 
we need to show that the extension condition is satisfied in \cref{proposition_Sobolev_homotopy_extension_triangulation}.
We observe that this is the case up to a generic diffeomorphism for the restriction to \(\manifold{M}\) of a mapping in \(U \in R^{1}_{m -\ell - 1} \brk{\manifold{M}, \manifold{N}}\).

\begin{proposition}
\label{proposition_extension_R1}
If \(\ell \in \set{1, \dotsc, m}\) and if \(U \in R^{1}_{m -\ell - 1} \brk{\manifold{M}, \manifold{N}}\), then for almost every \(h \in \Bset^{\mu}_{\delta}\), there exists \(W_h \in C\brk{\manifold{M}^{\ell}, \manifold{N}}\) such that for every \(x' \in \partial \manifold{M}^{\ell - 1}\),
\(
  U\brk{\Pi_{\partial \manifold{M}} \brk{x' + h }}
  = W_h \brk{x'}
\).
\end{proposition}
\begin{proof}
Without loss of generality \(\manifold{M}\) is triangulated and isometrically embedded into \(\Rset^\mu\).
By a Fubini-type argument and by the transversality of the nearest-point projection \(\Pi_{\partial \manifold{M}}\), the mapping 
\(U\brk{\Pi_{\partial \manifold{M}} \brk{\cdot + h }}\) is continuous for almost every \(h\);
the mapping \(W\) can then be defined by a generic transversality argument.
\end{proof}

We are in position to prove the sufficiency of the approximation condition in \cref{theorem_characterization_global}; 
necessity will follow from \cref{proposition_approximation_global} below.

\begin{proof}[Proof of the sufficiency part of \cref{theorem_characterization_global}]
We first assume that \(u = U\restr{\partial \manifold{M}}\) for \(U\in \smash{R^{1}_{m -\floor{p} - 1}} \brk{\manifold{M}, \manifold{N}}\).
By \cref{theorem_characterization_collar}, we get a map 
\(\smash{\widebar{U}} \in \smash{\dot{W}}^{1, p} \brk{\partial \manifold{M} \times \intvo{0}{1}, \manifold{N}}\) such that \(\smash{\tr_{\partial \manifold{M}\times \set{0}} \widebar{U} = u}\).
Applying \cref{proposition_collar_generic_restriction_triangulation_trace} to \(\smash{\widebar{U}}\) and choosing some \(h_* \in \smash{\Bset^{\mu}_{\delta}}\) such that the conclusions hold with suitable estimates, we have in particular \(\smash{u_{h_*}^{\floor{p - 1}} = \smash{\tr_{\partial \manifold{M}^{\floor{p - 1}} \times \set{0}}} \smash{\widebar{U}_{h_*}^{\floor{p}}}}\) and \(\smash{v_{h_*}^{\floor{p - 1}}} = \smash{\tr_{\partial \manifold{M}^{\floor{p - 1}} \times \set{1}} \smash{\widebar{U}_{h_*}^{\floor{p}}}}\).
Since we also have \(\smash{U_{h_*}^{\floor{p}}} \in \smash{\smash{\dot{W}}^{1, p} \brk{\partial \manifold{M}^{\floor{p - 1}}\times \intvo{0}{1}, \manifold{N}}}\), the maps \(\smash{u_{h_*}}\) and \(\smash{v_{h_*}}\) are homotopic (see \citelist{\cite{Brezis_Nirenberg_1996}*{Cor.\thinspace{}3}\cite{Bethuel_Demengel_1995}*{Th.\thinspace{}2}}), and \(\smash{v_{h_*}^{\floor{p - 1}}}\) has a continuous extension to \(\smash{\manifold{M}^{\floor{p}}}\) thanks to \cref{proposition_extension_R1} and the homotopy extension property.
We proceed then as in the proof of \cref{theorem_extension_global}, applying \cref{proposition_Sobolev_homotopy_extension_triangulation} and performing a suitable gluing.
\end{proof}

The previous argument also yields a proof of \cref{theorem_estimate_critical_global}.

\begin{proof}[Proof of \cref{theorem_estimate_critical_global}]
We proceed as in the proof of \cref{theorem_characterization_global}, relying on \cref{theorem_estimate_critical_collar} instead of \cref{theorem_characterization_collar} and \cref{proposition_skeleton_close_extension} instead of \cref{proposition_Sobolev_homotopy_extension_triangulation}.
\end{proof}

\section{Necessity of the approximation conditions}
\label{section_necessity_approximation}

We prove that traces have strong approximations, beginning  with the case of the half-space.

\begin{proposition}
  \label{proposition_approximation_Rm}
If \(1 < p < m\), then
\[
 \operatorname{tr}_{\Rset^{m - 1}} \brk{\smash{\smash{\dot{W}}^{1, p} \brk{\Rset^m_+, \manifold{N}}}}
 \subseteq
 \overline{C^1 \brk{\Rset^{m - 1}, \manifold{N}}}^{\smash{\dot{W}}^{1 - 1/p, p} \brk{\Rset^{m - 1}, \manifold{N}}}.
\]
\end{proposition}
\begin{proof}
We assume that \(u = \smash{\tr_{\Rset^{m - 1}} U}\) for some mapping \(U \in \smash{\smash{\dot{W}}^{1, p}} \brk{\Rset^m_+, \manifold{N}}\).
From the classical approximation result for Sobolev mappings \cite{Bethuel_1991}*{Th.\thinspace{}2} (see also \cite{Hang_Lin_2003_II}*{Th.\thinspace{}6.1} and \cref{theorem_R_dense_W1p} below), we get a sequence \(\brk{U_n}_{n \in \Nset}\) in \(\smash{R^1_{m - \floor{p} - 1,c}} \brk{\Rset^m_+, \manifold{N}} \cap \smash{\dot{W}}^{1, p} \brk{\Rset^m_+, \manifold{N}}\), as defined in \eqref{eq_io9catheshejiengiuhaeWoc}, that converges strongly to \(U\) in \(\smash{\dot{W}}^{1, p} \brk{\Rset^m_+, \manifold{N}}\).

By the continuity of the trace operator, the sequence \(\brk{\tr_{\Rset^{m - 1}} U_n}_{n \in \Nset}\) converges to \(u\) in the space \(\smash{\dot{W}^{1 - 1/p, p}} \brk{\Rset^{m - 1}, \manifold{N}}\). 
However, we do not necessarily have that \(\tr_{\Rset^{m - 1}} U_n \in \smash{R^1_{m - \floor{p} - 2}} \brk{\Rset^{m - 1}, \manifold{N}}\).
In order to remedy this, we observe that since the singular set is contained in locally finitely many smooth manifolds, by Sard's lemma and the trace theory, we have for almost every \(x_m \in \intvo{0}{\infty}\),
\begin{equation*}
\tr_{\Rset^{m - 1}\times \set{x_m}} U_n = U_n \restr{\Rset^{m - 1} \times \set{x_m}} \in \smash{R^1_{m - \floor{p} - 2}} \brk{\Rset^{m - 1}, \manifold{N}}.
\end{equation*}
Since moreover \(\tr_{\Rset^{m - 1} \times \set{t}} U_n\) converges  as \(t \to 0\) to \(\smash{\tr_{\Rset^{m - 1} \times \set{0}} U_n}\) strongly in the space \(\smash{\smash{\dot{W}^{1 - 1/p, p}}} \brk{\Rset^{m - 1}, \manifold{N}}\), there exists some sequence \((t_n)_{n \in \Nset}\) in \(\intvo{0}{\infty}\) converging to \(0\) such that the sequence \(\smash{\brk{\tr_{\Rset^{m - 1} \times \set{t_n}} U_n}_{n \in \Nset}}\) converges strongly to the mapping \(\smash{\tr_{\Rset^{m - 1} \times \set{0}} U}\) in \(\smash{\smash{\dot{W}}^{1 - 1/p, p}} \brk{\smash{\Rset^{m - 1}}, \manifold{N}}\) and such that for every \(n \in \Nset\),
\begin{equation*}
\tr_{\Rset^{m - 1}\times \set{t_n}} U_n = U_n \restr{\Rset^{m - 1} \times \set{t_n}} \in \smash{R^1_{m - \floor{p} - 2},c} \brk{\Rset^{m - 1}, \manifold{N}}.
\end{equation*}
We conclude by approximating for each \(n \in \Nset\) the map \(\tr_{\Rset^{m - 1} \times \set{t_n}} U_n\) by smooth maps  strongly in \(\smash{\smash{\dot{W}}^{1 - 1/p, p}} \brk{\smash{\Rset^{m - 1}}, \manifold{N}}\) (cf.\ \citelist{\cite{Hang_Lin_2003_II}*{Cor.\thinspace{}1.6 \& 6.2}};
since \(U_n\) is constant outside a compact set, the approximation on \(\Rset^{m - 1}\) can be reduced to an approximation on \(\Sset^{m - 1}\) thanks to a suitable stereographic projection).
\end{proof}

We continue with the characterisation of the maps that have a collar neighbourhood extension.

\begin{proposition}
Let \(\manifold{M}'\) be an \((m - 1)\)-dimensional compact Riemannian manifold.
If \(1 < p < m\), then
\[
\tr_{\manifold{M}' \times \set{0}}
\brk{ \smash{\smash{\dot{W}}^{1, p} \brk{\manifold{M}' \times \intvr{0}{1}, \manifold{N}}}}
  \subseteq \smash{\overline{R^1_{m - \floor{p} - 2} \brk{\smash{\manifold{M}'}, \manifold{N}}}}^{\smash{\smash{\dot{W}}^{1 - 1/p, p} \brk{\manifold{M}', \manifold{N}}}}.
\]
\end{proposition}
\begin{proof}
Given \(U \in \smash{\smash{\dot{W}}^{1, p} \brk{\manifold{M}' \times \intvr{0}{1}, \manifold{N}}}\), we proceed as in the proof of \cref{proposition_approximation_Rm} to get a sequence \(\brk{U_n}_{n \in \Nset}\) in \(\smash{R^1_{m - \floor{p} - 1}} \brk{\manifold{M}' \times \intvr{0}{1}, \manifold{N}}\) converging strongly to \(U\) in \(\smash{\smash{\dot{W}}^{1, p} \brk{\manifold{M}' \times \intvr{0}{1}, \manifold{N}}}\)
and a sequence \(\brk{t_n}_{n \in \Nset}\) in \(\intvo{0}{1}\) converging to \(0\) such that for every \(n \in \Nset\), 
\begin{equation*}
\tr_{\manifold{M}'\times \set{t_n}} U_n = U_n \restr{\manifold{M}' \times \set{t_n}} \in \smash{R^1_{m - \floor{p} - 2}} \brk{\manifold{M}', \manifold{N}},
\end{equation*}
and \(\brk{\tr_{\manifold{M}'\times \set{t_n}} U_n}_{n \in \Nset}\) is then the required approximating sequence.
\end{proof}

\medbreak

Finally we consider the problem of global extension of \cref{theorem_characterization_global}.

\begin{proposition}
\label{proposition_approximation_global}
Let \(\manifold{M}\) be an \(m\)-dimensional manifold with compact boundary \(\partial \manifold{M}\).
If \(1 < p < m\),
then
\[
\tr_{\partial \manifold{M}}
 \smash{ \smash{\dot{W}}^{1, p} \brk{\manifold{M}, \manifold{N}}}
 \subseteq \smash{\overline{\tr_{\partial \manifold{M}} R^1_{m - \floor{p} - 1} \brk{\manifold{M}, \manifold{N}}}}^{\smash{\dot{W}^{1 - 1/p, p} \brk{\partial \manifold{M}, \manifold{N}}}}.
\]
\end{proposition}

\begin{proof}
Assuming that \(u = \tr_{\partial \manifold{M}} U\) for some mapping \(U \in \smash{\smash{\dot{W}}^{1, p}} \brk{\manifold{M}, \manifold{N}}\),
we define \(\smash{\manifold{M}^{\mathrm{dbl}}}\) to be the double of the manifold \(\manifold{M}\) (to keep a smooth structure, \(\smash{\manifold{M}^{\mathrm{dbl}}}\) is endowed with a Riemannian metric which is equivalent but not necessarily isometric to the original one on the two copies of \(\manifold{M}\) it consists of) and we take the map \(\smash{U^{\mathrm{dbl}}} \in  \smash{\smash{\dot{W}}^{1, p}} \brk{\smash{\manifold{M}^{\mathrm{dbl}}}, \manifold{N}}\) to be the corresponding extension of the map \(U\) to \(\smash{\manifold{M}^{\mathrm{dbl}}}\).
We proceed then through approximation by maps in \(\smash{R^{1}_{m -  \floor{p} - 1}} \brk{\smash{\manifold{M}^{\mathrm{dbl}}}, \manifold{N}}\), Sard's lemma and the continuity of traces to reach the conclusion.
\end{proof}

\appendix

\section{Approximation of Sobolev maps on the Euclidean space}
\label{section_approximation_compact}
In the present appendix, we explain how approximation results for first-order and fractional Sobolev mappings can be obtained for mappings defined on the entire Euclidean space \(\Rset^m\).

\subsection{First-order Sobolev mappings}

As a nonlinear counterpart to the space of compactly supported smooth functions \(C^1_c \brk{\Rset^m, \manifold{N}}\), we define the set 
\begin{equation}
\label{eq_io9catheshejiengiuhaeWoc}
 R^1_{\ell,c} \brk{\Rset^m, \manifold{N}}
 \defeq  \set{u \in R^1_{\ell} \brk{\Rset^m, \manifold{N}}\st u \text{ is constant outside a compact set} },
\end{equation}
where \(R^1_{\ell} \brk{\Rset^m, \manifold{N}}\) has been defined in \eqref{eq_Quaibaibai5iajiel3yah9ie}.
In contrast to the definition of \(C^1_c \brk{\Rset^m, \manifold{N}}\), the constant appearing on the right-hand side of \eqref{eq_io9catheshejiengiuhaeWoc} can differ from one map to another.

We will prove that \(R^1_{\ell,c} \brk{\Rset^m, \manifold{N}}\) is strongly dense in the homogeneous Sobolev space \(\dot{W}^{1, p}\brk{\Rset^m, \manifold{N}}\).

\begin{theorem}
\label{theorem_R_dense_W1p}
Let \(m \in \Nset\setminus \set{0, 1}\) and let \(\manifold{N}\) be a compact Riemannian manifold. 
For every \(p \in \intvc{1}{m}\), the set 
\(R^1_{m - \floor{p} - 1, c} \brk{\Rset^m, \manifold{N}}\)
is strongly dense in \(\dot{W}^{1, p}\brk{\Rset^m, \manifold{N}}\).
\end{theorem}

The proof of \cref{theorem_R_dense_W1p} is based on the proof of the corresponding result on a bounded domain with a smooth boundary \citelist{\cite{Bethuel_1991}*{Th.\thinspace{}2}\cite{Hang_Lin_2003_II}*{Th.\thinspace{}6.1}} and the following approximation result.

\begin{proposition}
\label{proposition_W1p_compact_approximation}
Let \(m \in \Nset \setminus \set{0, 1}\) and let \(\manifold{N}\) be a compact Riemannian manifold.
For every \(p \in \intvc{1}{m}\) and every \(u \in \dot{W}^{1, p}\brk{\Rset^m, \manifold{N}}\), there exists a sequence \(\brk{u_n}_{n \in \Nset}\) of maps in \(\dot{W}^{1, p}\brk{\Rset^m, \manifold{N}}\) such that
\begin{enumerate}[label=(\roman*)]
 \item \(\brk{u_n}_{n \in \Nset}\) converges strongly to \(u\) in \(\dot{W}^{1, p}\brk{\Rset^m, \manifold{N}}\),
 \item for every \(n \in \Nset\), the map \(u_n\) is constant outside a compact set of \(\Rset^m\).
\end{enumerate}
\end{proposition}

\begin{remark}
\label{remark_compact_approximation_supercritical_W1p}
 Although we will not need it in the sequel,  \cref{proposition_W1p_compact_approximation} still holds for \(p>m  \), but for quite different reasons.
Indeed, given \(u \in \dot{W}^{1, p}\brk{\Rset^m, \manifold{N}}\), taking for each \(R \in \intvo{0}{\infty}\) the mapping \(u\colon \Rset^m \to \manifold{N}\) defined for \(x \in \Rset^m\) by 
\begin{equation}
\label{eq_Geis2eeThohsheisieRe1ahc}
 u_R \brk{x} \defeq 
 \begin{cases}
   u  \brk{x}& \text{if \(\abs{x} \le R\)},\\
   u \brk{R^2 x/\abs{x}^2} & \text{if \(\abs{x} > R\)},
 \end{cases}
\end{equation}
one checks that 
\[
 \int_{\Rset^m} \abs{\Deriv u_R}^p = 
 \int_{B_R} \brk[\bigg]{1 + \frac{\abs{x}^{2\brk{p - m}}}{R^{2 \brk{p - m}}} } \abs{\Deriv u \brk{x}}^p \dif x, 
\]
so that 
\(u_R \to u\) in \(\dot{W}^{1, p}\brk{\Rset^m, \manifold{N}}\) as \(R \to \infty\) when \(p > m\). Since 
\[ 
 \lim_{\abs{x} \to \infty} u_R \brk{x} = u \brk{0},
\]
\(u_R\) can then be approximated by maps that are constant outside a compact set.
\end{remark}

\Cref{proposition_W1p_compact_approximation} will be deduced from the following Luckhaus interpolation property.

\begin{lemma}
\label{lemma_Luckhaus_W1p}
Let \(\manifold{M}'\) be a compact \(\brk{m - 1}\)-dimensional Riemannian manifold and let \(\manifold{N}\) be a compact Riemannian manifold.
If \(p \ge 1\), then there exist constants \(\eta \in \intvo{0}{\infty}\) and \(C \in \intvo{0}{\infty}\) such that if \(u_0\) and \(u_1 \in W^{1, p} \brk{\manifold{M}', \manifold{N}}\) satisfy
\begin{equation}
\label{eq_yikoh1ceiPhoNga6ohtoo7da}
\brk[\bigg]{\int_{\manifold{M}'} \abs{\Deriv u_0}^p + \abs{\Deriv u_1}^p + \abs{u_1 - u_0}^p}^{\frac{\ell}{p}}\brk[\bigg]{\int_{\manifold{M}'} \abs{u_1 - u_0}^p}^{1 - \frac{\ell}{p}} \le \eta, 
\end{equation}
with \(\ell = \min \brk{m - 1, \floor{p - 1}}\),
then there exists \(U \in W^{1, p} \brk{\manifold{M}' \times \intvo{0}{1}, \manifold{N}}\) such that 
\begin{enumerate}[label=(\roman*)]
 \item \(\tr_{\manifold{M}'\times \set{0}} U = u_0\),
 \item \(\tr_{\manifold{M}'\times \set{1}} U = u_1\),
 \item one has 
 \[
  \int_{\manifold{M}'\times \intvo{0}{1}} \abs{\Deriv U}^p
  \le C \int_{\manifold{M}'} \abs{\Deriv u_0}^p + \abs{\Deriv u_1}^p + \abs{u_1 - u_0}^p.
 \]
\end{enumerate}
\end{lemma}

\Cref{lemma_Luckhaus_W1p} is a version of the Luckhaus lemma \cite{Luckhaus_1993}*{Lemma 3}; 
the version here has a slightly improved power in \eqref{eq_yikoh1ceiPhoNga6ohtoo7da} and provides an extension on \(\manifold{M}' \times \intvo{0}{1}\) instead of \(\manifold{M}' \times \intvo{0}{L}\).

\begin{proof}[Proof of \cref{lemma_Luckhaus_W1p}]
\resetconstant%
As in \cref{proposition_collar_generic_restriction_triangulation_trace}, taking a suitable triangulation of \(\manifold{M}'{}\), we can assume that for every \(j \in \set{\ell, \dotsc, m}\) and \(i \in \set{0, 1}\),
\[
  \tr_{\manifold{M}'{}^j} u_i = u_i\restr{\manifold{M}'{}^j}
  \in W^{1, p}\brk{\manifold{M}'{}^j, \manifold{N}},
\]
\[
 \int_{\manifold{M}'{}^j} \abs{\Deriv u_i}^p \le \C \int_{\manifold{M}'} \abs{\Deriv u_i}^p
\]
and 
\[
  \int_{\manifold{M}'{}^j} \abs{u_0 - u_1}^p \le \C \int_{\manifold{M}'{}^j} \abs{u_0 - u_1}^p.
\]
Setting \(\ell = \min \brk{\floor{p - 1}, m - 1}\), by the Gagliardo-Nirenberg interpolation inequality (see \cite{Brezis_Mironescu_2018}), we have, if \(\ell < p\), 
\begin{equation}
\label{eq_aiShiev7aithahbuG0naefah}
 \operatorname*{ess\,sup}_{\manifold{M}'{}^\ell} \abs{u_0 - u_1}
 \le \C 
 \brk[\bigg]{\int_{\manifold{M}'} \abs{\Deriv u_0}^p + \abs{\Deriv u_1}^p + \abs{u_1 - u_0}^p}^{\frac{\ell}{p}}\brk[\bigg]{\int_{\manifold{M}'} \abs{u_1 - u_0}^p}^{1 - \frac{\ell}{p}}.
\end{equation}
Letting \(\Pi_{\manifold{N}} \colon \manifold{N} + \smash{\overline{\Bset^\nu_{\delta_{\manifold{N}}}}} \to \manifold{N}\) be the nearest-point projection, it follows from \eqref{eq_aiShiev7aithahbuG0naefah} that if \(\eta > 0\) is small enough, the line segment \(\intvc{u_0 \brk{x}}{u_1 \brk{x}}\) is contained in \(\manifold{N} + B_{\delta_{\manifold{N}}}\) and we can define for every \(\brk{x, t} \in \manifold{M}'{}^{\ell}\times \intvo{0}{1}\),
\[
  U \brk{x, t} = \Pi_{\manifold{N}} \brk{\brk{1- t}u_0 \brk{x} + t u_1 \brk{x}}.
\]
We then extend \(U\) by homogeneous extension to faces of \(\manifold{M}' \times \intvc{0}{1}\) of increasing dimension \(j \in \set{\ell + 2, \dotsc, m}\), since \(\ell + 2 = \floor{p + 1} > p\) as in the proof of \cref{proposition_Sobolev_homotopy_extension_triangulation}.
\end{proof}

We now explain how \cref{proposition_W1p_compact_approximation} can be proved thanks to \cref{lemma_Luckhaus_W1p}.

\begin{proof}%
[Proof of \cref{proposition_W1p_compact_approximation}]%
\resetconstant%
Since \(u \in \dot{W}^{1, p}\brk{\Rset^m, \manifold{N}}\), we have 
\[
 \int_0^\infty \int_{\partial B_R} \abs{\Deriv u}^p \dif R <\infty,
\]
and there exists thus a sequence \(\brk{R_n}_{n \in \Nset}\) in \(\intvo{0}{\infty}\) such that \(R_n \to \infty\) as \(n \to \infty\) and 
\begin{equation}
\label{eq_aboh6biuQuaiz2Yu4za9Kito}
 \lim_{n \to \infty} R_n \int_{\partial B_{R_n}} \abs{\Deriv u}^p = 0.
\end{equation}
For every \(n \in \Nset\), since \(m \ge 2\) the Poincaré inequality applies and we get a point \(a_n \in \partial B_{R_n}\) such that 
\begin{equation}
\label{eq_aeweenai0oochaiGeinaich5}
\begin{split}
  \int_{\partial B_{R_n}} \abs{u\brk{x} - u \brk{a_n}}^p \dif x 
  &\le \frac{\C}{R_n^{m - 1}} \int_{\partial B_{R_n}} \int_{\partial B_R} \abs{u\brk{x} - u \brk{y}}^p \dif x \dif y \\
  &\le \C R_n^p \int_{\partial B_{R_n}} \abs{\Deriv u}^p.
\end{split}
\end{equation}
We deduce from \eqref{eq_aeweenai0oochaiGeinaich5} that
\[
\begin{split}
  &\brk[\bigg]{\frac{1}{R_n^{m - 1 - p}} \int_{\partial B_{R_n}} \abs{\Deriv u}^p + \frac{\abs{u - u \brk{a_n}}^p}{R_n^p}}^{\frac{\floor{p - 1}}{p}}
  \brk[\bigg]{\frac{1}{R_n^{m - 1}}\int_{\partial B_{R_n}} \abs{u - u\brk{a_n}}^p }^{1 - \frac{\floor{p - 1}}{p}}\\
  &\qquad \le \C \frac{R_n}{R_n^{m - p}} \int_{\partial B_{R_n}} \abs{\Deriv u}^p,
\end{split}
\]
where right-hand side goes to \(0\) as \(n \to \infty\) thanks to \eqref{eq_aboh6biuQuaiz2Yu4za9Kito}, since \(R_n \to \infty\) and \(p \le m\).
By \cref{lemma_Luckhaus_W1p}, a suitable diffeomorphism between \(\partial B_R \times \intvo{0}{1}\) and \(B_{2R}\setminus B_R\), and a scaling argument, there exists a function \(w_n \in W^{1, p} \brk{B_{2R_n} \setminus B_{R_n}, \manifold{N}}\) such that 
\begin{align*}
 \tr_{\partial B_{R_n}} w_n &= \tr_{\partial B_{R_n}} u ,
 &\tr_{\partial B_{2 R_n}} w_n &= u \brk{a_n} ,
\end{align*}
and 
\[
 \int_{B_{2R_n} \setminus B_{R_n}} \abs{\Deriv w_n}^p
 \le \C R_n \int_{\partial B_{R_n}} \abs{\Deriv u}^p.
\]
Defining the mapping \(u_n \colon \Rset^m \to \manifold{N}\) for \(x \in \Rset^m\) by 
\[
 u_n \brk{x}\defeq  
 \begin{cases}
  u \brk{x} &\text{if \(x \in B_{R_n}\)},\\
  w_n \brk{x} &\text{if \(x \in B_{2R_n} \setminus B_{R_n}\)},\\
  u \brk{a_n}&\text{if \(x \in \Rset^m \setminus B_{2R_n}\)}.
 \end{cases}
\]
we have \(u_n = u \brk{a_n}\) in \(\Rset^m \setminus B_{2 R_n}\),
\(u_n \to u\) and \(\Deriv u_n \to \Deriv u\) almost everywhere in \(\Rset^m\) as \(n \to \infty\) and 
\[
 \int_{\Rset^m} \abs{\Deriv u_n}^p \le  \int_{\Rset^m} \abs{\Deriv u}^p 
 + \C R_n \int_{\partial B_{R_n}} \abs{\Deriv u}^p,
\]
and the conclusion follows.

\end{proof}

We are now in position to prove \cref{theorem_R_dense_W1p}.

\begin{proof}[Proof of \cref{theorem_R_dense_W1p}]
By \cref{proposition_W1p_compact_approximation} and a diagonal argument, we need to approximate a map \(u \in \smash{\dot{W}}^{1, p} \brk{\Rset^m, \manifold{N}}\) which is constant outside a bounded set. An inspection of the proofs of the strong approximation in \citelist{\cite{Bethuel_1991}*{Th.\thinspace{}2}\cite{Hang_Lin_2003_II}*{Th.\thinspace{}6.1}} shows that the approximating maps will be constant outside a compact set.
\end{proof}

\subsection{Fractional Sobolev mappings}

We now turn to the fractional counterpart of \cref{theorem_R_dense_W1p}.

\begin{theorem}
\label{theorem_R_dense_Wsp}
Let \(\manifold{N}\) be a compact Riemannian manifold. 
If \(s \in \intvo{0}{1}\), if \(p \in \intvr{1}{\infty}\) and if 
\(sp \le m\),
then \(R^1_{m - \floor{sp} - 1, c}\brk{\Rset^m, \manifold{N}}\) is dense in 
\(\dot{W}^{s, p}\brk{\Rset^m, \manifold{N}}\).
\end{theorem}

The proof of \cref{theorem_R_dense_Wsp} will rely on the fractional version of \cref{proposition_W1p_compact_approximation}.

\begin{proposition}
\label{proposition_Wsp_compact_approximation}
Let \(\manifold{N}\) be a compact Riemannian manifold.
If \(s \in \intvo{0}{1}\), if \(p \in \intvr{1}{\infty}\), if \(sp \le m\) and if  \(u \in \smash{\dot{W}}^{s, p}\brk{\Rset^m, \manifold{N}}\), then there exists a sequence \(\brk{u_n}_{n \in \Nset}\) of maps in \(W^{s, p}\brk{\Rset^m, \manifold{N}}\) such that
\begin{enumerate}[label=(\roman*)]
 \item \(\brk{u}_{n \in \Nset}\) converges strongly to \(u\) in \(W^{s, p}\brk{\Rset^m, \manifold{N}}\),
 \item for every \(n \in \Nset\), \(u_n\) is constant outside a compact set of \(\Rset^m\).
\end{enumerate}
\end{proposition}

\begin{remark}
As noted in \cref{remark_compact_approximation_supercritical_W1p} for \(\dot{W}^{1, p}\brk{\Rset^m, \manifold{N}}\), \cref{proposition_Wsp_compact_approximation} still holds for \(sp > m\), with the same construction by inversion \eqref{eq_Geis2eeThohsheisieRe1ahc}.
\end{remark}

The proof of \cref{proposition_Wsp_compact_approximation} relies on a fractional version of the Luckhaus lemma \cref{lemma_Luckhaus_W1p}.

\begin{lemma}
\label{lemma_Luckhaus_Wsp}
Let \(\manifold{M}'\) be a compact \(\brk{m - 1}\)-dimensional Riemannian manifold.
If \(s \in \intvo{0}{1}\) and if  \(p \in \intvr{1}{\infty}\), then
there exists \(\eta \in \intvo{0}{\infty}\) and \(C \in \intvo{0}{\infty}\) such that for every \(u_0\) and \(u_1 \in W^{s, p} \brk{\manifold{M}', \manifold{N}}\)
if either \(sp < 1\) or
\[
  \brk[\bigg]{\sum_{i = 0}^1 \smashoperator{\iint_{\manifold{M}'\times \manifold{M}'}} \frac{\abs{u_i \brk{x} - u_i \brk{y}}^p}{\abs{x - y}^{m+ sp}} \dif x \dif y + \int_{\manifold{M}'} \abs{u_1 - u_0}^p}^{\frac{\ell}{p}}\brk[\bigg]{\int_{\manifold{M}'} \abs{u_1 - u_0}^p}^{1 - \frac{\ell}{p}} \le \eta, 
\]
with \(\ell = \min \brk{m - 1, \floor{sp - 1}}\),
then there exists \(U \in W^{s, p} \brk{\manifold{M}' \times \intvo{0}{1}, \manifold{N}}\) such that 
\begin{enumerate}[label=(\roman*)]
 \item for every \(t \in \intvc{0}{\frac{1}{3}}\), \(U\brk{t, \cdot} = u_0\),
 \item for every \(t \in \intvo{\frac{2}{3}}{1}\)
 \(U \brk{t, \cdot}= u_1\),
 \item one has 
 \[
 \begin{split}
  &\iint_{\brk{\manifold{M}'\times \intvo{0}{1}}^2} \frac{\abs{U \brk{x} - U \brk{y}}^p}{\abs{x - y}^{m+ sp}} \dif x \dif y\\
  &\qquad \le \sum_{i = 0}^1\iint_{\brk{\manifold{M}'}^2} \frac{\abs{u_i \brk{x} - u_i \brk{y}}^p}{\abs{x - y}^{m - 1 + sp}} \dif x \dif y +
   \int_{\manifold{M}'} \abs{u_1 - u_0}^p.
  \end{split}
 \]
\end{enumerate}

\end{lemma}

\begin{proof}[Proof of \cref{lemma_Luckhaus_Wsp}]
If \(sp \ge 1\), we follow the proof of \cref{lemma_Luckhaus_W1p}, relying on suitable Fubini properties, Gagliardo-Nirenberg interpolation inequalities \cite{Brezis_Mironescu_2018} and homogeneous extensions \cite{Brezis_Mironescu_Shafrir_2019}.

When \(sp < 1\), we simply take
\[
 U \brk{t, x}\defeq
 \begin{cases}
   u_0 \brk{x} & \text{if \(t \le \frac{1}{2}\),}\\
   u_1 \brk{x}& \text{if \(t > \frac{1}{2}\),}
 \end{cases}
\]
and conclude with a suitable gluing property for Sobolev mappings.
\end{proof}

\begin{proof}%
[Proof of \cref{proposition_Wsp_compact_approximation}]%
\resetconstant%
Since
\[
 \int_{0}^\infty \smashoperator[r]{\iint_{\partial B_R \times \Rset^m}} \frac{\abs{u \brk{x} - u \brk{y}}^p}{\abs{x - y}^{m + sp}} \dif x \dif y \dif R
 =
 \smashoperator{\iint_{\Rset^m \times \Rset^m}} \frac{\abs{u \brk{x} - u \brk{y}}^p}{\abs{x - y}^{m + sp}} \dif x \dif y < \infty,
\]
there exists a sequence \(\brk{R_n}_{n \in \Nset}\) such that
\begin{align*}
\lim_{n \to \infty} R_n &= \infty&
&\text{and}&
  \lim_{n \to \infty} R_n \smashoperator{\iint_{\partial B_{R_n} \times \Rset^m}}
 \frac{\abs{u \brk{x} - u \brk{y}}^p}{\abs{x - y}^{m + sp}}\dif x \dif y &= 0.
\end{align*}
Defining \(a_n \in \partial B_{R_n}\) such that
\[
\int_{\Rset^m} \frac{\abs{u \brk{x} - u \brk{a_n}}^p}{\abs{x - a_n}^{m + sp}}\dif x
\le \frac{\C}{R_n^{m - 1}}
\smashoperator{\iint_{\partial B_{R_n} \times \Rset^m}}
 \frac{\abs{u \brk{x} - u \brk{y}}^p}{\abs{x - y}^{m + sp}}\dif x \dif y.
\]
we observe that
\[
  \smashoperator[r]{\iint_{\partial B_{R_n} \times \partial B_{R_n}}}
 \frac{\abs{u \brk{x} - u \brk{y}}^p}{\abs{x - y}^{m - 1 + sp}} \dif x \dif y
 \le
 \C \smashoperator{\iint_{\partial B_{R_n} \times \Rset^m}}
 \frac{\abs{u \brk{x} - u \brk{y}}^p}{\abs{x - y}^{m + sp}} \dif x \dif y
\]
and
\[
 \int_{\partial B_{R_n}} \abs{u \brk{x} - u \brk{a_n}}^p\dif x
\le \C R_n^{sp - 1}
\smashoperator{\iint_{\partial B_{R_n} \times \Rset^m}}
 \frac{\abs{u \brk{x} - u \brk{y}}^p}{\abs{x - y}^{m + sp}}\dif x \dif y
\]

Defining as previously the mapping \(w_n \colon \Rset^m \setminus B_{R_n} \to \manifold{N}\), thanks to \cref{lemma_Luckhaus_Wsp} instead of \cref{lemma_Luckhaus_W1p}, and the mapping \(u_n\colon \Rset^m \to \manifold{N}\), we get the estimates
\begin{gather*}
\smashoperator[r]{\iint_{\brk{B_{R_n}}^2}}  \frac{\abs{u_n \brk{x} - u_n \brk{y}}^p}{\abs{x - y}^{m + sp}}\dif x \dif y
 \le \smashoperator{\iint_{\Rset^m \times \Rset^m}} \frac{\abs{u \brk{x} - u \brk{y}}^p}{\abs{x - y}^{m + sp}}\dif x \dif y,\\
\begin{split}
  &\smashoperator[r]{\iint_{\brk{\Rset^m \setminus B_{R_n/2}}^2}} \frac{\abs{u_n \brk{x} - u_n \brk{y}}^p}{\abs{x - y}^{m + sp}}\dif x \dif y\\
   &\qquad \qquad \le \C \brk[\bigg]{\smashoperator{\iint_{\brk{B_{R_n} \setminus B_{R_n/2}}^2}} \frac{\abs{u \brk{x} - u \brk{y}}^p}{\abs{x - y}^{m + sp}}\dif x \dif y  +
   \smashoperator{\iint_{\brk{\Rset^m \setminus B_{R_n}}^2}} \frac{\abs{w_n \brk{x} - w_n \brk{y}}^p}{\abs{x - y}^{m + sp}}\dif x \dif y}\\
  &\qquad \qquad \le \C \brk[\bigg]{\smashoperator{\iint_{\brk{B_{R_n} \setminus B_{R_n/2}}^2}} \frac{\abs{u \brk{x} - u \brk{y}}^p}{\abs{x - y}^{m + sp}}\dif x \dif y  + R_n \smashoperator{\iint_{\partial B_{R_n} \times \Rset^m}}
 \frac{\abs{u \brk{x} - u \brk{y}}^p}{\abs{x - y}^{m + sp}}\dif x \dif y},
\end{split}\\
\begin{split}
  \smashoperator[r]{\iint_{B_{R_n/2} \times \brk{\Rset^m \setminus B_{2 R_n}}}} \frac{\abs{u_n \brk{x} - u_n \brk{y}}^p}{\abs{x - y}^{m + sp}}\dif x \dif y
  & =
  \smashoperator{\iint_{B_{R_n/2} \times \brk{\Rset^m \setminus B_{R_n}}}} \frac{\abs{u \brk{x} - u \brk{a_n}}^p}{\abs{x - y}^{m + sp}}\dif x \dif y\\
  &\le \C \int_{B_{R_n/2}} \frac{\abs{u \brk{x} - u \brk{a_n}}^p}{R_n^{sp}}\dif x
\end{split}
\intertext{and finally}
\begin{split}
  &\smashoperator[r]{\iint_{B_{R_n/2} \times \brk{B_{2 R_n} \setminus B_{R_n}}}} \frac{\abs{u_n \brk{x} - u_n \brk{y}}^p}{\abs{x - y}^{m + sp}}\dif x \dif y\\
  & \qquad \le\C
  \int_{B_{R_n/2}} \frac{\abs{u \brk{x} - u \brk{a_n}}^p}{R^{sp}}\dif x
  + \C \int_{B_{2R_n} \setminus B_{R_n}} \frac{\abs{w_n \brk{x} - u \brk{a_n}}^p}{R^{sp}} \dif x\\
  &\qquad \le \C R_n^{m} \int_{\Rset^m} \frac{\abs{u \brk{x} - u \brk{a_n}}^p}{\abs{x - a_n}^{m + sp}}\dif x
\end{split}
\end{gather*}
from which we deduce, that
\[
 \limsup_{n \to \infty} \smashoperator{\iint_{\Rset^m \times \Rset^m}} \frac{\abs{u_n \brk{x} - u_n \brk{y}}^p}{\abs{x - y}^{m + sp}}\dif x \dif y
 \le \smashoperator{\iint_{\Rset^m \times \Rset^m}}  \frac{\abs{u\brk{x} - u \brk{y}}^p}{\abs{x - y}^{m + sp}}\dif x \dif y,
\]
which implies the announced strong convergence since \(u_n \to u\) almost everywhere on \(\Rset^m\) as \(n \to \infty\) thanks to Scheffé’s lemma.
\end{proof}

\begin{proof}[Proof of \cref{theorem_R_dense_Wsp}]
This follows from \cref{proposition_Wsp_compact_approximation} and noting that Brezis and Mironescu’s proof on the strong approximation of Sobolev mappings in fractional Sobolev spaces is based on a local construction by homogeneous extension \cite{Brezis_Mironescu_2015}*{Lemma 5.1} which ensures that a map constant outside a ball is approximated by maps that are constant outside a (slightly larger) ball.
\end{proof}

\end{document}